\documentclass{memo-l}
\usepackage{amsthm,amsmath,amssymb,bbm}
\usepackage{epsfig}

\newcommand{\R}{\mathbb{R}}
\newcommand{\N}{\mathbb{N}}
\newcommand{\E}{\mathbb{E}}
\renewcommand{\P}{\mathbb{P}}
\newcommand{\tr}{\operatorname{trace}}
\newcommand{\citationand}{\&}

\newtheorem{theorem}{Theorem}[chapter]
\newtheorem{lemma}[theorem]{Lemma}
\newtheorem{prop}[theorem]{Proposition}
\newtheorem{cor}[theorem]{Corollary}

\theoremstyle{definition}
\newtheorem{definition}[theorem]{Definition}

\theoremstyle{remark}

\numberwithin{section}{chapter}
\numberwithin{equation}{chapter}

\begin{document}
\frontmatter
\title{Numerical approximations
of stochastic
\linebreak
differential 
equations with non-globally
\linebreak
Lipschitz continuous 
coefficients}

\author{Martin Hutzenthaler}
\address{
LMU Biozentrum,
Department Biologie II,
University of Munich (LMU),
82152~Planegg-Martinsried, Germany
}
\email{hutzenthaler$\,$(at)$\,$bio.lmu.de}
\thanks{
}

\author{Arnulf Jentzen}
\address{
Seminar for Applied Mathematics,
Swiss Federal Institute of Technology
Zurich, 
8092 Zurich, Switzerland;
Program in Applied 
and Computational Mathematics, 
Princeton University,
Princeton, NJ 08544-1000, USA
}
\email{arnulf.jentzen$\,$(at)$\,$sam.math.ethz.ch
}
\thanks{
This work has been partially
supported
by the 
research project
``Numerical solutions of stochastic
differential equations with
non-globally Lipschitz
continuous coefficients''
and by the research project
``Numerical approximation of stochastic
differential equations with
non-globally Lipschitz
continuous coefficients''
both funded by the German Research
Foundation.
}

\date{}

\begin{abstract}
Many 
stochastic differential equations 
(SDEs) in the literature have a 
superlinearly growing nonlinearity in 
their drift or diffusion coefficient.
Unfortunately, moments of 
the computationally efficient 
Euler-Maruyama approximation method
diverge for these SDEs
in finite time.
This article 
develops a 
general theory based on rare
events for studying
integrability properties such as moment bounds
for discrete-time 
stochastic processes.
Using this approach, 
we establish moment bounds
for fully and partially 
drift-implicit Euler methods
and for a class of new explicit
approximation methods
which require only a 
few more arithmetical 
operations 
than the Euler-Maruyama method.
These moment bounds
are then used to prove strong
convergence of the proposed
schemes.
Finally, we illustrate our
results for several SDEs
from finance,
physics,
biology
and chemistry.
\end{abstract}

\maketitle

\tableofcontents

\mainmatter

\chapter{Introduction}
\label{sec:intro}

This article investigates 
integrability and convergence
properties of 
numerical approximation
processes for stochastic differential
equations (SDEs).
In order to illustrate one of our
main results,
the following 
general setting is considered
in this
introductory chapter.
Let $ T \in (0,\infty) $,
$ d, m \in \mathbb{N}:=\{1,2,\ldots\} $,
let 
$ 
  \left( \Omega, \mathcal{F},
  \mathbb{P} \right) 
$
be a probability space 
with a normal filtration 
$ 
  ( \mathcal{F}_t )_{ t \in [0,T] }
$,
let
$
  W \colon [0,T] \times \Omega
  \rightarrow \mathbb{R}^m
$
be a
standard
$ 
  ( \mathcal{F}_t )_{ t \in [0,T] } 
$-Brownian motion,
let $ D \subset \R^d $
be an open set,
let
$
  \mu
  = ( \mu_1, \dots, \mu_d )
  \colon D
  \rightarrow \mathbb{R}^d
$
and
$
  \sigma 
  =
  ( \sigma_{ i, j } )_{
    i \in \{ 1, 2, \dots, d \} ,
    j \in \{ 1, 2, \dots, m \}
  } 
  \colon D
  \rightarrow 
  \mathbb{R}^{ d \times m } 
$
be 
locally Lipschitz continuous
functions
and let
$
  X \colon [0,T] \times \Omega
  \rightarrow D
$
be an
$ ( \mathcal{F}_t )_{ t \in [0,T] }
$-adapted stochastic process
with continuous sample
paths 
satisfying the SDE
\begin{equation}
\label{eq:SDE.intro}
   X_t = X_0
   + 
   \int_0^t \mu(X_s)\,ds
    +
    \int_0^t \sigma(X_s)\,dW_s
\end{equation}
$\P$-almost surely for all 
$ t \in [0,T] $.
Here $ \mu $ is the infinitesimal 
mean and $ \sigma\cdot\sigma^{ * } $
is the infinitesimal covariance matrix
of the solution process
$ X $ of the 
SDE~\eqref{eq:SDE.intro}.
To guarantee finiteness of 
some moments of
the SDE~\eqref{eq:SDE.intro},
we assume existence of a 
Lyapunov-type function.
More precisely, 
let $ q \in (0,\infty) $,
$ \kappa \in \R $
be real numbers 
and let
$ 
  V \colon D \to [1,\infty)
$
be a twice continuously differentiable
function with
$ 
  \mathbb{E}[ V( X_0 ) ]
  < \infty 
$
and with
$
  V(x) \geq
  \| x \|^q 
$ 
and
\begin{equation}
\label{eq:Lyapunov.type.condition.intro}
  \sum_{ i = 1 }^d
  \left(
    \frac{
      \partial 
      V
    }{
      \partial x_i
    } 
  \right)\!( x )
  \cdot
  \mu_{ i }(x)
  +
  \frac{ 1 }{ 2 }
  \sum_{ i, j = 1 }^{ d }
  \sum_{ k = 1 }^{ m }
  \left(
    \frac{
      \partial^2 
      V
    }{
      \partial x_i
      \partial x_j
    } 
  \right)\!( x )
  \cdot
  \sigma_{ i, k }(x)
  \cdot
  \sigma_{ j, k }(x)
\leq
  \kappa \cdot V(x)
\end{equation}
for all $ x \in D $.
These assumptions ensure
\begin{equation}
  \E\big[
    V( X_t )
  \big]
  \leq
  e^{ \kappa t } \cdot
  \E\big[ V( X_0 ) \big]
\end{equation}
for all $ t \in [0,T] $
and, therefore, 
finiteness of the $ q $-th
absolute moments of the
solution process $ X_t $, $ t \in [0,T] $,
of the SDE~\eqref{eq:SDE.intro},
i.e.,
$ 
  \sup_{ t \in [0,T] }
  \mathbb{E}\big[ \| X_t \|^q \big]
  < \infty 
$. 
Note that in this setting both
the drift coefficient $ \mu $
and the diffusion coefficient
$ \sigma $ of the
SDE~\eqref{eq:SDE.intro}
may grow superlinearly
and are, in particular, not
assumed to be globally
Lipschitz continuous.
Our main goal
in this introduction
is to construct 
and to analyze
numerical approximation processes
that converge strongly to the exact
solution of the SDE~\eqref{eq:SDE.intro}.
The standard literature
in computational stochastics
(see, for instance, 
Kloeden \citationand\ Platen~\cite{kp92}
and Milstein~\cite{m95}) 
concentrates on SDEs
with globally Lipschitz continuous
coefficients and can therefore
not be applied here.
Strong numerical approximations
of the SDE~\eqref{eq:SDE.intro}
are of particular interest for 
the computation
of statistical quantities
of the solution process of 
the SDE~\eqref{eq:SDE.intro}
through computationally efficient 
multilevel Monte Carlo methods
(see Giles~\cite{g08b}, 
Heinrich~\cite{h01}
and, e.g., in
Creutzig et al.~\cite{cdmr09},
Hickernell et al.~\cite{Hickernelletal2010},
Barth, Lang \& Schwab~\cite{bls11}
and in the references therein
for further recent results
on multilevel Monte Carlo methods).

Several SDEs from the literature 
satisfy the above setting
(see Sections~\ref{sec:vanderpol}--\ref{sec:langevin}
below).
For instance, the function
$ 
  V(x) = \left( 1 + \| x \|^2 \right)^r
$,
$ x \in D $,
for an arbitrary $ r \in (0,\infty) $
serves as a Lyapunov-type function
for the stochastic
van der Pol 
oscillator~\eqref{eq:ex_Vanderpol},
for the stochastic Lorenz equation~\eqref{eq:Lorenz},
for the Cox-Ingersoll-Ross process~\eqref{eq:ex_CIR}
and for the simplified 
Ait-Sahalia interest
rate 
model~\eqref{eq:ex_AitSahalia}
but not for the
stochastic Duffing-van der
Pol 
oscillator~\eqref{eq:ex_Duffing},
not for the stochastic
Brusselator~\eqref{eq:ex_Brusselator},
not for the stochastic SIR 
model~\eqref{eq:ex_SIR},
not for the Lotka-Volterra predator
prey model~\eqref{eq:ex_PredatorPrey}
and, in general, also not
for the Langevin 
equation~\eqref{eq:ex_Langevin}.
The function
$ 
  V(x) = \left( 
    1 + (x_1)^4 + 2 (x_2)^2 
  \right)^r
$,
$ x = (x_1, x_2) \in D = \R^2 $,
for an arbitrary $ r \in (0,\infty) $
is a Lyapunov-type function
for the stochastic Duffing-van
der Pol 
oscillator~\eqref{eq:ex_Duffing}.
For the stochastic SIR model~\eqref{eq:ex_SIR},
the function
$ 
  V(x) = 
  \left( 1 + ( x_1 + x_2 )^2 
  + (x_3)^2 \right)^r 
$,
$
  x = (x_1, x_2, x_3) \in 
  D = (0,\infty)^3 
$,
for an arbitrary $ r \in (0,\infty) $
serves as a Lyapunov-type function
and for the
stochastic Lotka-Volterra
system~\eqref{eq:ex_Volterra},
the function
$ 
  V(x) = 
  \left( 1 + 
    ( v_1 x_1 + \dots v_d x_d )^2 
  \right)^r 
$,
$ x \in D = (0,\infty)^d $,
for an arbitrary $ r \in (0,\infty) $
and an appropriate 
$ v = (v_1, \dots, v_d) \in \R^d $
is a Lyapunov-type function.
More details on the examples
can be found in Chapter~\ref{sec:examples}.

The standard method
for approximating SDEs
with globally Lipschitz continuous
coefficients is the Euler-Maruyama
method.
Unfortunately,
the Euler-Maruyama method often fails
to converge strongly to the exact
solution of nonlinear SDEs of the
form~\eqref{eq:SDE.intro};
see \cite{hjk11}.
Indeed, if at least one of the
coefficients of the SDE 
grows superlinearly, then 
the Euler-Maruyama scheme
diverges in the strong sense. 
More precisely,
let
$
  Z^N
  \colon
  \{ 0, 1, \ldots, N\}
  \times\Omega\to\R^d
$, 
$ N \in \N $,
be Euler-Maruyama approximations
for the 
SDE~\eqref{eq:SDE.intro}
defined recursively through 
$ Z^N_0 := X_0 $ 
and
\begin{equation}
\label{eq:Euler.intro}
  Z_{n+1}^N
  :=
  Z_n^N
  +
  \bar{\mu}( Z_n^N )
  \tfrac{T}{N}
  +
  \bar{\sigma}( Z_n^N )
  \big(
    W_{ \frac{ ( n+1 ) T }{ N } }
    -
    W_{ \frac{ n T }{ N } }
  \big)
\end{equation}
for all $ n \in \{ 0, 1, \ldots, N-1 \} $ 
and all $ N \in \mathbb{N} $.
Here
$ 
  \bar{\mu} \colon \R^d
  \to \R^d
$
and
$
  \bar{\sigma} \colon \R^d
  \to \R^{ d \times m }
$
are extensions of $ \mu $ and $ \sigma $
given by
$ \bar{\mu}(x) = 0 $,
$ \bar{\sigma}(x) = 0 $
for all $ x \in D^c $ 
and by
$ \bar{\mu}(x) = \mu(x) $,
$ \bar{\sigma}(x) = \sigma(x) $
for all $ x \in D $ respectively
(see also Sections~\ref{sec:consistent}
and \ref{sec:schemes} for more general extensions).
Theorem~2.1 of \cite{hjk11b} 
(which 
generalizes Theorem 2.1 of \cite{hjk11})
then implies in the case 
$ d = m = 1 $ that  
if there exists a real number 
$ \varepsilon \in (0,\infty) $
such that 
$
  |\bar{\mu}(x)| + 
  |\bar{\sigma}(x)|
  \geq 
  \varepsilon |x|^{ (1 + \varepsilon) }
$ 
for all 
$
  |x| \geq 1 / \varepsilon
$
and if 
$
  \P[ \sigma(X_0) \neq 0 ] > 0
$, 
then 
$
  \lim_{ N \to \infty }
  \E\big[
    |Y_N^N|^r
  \big] = \infty
$ 
for all $ r \in (0,\infty) $
and therefore
$
  \lim_{ N \rightarrow \infty }
  \mathbb{E}\big[
    | X_T - Y^N_N |^r
  \big]
  = \infty
$
for all $ r \in (0,q] $
(see also Sections~4 and 5
in \cite{hjk11b} for divergence results 
for the corresponding
multilevel Monte Carlo Euler method).
Due to these deficiencies of the 
Euler-Maruyama method, 
we look for 
numerical 
approximation methods whose 
computational cost is 
close to that of
the Euler-Maruyama method
and which converge strongly
even in the case of SDEs
with superlinearly growing
coefficients.

There are a number
of strong convergence results
for temporal numerical
approximations 
of SDEs of the 
form~\eqref{eq:SDE.intro}
with possibly superlinearly growing
coefficients
in the literature.
Many of these results 
assume beside other
assumptions
that the drift coefficient $ \mu $
of the SDE~\eqref{eq:SDE.intro} 
is globally one-sided
Lipschitz continuous
and also prove rates of convergence 
in that case.
In particular, 
if the drift coefficient 
is globally 
one-sided Lipschitz
continuous and 
if the diffusion coefficient 
is globally Lipschitz
continuous beside
other assumptions,
then strong convergence of 
the fully drift-implicit Euler 
method
follows
from Theorem 2.4 in 
Hu~\cite{h96}
and
from Theorem 5.3 in 
Higham, Mao \citationand~Stuart~\cite{hms02},
strong convergence of the split-step backward Euler method
follows from Theorem 3.3 in Higham, 
Mao \citationand~Stuart~\cite{hms02},
strong convergence of a drift-tamed 
Euler-Maruyama method follows 
from Theorem~1.1 in 
\cite{hjk10b}
and
strong convergence of a drift-tamed Milstein scheme follows
from Theorem 3.2 in Gan 
\citationand~Wang~\cite{gw11}.
Theorem 2 and Theorem 3 in 
Higham \citationand~Kloeden~\cite{hk07}
generalize Theorem 3.3 and Theorem 5.3 in Higham, Mao \citationand~Stuart~\cite{hms02}
to SDEs with Poisson-driven jumps.
In addition,
Theorem 6.2 in Szpruch et al.~\cite{hmps11}
establishes strong convergence of the 
fully drift-implicit 
Euler method
of a one-dimensional 
Ait-Sahalia-type interest rate model
having a superlinearly growing diffusion coefficient $\sigma$
and a globally one-sided Lipschitz continuous drift coefficient $\mu$ which
is unbounded near $0$.
Moreover, Theorem~4.4 
in 
Mao \citationand\ 
Szpruch~\cite{MaoSzpruch2012pre} 
generalizes this result to 
a class of SDEs
which have globally one-sided Lipschitz 
continuous drift coefficients 
and in which the function
$
  V(x) = 1 + \| x \|^2
$, 
$
  x \in D
$,
is a Lyapunov-type function
(see also 
Mao \citationand\ Szpruch~\cite{MaoSzpruch2012B} for 
related results but with rates 
of convergence).
A similar method is used in
Proposition~3.3 in 
Dereich, Neuenkirch 
\citationand~Szpruch~\cite{DereichNeuenkirchSzpruch2012}
to obtain strong convergence 
of a drift-implicit Euler 
method for a class of
Bessel type processes.
Moreover,
Gy\"ongy \citationand~Millet
establish
in
Theorem~2.10
in \cite{GyoengyMillet2005}
strong convergence of
implicit numerical approximation
processes
for a class of possibly infinite
dimensional SDEs whose 
drift $ \mu $ and diffusion
$ \sigma $ 
satisfy a suitable 
one-sided Lipschtz condition
(see Assumption~(C1) 
in \cite{GyoengyMillet2005} 
for details).
Strong convergence of
temporal numerical 
approximations for 
two-dimensional stochastic
Navier-Stokes equations
is obtained in
Theorem~7.1 in Brze\'zniak, 
Carelli \citationand\ Prohl~\cite{BrzezniakCarelliProhl2010}.
In all of the above 
mentioned results 
from the literature,
the function
$
  V(x) = 1 + \| x \|^2
$, 
$
  x \in D
$,
is a Lyapunov-type function
of the considered
SDE.
A result on more 
general Lyapunov-type functions is the framework in
Schurz~\cite{Schurz2006}
which assumes general abstract conditions
on the numerical approximations.
The applicability of this framework
is demonstrated in the case
of SDEs which have globally one-sided
Lipschitz continuous drift 
coefficients and in 
which the function
$
  V(x) = 1 + \| x \|^2
$, 
$
  x \in D
$,
is a Lyapunov-type function;
see \cite{Schurz2003,Schurz2005,Schurz2006}.
To the best 
of our knowledge,
no strong numerical approximation 
results are known
for the stochastic van der Pol
oscillator~\eqref{eq:ex_Vanderpol},
for the stochastic Duffing-van der Pol oscillator~\eqref{eq:ex_Duffing},
for the stochastic Lorenz equation~\eqref{eq:Lorenz},
for the stochastic Brusselator~\eqref{eq:ex_Brusselator},
for the stochastic SIR model~\eqref{eq:ex_SIR},
for the experimental
psychology
model~\eqref{eq:experimental.psychology}
and for the Lotka-Volterra predator-prey model~\eqref{eq:ex_Volterra}.

In this article, the following 
increment-tamed Euler-Maruyama 
scheme is proposed to 
approximate the solution process
of the SDE~\eqref{eq:SDE.intro}
in the strong sense. 
Let 
$
  Y^N \colon \{0,1 \ldots,N\}
  \times\Omega\to\R^d
$, 
$ N \in \N $,
be numerical approximation
processes
defined through
$ Y_0^N := X_0 $
and
\begin{equation}  \label{eq:increment.tamed.Euler.intro}
  Y^N_{ n+1 }
:=
  Y^N_{ 
    n
  }
+
  \frac{
    \bar{\mu}\big( 
      Y^N_{  n } 
    \big) 
    \frac{ T }{ N }
    + 
        \bar{\sigma}\big( 
          Y^N_{  n } 
        \big) 
        \big(
          W_{ \frac{ (n + 1) T }{ N }  
          } 
          -
          W_{ \frac{ n T }{ N } } 
        \big)
  }{
    \max\!\big( 1, 
      \frac{ T }{ N }
      \| 
        \bar{\mu}( 
          Y^N_{  n } 
        ) 
        \frac{ T }{ N }
        + 
        \bar{\sigma}( 
          Y^N_{  n } 
        ) 
        \big(
          W_{ \frac{ ( n + 1 ) T 
          }{ N } } 
          -
          W_{ \frac{ n T }{ N } } 
        \big)
      \| 
    \big)
  }
\end{equation}
for all 
$ 
  n \in \{ 0, 1, \dots, N - 1 \} 
$
and all
$ N \in \mathbb{N} $.
Note that the computation 
of \eqref{eq:increment.tamed.Euler.intro}
requires only a few additional arithmetical operations when
compared to the computation
of the Euler-Maruyama 
approximations~\eqref{eq:Euler.intro}.
Moreover, we emphasize that the 
scheme~\eqref{eq:increment.tamed.Euler.intro}
is a special case of a more general
class of suitable 
tamed 
schemes proposed in
Subsection~\ref{sec:taming}
below.
Next let
$
  \bar{Y}^N\colon[0,T]\times\Omega\to\R^d
$, 
$ N \in \N $,
be linearly interpolated 
continuous-time versions of
\eqref{eq:increment.tamed.Euler.intro}
defined
through
$
  \bar{Y}^N_t
:=
  \big(
    n + 1 - \tfrac{tN}{T}
  \big)
  Y_n^N
  +
  \big(
    \tfrac{tN}{T}-n
  \big)
  Y_{ n + 1 }^N
$
for all
$
  t \in 
  [
    n T / N,
    (n + 1) T / N 
  ]
$,
$ 
  n \in \{ 0, 1, \dots, N - 1 \} 
$
and all
$ N \in \mathbb{N} $.
For proving strong convergence
of the numerical
approximation processes
$ \bar{Y}^N $, $ N \in \N $,
to the exact solution process
$ X $ of the 
SDE~\eqref{eq:SDE.intro}, 
we additionally assume
that
$ 
  \E\big[
    \| X_0 \|^r
  \big] 
  < \infty
$
for all $ r \in [0,\infty) $
and that 
there exist real numbers
$ 
  \gamma_0, \gamma_1, c 
  \in [0,\infty)
$,
$ p \in [3,\infty) $
and a three times 
continuously
differentiable extension 
$ 
  \bar{V}
  \colon \R^d \rightarrow [1,\infty) 
$
of 
$ 
  V \colon D \rightarrow [1,\infty)
$
such that
$ 
  \bar{V}(x) 
  \geq
  \| x \|^q
$,
$
  \|
    \bar{V}^{(i)}( x )
  \|_{
    L^{ (i) }( 
      \mathbb{R}^d, \mathbb{R} 
    )
  }
  \leq
  c \,
  |
    \bar{V}(x)
  |^{ 
    \left[
      1 - \frac{ i }{ p }
    \right]
  }
$
and
\begin{equation}
  \| \bar{\mu}(x) \|                  
\leq 
  c \, 
  | \bar{V}(x) |^{ 
    \left[ 
      \frac{ \gamma_0 + 1 }{ p } 
    \right]
  }
\qquad
  \text{and}
\qquad
  \|\bar{\sigma}(x)
  \|_{ L(\R^m,\R^d) }
  \leq 
  c \, |\bar{V}(x)|^{ 
    \left[ 
      \frac{ \gamma_1 + 2 }{ 2 p } 
    \right] 
  } 
\end{equation}
for all 
$ x \in \mathbb{R}^d $ and all 
$ i \in \{1,2,3 \} $.
These assumptions are 
satisfied in all of the example SDEs 
from 
Sections~\ref{sec:vanderpol}--\ref{sec:lotka}.
In the case of 
the squared
volatility process~\eqref{eq:squaredvola}
in Section~\ref{sec:vola}
and in case of 
the Langevin
equation~\eqref{eq:ex_Langevin}
in Section~\ref{sec:langevin},
these assumptions are also satisfied
if the model parameters satisfy
suitable regularity conditions
(see Sections~\ref{sec:vola}
and \ref{sec:langevin} for details).
Under these assumptions,
Theorem~\ref{thm:strongC} 
below shows that
\begin{equation}  \label{eq:strong.convergence.intro}
  \lim_{ N \to \infty }
  \sup_{ t \in [0,T] }
  \E\Big[
    \big\|
     X_t - \bar{Y}_t^N
    \big\|^r
  \Big]=0
\end{equation}
for all $ r \in (0,q) $
satisfying
$
  r < 
  \frac{ p }{
    2 \gamma_1 +
    4 \max( \gamma_0, \gamma_1, 1 / 2 )
  }
  - 
  \frac{ 1 }{ 2 }
$.
Theorem~\ref{thm:strongC} 
thereby proves
strong convergence of 
the increment-tamed 
Euler-Maruyama 
method~\eqref{eq:increment.tamed.Euler.intro}
for all example SDEs 
from Sections~\ref{sec:vanderpol}--\ref{sec:lotka}
and in parts also for
the example SDEs
from
Sections~\ref{sec:vola}--\ref{sec:langevin}.
Moreover, using a whole family 
of Lyapunov-type functions, 
we will deduce from Theorem~\ref{thm:strongC} 
for most of the examples of
Chapter~\ref{sec:examples}
that strong
$
  L^r
$-convergence~\eqref{eq:strong.convergence.intro}
holds for all $ r \in (0,\infty) $
(see Corollary~\ref{cor:CT1} below
for details).
To the best of our knowledge,
Theorem~\ref{thm:strongC} 
is
the first result in the literature
that proves strong convergence 
of a numerical approximation method
for the stochastic van der Pol
oscillator~\eqref{eq:ex_Vanderpol},
for the stochastic Duffing-van der Pol 
oscillator~\eqref{eq:ex_Duffing},
for the stochastic Lorenz 
equation~\eqref{eq:Lorenz},
for the stochastic
Brusselator~\eqref{eq:ex_Brusselator},
for the stochastic SIR model~\eqref{eq:ex_SIR},
for the experimental
psychology
model~\eqref{eq:experimental.psychology}
and for the stochastic 
Lotka-Volterra 
predator-prey 
model~\eqref{eq:ex_PredatorPrey}.
%
%
%

Theorem~\ref{thm:strongC} proves
the strong convergence~\eqref{eq:strong.convergence.intro}
in a quite general setting.
One may ask whether 
it is also possible to establish
a strong convergence 
rate in this setting.
There is a strong hint that this is 
not possible in this general setting.
More precisely, 
Theorem~1.2 in
Hairer et 
al.~\cite{hhj12} shows 
that in this setting
there exist SDEs with smooth 
and globally bounded 
coefficients whose solution processes are 
nowhere 
locally H\"{o}lder continuous in 
the strong mean square sense 
with respect to the initial values.
This instability suggests 
that 
there exist SDEs 
with smooth and globally
bounded coefficients
for which
there exist no one-step 
numerical approximation 
processes 
which converge in the strong sense
with a 
convergence rate.
In addition,
Theorem~1.3 in
Hairer et 
al.~\cite{hhj12} proves
that there exist
SDEs
with smooth and globally
bounded coefficients
to which the Euler-Maruyama
scheme (see \eqref{eq:Euler.intro})
and other schemes
such as the Milstein
scheme
converge in the strong
mean square sense 
without
any arbitrarily small
positive rate of convergence.
It remains an open question
which conditions on the coefficients 
$ \mu $ (more general than globally 
one-sided Lipschitz continuous)
and $ \sigma $ 
of the SDE~\eqref{eq:SDE.intro}
are sufficient to ensure
strong convergence of 
appropriate one-step numerical 
approximation processes
to the exact solution of
the SDE~\eqref{eq:SDE.intro}
with the standard strong
convergence order 
$ 1/2 $
at least.

Finally, we summarize a few more 
results of this article.
In Chapter~\ref{sec:integrability properties},
we establish uniform moment bounds
of approximation processes for SDEs
which are typically the 
first step in proving 
strong and numerically weak
convergence results.
In particular, Corollary~\ref{cor:apriori_increment}
in Subsection~\ref{sec:momentbounds2}
proves uniform moment bounds for the
increment-tamed Euler-Maruyama scheme~\eqref{eq:increment.tamed.Euler.intro}.
Moreover,
Corollary~\ref{c:Lyapunov.implicit.Euler}
in Subsection~\ref{sec:Full.drift.implicit.approximation.schemes}
yields uniform moment bounds
for the fully 
drift-implicit Euler scheme
and Lemma~\ref{l:more.Lyapunov.implicit.Euler}
in Subsection~\ref{sec:Partial.drift.implicit.approximation.schemes}
establishes uniform moment bounds
for partially drift-implicit approximation schemes.
These results on uniform moment bounds are
applications of a general theory 
which
we develop in 
Section~\ref{sec:general}.
In this theory 
(see Propositions~\ref{p:stability} 
and \ref{p:PG}
and Corollaries~\ref{cor:stability2}, \ref{cor:stability} and \ref{cor:FV})
we assume a 
Lyapunov-type inequality
to be satisfied
by the approximation processes
on large subevents of the 
probability space, i.e., on
complements of {\it rare events};
see inequality~\eqref{eq:semi.V.stable2}
in Corollary~\ref{cor:stability2}.
One of our main results
(Theorem~\ref{thm:SVstability} in Subsection~\ref{sec:SVEM})
establishes
this Lyapunov-type condition
for the
Euler-Maruyama 
approximations~\eqref{eq:Euler.intro}.
More precisely, whereas
the Euler-Maruyama approximations
often do not satisfy a Lyapunov-type inequality
on events of probability one
in the case of superlinearly growing coefficients
according to 
Corollary~\ref{cor:disprove}
in Subsection~\ref{sec:SVEM},
the Euler-Maruyama approximations
do satisfy the Lyapunov-type inequality~\eqref{eq:semi.V.stable2}
on large subevents of the probability 
space
according to
Theorem~\ref{thm:SVstability} in Subsection~\ref{sec:SVEM}.
This integrability result on the Euler-Maruyama approximation
processes
can then be transfered to a large class 
of other one-step approximation processes.
To be more precise, Lemma~\ref{lem:compareSVstability} 
in Subsection~\ref{sec:SVtaming}
proves that if
two general one-step approximation schemes are close to each other
in the sense of~\eqref{eq:distPhi}
(see Lemma~\ref{lem:compareSVstability} 
for the details)
and if
one approximation scheme satisfies the
Lyapunov-type inequality~\eqref{eq:semi.V.stable2} on large subevents,
then the other approximation scheme
satisfies the Lyapunov-type inequality~\eqref{eq:semi.V.stable2} 
on large subevents of the probability
space as well.
After having established the 
Lyapunov-type 
inequality~\eqref{eq:semi.V.stable2} 
on such complements of rare events, the general rare event
based theory in
Section~\ref{sec:general} can 
be applied to derive moment
bounds and further integrability
properties of the approximation
processes.
In Chapter~\ref{chap:convergence},
we then proceed to study 
convergence in probability (see Section~\ref{sec:convergence in probability}), 
strong convergence (see Section~\ref{sec:strongconvergence})
and weak convergence 
(see Section~\ref{sec:weakconvergence})
of approximation processes
for SDEs.
Definition~\ref{def:consistent} 
in Section~\ref{sec:consistent}
specifies a local consistency condition on approximation schemes
which is, according to Theorem~\ref{thm:convergence} 
in 
Section~\ref{sec:convergence in probability},
sufficient for convergence in 
probability of the approximation
processes
to the exact solution of 
the SDE~\eqref{eq:SDE.intro}.
This convergence in probability
and
the uniform moment bounds in
Corollary~\ref{cor:apriori_increment}
then result in the strong convergence~\eqref{eq:strong.convergence.intro}
of the increment-tamed 
Euler-Maruyama approximations~\eqref{eq:increment.tamed.Euler.intro};
see Theorem~\ref{thm:strongC} in Subsection~\ref{sec:strongCtamed} 
for the details.
Moreover, we obtain results
for approximating moments and 
more general statistical quantities
of 
solutions of SDEs 
of the form~\eqref{eq:SDE.intro} 
in Section~\ref{sec:weakconvergence}.
In particular,
Corollary~\ref{cor:MCEmethod} in Subsection~\ref{sec:montecarloeuler}
establishes convergence of the
Monte Carlo Euler approximations 
for SDEs of the 
form~\eqref{eq:SDE.intro}.

\section{Notation}

Throughout this article, the
following notation is used.
For a set $ \Omega $,
a measurable space
$
  \left( E, \mathcal{E} \right)
$
and a mapping
$ 
  Y \colon \Omega \rightarrow E 
$
we denote by
$
  \sigma_{ \Omega }( Y )
:=
  \{
    Y^{ - 1 }(A) \subset \Omega
    \colon
    A \in \mathcal{E}
  \}
$ 
the smallest sigma algebra
with respect to which
$ Y \colon \Omega \rightarrow E 
$ 
is measurable.
Furthermore,
for a topological space
$
  \left( E, \mathcal{E} \right)
$
we denote by
$
  \mathcal{B}( E ) :=
  \sigma_{ E }( \mathcal{E} )
$
the Borel sigma-algebra of
$
  \left( E, \mathcal{E} \right)
$.
Moreover, for a natural number $ d \in \N $
and two sets $ A, B \subset \R^d $ 
we denote by
\begin{equation}
  \text{dist}(A,B) :=
  \begin{cases}
    \inf\!\left\{
      \| a - b \| 
      \in [0,\infty)
      \colon
      (a,b) \in A \times B
    \right\}
  &
    \colon 
    A \neq \emptyset
    \text{ and }
    B \neq \emptyset
  \\
    \infty 
    & \colon 
    \text{else}
  \end{cases}
\end{equation}
the distance of $ A $ and $ B $.
In addition, for a natural number
$ d \in \N $, 
an element $ x \in \R^d $
and a set $ A \subset \R^d $
we denote by
$
  \text{dist}(x,A) :=
  \text{dist}( \{ x \}, A )
$
the distance of $ x $ and $ A $.
Throughout this article we also often calculate
and formulate expressions 
in the extended 
positive real numbers
$ 
  [0,\infty]
  =
  [0,\infty) 
  \cup \{ \infty \}
$.
For instance,
we frequently 
use the conventions
$ \frac{ a }{ \infty } = 0 $
for all $ a \in [0,\infty) $,
$ \frac{ a }{ 0 } = \infty $
for all $ a \in (0,\infty] $
and
$ 0 \cdot \infty = 0 $.
Moreover, 
let 
$ \chi_p \in [0,\infty) $,
$ p \in [1,\infty) $, be a family
of real numbers 
such that
for every $ p \in [1,\infty) $,
every probability space
$ \left( \Omega, \mathcal{F}, \P \right) $,
every stochastic processes
$ 
  Z \colon \N \times \Omega
  \to \R 
$
with the property that
$
  (
    \sum_{ k = 1 }^n
    Z_k
  )_{ n \in \N_0 }
$
is a martingale
and every
$ N \in \N_0 := \{ 0, 1, \dots \} $
it holds that
\begin{equation}
\label{eq:burkholder}
\begin{split}
    \left\| 
      \sup_{ n \in \{ 0, 1, \dots, N \} }
      \left|
       \sum\nolimits_{ k = 1 
       }^{ n } Z_k
      \right|
    \right\|_{ L^p( \Omega; \R ) 
    }^2
  \leq
  \chi_p
  \left(
    \sum\nolimits_{ n = 1 }^{ N 
    }
    \|
      Z_n
    \|_{ L^p( \Omega; \R ) }^2
  \right)
  .
\end{split}
\end{equation}
The Burkholder-Davis-Gundy inequality
(see, e.g., Theorem~48
in Protter~\cite{Protter2004})
ensures that the 
real numbers
$ \chi_p \in [0,\infty) $,
$ p \in [1,\infty) $, 
in \eqref{eq:burkholder}
do indeed exist.
Next for two sets $ A $ and $ B $
we denote 
by
$ 
  \mathcal{M}( A, B ) 
$
the set of all mappings from
$ A $ to $ B $.
Furthermore, for 
natural numbers
$ d, m \in \N $ 
and a
$ d \times m $-matrix
$ A \in \R^{ d \times m } $
we denote by
$ A^{*} \in \R^{ m \times d } $
the transpose of the matrix $ A $.
In addition,
for 
$ d, m \in \mathbb{N} $
and arbitrary functions
$ 
  \mu =
  ( \mu_1, \dots, \mu_d )
  \colon \mathbb{R}^d 
  \rightarrow \mathbb{R}^d 
$
and
$
  \sigma 
  =
  ( \sigma_{ i, j } )_{
    i \in \{ 1, 2, \dots, d \} ,
    j \in \{ 1, 2, \dots, m \}
  } 
  =
  ( \sigma_k )_{
    k \in \{ 1, 2, \dots, m \}
  }
  \colon \mathbb{R}^d
  \rightarrow \mathbb{R}^{ d \times m } 
$
we denote
by
$  
  \mathcal{G}_{ \mu, \sigma }
  \colon
  C^2( \mathbb{R}^d, \mathbb{R} )
  \rightarrow
  \mathcal{M}( \mathbb{R}^d,
    \mathbb{R}
  )
$
and
$  
  \mathcal{
    \tilde{G}
  }_{ \mu, \sigma }
  \colon
  C^2( \mathbb{R}^d, \mathbb{R} )
  \rightarrow
  \mathcal{M}( 
    \mathbb{R}^d \times \mathbb{R}^d,
    \mathbb{R}
  )
$
linear operators 
defined through
\begin{equation}
\label{eq:generator}
\begin{split}
&
  ( \mathcal{G}_{ \mu, \sigma } \varphi )  
  (x)
\\ & :=
  \varphi'(x) \, \mu(x)
+
  \frac{ 1 }{ 2 }
  \sum_{ k = 1 }^{ m }
  \varphi''(x)\big(
    \sigma_k(x) ,
    \sigma_k(x) 
  \big)
\\ & =
  \left< 
    \mu(x),
    (\nabla \varphi)(x) 
  \right>
+
    \frac{ 1 }{ 2 }
  \tr\!\big(
    \sigma(x) 
    \sigma(x)^{ * }
    (\text{Hess } \varphi)(x) 
  \big)
\\ & =
  \sum_{ i = 1 }^d
  \left(
    \frac{
      \partial 
      \varphi
    }{
      \partial x_i
    } 
  \right)\!( x )
  \cdot
  \mu_{ i }(x)
  +
  \frac{1}{2}
  \sum_{ i, j = 1 }^{ d }
  \sum_{ k = 1 }^{ m }
  \left(
    \frac{
      \partial^2 
      \varphi
    }{
      \partial x_i
      \partial x_j
    } 
  \right)\!( x )
  \cdot
  \sigma_{ i, k }(x)
  \cdot
  \sigma_{ j, k }(x)
\end{split}
\end{equation}
and
\begin{equation}
\label{eq:tildeG}
\begin{split}
&
  ( 
    \mathcal{\tilde{G}}_{ \mu, \sigma }
    \varphi 
  )(x,y)
\\ & := 
  \varphi'(x) \, \mu(y)
+
  \frac{ 1 }{ 2 }
  \sum_{ k = 1 }^{ m }
  \varphi''(x)\big(
    \sigma_k(y) ,
    \sigma_k(y) 
  \big)
\\ & =
  \left< 
    \mu(y),
    (\nabla \varphi)(x) 
  \right>
+
    \frac{ 1 }{ 2 }
  \tr\!\big(
    \sigma(y) 
    \sigma(y)^{ * }
    (\text{Hess } \varphi)(x) 
  \big)
\\ & =
  \sum_{ i = 1 }^d
  \left(
    \frac{
      \partial 
      \varphi
    }{
      \partial x_i
    } 
  \right)\!( x )
  \cdot
  \mu_{ i }(y)
  +
  \frac{1}{2}
  \sum_{ i, j = 1 }^{ d }
  \sum_{ k = 1 }^{ m }
  \left(
    \frac{
      \partial^2 
      \varphi
    }{
      \partial x_i
      \partial x_j
    } 
  \right)\!( x )
  \cdot
  \sigma_{ i, k }(y)
  \cdot
  \sigma_{ j, k }(y)
\end{split}
\end{equation}
for all $ x, y \in \mathbb{R}^d $
and all $ 
  \varphi \in C^2( \mathbb{R}^d,
  \mathbb{R} )
$
where 
$ 
  \sigma_k \colon \mathbb{R}^d
  \rightarrow 
  \mathbb{R}^d
$,
$ k \in \{ 1, 2, \dots, m \} $,
fulfill
$
  \sigma_k(x)
=
  \left(   
    \sigma_{ 1, k }(x) ,
    \dots ,
    \sigma_{ d, k }(x) 
  \right)
$
for all $ x \in \mathbb{R}^d $
and all $ k \in \{ 1, 2, \dots, m \} 
$.
The linear operator
in \eqref{eq:generator}
is associated to the exact solution
of the SDE~\eqref{eq:SDE.intro}
and the linear operator 
in \eqref{eq:tildeG}
is associated to the Euler-Maruyama
approximations of the 
SDE~\eqref{eq:SDE.intro}
(see, e.g., \eqref{eq:lem_SVstabilityI} 
in the proof
of Lemma~\ref{lem:SVstabilityI} 
below).
Furthermore,
for $ d \in \mathbb{N} $
and a Borel
measurable set
$ A \in \mathcal{B}( \mathbb{R}^d ) $
we denote by
$
  \lambda_{ A }
  \colon
  \mathcal{B}( A )
  \rightarrow [0,\infty]
$
the Lebesgue-Borel 
measure on 
$ A \subset \mathbb{R}^d $.
In addition, 
for 
$ n, d \in \mathbb{N} $,
$ 
p \in (0,\infty] $
and a set
$ 
  A \subset \mathbb{R} 
$
we denote 
by 
$ C^n_p( \mathbb{R}^d, A ) $
the set
\begin{equation}
\begin{split}
&
  \mathcal{C}^n_p( \mathbb{R}^d, A )
\\ & :=
  \left\{  
    f \in C^{ n - 1 }( \mathbb{R}^d, A )
    \colon
    \begin{array}{c}
      f^{(n-1)} 
      \text{ is locally Lipschitz continuous 
      and there} 
    \\
      \text{exists a real number } c \in [0,\infty)
      \text{ such that for}
    \\
      \lambda_{ \mathbb{R}^d 
      }\text{-almost all }
      x \in \mathbb{R}^d
      \text{ and all } i \in \{ 1, 2, \dots, n \}
    \\ 
      \text{ we have }
      \| 
        f^{(i)}(x)
      \|_{ 
        L^{(i)}( \mathbb{R}^d, 
        \mathbb{R} )
      }
      \leq c \, | f(x) |^{ [ 1 - i/p ] }
    \end{array}
  \right\} 
\end{split}
\end{equation}
throughout the rest of
this article.
Note that this definition
is well-defined since
Ra\-de\-ma\-cher's theorem
proves that a locally Lipschitz
continuous function 
is almost everywhere differentiable.

\chapter{Integrability properties 
of approximation processes for SDEs}
\label{sec:integrability properties}

A central step in establishing 
strong and numerically
weak convergence 
of approximation processes
is to prove uniform moment bounds.
For this analysis, we propose 
a Lyapunov-type condition
on the one-step function of 
a one-step approximation scheme
(see Definition~\ref{def:SVstability}
in Subsection~\ref{sec:onestep}).
In Sections~\ref{sec:SVstability}
and~\ref{sec:implicit}, we will show that many numerical approximation
schemes including the 
Euler-Maruyama scheme,
the increment-tamed
Euler-Maruyama
scheme~\eqref{eq:increment.tamed.Euler.intro}
and some implicit approximation 
schemes
satisfy this condition
in the case of several nonlinear SDEs.
Subject of Section~\ref{sec:general}
is to infer from this Lyapunov-type condition on the one-step function
that the associated numerical approximations have certain
uniform integrability properties.

\section{General 
discrete-time stochastic
processes}
\label{sec:general}

This section introduces a general approach
for studying integrability and stability
properties
of discrete-time
stochastic processes.
We assume a Lyapunov-type 
estimate on a subevent of
the probability space for each time step.
From this, we derive Lyapunov-type estimates
for the process uniformly in the time variable
in Subsection~\ref{sec:stabest}.
This approach is then applied to derive uniform
moment bounds in finite time 
(see Subsection~\ref{sec:moment_bounds000};
see also
Proposition~\ref{p:stability}
for infinite time)
for a large class of possibly
infinite dimensional
approximation processes.
Note that the state space
$ 
  \left( E, \mathcal{E} \right) 
$
appearing in 
Propositions~\ref{p:stability} 
and \ref{p:PG}
and in Corollaries~\ref{cor:stability2}, \ref{cor:stability} and \ref{cor:FV}
is an arbitrary
measurable space.
In our examples in 
Sections~\ref{sec:SVstability}
and \ref{sec:implicit}
below we restrict ourself,
however, to explicit 
(see Section~\ref{sec:SVstability})
and implicit 
(see 
Section~\ref{sec:implicit})
approximation schemes
for finite dimensional SDEs
driven by standard
Brownian motions.
Our approach is mainly influenced
by ideas in \cite{hj11,bh11,hjk10b};
see the end of 
Subsection~\ref{sec:onestep}
for more details on these 
articles.

\subsection{Lyapunov-type 
estimates on complements
of rare events}
\label{sec:stabest}

The following result (Proposition~\ref{p:stability}) proves
Lyapunov-type estimates for 
discrete-time 
stochastic processes which 
do,
in general, not hold on the whole
probability space but only on 
a family of typically large subevents 
of the probability space.
These subevents
are defined in terms of an
appropriate {\it Lyapunov-type function}
$ V \colon E \rightarrow [0,\infty) $
on the measurable state space 
$ \left( E, \mathcal{E} \right) $ 
and in terms of a suitable {\it truncation function}
$
  \zeta \colon [0,\infty)
  \rightarrow (0,\infty]
$.
In the case where the trunction
function is infinity, i.e.,
$
  \zeta(t) = \infty
$
for all $ t \in [0,\infty) $,
Proposition~\ref{p:stability}
and most of its consequences in Section~\ref{sec:general} 
are 
well-known;
see, e.g.,
Mattingly, Stuart
\citationand\ Higham~\cite{msh02}
and
Schurz~\cite{Schurz2005}.

\begin{prop}
\label{p:stability}
Let 
$ \rho \in \mathbb{R} $,
let
$ 
  \left( 
    \Omega, \mathcal{F}, 
    \mathbb{P} 
  \right)
$
be a probability 
space,
let
$
  \left( E, \mathcal{E} \right)
$
be a measurable space,
let 
$ t_n \in \mathbb{R} $, 
$ n \in \mathbb{N}_0 $,
be a non-decreasing sequence,
let
$
  \zeta \colon [0,\infty)
  \rightarrow (0,\infty]
$
be a function,
let
$
  V  \colon E
  \rightarrow [0,\infty)
$
be an
$ \mathcal{E} $/$ \mathcal{B}( [0,\infty) ) 
$-measurable function
and let
$ 
  Y 
  \colon \N_0 \times 
  \Omega
  \rightarrow E
$,
$ 
  Z 
  \colon \mathbb{N} \times 
  \Omega
  \rightarrow \mathbb{R}
$
be stochastic processes with
$ 
  \mathbb{E}[ 
    \mathbbm{1}_{ \Omega_n }
    | Z_{ n } |   
  ] < \infty 
$
and
\begin{equation}
\label{eq:semi.V.stable}
  \mathbbm{1}_{
    \Omega_n
  }
    V( Y_{ n } )
\leq 
  e^{ \rho \left( t_{ n } - t_{ n - 1 } \right) }
  V( Y_{ n - 1 } )
  +
  \mathbbm{1}_{
    \Omega_n
  }
  Z_{ n }
\end{equation}
for 
all $ n \in \mathbb{N} $
where
$
  \Omega_n :=
  \cap_{ k = 0 }^{ n - 1 }
  \{ 
    V( Y_k ) \leq 
    \zeta( t_{ k + 1 } - t_k )
  \}
  \in \mathcal{F}
$
for all $ n \in \mathbb{N}_0 $.
Then 
\begin{align}
\label{eq:exactrep}
  \mathbbm{1}_{ 
    \Omega_n 
  }
  V( Y_n )
&\leq
  e^{ \rho \left( t_n - t_0 \right) }
  V( Y_0 )
  +
  \sum_{ k = 1 }^{ n }
  e^{ \rho \left( t_n - t_k \right) }
  \mathbbm{1}_{ \Omega_k }
  Z_k ,
\\
\label{eq:rareevents}
  \mathbb{P}\big[
    ( 
      \Omega_n 
    )^{ c }
  \big]
  &\leq
  \sum_{ k = 0 }^{ n - 1 }
  \left(
  \frac{
    e^{ 
      \rho \left( t_k - t_0 \right) 
    }   
    \,
    \mathbb{E}[
      V( 
        {Y}_0
      )
    ] 
    +
    \sum_{ l = 1 }^{ k }
    e^{ \rho \left( t_k - t_l \right) }
    \,
    \mathbb{E}[
      \mathbbm{1}_{ \Omega_l }
      Z_l 
    ]
  }{
    \zeta\!\left(
     t_{ k+1 } - t_k 
    \right)
  }
  \right)
\end{align}
for all 
$ n \in \mathbb{N}_0 $.
\end{prop}

\begin{proof}[Proof
of Proposition~\ref{p:stability}]
First, observe that
assumption~\eqref{eq:semi.V.stable}
and the relation 
$ \Omega_n \subset \Omega_{ n - 1 } $
for all $ n \in \mathbb{N} $
show that
\begin{equation}
\label{eq:firstineq}
  \mathbbm{1}_{
    \Omega_{ n }
  }
    V( Y_{ n } )
\leq 
  \mathbbm{1}_{
    \Omega_{ n - 1 }
  }
  e^{ \rho \left( t_{ n } - t_{ n - 1 } \right) }
  V( Y_{ n - 1 } )
  +
  \mathbbm{1}_{
    \Omega_{ n }
  }
  Z_{ n }
\end{equation}
for all $ n \in \mathbb{N} $.
Estimate~\eqref{eq:firstineq} 
is equivalent to the inequality
\begin{equation}
\label{eq:difference}
  \mathbbm{1}_{
    \Omega_{ n }
  }
  e^{ - \rho t_{ n } }
  V( Y_{ n } )
  -
  \mathbbm{1}_{
    \Omega_{ n - 1 }
  }
  e^{ - \rho t_{ n - 1 } }
  V( Y_{ n - 1 } )
\leq
  \mathbbm{1}_{ \Omega_n }
  e^{ - \rho t_{ n } }
  Z_{ n }
\end{equation}
for all $ n \in \mathbb{N} $.
Next note that
\eqref{eq:difference} 
and the fact $ \Omega_0 = \Omega $
imply
\begin{equation}
\label{eq:firstresult}
\begin{split}
&
    \mathbbm{1}_{
    \Omega_{ n }
  }
  e^{ - \rho t_{ n } }
  V( Y_{ n } )
\\ & =
  \mathbbm{1}_{
    \Omega_{ 0 }
  }
  e^{ - \rho t_{ 0 } }
  V( Y_{ 0 } )
  +
  \sum_{ k = 1 }^{ n }  
  \left(
    \mathbbm{1}_{
      \Omega_{ k }
    }
    e^{ - \rho t_{ k } }
    V( Y_{ k } )
    -
    \mathbbm{1}_{
      \Omega_{ k - 1 }
    }
    e^{ - \rho t_{ k - 1 } }
    V( Y_{ k - 1 } )
  \right)
\\ & \leq
  e^{ - \rho t_{ 0 } }
  V( Y_{ 0 } )
  +
  \sum_{ k = 1 }^{ n }  
    \mathbbm{1}_{ \Omega_k }
    e^{ - \rho t_{ k } }
    Z_{ k }
\end{split}
\end{equation}
for all $ n \in \mathbb{N}_0 $.
This
implies \eqref{eq:exactrep}.
For proving~\eqref{eq:rareevents},
note that
the relation
$ \Omega_{ n } \subset \Omega_{ n - 1 } $
for all $ n \in \N $
implies
\begin{equation}
\label{eq:sets}
  \left( 
    \Omega_n 
  \right)^c
  =
  \big(
    \Omega_{
      n-1
    } 
    \setminus
    \Omega_n 
  \big)
  \uplus
  \big(
  \left( 
    \Omega_{
      n-1
    } 
  \right)^{ c }
  \setminus \Omega_n
  \big)
  =
  \big(
    \Omega_{
      n-1
    } 
    \setminus
    \Omega_{n} 
  \big)
  \uplus
  \big(
  \left( 
    \Omega_{
      n-1
    } 
  \right)^{ c }
  \big)
\end{equation}
for all 
$
  n \in \mathbb{N}
$.
Iterating 
equation~\eqref{eq:sets} and 
using again 
$ \Omega_0 = \Omega $ 
shows
\begin{equation}
\begin{split}
&
  \left( 
    \Omega_n 
  \right)^c
\\ & =
  \bigg(
  \biguplus_{ k = 0 }^{ n-1 }
  \left( 
    \Omega_{
      k
    }
    \! 
    \setminus
    \Omega_{
      { k + 1 }
    } 
  \right)
  \bigg)
  \;\biguplus\;
  \big(
  \left( 
    \Omega_{
      0
    } 
  \right)^{ c }
  \big)
  =
  \biguplus_{ k = 0 }^{ n-1 }
  \left( 
    \Omega_{
      k
    }
    \! 
    \setminus
    \Omega_{
      { k + 1 }
    } 
  \right)
 \\ & =
  \biguplus_{ k = 0 }^{ n-1 }
  \left(
    \Omega_{
      k
    } 
    \cap
    \left\{
      V( 
        {Y}_{ 
          k
        } 
      )
      > 
      \zeta\!\left(
        t_{ k + 1 } - t_k 
      \right)
    \right\} 
  \right)
  =
  \biguplus_{ k = 0 }^{ n-1 }
    \left\{
      \mathbbm{1}_{
        \Omega_{
          k
        } 
      }
      V( 
        {Y}_{ 
          k
        } 
      )
      > 
      \zeta\!\left(
        t_{k+1} - t_k 
      \right)
    \right\} 
\end{split}
\end{equation}
for all 
$
  n \in \mathbb{N}_0  
$.
Additivity of
the probability measure
$ \mathbb{P} $,
Markov's inequality
and 
inequality~\eqref{eq:exactrep}
therefore 
imply
\begin{equation}
\begin{split}
  \mathbb{P}\big[
  ( 
    \Omega_n 
  )^c
  \big]
& =
  \sum_{ k = 0 }^{ n-1 }
  \mathbb{P}\Big[ \,
    \mathbbm{1}_{
      \Omega_{
        k
      } 
    }
      V( 
        {Y}_{ 
          k
        } 
      )
      > 
      \zeta\!\left(
        t_{k+1} - t_k 
      \right)
  \Big]
 \leq
  \sum_{ k = 0 }^{ n-1 }
  \left[
    \frac{
      \mathbb{E}\big[ 
      \mathbbm{1}_{
        \Omega_{
          k
        } 
      }
        V( 
          Y_{ 
            k
          } 
        )
      \big]
    }{
      \zeta\!\left(
        t_{k+1} - t_k 
      \right)
    }
  \right]
\\ & \leq
  \sum_{ k = 0 }^{ n - 1 }
  \left[
  \frac{
    e^{ 
      \rho \left( t_k - t_0 \right) 
    }   
    \,
    \mathbb{E}[
      V( 
        {Y}_0
      )
    ] 
    +
    \sum_{ l = 1 }^{ k }
    e^{ \rho \left( t_k - t_l \right) }
    \,
    \mathbb{E}[
      \mathbbm{1}_{ \Omega_l }
      Z_l 
    ]
  }{
    \zeta\!\left(
     t_{ k+1 } - t_k 
    \right)
  }
  \right]
\end{split}
\end{equation}
for all 
$
  n \in \mathbb{N}_0
$.
This is inequality~\eqref{eq:rareevents}
and the proof of 
Proposition~\ref{p:stability}
is thus completed.
\end{proof}

Let us illustrate Proposition~\ref{p:stability}
with the following 
simple implication.
If the assumptions
of Proposition~\ref{p:stability}
are fulfilled, 
if
$ \rho \in (-\infty,0) $, 
$ \zeta \equiv \infty $,
$
  \E[ V( Y_0 ) ] 
  < \infty
$
and if there exist real numbers 
$ h \in (0,\infty) $
and $ c \in [0,\infty) $
such that
$
  \sup_{ k \in \N }
  \E[ Z_k ] \leq c h
$
and 
$ t_n = n h $
for all $ n \in \N_0 $,
then we infer from
inequality~\eqref{eq:exactrep} 
that
\begin{equation}
\begin{split}
&
  \limsup_{ n \to \infty 
  }
  \E\big[
    V(Y_n)
  \big]
\leq
  \limsup_{ n \to \infty 
  }
  \left(
    e^{
      \rho n h
    }
    \,
    \E\!\left[
      V( Y_0 )
    \right]
    +
    \sum_{ k = 1 }^{ n }
    e^{
      \rho ( n - k ) h
    }
    \,
    \E\!\left[ Z_k \right]
  \right)
\\ & \leq
  \limsup_{ n \to \infty 
  }
  \left(
    e^{
      \rho n h
    }
    \,
    \E\!\left[
      V( Y_0 )
    \right]
    +
    c h
    \left[
    \sum_{ k = 1 }^{ n }
    e^{
      \rho ( n - k ) h
    }
    \right]
  \right)
\\ & =
  \limsup_{ n \to \infty 
  }
  \left(
    e^{
      \rho n h
    }
    \,
    \E\!\left[
      V( Y_0 )
    \right]
    +
    \frac{ c h 
      \left(
        1 - e^{ \rho n h }
      \right)
    }{
      \left( 1 - e^{ \rho h } \right)
    }
  \right)
  =
  \frac{ c h }{ 
    ( 1 - \exp( \rho h ) )
  }
\\ & =
  \frac{ c h }{ 
    \left| \rho \right|
    \left(
    \int_0^h
    \exp( \rho s ) \,
    ds
    \right)
  }
\leq
  \frac{ 
    c 
  }{
    \left| \rho \right|
    e^{ \rho h }
  }
=
  \frac{ 
    c \,
    e^{ |\rho| h }
  }{
    | \rho |
  }
  < \infty
  .
\end{split}
\end{equation}
In many situations, the random variables
$ ( Z_n )_{ n \in \mathbb{N} } $
appearing in 
Proposition~\ref{p:stability}
are centered or even
appropriate martingale differences.
This is the subject of the next
two corollaries 
(Corollary~\ref{cor:stability2}
and Corollary~\ref{cor:stability})
of 
Proposition~\ref{p:stability}.

\begin{cor}
\label{cor:stability2}
Let 
$ \rho \in \mathbb{R} $,
let
$ 
  \left( 
    \Omega, \mathcal{F}, 
    \mathbb{P} 
  \right)
$
be a probability 
space,
let
$
  \left( E, \mathcal{E} \right)
$
be a measurable space,
let 
$ t_n \in \mathbb{R} $, 
$ n \in \mathbb{N}_0 $,
be a non-decreasing sequence,
let
$
  \zeta \colon [0,\infty)
  \rightarrow (0,\infty]
$
be a function,
let
$
  V  \colon E
  \rightarrow [0,\infty)
$
be an
$ \mathcal{E} $/$ \mathcal{B}( [0,\infty) ) 
$-measurable function
and
let
$ 
  Y \colon \mathbb{N}_0 \times 
  \Omega
  \rightarrow E
$
be a stochastic process with
$ 
  \mathbb{E}[ V( Y_0 ) ] < \infty 
$
and
\begin{equation}
\label{eq:semi.V.stable2}
  \mathbbm{1}_{
      \cap_{ k = 0 }^{ n }
      \{ 
        V( Y_k ) 
        \leq 
        \zeta( t_{ k + 1 } - t_k ) 
      \}
  }
  \cdot
  \mathbb{E}\big[
    V( Y_{ n + 1 } )
    \, | \,
    ( Y_k )_{ k \in \{ 0, 1, \dots, n \} }
  \big]
\leq 
  e^{ \rho \left( t_{ n + 1 } - t_n \right) }
  \cdot
  V( Y_n )
\end{equation}
$ 
  \mathbb{P}
$-a.s.\ 
for 
all $ n \in \mathbb{N}_0 $.
Then the stochastic process
$
  \mathbbm{1}_{ \Omega_n }
  e^{ - \rho t_n }
  V( Y_n )
$, 
$ n \in \mathbb{N}_0 $,
is a non-negative 
supermartingale and
\begin{equation}
\label{eq:EV.inequality}
    \mathbb{E}\big[
      \mathbbm{1}_{ \Omega_n }
      V( 
        {Y}_n
      )
    \big] 
   \leq
    e^{ \rho \left( t_n - t_0 \right) }
    \,
    \mathbb{E}\big[
      V( 
        {Y}_0 
      )
    \big] ,
  \;
  \mathbb{P}\big[
    ( 
      \Omega_n 
    )^{ c }
  \big]
  \leq
  \bigg(
  \sum_{ k = 0 }^{ n - 1 }
  \frac{
    e^{ 
      \rho \left( t_k - t_0 \right) 
    } 
  }{
    \zeta\!\left(
     t_{ k+1 } - t_k 
    \right)
  }
  \bigg)
  \mathbb{E}\big[
    V( 
      {Y}_0
    )
  \big] ,
\end{equation}
\begin{equation}
\label{eq:EVbar}
\begin{split}
  \mathbb{E}\big[
    \bar{V}({Y}_n)
  \big]
&
  \leq 
  e^{ \rho (t_n - t_0) } \,
  \mathbb{E}\big[ 
    V( { Y }_{ 0 } )
  \big]
\\ & \quad
  +
  \|
    \bar{V}( Y_n )
  \|_{ 
    L^p( \Omega; \mathbb{R} )
  }
  \bigg[
  \bigg(
  \sum_{ k = 0 }^{ n - 1 }
  \frac{e^{\rho(t_k-t_0)}
  }{
    \zeta\!\left(
      t_{ k+1 } - t_k 
    \right)
  }
  \bigg)
  \mathbb{E}\big[
    V( {Y}_0 )
  \big]
  \bigg]^{ \! ( 1 - 1/p ) }
\end{split}
\end{equation}
for all 
$ n \in \mathbb{N}_0 $,
$ p \in [1,\infty] $
and all
$ \mathcal{E} $/$ \mathcal{B}( [0,\infty) ) 
$-measurable functions
$ \bar{V} \colon E \rightarrow [0,\infty) $
with $ \bar{V}(x) \leq V(x) $
for all $ x \in E $
where
$
  \Omega_n :=
  \cap_{ k = 0 }^{ n - 1 }
  \{ 
    V( Y_k ) \leq 
    \zeta( t_{ k + 1 } - t_k )
  \}
  \in \mathcal{F}
$
for all $ n \in \mathbb{N}_0 $.
\end{cor}

\begin{proof}[Proof
of Corollary~\ref{cor:stability2}]
The relation 
$ \Omega_{n+1} \subset \Omega_{ n } $
for all $ n \in \mathbb{N}_0 $
shows that
assumption~\eqref{eq:semi.V.stable2}
is equivalent to the estimate
\begin{equation}
\label{eq:difference2}
  \mathbb{E}\big[
    \mathbbm{1}_{
      \Omega_{ n + 1 }
    }
    e^{ - \rho t_{ n + 1 } }
    V( Y_{ n + 1 } )
    \, | \,
    ( Y_k )_{ k \in \{ 0, 1, \dots, n \} }
  \big]
\leq 
  \mathbbm{1}_{
    \Omega_{ n }
  }
  e^{ - \rho t_{ n } }
  V( Y_{ n } )
\end{equation}
$ \mathbb{P} $-a.s.\ for 
all $ n \in \mathbb{N}_0 $.
Combining \eqref{eq:difference2},
the assumption
$ \mathbb{E}[ V( Y_0 ) ] < \infty $
and the relation
\begin{equation}
  \sigma_{ \Omega }\big(
    (
    \mathbbm{1}_{
      \Omega_{ k }
    }
    e^{ - \rho t_{ k } }
    V( Y_{ k } ) 
    )_{
      k \in \{ 0, 1, \dots, n \}
    }
  \big)
  \subset
  \sigma_{ \Omega }\big(
    ( Y_k )_{
      k \in \{ 0, 1, \dots, n \} 
    }
  \big)
\end{equation}
for all $ n \in \mathbb{N}_0 $
proves that the process
$
  \mathbbm{1}_{ \Omega_n }
  e^{ - \rho t_n }
  V( Y_n )
$,
$ n \in \mathbb{N}_0 $,
is a non-negative 
supermartingale.
This implies the first
inequality in \eqref{eq:EV.inequality}.
In addition, this ensures
that the stochastic process
$ 
  Z \colon \mathbb{N} \times \Omega
  \rightarrow \R 
$
given
by
\begin{equation}
  Z_n =
  \mathbbm{1}_{ \Omega_n }
  V( Y_{ n } )
  -
  \E\!\left[
    \mathbbm{1}_{ \Omega_n }
    V( Y_n ) \, | \,
    ( Y_k )_{ 
      k \in \{ 0, 1, \dots, n-1 \} 
    }
 \right]
\end{equation}
$ \mathbb{P} $-a.s.\ for 
all $ n \in \mathbb{N} $
satisfies 
$ 
  \mathbb{E}\big[ 
  \mathbbm{1}_{ \Omega_n }
    |Z_n| 
  \big] < \infty 
$
and
$ 
  \mathbb{E}[ 
    \mathbbm{1}_{ \Omega_n } 
    Z_n 
  ] = 0 
$
for all $ n \in \mathbb{N} $.
Moreover, the definition
of 
$ 
  Z \colon \mathbb{N} \times
  \Omega \rightarrow \R 
$ 
and assumption~\eqref{eq:semi.V.stable2}
ensure
\begin{equation}
\begin{split}
  \mathbbm{1}_{ \Omega_n }
  V( Y_n )
& =
  \E\!\left[
    \mathbbm{1}_{ \Omega_{ n  } }
    V( Y_n ) \, | \,
    ( Y_k )_{ 
      k \in \{ 0, 1, \dots, n-1 \}   
    }
   \right]
  +
  Z_n
\\ & =
  \mathbbm{1}_{ \Omega_{ n  } }
  \E\!\left[
    V( Y_n ) \, | \,
    ( Y_k )_{ 
      k \in \{ 0, 1, \dots, n-1 \}   
    }
   \right]
  +
  \mathbbm{1}_{ \Omega_n }
  Z_n
\\ & \leq
  e^{ 
    \rho \left( t_{ n } - t_{ n - 1 } 
    \right) 
  }
  V( Y_{ n - 1 } )
  +
  \mathbbm{1}_{ \Omega_n }
  Z_n
\end{split}
\end{equation}
$ \mathbb{P} $-a.s.\ for 
all $ n \in \mathbb{N} $.
An application of 
Proposition~\ref{p:stability}
thus proves the second inequality
in \eqref{eq:EV.inequality}.
Next observe that H\"{o}lder's
inequality implies
\begin{equation}
\label{eq:simpleest}
  \mathbb{E}[
    X
  ]
\leq
  \mathbb{E}\!\left[
    \mathbbm{1}_{ \tilde{ \Omega } }
    X
  \right]
  +
  \big(
    \mathbb{P}\big[
      ( \tilde{\Omega} )^c
    \big]
  \big)^{ 
    \! ( 1 - 1/p )
  }
  \|
    X
  \|_{ 
    L^{ p }( \Omega; \mathbb{R} ) 
  }
\end{equation}
for all 
$ 
  \tilde{\Omega} \in \mathcal{F} 
$,
$ p \in [1,\infty] $
and all
$ \mathcal{F} $/$
  \mathcal{B}( [0,\infty) )
$-measurable
mappings
$ 
  X \colon \Omega 
  \rightarrow [0,\infty)
$.
Combining \eqref{eq:EV.inequality}
and \eqref{eq:simpleest}
finally results in
\begin{equation}  \begin{split}
  \mathbb{E}&\big[
    \bar{V}({Y}_n)
  \big]
  \leq 
  \mathbb{E}\big[ 
    \mathbbm{1}_{\Omega_n}
    V( { Y }_{ n } )
  \big]
  +
  \|
    \bar{V}( Y_n )
  \|_{ 
    L^p( \Omega; \mathbb{R} )
  }
  \left(
    \mathbb{P}\big[
      ( \Omega_n )^c
    \big]
  \right)^{ 
    ( 1 - 1/p )
  }
\\
  &\leq 
  e^{ \rho ( t_n - t_0 ) } \,
  \mathbb{E}\big[ 
    V( { Y }_{ 0 } )
  \big]
  +
  \|
    \bar{V}( Y_n )
  \|_{ 
    L^p( \Omega; \mathbb{R} )
  }
  \bigg[
  \bigg(
  \sum_{ k = 0 }^{ n - 1 }
  \frac{e^{\rho(t_k-t_0)}
  }{
    \zeta\!\left(
      t_{ k+1 } - t_k 
    \right)
  }
  \bigg)
  \mathbb{E}\big[
    V( {Y}_0 )
  \big]
  \bigg]^{
    \! (1 - 1/p)
  }
\end{split}     
\end{equation}
for all 
$ n \in \mathbb{N}_0 $,
$ p \in [1,\infty] $
and all
$ \mathcal{E} $/$ \mathcal{B}( [0,\infty) ) 
$-measurable functions
$ \bar{V} \colon E \rightarrow [0,\infty) $
with $ \bar{V}(x) \leq V(x) $
for all $ x \in E $.
The proof of 
Corollary~\ref{cor:stability2}
is thus completed.
\end{proof}


Corollary~\ref{cor:stability2},
in particular,
proves 
estimates
on the quantities
\begin{equation}
  \sup_{ k \in \{ 0, 1, \dots, n \} }
    \mathbb{E}\big[
      \mathbbm{1}_{ \Omega_k }
      V( 
        Y_k
      )
    \big] 
\end{equation}
for $ n \in \mathbb{N} $
(see the first
inequality in \eqref{eq:EV.inequality}).
Here $ \Omega_n \subset \Omega $,
$ n \in \mathbb{N}_0 $,
are typically large 
subevents of the
probability space
$ 
  \left( \Omega, \mathcal{F}, \mathbb{P}
  \right) 
$
and $ V \colon E \rightarrow [0,\infty) $
is an appropriate Lyapunov-type
function (see 
Corollary~\ref{cor:stability2}
for details).
Under suitable additional 
assumptions, one can also
obtain an estimate on the 
larger quantities
\begin{equation}
    \mathbb{E}\!\left[
      \sup_{ k \in \{ 0, 1, \dots, n \} }
      \mathbbm{1}_{ \Omega_k }
      V( 
        Y_k
      )
    \right] 
\end{equation}
for $ n \in \mathbb{N} $.
This is the subject of the next
corollary.

\begin{cor}
\label{cor:stability}
Let 
$ \rho \in \mathbb{R} $, 
$ p \in [1,\infty) $,
let
$ 
  \left( 
    \Omega, \mathcal{F}, 
    \mathbb{P}
  \right)
$
be a probability 
space,
let
$
  \left( E, \mathcal{E} \right)
$
be a measurable space,
let 
$ t_n \in \mathbb{R} $, 
$ n \in \mathbb{N}_0 $,
be a non-decreasing sequence,
let
$
  \zeta \colon [0,\infty)
  \rightarrow (0,\infty]
$,
$
  \nu \colon \mathbb{N}
  \rightarrow [0,\infty)
$
be functions,
let
$
  V  \colon E
  \rightarrow [0,\infty)
$
be an
$ \mathcal{E} $/$ \mathcal{B}( [0,\infty) ) 
$-measurable
function and
let
$ 
  Y \colon \mathbb{N}_0 \times 
  \Omega
  \rightarrow E
$,
$ 
  Z \colon \mathbb{N} \times 
  \Omega
  \rightarrow \R
$
be stochastic processes 
such that the process
$ 
  \sum_{ k = 1 }^{ n }
  \mathbbm{1}_{ \Omega_k }
  Z_k
$,
$ n \in \N $,
is a martingale and
such that
\begin{align}
  \label{eq:semi.Lyapunov.martingal}
  \mathbbm{1}_{
    \Omega_{ n }
  }
    V( Y_{ n } )
& \leq 
  e^{ \rho \left( t_{ n } - t_{ n - 1 } \right) }
  V( Y_{ n - 1 } )
  +
  \mathbbm{1}_{ \Omega_n }
  Z_{ n } 
  \qquad
  \P\text{-a.s.},
\\ 
  \label{eq:variationZbounded.by.Y}
  \|
     \mathbbm{1}_{ \Omega_n }
     Z_n
  \|_{ L^p( \Omega; \R ) }
  & \leq
  \nu_{n} \,
  \bigg\|
     \sup_{ 
       k \in \{ 0, 1, \ldots, n - 1 \}
     }
     \mathbbm{1}_{ \Omega_{k} } 
     e^{ \rho ( t_{ n } - t_{ k } ) } 
     V( Y_k )
  \bigg\|_{ L^p(\Omega;\R) } 
\end{align}
for all $ n \in \mathbb{N} $
where
$
  \Omega_n :=
  \cap_{ k = 0 }^{ n - 1 }
  \{ 
    V( Y_k ) \leq 
    \zeta( t_{ k + 1 } - t_k )
  \}
  \in \mathcal{F}
$
for all $ n \in \mathbb{N}_0 $.
Then
\begin{equation}  
\label{eq:cor:stability}
\begin{split}  
  \bigg\|
    \sup_{
      k
      \in \{0,1,\ldots,n\}
    }
    \mathbbm{1}_{ \Omega_k }
    e^{ - \rho t_k }
    V(Y_k)
  \bigg\|_{ L^p(\Omega;\R) }
  \!\!\!\!\!\!\!
\leq
  \sqrt{ 2 }
  \left\| 
    V( Y_0 )
  \right\|_{ L^p( \Omega; \R ) }
  \exp\!\left(
    \!
    \chi_p
    \!
    \left[
    \sum_{ k = 1 }^{ n }
    \left| \nu_k \right|^2
    \right]
    - \rho t_0
    \!
  \right)
\end{split}     
\end{equation}
for all $ n \in \N_0 $.
\end{cor}


\begin{proof}[Proof
of Corollary~\ref{cor:stability}]
Inequality~\eqref{eq:exactrep}
in Proposition~\ref{p:stability}
implies
\begin{equation}
    \sup_{ 
      k \in \{ 0, 1, \ldots, n \}
    }
    \mathbbm{1}_{ \Omega_k }
    e^{ - \rho t_k }
    V( Y_k )
  \leq
    \sup_{ 
      k \in \{ 0, 1, \ldots, n \}
    }
    e^{ - \rho t_0 }
    V( Y_0 )
    +
    \sup_{ 
      k \in \{ 0, 1, \ldots, n \}
    }
    \left[
      \sum_{ l = 1 }^k
      \mathbbm{1}_{
        \Omega_k
      }
      e^{ - \rho t_k }
      Z_k
    \right]
\end{equation}
$ \P $-a.s.\ for all $ n \in \N_0 $.
The triangle inequality
and the estimate 
$ (a+b)^2 \leq 2 a^2 + 2b^2 $ 
for all $ a, b \in \R $
hence yield
\begin{equation}
\begin{split}
&
  \left\|
    \sup_{ 
      k \in \{ 0, 1, \ldots, n \}
    }
      \mathbbm{1}_{ \Omega_k }
    e^{ - \rho t_k }
    V( Y_k )
  \right\|_{ 
    L^p( \Omega; \R) 
  }^2
\\ & \leq
  2 
  \left\| 
    e^{ - \rho t_0 } V(Y_0)
  \right\|_{ L^p( \Omega; \R ) }^2
  +
  2
  \left\|
    \sup_{ 
      k \in \{ 0, 1, \ldots, n \}
    }
    \left|
      \sum_{ l = 1 }^k 
      \mathbbm{1}_{ \Omega_l }
      e^{ - \rho t_l }
      Z_l
    \right|
  \right\|_{ L^p( \Omega; \R) }^2
\end{split}
\end{equation}
for all $ n \in \N_0 $.
The definition~\eqref{eq:burkholder}
of
$ \chi_p \in [0,\infty) $,
$ p \in [1,\infty) $,
applied to
the martingale
$
  \sum_{ l = 1 }^{ k }    
  \mathbbm{1}_{ \Omega_l }
  e^{ - \rho t_l } 
  Z_l
$,
$ k \in \mathbb{N}_0 $,
therefore shows
\begin{equation}  
\begin{split}
&
  \left\|
    \sup_{ 
      k \in \{ 0, 1, \ldots, n \}
    }
      \mathbbm{1}_{ \Omega_k }
    e^{ - \rho t_k }
    V( Y_k )
  \right\|_{ 
    L^p( \Omega; \R) 
  }^2
\\ & \leq
  2 
  \left\| 
    e^{ - \rho t_0 } V(Y_0)
  \right\|_{ L^p( \Omega; \R) }^2
  +
  2 \chi_p
  \sum_{ k = 1 }^{ n }
  \left\| 
    \mathbbm{1}_{ \Omega_k }
    e^{ - \rho t_k }
    Z_k
  \right\|_{
    L^p( \Omega; \R)
  }^2
\\ & \leq
  2 
  \left\| 
    e^{ - \rho t_0 } V(Y_0)
  \right\|_{ L^p( \Omega; \R ) }^2
  +
  2 \chi_p
  \sum_{ k = 0 }^{ n - 1 }
  \left| \nu_{ k + 1 } \right|^2
  \bigg\|
     \sup_{ 
       l \in \{ 0, 1, \ldots, k \}
     }
     \mathbbm{1}_{ \Omega_l }
     e^{ - \rho t_l }
    V(Y_l)
  \bigg\|_{ L^p( \Omega; \R ) }^2
\end{split}     
\end{equation}
for all $ n \in \N_0 $
where the last inequality
follows from assumption
\eqref{eq:variationZbounded.by.Y}.
Consequently, 
Gronwall's lemma for 
discrete time yields
\begin{equation}  
\begin{split}
&
  \bigg\|
    \sup_{
      k \in \{ 0, 1, \ldots, n \}
    }
    \mathbbm{1}_{ \Omega_k }
    e^{ - \rho t_k }
    V( Y_k )
  \bigg\|_{
    L^p( \Omega; \R) 
  }^2
\\ & \leq
  2 
  \left\| 
    e^{ - \rho t_0 } V(Y_0)
  \right\|_{
    L^p( \Omega; \R) 
  }^2
  \exp\!\left(
    2 \chi_p
    \sum_{ k = 0 }^{ n - 1 } 
    \left| \nu_{ k + 1 } \right|^2
  \right)
\end{split}     
\end{equation}
for all $ n \in \N_0 $.
This finishes the proof 
of 
Corollary~\ref{cor:stability}.
\end{proof}

\noindent
An application of 
Corollary~\ref{cor:stability}
can be found in 
Lemma~\ref{l:more.Lyapunov.implicit.Euler}
below.

\subsection{Moment 
bounds on complements of
rare events}
\label{sec:semimoment}

In the case of nonlinear
SDEs,
it has been shown in \cite{hjk11} 
that the Euler-Maruyama approximations
often fail to satisfy moment 
bounds although the exact solution
of the SDE does satisfy such moment 
bounds.
Nonetheless, the Euler-Maruyama
approximations often
satisfy suitable moment
bounds restricted to events
whose probabilities converge to one sufficiently
fast; see Corollary~4.4
and Lemma~4.5
in \cite{hj11}. 
This is one motivation for the
next definition.

\begin{definition}[Semi
boundedness]
\label{def:RB}
Let 
$ \alpha \in (0,\infty] $,
let
$ I \subset \mathbb{R} $
be a subset of $ \mathbb{R} $,
let
$
  \left( 
    E, \mathcal{E} 
  \right)
$
be a measurable space,
let
$
  \left( 
    \Omega, \mathcal{F},
    \mathbb{P}
  \right)
$
be a probability space
and let
$
  V \colon E 
  \rightarrow [0,\infty)
$
be an 
$ 
  \mathcal{E} 
$/$ 
  \mathcal{B}( 
    [0,\infty) 
  ) 
$-measurable 
mapping.
A sequence
$
  Y^N \colon I
  \times \Omega
  \rightarrow E
$,
$ 
  N \in \mathbb{N} 
$,
of stochastic processes
is then said to be
$ \alpha $-semi 
$ V $-bounded 
(with respect to $ \mathbb{P} $)
if there exists
a sequence 
$ 
  \Omega_N \in 
  \sigma_{ \Omega }\big( 
    ( Y_t^N )_{ t \in I } 
  \big)
  =
  \sigma_{ \Omega }( Y^N )
  \subset
  \mathcal{F} 
$,
$ N \in \mathbb{N} $,
of events
such that
\begin{equation}
\label{eq:semibounded}
  \limsup_{ N \rightarrow \infty }
  \left(
  \sup_{ t \in I }
  \mathbb{E}\big[
    \mathbbm{1}_{ \Omega_N }
    V( Y^N_t )
  \big]
  +
  N^{ \alpha } \cdot
  \mathbb{P}\big[
    ( \Omega_N )^c
  \big]
  \right)
  < \infty .
\end{equation}
Moreover, a sequence
$
  Y^N \colon I
  \times \Omega
  \rightarrow E
$,
$ 
  N \in \mathbb{N} 
$,
of stochastic processes
is said to be
$ 0 $-semi 
$ V $-bounded 
(with respect to $ \mathbb{P} $)
if there exists
a sequence 
$ 
  \Omega_N \in 
  \sigma_{ \Omega }( Y^N )
$,
$ N \in \mathbb{N} $,
of events
such that
$
  \limsup_{ N \rightarrow \infty }
  \sup_{ t \in I }
  \mathbb{E}\big[
    \mathbbm{1}_{ \Omega_N }
    V( Y^N_t )
  \big]
  < \infty
$
and
$
  \lim_{ N \to \infty }
  \mathbb{P}\big[
    ( \Omega_N )^c
  \big]
  = 0
$.
\end{definition}

We now present some remarks
concerning Definition~\ref{def:RB}.
First, note that the concept
of semi boundedness
in the sense of Definition~\ref{def:RB}
is a property of the 
probability measures associated
to the stochastic processes.
More precisely,  in 
the setting of Definition~\ref{def:RB},
a sequence of stochastic
processes
$ 
  Y^N \colon I \times \Omega
  \rightarrow E
$,
$ N \in \mathbb{N} 
$,
is $ \alpha $-semi $ V $-bounded
if and only if there exists
a sequence 
$ A_N \in \mathcal{E}^{ \otimes I } $,
$ N \in \mathbb{N} $,
of sets
such that
\begin{equation}
  \limsup_{ N \rightarrow \infty }
  \left\{
  \sup_{ t \in [0,T] }
  \int_{ E^{ \times I } }
  \mathbbm{1}_{ A_N }(x) \,
  V(\pi_t(x)) \,
  \mathbb{P}_{ Y^N }(dx)
  +
    N^{ \alpha }
    \cdot
    \mathbb{P}_{ Y^N }\!\big[ 
      ( A_N )^c 
    \big]
  \right\}
  < \infty
\end{equation}
where
$ 
  \mathbb{P}_{ Y^N }
  \colon
  \mathcal{E}^{ \otimes I }
  \rightarrow [0,1]
$,
$ N \in \mathbb{N} $,
with
$
  \mathbb{P}_{ Y^N }[ A ]
  = \mathbb{P}[ Y^N \in A ]
$
for all 
$ 
  A \in \mathcal{E}^{ \otimes I } 
$
and all
$ N \in \mathbb{N} $
are the probability measures
associated to 
$ 
  Y^N \colon I \times \Omega
  \rightarrow E
$, 
$ N \in \mathbb{N} $,
and where
$ 
  \pi_t \colon E^{ \times I }
  \rightarrow E
$,
$ t \in I $,
with
$   
  \pi_t(x) = x(t)
$
for all $ x \in E^{ \times I } $
and all $ t \in I $
are projections from
$ E^{ \times I } $
to $ E $.
Semi boundedness
in the sense of Definition~\ref{def:RB}
is thus a property of the sequence
$
  \mathbb{P}_{ Y^N }
  \colon \mathcal{E}^{ \otimes I }
  \rightarrow [0,1]
$,
$ N \in \N $,
of probability measures
on the measurable path space 
$ 
  \left( 
    E^{ \times I }, 
    \mathcal{E}^{ \otimes I } 
  \right)
$. 
Moreover, observe that in the case
where the index set $ I $ appearing
in Definition~\ref{def:RB} consists
of only one real number $ t_0 \in \R $, i.e., $ I = \{ t_0 \} $, 
the sequence
$ 
  Y^N \colon I \times \Omega
  \to E
$, $ N \in \N $,
of stochastic processes reduces to
a sequence 
$
  Z_N \colon \Omega \to E
$,
$ N \in \N $,
of random variables with
$ Z_N(\omega) = Y^N_{ t_0 }(\omega) $
for all $ \omega \in \Omega $
and all $ N \in \N $.
Furthermore, observe that
in the case $ \alpha = \infty $
in Definition~\ref{def:RB},
condition~\eqref{eq:semibounded}
is equivalent to the condition
$
  \limsup_{ N \rightarrow \infty }
  \sup_{ t \in I }
  \mathbb{E}\big[
    V( Y^N_t )
  \big]
  < \infty 
$.
Next we would like to add
some further comments 
to Definition~\ref{def:RB}.
For this we first note
in Lemma~\ref{lem:convergenceprobab}
a well-known 
characterization
of convergence in probability
(see, e.g., Exercise~6.2.1 (i)
in Klenke~\cite{k06}
for a related excercise
or, e.g., also Remark~9 in \cite{jk12}).
The proof of Lemma~\ref{lem:convergenceprobab}
is for completeness given below.

\begin{lemma}[Convergence
in probability]
\label{lem:convergenceprobab}
Let 
$ 
  \left(   
    \Omega, \mathcal{F},
    \mathbb{P}
  \right)
$
be a probability space,
let
$ 
  \left( E, d_E \right)
$
be a separable metric space and
let
$
  X \colon 
  \Omega \rightarrow E
$
and
$
  Y_N \colon 
  \Omega \rightarrow E
$,
$ N \in \mathbb{N} $,
be 
$ \mathcal{F} $/$
  \mathcal{B}(E)
$-measurable mappings.
Then the following three
assertions
\begin{itemize}

\item[(i)]
the sequence
$ Y_N $,
$ N \in \mathbb{N} $,
converges to $ X $ in
probability, i.e., 
it holds that
$
  \lim_{ N \rightarrow \infty }
  \mathbb{P}\big[
    d_E( X, Y_N ) > \varepsilon
  \big] = 0
$
for all $ \varepsilon \in (0,\infty) $,

\item[(ii)]
it holds that
$
  \lim_{ N \rightarrow \infty }
  \mathbb{P}\big[
    d_E( X, Y_N ) \leq 1
  \big] = 1
$
and it holds that
$
  \lim_{ N \rightarrow \infty }
$
$
  \mathbb{E}\big[
    \mathbbm{1}_{
      \left\{
        d_E( X, Y_N ) \leq 1
      \right\}
    }
    d_E( X, Y_N )
  \big] = 0
$,

\item[(iii)]
there exists
a sequence 
$ \Omega_N \in \mathcal{F} $,
$ N \in \mathbb{N} $,
with
$ 
  \lim_{ N \rightarrow \infty }
  \mathbb{P}\big[  
    \Omega_N  
  \big] = 1
$
and
$
  \lim_{ N \rightarrow \infty }
  \mathbb{E}\big[
    \mathbbm{1}_{ \Omega_N }
    d_E( X, Y_N )
  \big] = 0
$ 

\end{itemize}
are equivalent and each of these
assertions implies
for every $ p \in (0,\infty) $
and every $ x \in E $
with
$
  \E\big[
    \left| d_E( x, X ) \right|^p
  \big] < \infty
$
that
\begin{equation}
\label{eq:bound0semi}
  \lim\nolimits_{ N \to \infty }
  \P\big[
    d_E( X, Y_N ) > 1
  \big]
  = 0 
  \text{ and }
  \sup\nolimits_{ N \in \N }
  \E\big[
    \mathbbm{1}_{
      \left\{
        d_E( X, Y_N ) \leq 1
      \right\}
    }
    \left| d_E( x, Y_N ) \right|^p
  \big]
  < \infty 
  .
\end{equation}
\end{lemma}

\begin{proof}[Proof
of Lemma~\ref{lem:convergenceprobab}]
Note that if
$
  \lim_{ N \rightarrow \infty }
  \!\mathbb{P}\big[
    d_E( X, Y_N ) > \varepsilon
  \big]\! = 0
$
for all $ \varepsilon \in (0,\infty) $,
then
\begin{equation}
  \lim_{ N \rightarrow \infty }
  \!\mathbb{P}\big[
    d_E( X, Y_N ) \leq 1
  \big] = 1
\quad
\text{and}
\quad
  \lim_{ N \rightarrow \infty }
  \mathbb{E}\big[
    \mathbbm{1}_{
      \left\{
        d_E( X, Y_N ) \leq 1
      \right\}
    }
    d_E( X, Y_N )
  \big] 
  = 0
\end{equation}
due to Lebesgue's theorem 
of dominated convergence.
This shows that (i) implies (ii).
Next observe that
(ii), clearly, implies (iii).
Moreover, if there exists
a sequence 
$ \Omega_N \in \mathcal{F} $,
$ N \in \mathbb{N} $,
with
$
  \lim_{ N \rightarrow \infty }
  \mathbb{E}\big[
    \mathbbm{1}_{ \Omega_N }
    d_E( X, Y_N )
  \big] = 0
$
and
$ 
  \lim_{ N \rightarrow 0 }
  \mathbb{P}\big[  
    \Omega_N
  \big] = 1
$,
then
$
  \lim_{ N \rightarrow \infty }
  \mathbb{P}\big[
    \mathbbm{1}_{ \Omega_N }
    d_E( X, Y_N ) > \varepsilon 
  \big] 
  = 0
$
for all $ \varepsilon \in (0,\infty) $
and therefore
\begin{equation}
\begin{split}
  \lim_{ N \rightarrow \infty }
  \mathbb{P}\big[
    d_E( X, Y_N ) > \varepsilon
  \big] 
& \leq
  \lim_{ N \rightarrow \infty }
  \mathbb{P}\big[
    \Omega_N \cap 
    \{ d_E( X, Y_N ) > \varepsilon \}
  \big] 
  +
  \lim_{ N \rightarrow \infty }
  \mathbb{P}\big[
    ( \Omega_N )^c
  \big]
\\ & =
  \lim_{ N \rightarrow \infty }
  \mathbb{P}\big[
    \mathbbm{1}_{ \Omega_N }
    d_E( X, Y_N ) > \varepsilon 
  \big] 
  = 0
\end{split}
\end{equation}
for all $ \varepsilon \in (0,\infty) $.
This proves that
(iii) implies (i).
Finally, observe that 
\begin{equation}
\begin{split}
&
  \big\|
    \mathbbm{1}_{
      \left\{
        d_E( X, Y_N ) \leq 1
      \right\}
    }
    d_E( x, Y_N )
  \big\|_{ L^p( \Omega; \R ) }
\\ & \leq
  \big\|
    \mathbbm{1}_{
      \left\{
        d_E( X, Y_N ) \leq 1
      \right\}
    }
    d_E( x, X ) 
  \big\|_{ L^p( \Omega; \R ) }
  +
  \big\|
    \mathbbm{1}_{
      \left\{
        d_E( X, Y_N ) \leq 1
      \right\}
    }
    d_E( X, Y_N ) 
  \big\|_{ L^p( \Omega; \R ) }
\\ & \leq
  \big\|
    d_E( x, X ) 
  \big\|_{ L^p( \Omega; \R ) }
  + 1
\end{split}
\end{equation}
for all $ x \in E $,
$ N \in \N $
and all
$ p \in (0,\infty) $.
The proof
of Lemma~\ref{lem:convergenceprobab}
is thus completed.
\end{proof}

Let us now study 
under which conditions 
Euler-Maruyama approximations
are $ \alpha $-semi $ V $-bounded with
$ \alpha \in [0,\infty] $ 
and
$
  V \colon \mathbb{R}^d
  \rightarrow [0,\infty) 
$
appropriate
and $ d \in \mathbb{N} $.
First, note that convergence 
in probability of the Euler-Maruyama
approximations has been established
in the literature
for a large class of possibly
highly nonlinear SDEs
(see, e.g., Krylov~\cite{Krylov1990},
Gy\"{o}ngy \& Krylov~\cite{GyoengyKrylov1996},
Gy\"{o}ngy~\cite{g98b}
and Jentzen, Kloeden
\citationand\ Neuenkirch~\cite{jkn09a}).
For these SDEs, one can thus
apply Lemma~\ref{lem:convergenceprobab}
to obtain the existence of
a sequence of events 
whose probabilities converge 
to one and on which
moments of the Euler 
approximations are bounded
in the sense of \eqref{eq:bound0semi}.
This is, however, not sufficient
to establish $ \alpha $-semi
$ \left\| \cdot \right\|^p 
$-boundedness of the
Euler-Maruyama approximations
with $ \alpha \in [0,\infty] $
and $ p \in (0,\infty) $
since the events
in \eqref{eq:bound0semi}
do, in general, not satisfy
the required measurability condition
in Definition~\ref{def:RB}.
Moreover, we are mainly interested
in $ \alpha $-semi
$ \left\| \cdot \right\|^p $-boundedness
of the Euler-Maruyama approximations
with $ \alpha, p \in (0,\infty) $.
Lemma~\ref{lem:convergenceprobab} 
only shows that
the complement of the events 
on the right side of \eqref{eq:bound0semi}
converge to zero but gives no 
information on the rate of 
convergence of the probabilities
of these events.
So, in general,
$ \alpha $-semi
$ \left\| \cdot \right\|^p $-boundedness
with $ \alpha, p \in (0,\infty) $
cannot be inferred from convergence in probability.
Here we employ the theory of Subsection~\ref{sec:stabest}
to obtain
semi boundedness for the
Euler-Maruyama approximations
(see 
Theorem~\ref{thm:SVstability}
and
Corollary~\ref{cor:Semi.Stability}
below).
In particular, Corollary~\ref{cor:stability2} 
immediately implies the 
next corollary.

\begin{cor}[Semi moment 
bounds]
\label{cor:FV}
Let 
$ 
  \left( 
    \Omega, \mathcal{F}, \mathbb{P}
  \right)
$
be a probability space,
let 
$
  \left( E, \mathcal{E} \right)
$
be a measurable space,
let 
$
  \rho, T \in (0,\infty)
$,
$
  \alpha, q \in (1,\infty]
$,
$ N_0 \in \N $,
let
$ 
  V \colon 
  E \rightarrow
  [0,\infty)
$
be an
$ \mathcal{E} 
$/$ \mathcal{B}( [0,\infty) ) 
$-measurable 
mapping and
let
$ 
  Y^N
  \colon
  \{ 0, 1, \dots, N \}
  \times
  \Omega  
  \rightarrow E 
$,
$ N \in \N $,
be a sequence of 
stochastic processes
satisfying
\begin{equation}
\label{eq:assumedsemimoment}
  \mathbbm{1}_{ 
    \left\{
      V( Y^N_n )
      \leq
      \left|
        \frac{ N }{ T }
      \right|^{ \alpha }
    \right\}
  }
  \cdot
  \mathbb{E}\big[
    V( Y_{ n + 1 }^N )
    \, | \,
    ( Y^N_k )_{ k \in \{ 0, 1, \dots, n \} }
  \big]
\leq
  e^{ 
    \frac{ \rho T }{ N }
  }
  \cdot
  V( Y_n^N )
\end{equation}
$ \mathbb{P} $-a.s.\ for 
all 
$ 
  n \in \{ 0, 1, \dots, N - 1 \} 
$
and all
$ 
  N \in 
  \{ N_0, N_0 + 1, \dots \}
$
and
\begin{equation}
\label{eq:CorSemiBounded_InitialAss}
  \limsup_{ N \rightarrow \infty }
  \mathbb{E}\big[ 
    V( 
      Y^N_0
    ) 
  \big] 
  < \infty 
  .
\end{equation}
Then the sequence 
$
  Y^N
  \colon  
  \{ 0, 1, \dots, N \}
  \times \Omega
  \rightarrow E
$,
$ N \in \mathbb{N} $,
of stochastic
processes
is
$   
(\alpha-1)
$-semi 
$ V $-bounded.
%
%
%
%
\end{cor}

\begin{proof}[Proof
of Corollary~\ref{cor:FV}]
Corollary~\ref{cor:stability2}
above 
with the truncation function
$ 
  \zeta \colon [0,\infty) 
  \rightarrow (0,\infty]
$
given by
\begin{equation}
  \zeta(0) = \infty
\qquad
\text{and}
\qquad
  \zeta(t) = 
  \frac{ 1 }{ t^{ \alpha } }
\end{equation}
for all $ t \in (0,\infty) $
and with the sequence
$ t_n \in \R $, $ n \in \N_0 $,
given by
$ t_n = n T/ N $
for all 
$ n \in \N_0 $
implies
\begin{equation}
\label{eq:semi1}
    \mathbb{E}\big[
      \mathbbm{1}_{ \Omega_N }
      V( 
        Y^N_n
      )
    \big] 
  \leq
    e^{ \frac{ \rho n T }{ N } }
    \cdot
    \mathbb{E}\big[
      V( 
        Y^N_0 
      )
    \big] 
  \leq
    e^{ \rho T }
    \cdot
    \mathbb{E}\big[
      V( 
        Y^N_0 
      )
    \big] 
\end{equation}
and
\begin{equation}
\label{eq:semi2}
  \mathbb{P}\big[
    ( 
      \Omega_{ N }
    )^{ c }
  \big]
\leq
  \left| \frac{ T }{ N } \right|^{ \alpha }
  \cdot
  \left(
  \sum_{ k = 0 }^{ N - 1 }
    e^{ 
      \frac{ \rho k T }{ N }
    } 
  \right)
  \cdot
  \mathbb{E}\big[
    V( 
      Y^N_0
    )
  \big] 
  \leq
  \left| \frac{ T }{ N } \right|^{ \alpha }
  \cdot
  N
  \cdot
  e^{ \rho T }
  \cdot
  \mathbb{E}\big[
    V( 
      Y^N_0
    )
  \big] 
\end{equation}
for all 
$ n \in \{ 0, 1, \dots, N \} $
and all
$ N \in \{ N_0,  N_0 + 1, \dots \} $
where
$
  \Omega_N :=
  \cap_{ k = 0 }^{ N - 1 }
  \{ 
    V( Y^N_k ) \leq 
    \left( N / T \right)^{ \alpha }
  \}
  \in \mathcal{F}
$
for all $ N \in \N $.
Combining \eqref{eq:semi1},
\eqref{eq:semi2} 
and assumption~\eqref{eq:CorSemiBounded_InitialAss}
implies
\begin{equation}
  \limsup_{ N \rightarrow \infty }
  \left(
  \sup_{ n \in \{ 0, 1, \dots, N \} }
    \mathbb{E}\big[
      \mathbbm{1}_{ \Omega_N }
      V( 
        Y^N_n
      )
    \big] 
    +
    N^{ ( \alpha - 1 ) }
    \cdot
  \mathbb{P}\big[
    ( 
      \Omega_N
    )^{ c }
  \big]
  \right)
  < \infty.
\end{equation}
The proof of 
Corollary~\ref{cor:FV}
is thus completed.
\end{proof}

\subsection{Moment
bounds}
\label{sec:moment_bounds000}

Corollary~\ref{cor:FV} above, in particular,
establishes moment bounds
restricted to the complements of rare events 
for sequences of stochastic processes.
In some situations, the sequence
of stochastic processes fulfills
an additional growth bound assumption 
(see \eqref{eq:PG}
and \eqref{eq:PG2} below for details)
which can be used to prove moment
bounds on the 
full probability space.
This is the subject of the next proposition.
The main idea of this proposition
is a certain bootstrap argument
which exploits 
inequality~\eqref{eq:EVbar}
in Corollary~\ref{cor:stability2}
(see also estimate~\eqref{eq:simpleest}).

\begin{prop}[Moment bounds]
\label{p:PG}
Let 
$ 
  \left( 
    \Omega, \mathcal{F}, \mathbb{P}
  \right)
$
be a probability space,
let 
$
  \left( E, \mathcal{E} \right)
$
be a measurable space,
let 
$
  T \in (0,\infty)
$,
$
  \rho \in [0,\infty)
$,
$
  \alpha \in (1,\infty]
$,
$ N_0 \in \N $,
let
$ 
  V, \bar{V} \colon 
  E \rightarrow
  [0,\infty)
$
be 
$ \mathcal{E} 
$/$ \mathcal{B}( [0,\infty) ) 
$-measurable 
mappings
with $ \bar{V}(x) \leq V(x) $
for all $ x \in E $ and
let
$ 
  Y^N
  \colon
  \{ 0, 1, \dots, N \}
  \times
  \Omega  
  \rightarrow E 
$,
$ N \in \N $,
be a sequence of stochastic
processes satisfying 
\begin{equation}
\label{eq:assumedsemimoment2}
  \mathbbm{1}_{ 
    \left\{
      V( Y^N_n )
      \leq
      \left|
        \frac{ N }{ T }
      \right|^{ \alpha }
    \right\}
  }
  \cdot
  \mathbb{E}\big[
    V( Y_{ n + 1 }^N )
    \, | \,
    ( Y^N_k )_{ k \in \{ 0, 1, \dots, n \} }
  \big]
\leq
  e^{ 
    \frac{ \rho T }{ N }
  }
  \cdot
  V( Y_n^N )
\end{equation}
$ \mathbb{P} $-a.s.\ for 
all 
$ 
  n \in \{ 0, 1, \dots, N - 1 \} 
$
and all
$ 
  N \in 
  \{ N_0, N_0 + 1, \dots \}.
$
Then
\begin{equation}  
\label{eq:PG}
\begin{split}
&
  \limsup_{ N \to \infty }
  \sup_{ n \in \{ 0, 1, \dots, N \} }
  \mathbb{E}\big[
    \bar{V}( 
      Y^N_n 
    )
  \big] 
\\ & \leq
  e^{ \rho T } 
  \left(
  1 +
  \limsup_{ N \to \infty }
  \mathbb{E}\big[ 
    V( Y^N_{ 0 } )
  \big]
  \right)
\\ & \quad
  \cdot
  \left(
  1 +
    T^{ 
    \alpha 
    (1 - 1/p)
  }
  \limsup_{ N \to \infty }
  \!
  \left[
  N^{(1-\alpha) (1 - 1/p) }
  \sup_{ 0 \leq n \leq N }
  \|
    \bar{V}( Y^N_n )
  \|_{ 
    L^p( \Omega; \mathbb{R} )
  }  
  \right]
  \right) 
\end{split}
\end{equation}
for all $ p \in [1,\infty] $.
\end{prop}

\begin{proof}[Proof
of Proposition~\ref{p:PG}]
We proof Proposition~\ref{p:PG}
through an application of
Corollary~\ref{cor:stability2}.
More precisely,
observe that with the truncation
function
$ 
  \zeta \colon [0,\infty)
  \rightarrow (0,\infty]
$
given by
\begin{equation}
  \zeta(0) = 0
\qquad
\text{and}
\qquad
  \zeta(t) = t^{ - \alpha }
\end{equation}
for all $ t \in (0,\infty) $
and with 
the sequence
$ t_n \in \R $, $ n \in \N_0 $,
given by
$ t_n = n T/ N $
for all 
$ n \in \N_0 $
we get 
from inequality~\eqref{eq:EVbar}
in 
Corollary~\ref{cor:stability2}
that
\begin{equation}
\begin{split}
&
  \mathbb{E}\big[
    \bar{V}( Y^N_n)
  \big]
\\ & \leq 
  e^{ \frac{ \rho n T }{ N } } \,
  \mathbb{E}\big[ 
    V( Y^N_{ 0 } )
  \big]
  +
  \|
    \bar{V}( Y^N_n )
  \|_{ 
    L^p( \Omega; \mathbb{R} )
  }
  \bigg[
  \left| \frac{ T }{ N } \right|^{ \alpha }
  \bigg(
    \sum_{ k = 0 }^{ n - 1 }
    e^{ \frac{ \rho k T }{ N } }
  \bigg)
  \mathbb{E}\big[
    V( Y^N_0 )
  \big]
  \bigg]^{ \! ( 1 - 1/p ) }
\\ & \leq
  e^{ \rho T } \,
  \mathbb{E}\big[ 
    V( Y^N_{ 0 } )
  \big]
  +
  e^{ \rho T } \,
  \|
    \bar{V}( Y^N_n )
  \|_{ 
    L^p( \Omega; \mathbb{R} )
  }
  \,
  \Big[
  \left| \tfrac{ T }{ N } \right|^{ \alpha }
  \cdot N \cdot
  \mathbb{E}\big[
    V( Y^N_0 )
  \big]
  \Big]^{ ( 1 - 1/p ) }
\\ & \leq
  e^{ \rho T } 
  \left(
  1 +
  \mathbb{E}\big[ 
    V( Y^N_{ 0 } )
  \big]
  \right)
  \left(
  1 +
  \left| 
    \tfrac{ T }{ N } 
  \right|^{ 
    \alpha 
    (1 - 1/p)
  }
  N^{ (1 - 1/p) }
  \,
  \|
    \bar{V}( Y^N_n )
  \|_{ 
    L^p( \Omega; \mathbb{R} )
  }  
  \right)
\end{split}
\end{equation}
for all $ n \in \{ 0, 1, \dots, N \} $,
$ N \in \{ N_0, N_0 + 1, \dots \} $
and all $ p \in [1,\infty] $.
This completes the proof
of Proposition~\ref{p:PG}.
\end{proof}

\subsection{One-step 
approximation schemes}
\label{sec:onestep}

Let $T\in(0,\infty)$
be a real number,
let
$ ( \Omega, \mathcal{F},
( \mathcal{F}_t )_{ t \in [0,T] },
\mathbb{P} ) $
be a filtered probability space
and let $(E,\mathcal{E})$ be a measurable space.
We are interested in general 
one-step approximation 
schemes for SDEs with 
state space $E$.
The driving noise of the SDE could be, e.g., a standard Brownian motion,
a fractional Brownian motion or a Lev\'y process.
In such a general situation,
a sequence of approximation processes of the solution process
with uniform time discretization is often
given by a measurable mapping $\Psi\colon E\times[0,T]^2\times \Omega\to E$
as follows. 
Let $\xi\colon\Omega\to E$ be a measurable mapping.
Define stochastic processes
$
  Y^N\colon\{0,1,\ldots,N\}
  \times\Omega\to E
$, 
$ N \in \N $, 
through $ Y_0^N = \xi $
and
\begin{equation} 
\label{eq:Psi}
  Y_{n+1}^N
  =
  \Psi\big(
    Y_n^N,
    \tfrac{ n T }{ N },
    \tfrac{ (n+1) T }{ N }
  \big)
\end{equation}
for all 
$ 
  n \in \{0, 1, \ldots, N - 1 \}
$ 
and all 
$ N \in \N $.
For example, in the setting of the introduction, the
Euler-Maruyama
scheme~\eqref{eq:Euler.intro}
is given by the one-step function 
\begin{equation}
  \Psi(x,s,t)
  =
  x +
  \bar{\mu}(x) (t-s)
  +
  \bar{\sigma}(x)\big(W_t-W_s\big)
\end{equation}
for all $x\in\R^d$, $s,t\in[0,T]$.
The general approach of Subsections~\ref{sec:stabest}--\ref{sec:moment_bounds000}
can be used to study 
moment
bounds for stochastic processes 
defined as in~\eqref{eq:Psi}.

In Sections~\ref{sec:SVstability}
and \ref{sec:implicit} below, 
we focus on 
finite-dimensional SDEs driven 
by standard Brownian motions.
Due to the Markov property of the Brownian motion,
a one-step approximation scheme
can in this case be specified by a function
of the current position, of the 
time increment and of the increment
of a Brownian motion.
More precisely, 
let 
$ d, m \in \N $, 
$ \theta \in (0,T] $
and let
$  
  \Phi \colon 
  \R^d \times
  [0,\theta] \times \R^m \to \R^d
$ 
be a Borel measurable
function.
Using this function and a uniform 
time discretization,
we define a family 
of stochastic processes
$
  Y^N\colon\{0,1,
  \ldots,N\}
  \times
  \Omega\to\R^d
$, 
$ 
  N \in \N \cap [\frac{ T }{ \theta }, \infty) 
$,
through $Y_0^N=\xi$
and
\begin{equation}  \label{eq:approximation.defined.through.Phi}
  Y_{n+1}^N
=
  \Phi\big(
    Y_n^N,
    \tfrac{T}{N},
    W_{\frac{(n+1)T}{N}} -
    W_{\frac{nT}{N}}
  \big)
\end{equation}
for all 
$ n \in \{0,1,\ldots,N-1\}$ 
and all 
$ 
  N \in \N \cap [\frac{ T }{ \theta }, \infty)
$.
In this notation, the 
Euler-Maruyama
scheme~\eqref{eq:Euler.intro}
is given by 
$
  \Phi(x,t,y)
  = x + 
  \bar{\mu}(x)t
  +
  \bar{\sigma}(x)y
$
for all 
$
  (x,t,y)
  \in
  \R^d
  \times
  [0, \theta]
  \times
  \R^m
$.
Moreover, in this setting, 
condition~\eqref{eq:semi.V.stable2}
in Corollary~\ref{cor:stability2}
on the approximation 
processes 
$
  Y^N\colon\{0,1,
  \ldots,N\}
  \times\Omega\to\R^d
$, $N\in\N$,
follows from the 
following condition
(see \eqref{eq:Vstable})
on the one-step function
$
  \Phi
  \colon
  \R^d \times [0,\theta]
  \times
  \R^m \to \R^d
$.

%
%
\begin{definition}[Semi stability with 
respect
to Brownian motion]
\label{def:SVstability}
Let 
$ 
  \theta \in (0,\infty)
$,
$ \alpha \in (0,\infty] $,
$ d, m \in \mathbb{N} $
and
let
$ 
  V \colon \mathbb{R}^d
  \rightarrow [0,\infty) 
$
be a Borel measurable function.
A Borel measurable function
$ 
  \Phi \colon 
  \mathbb{R}^d \times [0,\theta] \times 
  \mathbb{R}^m
  \rightarrow 
  \mathbb{R}^d
$
is then said to 
be
$ \alpha $-semi
$ V $-stable 
with respect to
Brownian motion
if there exists a real
number $ \rho \in \mathbb{R} $
such that
\begin{equation}
\label{eq:Vstable}
  \mathbb{E}\big[
    V(
      \Phi( x, t, 
        W_t
      )
    )
  \big]
  \leq
  e^{ \rho t } \cdot
  V(x)
\end{equation}
for all 
$
  (x,t) \in
  \{
    (y,s) \in \mathbb{R}^d \times (0,\theta]
    \colon
    \alpha = \infty \text{ or }
    V(y) \leq s^{ - \alpha }
  \}
$
where 
$
  W \colon [0,\theta] \times \Omega
  \rightarrow \mathbb{R}^m
$
is an arbitrary 
standard Brownian motion
on a probability space
$
  \left( 
    \Omega, \mathcal{F},
    \mathbb{P}
  \right)
$.
In addition,
a Borel measurable function
$ 
  \Phi \colon 
  \mathbb{R}^d \times [0,\theta] \times 
  \mathbb{R}^m
  \rightarrow 
  \mathbb{R}^d
$
is simply said to be
$ V $-stable 
with respect to Brownian 
motion if it is 
$ \infty $-semi $ V $-stable
with respect to Brownian motion.
\end{definition}

Let us add some remarks to 
Definition~\ref{def:SVstability}.
Inequalities of the form
\eqref{eq:Vstable}
with $ \alpha = \infty $
have been frequently
used 
in the literature; see, e.g.,
Assumption~2.2
in Mattingly, Stuart
\citationand\ Higham~\cite{msh02}
and
Section~3.1 in 
Schurz~\cite{Schurz2005}.
In particular, 
inequality~(16) in
Schurz~\cite{Schurz2005}
(see also Schurz~\cite{Schurz2006})
defines a numerical 
approximation of the 
form 
\eqref{eq:Psi}
to be 
(weakly) $ V $-stable
if 
inequality~\eqref{eq:Vstable}
holds with $ \alpha = \infty $ 
where 
$ V \colon \R^d
\to [0,\infty) $
is Borel measurable.

In the case 
$ \alpha \in (0,\infty) $,
inequality~\eqref{eq:Vstable} in
Definition~\ref{def:SVstability}
is restricted
to the subset
$
  \{ 
    y \in \R^d \colon V(y) \leq t^{ - \alpha } 
  \} 
  \subset \R^d
$
which
increases to the 
full state space $ \R^d $
as 
the time step-size $ t \in (0,T] $
decreases to zero.
Roughly speaking,
the parameter $ \alpha \in (0,\infty) $
describes the speed how fast
the subsets
$
  \{ 
    y \in \R^d \colon V(y) \leq t^{ - \alpha } 
  \} 
$,
$ t \in (0,T] $,
increase to the 
full state space $ \R^d $.

Next note that if
$ 
  \theta \in (0,\infty)
$,
$ \beta \in (0,\infty] $,
$ d, m \in \mathbb{N} $,
if
$ 
  V \colon \mathbb{R}^d
  \rightarrow [1,\infty) 
$
is a Borel measurable function
and if
$ 
  \Phi \colon \R^d \times
  [0,\theta] \times \R^m
  \to \R^d
$
is $ \beta $-semi $ V $-stable
with respect to Brownian
motion, then
it holds 
for every $ \alpha \in (0,\beta] $
that
$ 
  \Phi 
$
is also $ \alpha $-semi 
$ V $-stable
with respect to Brownian
motion.
However,
the converse is not true in general.
In particular,
Theorem~\ref{thm:SVstability}
and Corollary~\ref{cor:disprove}
show in many situations that 
the Euler-Maruyama scheme 
is
$ \alpha $-semi $ V $-stable
with respect to Brownian motion 
for some $ \alpha \in (0,\infty) $
but not $ V $-stable 
with respect to Brownian motion
if the coefficients of the 
underlying SDE grow 
superlinearly.

%
It follows immediately from the Markov property of the Brownian motion
that 
the approximation 
processes~\eqref{eq:approximation.defined.through.Phi}
associated to an $\alpha$-semi $V$-stable one-step function
satisfy condition~\eqref{eq:assumedsemimoment}
where
$ \alpha \in (0,\infty) $ 
and
$
  V \colon \mathbb{R}^d
  \rightarrow [0,\infty) 
$
are
appropriate
and where $ d \in \mathbb{N} $.
The next corollary collects consequences of this observation.

\begin{cor}[Semi moment
bounds and moment bounds
based on semi stability with
respect to Brownian motion]
\label{cor:Semi.Stability}
Let 
$ d, m \in \N $, 
$ T \in (0,\infty) $,
$ \theta \in (0,T] $,
$
  \alpha \in (1,\infty]
$,
$ 
  p \in [1,\infty] 
$,
let
$ 
  V, \bar{V} \colon 
  \R^d \rightarrow
  [0,\infty)
$
be 
Borel measurable 
mappings
with $ \bar{V}(x) \leq V(x) $
for all $ x \in \R^d $,
let
$ 
  \left( 
    \Omega, 
    \mathcal{F}, 
     ( \mathcal{F}_t )_{ t \in [0,T] }, 
    \mathbb{P}
  \right)
$
be a filtered probability space,
let
$ 
  W \colon [0,T] \times
  \Omega \to \R^m
$
be a standard
$ 
  ( \mathcal{F}_t )_{ t \in [0,T] } 
$-Brownian motion,
let
$  
  \Phi \colon 
  \R^d \times
  [0,\theta] \times \R^m \to \R^d
$ 
be $ \alpha $-semi
$ V $-stable with respect
to Brownian motion and
let 
$
  Y^N
  \colon
  \{ 0, 1, \ldots, N \}
  \times
  \Omega \to \R^d
$, 
$ 
  N \in \N 
$,
be a sequence of 
$ ( \mathcal{F}_t )_{ t \in [0,T] } 
$-adapted
stochastic
processes satisfying
$
  \limsup_{ N \to \infty }
  \E[
    V( Y^N_0 ) 
  ]
  < \infty
$
and
\begin{equation}
  Y_{n+1}^N
=
  \Phi\big(
    Y_n^N,
    \tfrac{T}{N},
    W_{ (n+1)T / N } -
    W_{ nT / N }
  \big)
\end{equation}
for all 
$ n \in \{0,1,\ldots,N-1\}
$ 
and all 
$ 
  N \in \N \cap [T/\theta,\infty) 
$.
Then the stochastic
processes $ Y^N $, $ N \in \N $,
are $ ( \alpha - 1 ) $-semi
$ V $-bounded. 
Moreover, if
\begin{equation}
\label{eq:PG2}
  \limsup_{ N \rightarrow \infty }
  \big(
    N^{ 
      ( 1 - \alpha ) 
      ( 1 - 1/p ) 
    }
    \cdot
    \sup\nolimits_{ 
      n \in \{ 0, 1, \dots, N \}
    }
    \|
      \bar{V}( Y^N_n )
    \|_{
      L^p( \Omega; \mathbb{R} )
    }
  \big)
  <
  \infty 
\end{equation}
in addition to the above 
assumptions,
then the stochastic processes
$ Y^N $, $ N \in \N $,
are also $ \infty $-semi
$ \bar{V} $-bounded, i.e.,
%
$
  \limsup_{ N \rightarrow \infty }
  \sup_{ n \in \{ 0, \dots, N \} }
  \mathbb{E}\big[
    \bar{V}( 
      Y^N_n 
    )
  \big] 
$
$
  < \infty 
$.
\end{cor}

Corollary~\ref{cor:Semi.Stability}
is an immediate consequence
of Corollary~\ref{cor:FV}
and Proposition~\ref{p:PG}.
Below in 
Sections~\ref{sec:SVstability}
and \ref{sec:implicit}
we will study both explicit
and implicit one-step numerical
approximation processes
of the form~\eqref{eq:approximation.defined.through.Phi}.
For these approximation processes
we will then give sufficient conditions
which ensure that the function
$ 
  \Phi \colon \R^d
  \times [0,\theta] \times \R^m 
  \rightarrow \R^d
$
appearing in \eqref{eq:approximation.defined.through.Phi}
is $ \alpha $-semi $ V $-stable
with respect to Brownian motion
with $ \alpha \in (1,\infty] $ 
appropriate
so that 
\eqref{eq:assumedsemimoment2}
is fulfilled and the
abstract results
developed in
Subsections~\ref{sec:stabest}--\ref{sec:onestep}
can thus be applied.

To the best of our knowledge,
the idea to restrict 
the Euler-Maruyama approximations
to large subevents of the probability
space which increase to the full 
probability space
as the time discretization 
step-size decreases to $0$
appeared first in~\cite{hj11} 
(see Section~4 in \cite{hj11}
and also \cite{hjk10b}
for details).
The idea to restrict a 
Lyapunov-type 
condition
of the form \eqref{eq:Vstable}
on the one-step function of 
(Metropolis-adjusted) 
Euler-Maruyama schemes
to subsets of the state space 
which increase to the full state space
as the time discretization 
step-size decreases to $ 0 $
appeared first in 
Bou-Rabee \citationand\ 
Hairer~\cite{bh11}
in the setting of the Langevin
equation
(see Lemma~3.5 and
Section~5 
in~\cite{bh11} for details).
In Section~5 in \cite{bh11} 
it is also proved in the setting of the Langevin equation
that the one-step function 
of (Metropolis-adjusted) Euler-Maruyama schemes
satisfy suitable Lyapunov-type conditions
restricted to subsets of the state space
which increase to the full state space
as the time discretization step-size
decreases to $ 0 $ 
(see Proposition~5.2 and Lemma~5.6 in \cite{bh11}).
In this article the Lyapunov-type condition~\eqref{eq:Vstable} is
proved in Subsection~\ref{sec:SVEM}
in the case of the Euler-Maruyama scheme
and in Subsection~\ref{sec:SVtaming}
in the case of suitable tamed methods.
Finally, 
to the best of our knowledge,
a bootstrap
argument 
similar as in \eqref{eq:EVbar}
and Proposition~\ref{p:PG}
appeared first
in \cite{hjk10b}
for a class of drift-tamed 
Euler-Maruyama approximations
(see the proof of
Lemma~3.9 in
\cite{hjk10b} for details).

\section{Explicit 
approximation schemes}
\label{sec:SVstability}

This section investigates 
stability properties and 
moment bounds of explicit
approximation schemes.
We begin with the Euler-Maruyama
scheme in 
Subsection~\ref{sec:SVEM}
and then analyze further 
approximation methods
in Subsections~\ref{sec:SVtaming}
and \ref{sec:momentbounds2}
below.

\subsection{Semi stability for the 
Euler-Maruyama 
scheme}
\label{sec:SVEM}

The main result of this subsection,
Theorem~\ref{thm:SVstability}
below, gives sufficient conditions
for the Euler-Maruyama scheme
to be $ \alpha $-semi 
$ V $-stable
with respect to Brownian motion
with 
$ 
  \alpha \in (0,\infty) 
$,
$ 
  V \colon \mathbb{R}^d \rightarrow
  [0,\infty) 
$
appropriate
and 
$ d \in \mathbb{N} $.
For proving this result,
we first present 
three auxiliary results 
(Lemmas~\ref{lem:SVstabilityI}--\ref{lem:Vest}).
The first one
(Lemma~\ref{lem:SVstabilityI})
establishes
$ \alpha $-semi 
$ V $-stability 
with respect to Brownian
motion
of the Euler-Maruyama 
scheme
with 
$ \alpha \in (0,\infty) $,
$ 
  V \colon \mathbb{R}^d \rightarrow
  [0,\infty) 
$
appropriate
and $ d \in \mathbb{N} $
under a general abstract
condition
(see 
inequality~\eqref{eq:basicass}
below for details).

\begin{lemma}[An abstract
condition for semi
$ V $-stability with respect
to Brownian motion of the 
Euler-Maruyama 
scheme]
\label{lem:SVstabilityI}
Let $ \alpha, T \in (0,\infty) $,
$ d, m \in \mathbb{N} $,
let
$
  \mu \colon \mathbb{R}^d
  \rightarrow \mathbb{R}^d
$,
$
  \sigma 
  \colon \mathbb{R}^d
  \rightarrow \mathbb{R}^{ d \times m } 
$
be Borel measurable functions,
$ \rho, \tilde{\rho} \in \mathbb{R} $,
let 
$ 
  \left( \Omega, \mathcal{F},
  \mathbb{P} \right) 
$
be a probability space,
let
$
  W \colon [0,T] \times \Omega
  \rightarrow \mathbb{R}^m
$
be a standard
Brownian motion
and let
$ 
  V \colon
  \mathbb{R}^d \rightarrow
  [0,\infty)
$
be a twice continuously differentiable
function with
$
  \left( 
    \mathcal{G}_{ \mu, \sigma } 
    V 
  \right)\!( x )
  \leq
  \rho \cdot V(x)
$
and
\begin{equation}
\label{eq:basicass}
  \mathbb{E}\Big[
  \big( 
    \mathcal{\tilde{G}}_{ \mu, \sigma } V 
  \big)
  \big(
    x + \mu(x) \cdot t
    + \sigma(x) W_t ,
    x
  \big)
  -
  \big( 
    \mathcal{G}_{ \mu, \sigma } V 
  \big)(x)
  \Big]
\leq
  \tilde{\rho} \cdot V(x)
\end{equation}
for all
$ 
  (x,t) \in 
  \{
    (y,s) \in 
    \mathbb{R}^d \times
    (0,T]
    \colon
    V(y) \leq s^{ -\alpha } 
  \} 
$.
Then
\begin{equation}
  \mathbb{E}\big[
    V\big(
      x + \mu(x) \cdot t
      + \sigma(x) W_t
    \big)
  \big]
  \leq
  e^{ \left( \rho + \tilde{\rho} \right) t } 
  \cdot
  V(x)
\end{equation}
for all
$ 
  (x,t) \in 
  \{
    (y,s) \in 
    \mathbb{R}^d \times
    (0,T]
    \colon
    V(y) \leq s^{ -\alpha } 
  \} 
$.
In particular,
the Euler-Maruyama scheme
$
  \R^d \times [0,T] \times \R^m
  \ni (x,t,y) 
  \mapsto
  x + \mu( x ) t + \sigma( x ) y
  \in \R^d
$
is $ \alpha $-semi
$ V $-stable
with respect to Brownian motion.
\end{lemma}

\begin{proof}[Proof
of Lemma~\ref{lem:SVstabilityI}]
Ito's formula shows
\begin{equation}
\label{eq:lem_SVstabilityI}
\begin{split}
&
  \mathbb{E}\Big[
    V\big( 
      x + \mu(x) \cdot t + \sigma(x) W_t 
    \big)
  \Big]
\\ & =
  V( x )
  +
  \int_0^t
    \mathbb{E}\Big[
  \big( 
    \mathcal{\tilde{G}}_{ \mu, \sigma } V 
  \big)
  \big(
    x + \mu(x) \cdot t
    + \sigma(x) W_t ,
    x
  \big)
    \Big]
  \, ds
\\ & =
  V( x )
  +
  \big( \mathcal{G}_{ \mu, \sigma } V \big)(x) \cdot t
\\ & \quad
  +
  \int_0^t
    \mathbb{E}\Big[
  \big( 
    \mathcal{\tilde{G}}_{ \mu, \sigma } V 
  \big)
  \big(
    x + \mu(x) \cdot t
    + \sigma(x) W_t ,
    x
  \big)
  -
  \big( 
    \mathcal{G}_{ \mu, \sigma } V 
  \big)(x)
    \Big]
  \, ds
\end{split}
\end{equation}
for all $ x \in \mathbb{R}^d $
and all
$ t \in [0,T] $
and the assumption 
$
  \left( 
    \mathcal{G}_{ \mu, \sigma } 
    V 
  \right)\!( x )
  \leq
  \rho \cdot V(x)
$
hence implies
\begin{equation}
\begin{split}
&
  \mathbb{E}\Big[
    V\big( 
      x + \mu(x) \cdot t + \sigma(x) W_t 
    \big)
  \Big]
\\ & \leq
  V( x )
  +
  \rho \cdot t \cdot V(x)
\\ & \quad
  +
  \int_0^t
    \mathbb{E}\Big[
  \big( 
    \mathcal{\tilde{G}}_{ \mu, \sigma } V 
  \big)
  \big(
    x + \mu(x) \cdot t
    + \sigma(x) W_t ,
    x
  \big)
  -
  \big( 
    \mathcal{G}_{ \mu, \sigma } V 
  \big)(x)
    \Big]
  \, ds
\end{split}
\end{equation}
for all $ x \in \mathbb{R}^d $
and all
$ t \in [0,T] $.
Inequality~\eqref{eq:basicass}
therefore shows
\begin{equation}
\begin{split}
  \mathbb{E}\Big[
    V\big( 
      x + \mu(x) \cdot t + \sigma(x) W_t 
    \big)
  \Big]
& \leq
  V( x )
  +
  \rho \cdot t \cdot V(x)
  +
  \tilde{\rho} \cdot t \cdot V(x)
\\ & =
  V( x )
  \left(
  1
  +
  \rho \cdot t 
  +
  \tilde{\rho} \cdot t 
  \right)
\leq
  e^{ \left( \rho + \tilde{\rho} \right) t }
  \cdot
  V(x)
\end{split}
\end{equation}
for all
$ 
  (x,t) \in 
  \{
    (y,s) \in 
    \mathbb{R}^d \times
    (0,T]
    \colon
    V(y) \leq s^{ -\alpha } 
  \} 
$.
The proof of
Lemma~\ref{lem:SVstabilityI} is thus
completed.
\end{proof}

In many situations,
it is tedious
to verify \eqref{eq:basicass}.
We therefore give
more concrete conditions
for $ \alpha $-semi 
$ V $-stability
with respect to Brownian motion
of the Euler-Maruyama scheme
with 
$ \alpha \in (0,\infty) $,
$ 
  V \colon \mathbb{R}^d \rightarrow
  [0,\infty) 
$
appropriate
and $ d \in \mathbb{N} $
in Theorem~\ref{thm:SVstability}
below.
For establishing this theorem,
the next simple lemma 
is used.

\begin{lemma}
\label{lem:funcest}
Let
$
  T \in (0,\infty),
  c \in [0,\infty),
  p \in [1,\infty)
$
be real numbers
and let
$
  y \colon [0,T]
  \rightarrow \mathbb{R}
$
be an absolute continuous
function
with
$
  y'(t) \leq
  c 
  \left| 
    y(t) 
  \right|^{ (1 - 1/p) }
$
for 
$ \lambda_{ [0,T] } $-almost
all $ t \in [0,T] $.
Then
\begin{equation}
\label{eq:yest}
  y(t)
\leq 
  \left[
    \left| y(0) \right|^{ 
      1 / p	 
    }
    +
    \tfrac{ c t }{ p }
  \right]^{
    p
  }
\leq
  2^{
    \left( p - 1 \right)
  }
  \bigg[
    \left| y(0) \right|
    +
    \left|
      \tfrac{ c t }{ p }
    \right|^{
      p 
    } 
  \bigg] 
\end{equation}
and
\begin{equation}
\label{eq:yest22}
\begin{split}
  y(t)
& \leq 
  \left| y(0) \right|
  +
  c \, t \,
  \big[
    \left| y(0) \right|^{
      1 / p
    }
    +
    c \, t \,
  \big]^{
    \left( p - 1 \right) 
  }
\\ & \leq
  \left| y(0) \right|
  +
  2^{
    \left( p - 1 \right)
  }
  \left[
    c \, t 
    \left| y(0) \right|^{
      \left( 1 - 1/p \right)
    }
    +
    \left|
      c \, 
      t
    \right|^{ p }
  \right]
\end{split}
\end{equation}
for all $ t \in [0,T] $.
\end{lemma}

\begin{proof}[Proof
of Lemma~\ref{lem:funcest}]
The assumption
$
  y'(t) \leq
  c 
  \left| 
    y(t) 
  \right|^{ (1 - 1/p) }
$
for 
$ \lambda_{ [0,T] } $-almost
all $ t \in [0,T] $
implies
\begin{equation}
\label{eq:yest2}
  \int_{ t_0 }^t
  y'(s) 
  \left(
    y(s)
  \right)^{ \left( 1 / p - 1 \right) }
  ds
  \leq 
  c \left( t - t_0 \right)
\end{equation}
for all 
$ 
  t_0 \in ( \tau(t), t ] 
$
and all
$ t \in [0,T] $
with $ y(t) > 0 $
where the function
$
  \tau
  \colon
  [0,T]
  \rightarrow [0,T] 
$
is defined through
\begin{equation}
  \tau(t)
:=
  \max\!\Big(
    \{ 0 \}
    \cap
    \big\{
      s \in [0,t]
      \colon
      y(s) = 0
    \big\}
  \Big)
\end{equation}
for all $ t \in [0,T] $.
Estimate~\eqref{eq:yest2}
then gives
\begin{equation}
  \big( y(t) \big)^{
    1 / p
  }
\leq
  \big( y(t_0) \big)^{
    1 / p
  }
  +
  \frac{ c \left( t - t_0 \right) }{ p }
\end{equation}
for all 
$ 
  t_0 \in ( \tau(t), t ] 
$
and all
$ t \in [0,T] $
with $ y(t) > 0 $.
This yields
\begin{equation}
\label{eq:yest3}
\begin{split}
  \big( y(t) \big)^{
    1 / p
  }
& \leq
  \big| y\big( \tau(t) \big) \big|^{
    1 / p
  }
  +
  \frac{
    c 
    \left( t - \tau(t) \right)
  }{
    p
  }
\leq
  \big| y(0) \big|^{
    1 / p
  }
  +
  \frac{ c t }{ p }
\end{split}
\end{equation}
for all $ t \in [0,T] $
with $ y(t) > 0 $.
Inequality~\eqref{eq:yest3}
implies \eqref{eq:yest}.
In addition, note that
combining \eqref{eq:yest}
and
the inequality
\begin{equation}
\begin{split}
  \left( x + y \right)^r
& =
  x^r
  +
  \int_0^1 
    r
    \left( x + s y \right)^{ ( r - 1 ) }
    y \,
  ds
\leq
  x^r 
  +
  r \, y
  \left( x + y \right)^{ ( r - 1 ) }
\end{split}
\end{equation}
for all $ x, y \in [0,\infty) $,
$ r \in [1,\infty) $
shows \eqref{eq:yest22}.
The proof of Lemma~\ref{lem:funcest}
is thus completed.
\end{proof}

\noindent
An immediate consequence of
Lemma~\ref{lem:funcest} 
are the following estimates.

\begin{lemma}
\label{lem:Vest}
Let 
$
  c,p \in [1,\infty)
$
be real numbers
and let
$
  V \in C^1_p( \mathbb{R}^d, 
  \mathbb{R} )
$
with
$
  \left\| (\nabla V)(x) \right\|
  \leq
  c \left| V(x) \right|^{ ( 1 - 1 / p ) }
$
for 
$ \lambda_{ \mathbb{R}^d } $-almost
all $ x \in \mathbb{R}^d $.
Then
\begin{align}
\label{eq:Vest1}
  V( x + y )
&\leq
  c^p \,
  2^{ 
    ( p - 1 )
  }
  \Big(
    | V(x) |
    +
    \| y \|^{
      p
    }
  \Big) ,
\\
\label{eq:Vest2}
  V( x + y )
&\leq
  \left| V(x) \right|
  +
  c^p \,
  2^{ 
    ( p - 1 )
  }
  \left(
    \| y \|
    \left| V(x) \right|^{
      ( 1 - 1 / p ) 
    }
    +
    \left\| y \right\|^{ p }
  \right)
\end{align}
for all $ x, y \in \mathbb{R}^d $.
\end{lemma}

\begin{proof}[Proof
of Lemma~\ref{lem:Vest}]
First of all, note that
the assumption
$
  \left\| (\nabla V)(x) \right\|
  \leq
  c \left| V(x) \right|^{ ( 1 - 1 / p ) }
$
for 
$ \lambda_{ \mathbb{R}^d } $-almost
all $ x \in \mathbb{R}^d $
implies
\begin{equation}
\begin{split}
  \frac{ d }{ d t }
  \Big(
    V( x + t y )
  \Big)
& =
  V'\!\left( x + t y \right) y
\leq
  \left\| 
    \left( \nabla
    V \right)\!\left( x + t y \right) 
  \right\|
  \left\|
    y
  \right\|
\\ & \leq
  c
  \left| 
    V( x + t y ) 
  \right|^{ 
    ( 1 - 1/p ) 
  }
  \left\|
    y
  \right\|
\end{split}
\label{eq:V_assumption}
\end{equation}
for 
$ \lambda_{ \mathbb{R} } $-almost
all $ t \in \mathbb{R} $
and all $ x, y \in \mathbb{R}^d $.
Inequality~\eqref{eq:yest22}
in Lemma~\ref{lem:funcest}
hence gives
\eqref{eq:Vest2}.
Moreover, combining \eqref{eq:V_assumption}
and
\eqref{eq:yest}
implies
\begin{equation}
\begin{split}
  V( x + y )
& \leq	
  2^{ 
    \left( p - 1 \right)
  }
  \Big(
    \left| V(x) \right|
    +
    \left|
      \tfrac{ c \left\| y \right\| }{ p }
    \right|^{
      p 
    }
  \Big)
\leq
  c^p \,
  2^{ 
    ( p - 1 )
  }
  \Big(
    \left| V(x) \right|
    +
    \left\| y \right\|^{
      p
    }
  \Big)
\end{split}
\end{equation}
for all $ x, y \in \mathbb{R}^d $.
The proof
of Lemma~\ref{lem:Vest}
is thus completed.
\end{proof}

We are now ready to present the
promised theorem which shows
$ \alpha $-semi $ V $-stability 
with respect to Brownian motion
of the
Euler-Maruyama scheme
with 
$ \alpha \in (0,\infty) $,
$ 
  V \colon \mathbb{R}^d \rightarrow
  [0,\infty) 
$
appropriate
and $ d \in \mathbb{N} $.
It makes use of 
Lemma~\ref{lem:SVstabilityI}
and Lemma~\ref{lem:Vest}
above.

\begin{theorem}[Semi $V$-stability 
with respect to Brownian motion for 
the Euler-Maruyama
scheme]
\label{thm:SVstability}
Let $ T \in (0,\infty) $,
$ d, m \in \mathbb{N} $,
$
  p \in [3,\infty)
$,
$
  c,
  \gamma_0,
  \gamma_1
  \in [0,\infty)
$
be real numbers
with
$ 
  \gamma_0 + \gamma_1 > 0 
$,
let
$
  \mu \colon \mathbb{R}^d
  \rightarrow \mathbb{R}^d
$,
$
  \sigma 
  \colon \mathbb{R}^d
  \rightarrow \mathbb{R}^{ d \times m }
$
be Borel measurable functions and
let
$
  V \in C^3_p( \mathbb{R}^d, [1,\infty) )
$
with
$
  ( \mathcal{G}_{ \mu, \sigma } V)(x)
\leq
  c \cdot V(x) 
$
and
\begin{equation}
\label{eq:mu_sigma_Vgrowth_assumption}
  \left\|
    \mu(x)
  \right\|
  \leq
  c \,
  |
    V(x)
  |^{
    \left[
      \frac{ \gamma_0 + 1 }{ p }
    \right]
  } 
\qquad
  \text{and}
\qquad 
  \|
    \sigma(x)
  \|_{
    L( \mathbb{R}^m, \mathbb{R}^d )
  }
  \leq
  c \,
  |
    V(x)
  |^{ 
    \left[
      \frac{ \gamma_1 + 2 }{ 2p }
    \right]
  } 
\end{equation}
for all $ x \in \mathbb{R}^d $.
Then the Euler-Maruyama 
scheme 
$
  \mathbb{R}^d \times [0,T]
  \times \mathbb{R}^m
  \ni (x,t,y) 
  \mapsto
  x + \mu(x) t + \sigma(x) y 
  \in \mathbb{R}^d
$
is 
$ 
  p / (
    \gamma_1 + 
    2 ( \gamma_0 \vee \gamma_1 ) 
  )
$-semi
$ V $-stable
with respect to Brownian motion.
\end{theorem}

\begin{proof}[Proof
of Theorem~\ref{thm:SVstability}]
Throughout this proof, let 
$ 
  \left( \Omega, \mathcal{F},
  \mathbb{P} \right) 
$
be a probability space and
let
$
  W \colon [0,T] \times \Omega
  \rightarrow \mathbb{R}^m
$
be a standard
Brownian motion.
We will prove 
semi $ V $-stability
with respect to Brownian
motion
for the Euler-Maruyama scheme
by applying 
Lemma~\ref{lem:SVstabilityI}.
Our aim is thus to 
verify~\eqref{eq:basicass}.
For this note that
\begin{equation}
\begin{split}
\lefteqn{
  \mathbb{E}\Big[
  \big|
  ( \mathcal{\tilde{G}}_{ 
    \mu, \sigma } V 
  )
  \big(
    x + \mu(x) \cdot t
    + \sigma(x) W_t, x
  \big)
  -
  ( \mathcal{G}_{ \mu, \sigma } V )
  (
    x
  )
  \big|
  \Big] 
}
\\ & \leq
  \mathbb{E}\Big[
    \big\|
      V'\big( 
        x + \mu(x) \cdot t
        + \sigma(x) W_t
      \big)
      - V'(x)
    \big\|_{ 
      L( \mathbb{R}^d, \mathbb{R} )
    }
  \Big]
  \left\| \mu(x) \right\|
\\ & +  
  \frac{ 1 }{ 2 } \cdot
  \mathbb{E}\Big[
    \big\|
      V''\big( 
        x + \mu(x) \cdot t
        + \sigma(x) W_t
      \big)
      - V''(x)
    \big\|_{ 
      L^{ (2) }( \mathbb{R}^d, \mathbb{R} )
    }
  \Big]
  \left(
    \sum_{ k = 1 }^{ m }
    \left\| \sigma_k(x) \right\|^2
  \right)
\end{split}
\end{equation}
for all
$
  x \in \mathbb{R}^d
$
and all
$ t \in [0,T] $.
Assumption~\eqref{eq:mu_sigma_Vgrowth_assumption}
hence shows
\begin{equation}
\label{eq:derivativeVV1}
\begin{split}
&
  \mathbb{E}\Big[
  \big|
  ( \mathcal{\tilde{G}}_{ 
    \mu, \sigma } V 
  )
  \big(
    x + \mu(x) \cdot t
    + \sigma(x) W_t, x
  \big)
  -
  ( \mathcal{G}_{ \mu, \sigma } V )
  (
    x
  )
  \big|
  \Big] 
\\ & \leq
  m \left( c + 1 \right)^2 
\\ & 
  \cdot
  \left(
  \sum_{ i = 1 }^{ 2 }
  \mathbb{E}\Big[
    \big\|
      V^{ (i) }\big( 
        x + \mu(x) \cdot t
        + \sigma(x) W_t
      \big)
      - V^{ (i) }(x)
    \big\|_{ 
      L^{ (i) }( \mathbb{R}^d, \mathbb{R} )
    }
  \Big]
  \left| V(x) 
  \right|^{
    \left[
      \frac{ \gamma_{ (i - 1) } + i }{ p }
    \right]
  }
  \right)
\end{split}
\end{equation}
for all
$
  x \in \mathbb{R}^d
$
and all
$ t \in [0,T] $.
Next observe that
the assumption
\begin{equation}
  V \in C_p^3( \mathbb{R}^d,
  [1,\infty) ) 
\end{equation}
and estimate~\eqref{eq:Vest1}
imply the existence
of a real number 
$ \hat{c} \in [c + 1,\infty) $
such that
\begin{equation}
\label{eq:derivativeV1}
\begin{split}
&
  \left\| 
    V^{(i)}(y) - 
    V^{(i)}(x)
  \right\|_{
    L^{(i)}( \mathbb{R}^d, 
      \mathbb{R} 
    )
  }
\\ & \leq
  \int_0^1
  \left\| 
    V^{(i+1)}\big( x + r (y - x ) \big)
  \right\|_{
    L^{(i+1)}( \mathbb{R}^d, 
      \mathbb{R} 
    )
  }  
  \left\| y - x \right\|
  dr
\\ & \leq
  \hat{c}
  \left(
  \int_0^1
  \left| 
    V\big( x + r (y - x ) \big)
  \right|^{
    \frac{ (p - i - 1) }{ p }
  }
  dr
  \right)
  \left\| y - x \right\|
\\ & \leq
  \left| \hat{c} \right|^{ (p+1) }
  2^p 
  \left(
  \left| 
    V( x )
  \right|^{
    \frac{ (p - i - 1) }{ p }
  }
  +
  \left\| y - x
  \right\|^{
    \left( p - i - 1 \right)
  }
  \right)
  \left\| y - x \right\|
\\ & =
  \left| \hat{c} \right|^{ (p+1) }
  2^p 
  \left(
  \left| 
    V( x )
  \right|^{
    \frac{ (p - i - 1) }{ p }
  }
  \left\| y - x \right\|
  +
  \left\| y - x
  \right\|^{
    \left( p - i \right)
  }
  \right)
\end{split}
\end{equation}
for all 
$ x, y \in \mathbb{R}^d $
and all
$ i \in \{ 1, 2 \} $.
Putting \eqref{eq:derivativeV1}
into \eqref{eq:derivativeVV1}
then results in
\begin{equation}
\label{eq:derivativeVV3}
\begin{split}
\lefteqn{
  \mathbb{E}\Big[
  \big|
  ( \mathcal{\tilde{G}}_{ 
    \mu, \sigma } V 
  )
  \big(
    x + \mu(x) \cdot t
    + \sigma(x) W_t, x
  \big)
  -
  ( \mathcal{G}_{ \mu, \sigma } V )
  (
    x
  )
  \big|
  \Big] 
}
\\ & \leq
  m 
  \left| \hat{c} \right|^{ (p+3) }
  2^p 
  \left(
  \sum_{ i = 1 }^{ 2 }
  \mathbb{E}\Big[
    \|
      \mu(x) \cdot t
      + \sigma(x) W_t
    \|
  \Big]
  \left| V(x) 
  \right|^{
    \left[
      \frac{ \gamma_{ (i - 1) } + p - 1 }{ p }
    \right]
  }
  \right)
\\ & \quad
  +
  m 
  \left| \hat{c} \right|^{ (p+3) }
  2^p 
  \left(
  \sum_{ i = 1 }^{ 2 }
  \mathbb{E}\Big[
    \|
      \mu(x) \cdot t
      + \sigma(x) W_t
    \|^{ (p-i) }
  \Big]
  \left| V(x) 
  \right|^{
    \left[
      \frac{ \gamma_{ (i - 1) } + i }{ p }
    \right]
  }
  \right)
\end{split}
\end{equation}
for all
$
  x \in \mathbb{R}^d
$
and all
$ t \in [0,T] $.
In the next step, we 
put the estimate
\begin{equation}
\label{eq:y_est}
\begin{split}
&
  \mathbb{E}\Big[
    \big\|
      \mu(x) \cdot t
      + \sigma(x) W_t
    \big\|^{ 
      r 
    }
  \Big]
\\ & \leq
  \left( m + 1 \right)^r
  \left(
    \| \mu(x) \|^r \cdot t^r
    +
    \sum_{ k = 1 }^{ m }
    \| \sigma_k(x) \|^r \cdot 
    \mathbb{E}\Big[
      \big| W^{(k)}_t
      \big|^{ 
        r 
      }
    \Big]
  \right)
\\ & \leq
  c^r
  \left( m + 1 \right)^{ (r + 1) }
  \Big(
    \left| V(x) \right|^{
      \left[ \frac{r \gamma_0 + r}{p} \right]
	  }
    t^r
    +
    \left| V(x) \right|^{
      \left[ 
        \frac{ r \gamma_1 + 2r }{ 2p }
      \right] 
	  }
    \mathbb{E}\Big[
      \big| W^{(1)}_t
      \big|^{ 
        r 
      }
    \Big]
  \Big)
\\ & \leq
  c^r
  \left( m + r + 1 \right)^{ (2 r + 1) }
  \Big(
    \left| V(x) \right|^{
      \left[ \frac{r \gamma_0 + r}{p} \right]
	  }
    t^r
    +
    \left| V(x) \right|^{
      \left[ 
        \frac{ r \gamma_1 + 2r }{ 2p }
      \right] 
	  }
    t^{ \frac{ r }{ 2 } }
  \Big)
\\ & =
  c^r
  \left( m + r + 1 \right)^{ (2 r + 1) }
  \left(
    \sum_{ j = 1 }^{ 2 }
    \left| V(x) \right|^{
      \left[ 
        \frac{ r \gamma_{(j-1)} + jr }{j p }
      \right]
	  }
    t^{ \frac{ r }{ j } }
  \right)
\end{split}
\end{equation}
for all 
$ x \in \mathbb{R}^d $,
$ t \in [0,T] $,
$ r \in [0,\infty) $
into
\eqref{eq:derivativeVV3} 
to obtain
\begin{equation}
\begin{split}
\lefteqn{
  \mathbb{E}\Big[
  \big|
  ( \mathcal{\tilde{G}}_{ 
    \mu, \sigma } V 
  )
  \big(
    x + \mu(x) \cdot t
    + \sigma(x) W_t, x
  \big)
  -
  ( \mathcal{G}_{ \mu, \sigma } V )
  (
    x
  )
  \big|
  \Big] 
}
\\ & \leq
  \left( m + p \right)^{ 3 p }
  \left| \hat{c} \right|^{ (2 p+3) }
  2^p 
  \left(
  \sum_{ i, j = 1 }^{ 2 }
  \left| V(x) 
  \right|^{
    \left[
      \frac{ \gamma_{ (i - 1) } + \gamma_{ (j - 1) }/j }
           { p }
      + 
      1
    \right]
  }
  t^{ \frac{ 1 }{ j } }
  \right)
\\ & \quad
  +
  \left( m + p \right)^{ 3 p }
  \left| \hat{c} \right|^{ (2 p + 3) }
  2^p 
  \left(
  \sum_{ i, j = 1 }^{ 2 }
  \left| V(x) 
  \right|^{
    \left[
      \frac{
      \gamma_{ (i - 1) }
      +
        ( p - i )
        \gamma_{ (j - 1) } / j
      }{ p }
      +
      1
    \right]
  }
  t^{ \frac{ (p - i) }{ j } }
  \right)
\end{split}
\end{equation}
and hence
\begin{equation}
\label{eq:derivativeVV4}
\begin{split}
&
  \mathbb{E}\Big[
  \big|
  ( \mathcal{\tilde{G}}_{ 
    \mu, \sigma } V 
  )
  \big(
    x + \mu(x) \cdot t
    + \sigma(x) W_t, x
  \big)
  -
  ( \mathcal{G}_{ \mu, \sigma } V )
  (
    x
  )
  \big|
  \Big] 
\leq
  \left( m + p \right)^{ 3 p }
  \left| \hat{c} \right|^{ (2 p+3) }
  2^p 
\\ & \cdot
  \left(
  \sum_{ i, j = 1 }^{ 2 }
  \left[
  \left| V(x) 
  \right|^{
    \left[
      \frac{ \gamma_{ (i - 1) } + \gamma_{ (j - 1) }/j }
           { p }
    \right]
  }
  t^{ \frac{ 1 }{ j } }
  +
  \left| V(x) 
  \right|^{
    \left[
      \frac{
      \gamma_{ (i - 1) }
      +
        ( p - i )
        \gamma_{ (j - 1) } / j
      }{ p }
    \right]
  }
  t^{ \frac{ (p - i) }{ j } }
  \right]
  \right)
  V(x)
\end{split}
\end{equation}
for all
$
  x \in \mathbb{R}^d
$
and all
$ t \in [0,T] $.
This implies that there exists
a real number 
$ \tilde{ \rho } \in [0, \infty) $
such that
\begin{equation}
  \mathbb{E}\Big[
  \big|
  ( \mathcal{\tilde{G}}_{ 
    \mu, \sigma } V 
  )
  \big(
    x + \mu(x) \cdot t
    + \sigma(x) W_t, x
  \big)
  -
  ( \mathcal{G}_{ \mu, \sigma } V )
  (
    x
  )
  \big|
  \Big] 
\leq
  \tilde{ \rho }
  \cdot V(x)
\end{equation}
for all 
$ 
  (x,t) \in 
  \{
    (y,s) \in 
    \mathbb{R}^d \times
    (0,T]
    \colon
    V(y) \leq s^{ -\alpha } 
  \} 
$
where
\begin{equation}
\begin{split}
  \alpha 
& :=
  \min_{ i, j \in \{ 1, 2 \} } \!
  \left(
  \min\!
  \left\{
  \frac{
    p 
  }{
    j 
    \cdot
    \left(
      \gamma_{ (i-1) }
      +
      \frac{ \gamma_{(j-1)} }{ j }
    \right)
  } ,
  \frac{
    p \cdot \left( p - i \right) 
  }{
    j 
    \cdot
    \left(
      \gamma_{ (i-1) }
      +
      \frac{
        \left( p - i \right) 
        \gamma_{(j-1)} 
      }{ j }
    \right)
  }
  \right\}
  \right) .
\end{split}
\end{equation}
Lemma~\ref{lem:SVstabilityI}
therefore shows
that the Euler-Maruyama 
scheme is 
$ 
  \alpha
$-semi
$ V $-stable
with respect to Brownian motion, 
i.e.,
there exists
a real number
$ \rho \in \mathbb{R} $
such that
\begin{equation}
  \mathbb{E}\big[
    V(
      x + \mu(x) \cdot t + \sigma(x) W_t
    )
  \big]
\leq
  e^{ \rho t } \cdot V(x)
\end{equation}
for all 
$ 
  (x,t) \in 
  \{
    (y,s) \in 
    \mathbb{R}^d \times
    (0,T]
    \colon
    V(y) \leq s^{ -\alpha } 
  \} 
$.
Finally, note that
\begin{equation}
\begin{split}
  \alpha
 & =
  \min_{ i, j \in \{ 1, 2 \} } \!
  \left(
  \min\!
  \left\{
  \frac{
    p 
  }{
    j 
    \cdot
    \left(
      \gamma_{ (i-1) }
      +
      \frac{ \gamma_{(j-1)} }{ j }
    \right)
  } ,
  \frac{
    p
  }{
    j 
    \cdot
    \left(
      \frac{ 
        \gamma_{ (i-1) }
      }{ (p - i) }
      +
      \frac{
        \gamma_{(j-1)} 
      }{ j }
    \right)
  }
  \right\}
  \right) 
\\ & =
  \min_{ i, j \in \{ 1, 2 \} } \!
  \left(
  \frac{
    p 
  }{
    \left(
      j \cdot
      \gamma_{ (i-1) }
      +
      \gamma_{(j-1)} 
    \right)
  } 
  \right) 
  =
  \frac{ p }
  {
    \max_{ j \in \{1,2\} }\!
    \Big(
      j \cdot
      \max\!\left(
        \gamma_{ 0 }, \gamma_1
      \right) 
      +
      \gamma_{ (j-1) } 
    \Big)
  }
\\ & =
  \frac{ p }{
      \gamma_{ 1 } +
      2 
      \max\!\left(
        \gamma_{ 0 }, \gamma_1
      \right)
  } .
\end{split}
\end{equation}
This completes 
the proof of 
Theorem~\ref{thm:SVstability}.
\end{proof}

If 
$ 
  V \colon \mathbb{R}^d
  \rightarrow [1,\infty)
$ 
is a Lyapunov-type 
function
in the sense of 
Theorem~\ref{thm:SVstability}
with $ d \in \mathbb{N} $,
then the function
$
  \mathbb{R}^d \ni x 
  \mapsto 
  \left| V(x) \right|^q
  \in [1,\infty)
$
with
$ q \in (0,\infty) $
appropriate
is in many situations
a Lyapunov-type 
function too.
This is the subject of 
Corollary~\ref{cor:SVstability} below.
The next result is a simple lemma
which will be used in the proof of 
Corollary~\ref{cor:SVstability}.

\begin{lemma}
\label{lem:potencies}
Let 
$ q \in [1,\infty) $,
$ p \in [2,\infty) $,
$ d \in \mathbb{N} $
and let
$
  V \in C^3_p( \mathbb{R}^d, [1,\infty) ) 
$.
Then the function
$ 
  \hat{V}
  \colon 
  \mathbb{R}^d
  \rightarrow [1,\infty)
$
given by
$
  \hat{V}( x ) = 
  \left( V(x) \right)^q
$
for all $ x \in \mathbb{R}^d $
satisfies
$ 
  \hat{V}
  \in C^3_{ p q }( \mathbb{R}^d, [1,\infty) )
$.
\end{lemma}

\begin{proof}[Proof
of Lemma~\ref{lem:potencies}]
Note that
\begin{equation}
  \hat{V}'(x)(v_1)
=
  q
  \left( V(x) \right)^{ (q - 1) }
  V'(x)(v_1) ,
\end{equation}
\begin{equation}
\begin{split}
&
  \hat{V}''(x)(v_1, v_2)
\\ & =
  q \left( q - 1 \right)
  \left( V(x) \right)^{ (q - 2) } 
  V'(x)(v_1)
  \,
  V'(x)(v_2)
  +
  q
  \left( V(x) \right)^{ (q - 1) }
  V''(x)(v_1, v_2)
\end{split}
\end{equation}
for all 
$ 
  x, v_1, v_2 \in
  \mathbb{R}^d
$
and
\begin{equation}
\begin{split}
& 
  \hat{V}^{(3)}(x)(v_1, v_2, v_3)
\\ & =
  q 
  \left( q - 1 \right)
  \left( q - 2 \right)
  \left( V(x) \right)^{ (q - 3) } 
  V'(x)(v_1)
  \,
  V'(x)(v_2)
  \,
  V'(x)(v_3)
\\ & \quad
  +
  q \left( q - 1 \right)
  \left( V(x) \right)^{ (q - 2) }
  V''(x)(v_1, v_2)
  \,
  V'(x)( v_3 )
\\ & \quad +
  q \left( q - 1 \right)
  \left( V(x) \right)^{ (q - 2) }
  V''(x)(v_1, v_3)
  \,
  V'(x)( v_2 )
\\ & \quad
  +
  q \left( q - 1 \right)
  \left( V(x) \right)^{ (q - 2) }
  V'(x)(v_1)
  \,
  V''(x)( v_2, v_3 )
\\ & \quad
  +
  q
  \left( V(x) \right)^{ (q - 1) }
  V^{(3)}(x)(v_1, v_2, v_3)
\end{split}
\end{equation}
for all 
$ v_1, v_2, v_3 \in \mathbb{R}^d 
$
and
all
$ 
  x \in \{
    y \in \mathbb{R}^d \colon
    V'' \text{ is differentiable in } y
  \}
$.
The assumption
$
  V \in C^3_p( \mathbb{R}^d, [1,\infty) ) 
$
therefore
shows that 
$
  \hat{V} 
  \in C^3_{ p q }( \mathbb{R}^d, [1,\infty) ) 
$.
The proof of
Lemma~\ref{lem:potencies}
is thus completed.
\end{proof}

\begin{cor}[Powers of
the Lyapunov-type function]
\label{cor:SVstability}
Let 
$ T \in (0,\infty) $,
$ d, m \in \mathbb{N} $,
$
  p \in [3,\infty)
$,
$
  q \in [1,\infty)
$,
$
  c,
  \gamma_0,
  \gamma_1
  \in [0,\infty)
$
be real numbers
with
$ 
  \gamma_0 + \gamma_1 > 0 
$,
let
$
  \mu \colon \mathbb{R}^d
  \rightarrow \mathbb{R}^d
$,
$
  \sigma 
  \colon \mathbb{R}^d
  \rightarrow \mathbb{R}^{ d \times m }
$
be Borel measurable 
functions
and 
let
$ 
  V \in C^3_p( \mathbb{R}^d, [1,\infty) )
$
with
$
  \left\|
    \mu(x)
  \right\|
  \leq
  c \,
  |
    V(x)
  |^{
    \left[
      \frac{ \gamma_0 + 1 }{ p }
    \right]
  }
$,
$
  \|
    \sigma(x)
  \|_{
    L( \mathbb{R}^m, \mathbb{R}^d )
  }
  \leq
  c \,
  |
    V(x)
  |^{ 
    \left[
      \frac{ \gamma_1 + 2 }{ 2p }
    \right]
  } 
$
and
\begin{equation}
\label{eq:generatorplus}
  ( \mathcal{G}_{ \mu, \sigma } V)(x)
  + 
  \frac{
    \left( q - 1 \right)
    \| 
      V'(x) \sigma(x) 
    \|^2_{ 
      HS( \mathbb{R}^m, \mathbb{R} ) 
    }
  }{
    2 \cdot V(x)
  }
  \leq c \cdot V(x) 
\end{equation}
for  
all
$ x \in \mathbb{R}^d $.
Then the Euler-Maruyama 
scheme 
\begin{equation}
  \mathbb{R}^d \times [0,T]
  \times \mathbb{R}^m
  \ni (x,t,y) 
  \mapsto
  x + \mu(x) t + \sigma(x) y 
  \in \mathbb{R}^d
\end{equation}
is 
$ 
  p q / (
    \gamma_1 + 
    2 ( \gamma_0 \vee \gamma_1 ) 
  )
$-semi
$
  \left| V \right|^q 
$-stable
with respect to Brownian motion.
\end{cor}

\begin{proof}[Proof of 
Corollary~\ref{cor:SVstability}]
First, define the function
$ 
  \hat{V} \colon \mathbb{R}^d
  \rightarrow [1,\infty)
$
by
$
  \hat{V}(x) 
  = 
  \left( V(x) \right)^{ q }
$
for all $ x \in \mathbb{R}^d $.
Then note that this definition
ensures
\begin{equation}
\label{eq:growthplus2}
  \left\|
    \mu(x)
  \right\|
  \leq
  c \,
  |
    \hat{V}(x)
  |^{
    \left[
      \frac{ \gamma_0 + 1 }{ p q }
    \right]
  }
\qquad
  \text{and}
\qquad
  \|
    \sigma(x)
  \|_{
    L( \mathbb{R}^m, \mathbb{R}^d )
  }
  \leq
  c \,
  |
    \hat{V}(x)
  |^{
    \left[
       \frac{ \gamma_1 + 2 }{ 2 p q }
    \right]
  } 
\end{equation}
for all $ x \in \mathbb{R}^d $.
In addition, observe that
\begin{equation}
  \big(
    \mathcal{G}_{ \mu, \sigma } 
    ( V^r )
  \big)(x)
=
  r
  \left( V(x) \right)^{ ( r - 1 ) }
  \left(
  ( \mathcal{G}_{ \mu, \sigma } V)(x)
  + 
  \frac{
    \left( r - 1 \right)
    \left\| 
      V'(x) \sigma(x) 
    \right\|^2_{ 
      HS( \mathbb{R}^m, \mathbb{R} ) 
    }
  }{
    2 \cdot V(x)
  }
  \right)
\end{equation}
for all $ x \in \mathbb{R}^d $
and all $ r \in (0,\infty) $.
Therefore, we obtain
\begin{equation}
  (
    \mathcal{G}_{ \mu, \sigma } 
    \hat{V}
  )(x)
=
  q
  \left( V(x) \right)^{ ( q - 1 ) }
  \left(
  ( \mathcal{G}_{ \mu, \sigma } V)(x)
  + 
  \frac{
    \left( q - 1 \right)
    \left\| 
      V'(x) \sigma(x) 
    \right\|^2_{ 
      HS( \mathbb{R}^m, \mathbb{R} ) 
    }
  }{
    2 \cdot V(x)
  }
  \right)
\end{equation}
for all $ x \in \mathbb{R}^d $
and 
\eqref{eq:generatorplus}
hence gives
\begin{equation}
\label{eq:generatorplus2}
  (
    \mathcal{G}_{ \mu, \sigma } 
    \hat{V}
  )(x)
\leq
  q \cdot
  c \cdot \hat{V}(x)
\end{equation}
for all $ x \in \mathbb{R}^d $.
Combining \eqref{eq:growthplus2},
\eqref{eq:generatorplus2},
Lemma~\ref{lem:potencies}
and
Theorem~\ref{thm:SVstability}
then shows that
the Euler-Maruyama 
scheme is 
$ 
  \frac{ p q }{
    \gamma_1 + 
    2 ( \gamma_0 \vee \gamma_1 ) 
  }
$-semi
$ \hat{V} $-stable
with respect to Brownian motion.
The proof of 
Corollary~\ref{cor:SVstability}
is thus completed.
\end{proof}

Note that in \eqref{eq:generatorplus}
the norm in the Hilbert space
of {\it Hilbert-Schmidt operators}
from $ \R^m $ to $ \R $ is used
where $ m \in \N $.
The definition of that norm
and more details on Hilbert-Schmidt
operators can, e.g., be found in
Appendix~B
in Pr\'{e}v\^{o}t \citationand\ 
R\"{o}ckner~\cite{pr07}.
In the next step,
Theorem~\ref{thm:SVstability}
is illustrated by a simple 
example.
More precisely,
the next corollary
considers the special 
Lyapunov-type function
$ 
  V \colon \mathbb{R}^d
  \rightarrow [1,\infty)
$
given by
$
  V(x) = 1 + \left\| x \right\|^p
$
for all 
$ x \in \mathbb{R}^d $
with 
$ p \in [3,\infty) $
and
$ d \in \mathbb{N} $.

\begin{cor}[A special polynomial
like Lyapunov-type function]
\label{cor:SVstability2}
Let 
$ d, m \in \mathbb{N} $,
$ T \in (0,\infty) $,
$ 
  c,
  \gamma_0,
  \gamma_1
  \in [0,\infty) 
$,
$ p \in [3,\infty) $
be real numbers
with
$ 
  \gamma_0 + \gamma_1 > 0
$,
let
$
  \mu \colon \mathbb{R}^d
  \rightarrow \mathbb{R}^d
$,
$
  \sigma 
  \colon \mathbb{R}^d
  \rightarrow \mathbb{R}^{ d \times m }
$
be Borel measurable functions with
\begin{equation}
\label{eq:p-generator}
  \left< x, \mu(x) \right>
  + 
  \tfrac{ \left( p - 1 \right) }{ 2 }
  \| \sigma(x) 
  \|^2_{
    HS( \mathbb{R}^m, \mathbb{R}^d )
  }
  \leq 
  c \left(
    1 + \left\| x \right\|^2
  \right) ,
\end{equation}
\begin{equation}
  \left\|
    \mu(x)
  \right\|
  \leq
  c \,
  \big( 
    1 + 
    \left\| x \right\|^{ 
      \left[ \gamma_0 + 1 \right]
    }
  \big)
\qquad
  \text{and}
\qquad
  \| 
    \sigma(x) 
  \|_{ 
    L( \mathbb{R}^m , \mathbb{R}^d ) 
  }
  \leq
  c \,
  \big( 
    1 + 
    \left\| x \right\|^{ 
      \left[ \frac{ \gamma_1 + 2 }{ 2 } \right]
    }
  \big)
\end{equation}
for all
$ 
  x \in \mathbb{R}^d  
$.
Then the Euler-Maruyama 
scheme 
\begin{equation}
  \mathbb{R}^d \times [0,T]
  \times \mathbb{R}^m
  \ni (x,t,y) 
  \mapsto
  x + \mu(x) t + \sigma(x) y 
  \in \mathbb{R}^d
\end{equation}
is 
$
   p / (
    \gamma_1 + 
    2 ( \gamma_0 \vee \gamma_1 ) 
  )
$-semi
$ ( 1 + \| x \|^p )_{ x \in \mathbb{R}^d } 
$-stable
with respect to Brownian motion.
\end{cor}

\begin{proof}[Proof
of 
Corollary~\ref{cor:SVstability2}]
Let
$
  V \colon \mathbb{R}^d
  \rightarrow [1,\infty)
$
be given by
$
  V(x) = 1 + \left\| x \right\|^p
$
for all $ x \in \mathbb{R}^d $.
Then note that
inequality~\eqref{eq:p-generator}
implies
\begin{equation}
\label{eq:p-generator2}
\begin{split}
  ( \mathcal{G}_{ \mu, \sigma } V)(x)
& \leq
  p 
  \left\| x \right\|^{ (p - 2) }
  \left< x, \mu(x) \right>
  + 
  \tfrac{ p \left( p - 1 \right) }{ 2 }
  \left\| x \right\|^{ (p - 2) }
  \| \sigma(x) 
  \|^2_{
    HS( \mathbb{R}^m, \mathbb{R}^d )
  }
\\ & \leq
  2 \cdot
  p \cdot c \cdot V(x)
\end{split}
\end{equation}
for all $ x \in \mathbb{R}^d $.
Next
observe that
\begin{align}
\label{eq:pgrowth1}
  \left\|
    \mu(x)
  \right\|
  &\leq
  c \,
  \big(
    1 + 
    \left\| x \right\|^{ 
      \left[ \gamma_0 + 1 \right]
    }
  \big)
  \leq
  2 c 
  \left|
    V(x)
  \right|^{
    \left[
    \frac{
      \gamma_0 + 1
    }{
      p
    }
    \right]
  } ,
\\
\label{eq:pgrowth2}
  \left\|
    \sigma(x)
  \right\|_{
    L( \mathbb{R}^m, \mathbb{R}^d )
  }
  &\leq
  c
  \,
  \big(
    1 + 
    \left\| x \right\|^{   
      \left[ \frac{ \gamma_1 + 2 } { 2 } \right]
    }
  \big)
  \leq
  2 c
  \left|
    V(x)
  \right|^{ 
    \left[ \frac{ \gamma_1 + 2 } { 2p } \right]
  }
\end{align}
for all
$ x \in \mathbb{R}^d $.
Combining \eqref{eq:p-generator2}--\eqref{eq:pgrowth2},
the fact
$
  V \in C^3_p( \mathbb{R}^d, [1,\infty) )
$
and
Theorem~\ref{thm:SVstability} 
hence shows that the
Euler-Maruyama scheme
is 
$
   p / (
    \gamma_1 + 
    2 ( \gamma_0 \vee \gamma_1 ) 
  )
$-semi
$ V $-stable with respect
to Brownian motion.
The proof of 
Corollary~\ref{cor:SVstability2}
is thus completed.
\end{proof}

Theorem~\ref{thm:SVstability},
Corollary~\ref{cor:SVstability} and
Corollary~\ref{cor:SVstability2} 
give sufficient conditions for
the Euler-Maruyama scheme
to be $ \alpha $-semi $ V $-stable
with respect to Brownian motion
with
$ 
  \alpha \in (0,\infty)
$, 
$ 
  V \colon \mathbb{R}^d 
  \rightarrow [0,\infty) 
$
appropriate
and $ d \in \mathbb{N} $.
One may ask whether 
the Euler-Maruyama scheme
is also $ V $-stable
with respect to Brownian motion.
The next result,
which is a corollary 
of Theorem~2.1 in \cite{hjk11b}
(which generalizes Theorem~2.1 in \cite{hjk11}),
disproves this 
statement for one-dimensional
SDEs in which at least one of 
the coefficients $ \mu $ and 
$ \sigma $ grows more than 
linearly.
There is thus a large
class of SDEs in which the
Euler-Maruyama scheme
is $ \alpha $-semi $ V $-stable 
but not $ V $-stable
with respect to Brownian motion.

\begin{cor}[Disproof of $ V $-stability
with respect to Brownian motion
for the Euler-Maruyama scheme]
\label{cor:disprove}
Let 
$ T, \rho \in (0,\infty) $,
$ 
  \alpha, c \in (1,\infty)
$,
let 
$ 
  \mu, \sigma \colon \mathbb{R}
  \rightarrow \mathbb{R}
$
be Borel measurable functions
with
\begin{equation}
  \left| \mu(x) \right|
  + 
  \left| \sigma(x) \right|
  \geq
  \frac{ \left| x \right|^{ \alpha } }{ c }
\end{equation}
for all 
$ x \in ( - \infty, c] \cup [c, \infty) $ 
and with
$ \sigma \neq 0 $
(i.e., with the property that
there exists a real number
$ x_0 \in \R $ such that
$ 
  \sigma(x_0) \neq 0 
$).
Then there
exists no Borel measurable function
$
  V \colon \mathbb{R}
  \rightarrow [0,\infty)
$
which fulfills
\begin{equation}
  \limsup_{ r \searrow 0 }
  \sup_{ x \in \R }
  \left(
  \frac{
    | x |^{ r }
  }{
    \left( 1 + V(x) \right)
  }
  \right)
  < \infty
\end{equation}
and
\begin{equation}
\label{eq:disprove}
  \mathbb{E}\big[
    V\big(
      x + \mu(x) \cdot t
      + \sigma(x) W_t
    \big)
  \big]
  \leq
  e^{ \rho t } 
  \cdot
  V(x)
\end{equation}
for all $ (x,t) \in \mathbb{R} \times (0,T] $
where
$
  W \colon [0,T] \times \Omega
  \rightarrow \mathbb{R}
$
is an arbitrary standard Brownian
motion on a probability
space 
$ 
  \left( 
    \Omega, \mathcal{F}, \mathbb{P}
  \right)
$.
\end{cor}

\begin{proof}[Proof
of Corollary~\ref{cor:disprove}]
Suppose that
$
  V \colon \mathbb{R}
  \rightarrow [0,\infty)
$
is a Borel measurable function
which fulfills
\eqref{eq:disprove}
and
\begin{equation}
  \limsup_{ r \searrow 0 }
  \sup_{ x \in \R }
  \left(
  \frac{
    | x |^{ r }
  }{
    \left( 1 + V(x) \right)
  }
  \right)
  < \infty .
\end{equation}
Then there exists a real number
$ r \in (0,\infty) $ 
such that
\begin{equation}
  | x |^r \leq
  \frac{ 1 }{ r } 
  \left( 1 + V(x) \right)
\end{equation}
for all
$ x \in \R $.
Next observe that the
assumption
$
  \sigma \neq 0
$
implies that there exists
a real number $ x_0 \in \mathbb{R} $
with $ \sigma( x_0 ) \neq 0 $.
Now define a sequence
$ 
  Y^N \colon \{ 0, 1, \dots, N \} 
  \times
  \Omega
  \rightarrow \mathbb{R} 
$,
$ N \in \mathbb{N} $,
of stochastic processes by
$ Y^N_0 = x_0 $
and 
\begin{equation}
  Y^N_{ n + 1 } = 
  Y^N_n 
  + 
  \mu( Y^N_n ) \tfrac{ T }{ N }
  +
  \sigma( Y^N_n )
  \left( 
    W_{ \frac{ (n+1) T }{ N } }
    - 
    W_{ \frac{ n T }{ N } }
  \right)
\end{equation}
for all $ n \in \{ 0, 1, \dots, N - 1 \} $
and all
$ N \in \mathbb{N} $.
Corollary~\ref{cor:stability2}
then implies
\begin{equation}
  \limsup_{ N \rightarrow \infty }
  \mathbb{E}\big[
    \| Y^{ N }_N \|^{ r }
  \big]
\leq
  \tfrac{ 1 }{ r }
  \left(
    1 +
    \limsup_{ N \rightarrow \infty }
    \mathbb{E}\big[
      V( Y^{ N }_N )
    \big]
  \right)
\leq
  \tfrac{ 1 }{ r }
  \left(
    1 +
    e^{ \rho T }
    \cdot
    V( x_0 )
  \right)
< \infty .
\end{equation}
This contradicts to
Theorem~2.1 in \cite{hjk11b}.
The proof of Corollary~\ref{cor:disprove}
is thus completed.
\end{proof}

\subsection{Semi
stability
for tamed schemes}
\label{sec:SVtaming}

In the previous subsection,
semi $ V $-stability 
with respect to Brownian
motion for
the Euler-Maruyama scheme
has been analyzed.
In this subsection,
semi $ V $-stability 
with respect to Brownian
motion for
appropriately modified 
Euler-type methods is investigated.
We begin with a general 
abstract result
which shows that if two numerical
schemes are close in some sense
(see inequality~\eqref{eq:distPhi} below
for details) and if one of the two
numerical schemes is 
$ \alpha $-semi $V$-stable
with respect to Brownian
motion
with $ \alpha \in (0,\infty) $,
$ V \colon \mathbb{R}^d
\rightarrow [0,\infty) $
appropriate
and $ d \in \mathbb{N} $, 
then the other
scheme is 
$ \alpha $-semi $V$-stable
with respect to Brownian
motion too.

\begin{lemma}[A comparison
principle for semi $ V $-stability
with respect to Brownian motion]
\label{lem:compareSVstability}
Let 
$ \alpha, T \in (0,\infty) $,
$ c \in [0,\infty) $,
$
  p \in [1,\infty)
$,
$ d, m \in \mathbb{N} $,
let 
$ 
  \left( \Omega, \mathcal{F},
  \mathbb{P} \right) 
$
be a probability space,
let
$
  W \colon [0,T] \times \Omega
  \rightarrow \mathbb{R}^m
$
be a
standard
Brownian motion,
let
$
  \Phi,
  \tilde{\Phi}
  \colon \mathbb{R}^d \times
  [0,T] \times \mathbb{R}^m
  \rightarrow \mathbb{R}^d 
$
be Borel measurable functions
and let
$
  V \in C^1_p( \mathbb{R}^d,
  [0,\infty) )
$ 
be such that
$
  \Phi \colon \mathbb{R}^d \times
  [0,T] \times \mathbb{R}^m
  \rightarrow \mathbb{R}^d 
$
is $ \alpha $-semi $V$-stable
with respect to Brownian
motion and such that
\begin{equation}
\label{eq:distPhi}
  \Big(
    \mathbb{E}\Big[
      \big\|
        \Phi( x, t, 
          W_{ t }
        )
        -
        \tilde{\Phi}( x, t, 
          W_{ t }
        )
      \big\|^{ p }
    \Big]
  \Big)^{ 1 / p }
  \leq
    c \cdot
    t  
    \cdot
    \left|
      V(x)
    \right|^{ 1 / p }
\end{equation}
for all 
$ 
  (x,t) \in 
  \{
    (y,s) \in 
    \mathbb{R}^d \times
    (0,T]
    \colon
    V(y) \leq s^{ -\alpha } 
  \} 
$.
Then
$ 
  \tilde{\Phi} \colon 
    \mathbb{R}^d \times
    [0,T] \times \mathbb{R}^m
  \rightarrow 
    \mathbb{R}^d
$
is also
$ \alpha $-semi $V$-stable 
with respect to Brownian
motion.
\end{lemma}

\begin{proof}[Proof
of Lemma~\ref{lem:compareSVstability}]
Inequality~\eqref{eq:Vest2} in
Lemma~\ref{lem:Vest} implies
the existence of a real
number $ \hat{c} \in [c+1, \infty) $
such that
\begin{equation}
\begin{split}
&
  \mathbb{E}\Big[
    V\big(
      \tilde{\Phi}( x, t, 
        W_{ t }
      )
    \big)
  \Big]
\\ & =
  \mathbb{E}\Big[
    V\Big(
      \Phi( x, t, 
        W_{ t }
      )
      +
      \big\{
      \tilde{\Phi}( x, t, 
        W_{ t }
      )
      -
      \Phi( x, t, 
        W_{ t }
      )
      \big\}
    \Big)
  \Big]
\\ & \leq
  \mathbb{E}\Big[
    V\big(
      \Phi( x, t, 
        W_{ t }
      )
    \big)
  \Big]
 \\ & \quad
  +
  2^p 
  \left| \hat{c} \right|^p
    \mathbb{E}\Big[
      \big|
      V\big(
        \Phi( x, t, 
          W_{ t }
        )
      \big)
      \big|^{ \frac{ (p-1) }{ p } } \,
      \big\|
        \Phi( x, t, 
          W_{ t }
        )
        -
        \tilde{\Phi}( x, t, 
          W_{ t }
        )
      \big\|
    \Big]
\\ & \quad +
  2^p 
  \left| \hat{c} \right|^p
    \mathbb{E}\Big[
      \big\|
        \Phi( x, t, 
          W_{ t }
        )
        -
        \tilde{\Phi}( x, t, 
          W_{ t }
        )
      \big\|^{ p }
    \Big]
\end{split}
\end{equation}
and H\"{o}lder's inequality
hence gives
\begin{equation}
\begin{split}
&
  \mathbb{E}\Big[
    V\big(
      \tilde{\Phi}( x, t, 
        W_{ t }
      )
    \big)
  \Big]
\\ & \leq
  \mathbb{E}\Big[
    V\big(
      \Phi( x, t, 
        W_{ t }
      )
    \big)
  \Big]
\\ & \quad +
  2^p 
  \left| \hat{c} \right|^p 
  \Big(
    \mathbb{E}\Big[
      V\big(
        \Phi( x, t, 
          W_{ t }
        )
      \big)
    \Big]
  \Big)^{ 
    \!
    \frac{ (p-1) }{ p }
  }
  \,
  \Big(
    \mathbb{E}\Big[
      \big\|
        \Phi( x, t, 
          W_{ t }
        )
        -
        \tilde{\Phi}( x, t, 
          W_{ t }
        )
      \big\|^{
        p
      }
    \Big]
  \Big)^{
    \frac{ 1 }{ p }
  }
\\ & \quad 
  +
  2^p 
  \left| \hat{c} \right|^p
    \mathbb{E}\Big[
      \big\|
        \Phi( x, t, 
          W_{ t }
        )
        -
        \tilde{\Phi}( x, t, 
          W_{ t }
        )
      \big\|^{ 
        p
      }
    \Big]
\end{split}
\label{eq:ctaming1}
\end{equation}
for all $ x \in \mathbb{R}^d $
and all $ t \in [0,T] $.
Moreover, 
by assumption there exists
a real number 
$ \rho \in [0,\infty) $
such that
\begin{equation}
\label{eq:ctaming2}
    \mathbb{E}\Big[
      V\big(
        \Phi( x, t, 
          W_{ t }
        )
      \big)
    \Big]
\leq
  e^{ \rho t } \cdot
  V(x) 
\end{equation}
for all 
$ 
  (x,t) \in 
  \{
    (y,s) \in 
    \mathbb{R}^d \times
    (0,T]
    \colon
    V(y) \leq s^{ -\alpha } 
  \} 
$.
Putting 
\eqref{eq:ctaming2}
and
\eqref{eq:distPhi} 
into \eqref{eq:ctaming1}
then results in
\begin{equation}
\begin{split}
  \mathbb{E}\Big[
    V\big(
      \tilde{\Phi}( x, t, 
        W_{ t }
      )
    \big)
  \Big]
& \leq
  e^{ \rho t } \cdot V(x)
  +
  2^p 
  \left| \hat{c} \right|^p 
  \left(
    e^{ \rho t }  
    \cdot \hat{c} \cdot t \cdot
    V(x)
    +
    \left| \hat{c} \right|^p
    \cdot
    t^{ p } 
    \cdot
    V(x)
  \right)
\\ & \leq
  \Big(
  e^{ \rho t } 
  +
  t
  \left( T + 1 \right)^{
    p
  }
  e^{ \rho T }
  \,
  2^{
    (p + 1)
  }
  \left|
    \hat{c}
  \right|^{
    2 p
  }
  \Big)
  \cdot
  V(x)
\\ & \leq
  \exp\!
  \left(
    \left[
      \rho
      +
      \left( T + 1 \right)^{
        p
      }
      e^{ \rho T }
  2^{
    (p + 1)
  }
  \left|
    \hat{c}
  \right|^{
    2 p
  }
    \right] \cdot t 
  \right)
  \cdot
  V(x)
\end{split}
\end{equation}
for all 
$ 
  (x,t) \in 
  \{
    (y,s) \in 
    \mathbb{R}^d \times
    (0,T]
    \colon
    V(y) \leq s^{ -\alpha } 
  \} 
$.
The proof
of 
Lemma~\ref{lem:compareSVstability}
is thus completed.
\end{proof}

A direct consequence
of Lemma~\ref{lem:compareSVstability}
is the next corollary.
It proves $ \alpha $-semi $V$-stability 
with respect to Brownian
motion
with $ \alpha \in (0,\infty) $,
$ 
  V \colon \mathbb{R}^d 
  \rightarrow   
  [0,\infty) 
$
appropriate
and
$
  d \in \mathbb{N}
$
for a class of suitably
``tamed'' numerical methods.

\begin{cor}[Semi $V$-stability 
with respect to Brownian motion for
an increment-taming principle]
\label{cor:SVincrement}
Let $ \alpha, T \in (0,\infty) $,
$ \beta, c \in [0,\infty) $,
$
  p \in [1,\infty)
$,
$ d, m \in \mathbb{N} $,
let 
$ 
  \left( \Omega, \mathcal{F},
  \mathbb{P} \right) 
$
be a probability space,
let
$
  W \colon [0,T] \times \Omega
  \rightarrow \mathbb{R}^m
$
be a standard
Brownian motion,
let
$
  V \in C^1_p( \mathbb{R}^d ,
  [0,\infty) )
$,
let
$
  \Phi \colon \mathbb{R}^d \times
  [0,T] \times \mathbb{R}^m
  \rightarrow \mathbb{R}^d 
$
be a Borel measurable function
which is $ \alpha $-semi $V$-stable 
with respect to Brownian motion
and
assume that
\begin{equation}
\label{eq:checkinc}
    \mathbb{E}\!\left[
      \left\|
        \Phi( x, t, 
          W_{ t }
        )
        -
        x
      \right\|^{ 2 p }
    \right]
  \leq
   c \cdot
   t^{ p \left( 1 - \beta \right) }
   \cdot
    V(x)
\end{equation}
for all 
$ 
  (x,t) \in 
  \{
    (y,s) \in 
    \mathbb{R}^d \times
    (0,T]
    \colon
    V(y) \leq s^{ -\alpha } 
  \} 
$.
Then the function
\begin{equation}
  \mathbb{R}^d \times [0,T]
  \times \mathbb{R}^m
  \ni (x,t,y) 
  \mapsto
  x +
  \frac{
    \Phi(x,t,y) - x
  }{
    \max\!\left( 
      1 , 
      t^{ \beta } 
      \left\| \Phi(x,t,y) - x \right\| 
    \right)
  }
  \in \mathbb{R}^d
\end{equation}
is
$ \alpha $-semi $V$-stable
with respect to Brownian motion.
\end{cor}

\begin{proof}[Proof
of Corollary~\ref{cor:SVincrement}]
We show
Corollary~\ref{cor:SVincrement}
by an application of
Lemma~\ref{lem:compareSVstability}.
We thus need to
verify inequality~\eqref{eq:distPhi}.
For this let
$ 
  \tilde{ \Phi } \colon 
  \mathbb{R}^d \times [0,T]
  \times \mathbb{R}^m
  \rightarrow \mathbb{R}^d
$
be given by
\begin{equation}
  \tilde{ \Phi }(t,x,y) =
  x +
  \frac{
    \Phi(x,t,y) - x
  }{
    \max\left( 1, 
      t^{ \beta } \left\| \Phi(x,t,y) - x \right\| 
    \right)
  }
\end{equation}
for all
$ x \in \mathbb{R}^d $,
$ t \in [0,T] $,
$ y \in \mathbb{R}^m $
and note that
\begin{equation}
\begin{split}
  \big\|
    \Phi(x,t,W_t) - \tilde{\Phi}(x,t,W_t)
  \big\|
& =
  \left\|
    \Phi(x,t,W_t) - x
  \right\|
  \cdot
  \frac{
    \max\!\left( 1, 
      t^{ \beta } \left\| \Phi(x,t,y) - x \right\| 
    \right)
    - 1
  }{
    \max\!\left( 1, 
      t^{ \beta } \left\| \Phi(x,t,y) - x \right\| 
    \right)
  }
\\ & \leq
  t^{ \beta }
  \left\|
    \Phi(x,t,W_t) - x
  \right\|^2
\end{split}
\end{equation}
for all $ x \in \mathbb{R}^d $,
$ t \in [0,T] $
and all
$ y \in \mathbb{R}^m $.
Estimate~\eqref{eq:checkinc}
therefore shows 
\eqref{eq:distPhi}
and
Lemma~\ref{lem:compareSVstability} 
hence 
completes the proof of 
Corollary~\ref{cor:SVincrement}.
\end{proof}

The next result is an 
immediate
consequence of
Corollary~\ref{cor:SVincrement}
and Theorem~\ref{thm:SVstability}.
It shows $ \alpha $-semi 
$V$-stability
with respect to Brownian motion
with 
$ \alpha \in (0,\infty) $,
$ 
  V \colon \mathbb{R}^d \rightarrow
  [0,\infty) 
$
appropriate
and
$
  d \in \mathbb{N}
$
for a suitable
``tamed'' Euler-Maruyama 
scheme.

\begin{cor}[Semi $V$-stability 
with respect to Brownian motion for 
an increment-tamed Euler-Maruyama
scheme]
\label{cor:SVincrement2}
Let $ T \in (0,\infty) $,
$ d, m \in \mathbb{N} $,
$
  p \in [3,\infty)
$,
$
  c,
  \gamma_0,
  \gamma_1
  \in [0,\infty)
$,
let
$
  \mu \colon \mathbb{R}^d
  \rightarrow \mathbb{R}^d
$,
$
  \sigma  
  \colon \mathbb{R}^d
  \rightarrow \mathbb{R}^{ d \times m }
$
be Borel measurable functions and
let
$ 
  V \in C^3_p( \mathbb{R}^d, [1,\infty) )
$
with
\begin{equation*}
  ( \mathcal{G}_{ \mu, \sigma } V)(x)
\leq
  c \cdot V(x) ,
\;\;
  \left\|
    \mu(x)
  \right\|
  \leq
  c \,
  |
    V(x)
  |^{
    \left[
      \frac{ \gamma_0 + 1 }{ p }
    \right]
  } ,
\;\;
  \|
    \sigma(x)
  \|_{
    L( \mathbb{R}^m, \mathbb{R}^d )
  }
  \leq
  c \,
  |
    V(x)
  |^{ 
    \left[
      \frac{ \gamma_1 + 2 }{ 2p }
    \right]
  } 
\end{equation*}
for all $ x \in \mathbb{R}^d $.
Then the function
\begin{equation}
\label{eq:inc_tamed_Euler}
  \mathbb{R}^d \times [0,T]
  \times \mathbb{R}^m
  \ni (x,t,y) 
  \mapsto
  x + 
  \frac{ 
    \mu(x) t + \sigma(x) y 
  }{
    \max( 1, t \| \mu(x) t + \sigma(x) y \| )
  }
  \in \mathbb{R}^d
\end{equation}
is 
$ 
  \tfrac{ p }
  {
      \gamma_1 + 
      2 
      \max( 
        \gamma_0, \gamma_1, 1/2
      )
  }
$-semi
$ V $-stable with respect to 
Brownian motion.
\end{cor}

\begin{proof}[Proof
of Corollary~\ref{cor:SVincrement2}]
We prove
Corollary~\ref{cor:SVincrement2}
by using
Theorem~\ref{thm:SVstability}
and 
Corollary~\ref{cor:SVincrement}.
We thus need to verify
condition~\eqref{eq:checkinc}.
For this let $ \left( \Omega, \mathcal{F}, \P \right) $
be a probability space and let
$
  W = ( W^{ (1) }, \dots, W^{ (m) } ) 
  \colon [0, \infty) \times \Omega \to \R^m
$
be an $ m $-dimensional standard Brownian
motion throughout this proof.
Then note that
the assumptions
\begin{equation}
  \left\|
    \mu(x)
  \right\|
  \leq
  c \,
  |
    V(x)
  |^{
      (\gamma_0 + 1) / p
  }
\quad
\text{and}
\quad
  \|
    \sigma(x)
  \|_{
    L( \mathbb{R}^m, \mathbb{R}^d )
  }
  \leq
  c \,
  |
    V(x)
  |^{ 
      (\gamma_1 + 2)/(2p)
  } 
\end{equation}
for all $ x \in \mathbb{R}^d $
imply
\begin{equation}
\begin{split}
&
  \mathbb{E}\!\left[
  \left\|
    \mu(x) \cdot t +
    \sigma(x) W_t
  \right\|^{ 2 p }
  \right]
\\ & \leq
  \left(
    m + 1
  \right)^{ (2 p - 1) } \,
  \left(
    \left\|
      \mu(x)
    \right\|^{ 2 p }
    t^{ 2 p }
    +
    \sum_{ k = 1 }^m
    \left\|
      \sigma_k(x)
    \right\|^{ 2 p }
    \mathbb{E}\big[
      | W^{ (k) }_t |^{ 2 p }
    \big]
  \right)
\\ & \leq
  \left(
    m + 1
  \right)^{ (2 p - 1) } \,
  \mathbb{E}\big[
    | W^{ (1) }_1 |^{ 2 p }
  \big]
  \left(
    \left\|
      \mu(x)
    \right\|^{ 2 p }
    t^{ 2 p }
    +
    \sum_{ k = 1 }^m
    \left\|
      \sigma_k(x)
    \right\|^{ 2 p }
    t^{ p }
  \right)
\\ & \leq
  c^{ 2 p }
  \left(
    m + 1
  \right)^{ 2 p } \,
  \mathbb{E}\big[
    | W^{ (1) }_1 |^{ 2 p }
  \big]
  \left(
    \left|
      V(x)
    \right|^{ 
      \left(
        2 \gamma_0 + 1
      \right)
    }
    t^{ 2 p }
    +
    \left|
      V(x)
    \right|^{ 
      \left(
        \gamma_1 + 1
      \right)
    }
    t^{ p }
  \right)
  V(x)
\end{split}
\end{equation}
for all $ x \in \mathbb{R}^d $
and all $ t \in [0,T] $.
Applying
Theorem~\ref{thm:SVstability}
and
Corollary~\ref{cor:SVincrement}
therefore shows that
\eqref{eq:inc_tamed_Euler} 
is $ \alpha $-semi 
$ V $-stable 
with respect to Brownian motion
with
\begin{equation}
  \alpha
:= 
  \min\!\left(
    \frac{ 
      2 p 
    }{
      2  \gamma_0 + 1
    }
    ,
    \frac{ 
      p 
    }{
      \gamma_1 + 1
    }
    ,
    \frac{ 
      p 
    }{
      \gamma_1 + 
      2 ( \gamma_0 \vee \gamma_1 ) 
    }
  \right) 
  \in (0,\infty) .
\end{equation}
Next note that
\begin{equation}
\begin{split}
  \alpha
& = 
  \frac{ 
    p 
  }{
    \max\!\left(
      \gamma_0 + 
      \frac{1}{2}
      ,
      \gamma_1 + 1
      ,
      \gamma_1 + 
      2 ( \gamma_0 \vee \gamma_1 ) 
    \right)
  }
\\ & =
  \frac{ 
    p 
  }{
    \max\!\left(
      \gamma_0 + \frac{1}{2}
      ,
      \gamma_1 + 
      2 
      \max( 
        \gamma_0, \gamma_1, \frac{1}{2}
      )
    \right)
  }
  =
  \frac{ 
    p 
  }{
      \gamma_1 + 
      2 
      \max( 
        \gamma_0, \gamma_1, \frac{1}{2}
      )
  } .
\end{split}
\end{equation}
The proof of 
Corollary~\ref{cor:SVincrement2}
is thus completed.
\end{proof}

Note that, 
under the
assumptions of
Theorem~\ref{thm:SVstability},
the Euler-Maruyama scheme
is 
$ 
  p / (
      \gamma_1 + 
      2 
      \max( 
        \gamma_0, \gamma_1
      )
  )
$-semi $ V $-stable
with respect to Brownian motion
but the appropriate 
tamed numerical
method in 
Corollary~\ref{cor:SVincrement2}
is 
$ 
  p / (
      \gamma_1 + 
      2 
      \max( 
        \gamma_0, \gamma_1, 1/2
      )
  )
$-semi $V$-stable
with respect to Brownian motion.

\subsection{Moment bounds
for an increment-tamed Euler-Maruyama\\
scheme}
\label{sec:momentbounds2}

In 
Subsections~\ref{sec:SVEM}
and \ref{sec:SVtaming} above,
we established
under suitable assumptions
that the Euler-Maruyama scheme
(see Theorem~\ref{thm:SVstability}
in Subsection~\ref{sec:SVEM})
as well as appropriately tamed
numerical methods 
(see Corollary~\ref{cor:SVincrement}
and
Corollary~\ref{cor:SVincrement2}
in Subsection~\ref{sec:SVtaming})
are $ \alpha $-semi
$ V $-stable
with respect to Brownian
motion with $ \alpha \in (1,\infty) $,
$ V \colon \mathbb{R}^d
\rightarrow [0,\infty) $
appropriate
and $ d \in \mathbb{N} $.
Corollary~\ref{cor:FV}
can then be applied to show
that these approximations
have uniformly bounded
moments restricted to events 
whose probabilities converge to one with
convergence order $ (\alpha - 1) $.
In the case of appropriately
tamed numerical methods
(see Subsection~\ref{sec:SVtaming}
for a few simple examples
and Section~\ref{sec:schemes}
below for more examples),
it can even be shown that
the approximations have
uniformly bounded moments
without restricting to a sequence
of events whose probabilities converge to one
sufficiently fast.
In the next result this is
illustrated in the case
of the increment-tamed
Euler-Maruyama scheme
from 
Corollary~\ref{cor:SVincrement2}.
Its proof is based
on an application of
Corollary~\ref{cor:Semi.Stability} 
above.

\begin{cor}[Moment
bounds for an
increment-tamed 
Euler-Maruyama
scheme]
\label{cor:apriori_increment}
Let $ T \in (0,\infty) $,
$ d, m \in \mathbb{N} $,
$
  p \in [3,\infty)
$,
$
  c,
  \gamma_0,
  \gamma_1
  \in [0,\infty)
$,
let
$
  \mu \colon \mathbb{R}^d
  \rightarrow \mathbb{R}^d
$,
$
  \sigma 
  \colon \mathbb{R}^d
  \rightarrow \mathbb{R}^{ d \times m }
$
be Borel measurable functions,
let 
$ 
  \left( \Omega, \mathcal{F},
  ( \mathcal{F}_t )_{ t \in [0,T] } ,
  \mathbb{P} \right) 
$
be a filtered probability space,
let
$
  W \colon [0,T] \times \Omega
  \rightarrow \mathbb{R}^m
$
be a standard
$
  ( \mathcal{F}_t )_{ t \in [0,T] }
$-Brownian motion
and let
$ 
  \xi \colon \Omega
  \rightarrow \mathbb{R}^d
$
be an
$ 
  \mathcal{F}_0 
$/$ \mathcal{B}(\mathbb{R}^d) 
$-measurable mapping
with 
$
  \mathbb{E}\big[
    \| \xi \|^q 
  \big]
  < \infty
$
for all $ q \in [0,\infty) $.
Moreover, let
$ 
  V \in 
  C^3_p( \mathbb{R}^d, [1,\infty) )
$
with
\begin{equation*}
  ( \mathcal{G}_{ \mu, \sigma } V)(x)
\leq
  c \cdot V(x) ,
\;
  \left\|
    \mu(x)
  \right\|
  \leq
  c \,
  |
    V(x)
  |^{
    \left[
      \frac{ \gamma_0 + 1 }{ p }
    \right]
  } ,
\;
  \|
    \sigma(x)
  \|_{
    L( \mathbb{R}^m, \mathbb{R}^d )
  }
  \leq
  c \,
  |
    V(x)
  |^{ 
    \left[
      \frac{ \gamma_1 + 2 }{ 2p }
    \right]
  } 
\end{equation*}
for all $ x \in \mathbb{R}^d $.
Furthermore, let
$
  \bar{Y}^N \colon 
  [0,T] \times
  \Omega
  \rightarrow \mathbb{R}^d
$,
$ N \in \mathbb{N} 
$,
be a sequence of stochastic
processes 
given by
$
  \bar{Y}^N_0 = \xi 
$
and
\begin{equation}
\label{eq:defTamed}
  \bar{Y}^N_{ t }
=
  \bar{Y}^N_{ 
    \frac{ n T }{ N }
  }
+
  \left(
    \tfrac{ t N }{ T } - n
  \right) 
  \cdot
  \frac{
    \mu( 
      \bar{Y}^N_{ \frac{ n T }{ N } } 
    ) 
    \frac{ T }{ N }
    + 
    \sigma( 
      \bar{Y}^N_{ \frac{ n T }{ N } } 
    ) 
    (
      W_{ \frac{ (n + 1) T }{ N } } 
      -
      W_{ \frac{ n T }{ N } } 
    )
  }{
    \max\!\big( 1, 
      \frac{ T }{ N }
      \| 
        \mu( 
          \bar{Y}^N_{ \frac{ n T }{ N } } 
        ) 
        \frac{ T }{ N }
        + 
        \sigma( 
          \bar{Y}^N_{ \frac{ n T }{ N } } 
        ) 
        (
          W_{ \frac{ (n + 1) T }{ N } } 
          -
          W_{ \frac{ n T }{ N } } 
        )
      \| 
    \big)
  }
\end{equation}
for all 
$
  t \in 
  \big(
    \frac{ n T }{ N },
    \frac{ (n + 1) T }{ N }
  \big]
$,
$ 
  n \in \{ 0, 1, \dots, N - 1 \} 
$,
$ N \in \mathbb{N} $.
Then 
\begin{equation}
\label{eq:momentbound_incEuler}
  \sup_{ N \in \mathbb{N} }
    \sup_{ t \in [0,T] }
    \mathbb{E}\big[
      \| 
        \bar{Y}^N_t
      \|^{ 
        q
      }
    \big] 
  < \infty 
\end{equation}
for all 
$ 
  q 
  \in 
  [0,\infty)
$
with
$
  q <
  \frac{
      p 
    }{
      2 \gamma_1 + 
      4 
      \max( 
        \gamma_0, \gamma_1, 1/2
      )
    }
  -
  \frac{ 1 }{ 2 }
$
and
$
  \sup_{ 
    x \in \mathbb{R}^d
  }
  \| x \|^{ q } /
  V(x)
  < \infty
$.
\end{cor}

\begin{proof}[Proof
of Corollary~\ref{cor:apriori_increment}]
If there exists no real number
$ 
  q 
  \in 
  (0,\infty)
$
which satisfies
\begin{equation}
\label{eq:q_to_fulfill}
  q <
  \frac{
      p 
    }{
      2 \gamma_1 + 
      4 
      \max( 
        \gamma_0, \gamma_1, 1/2
      )
    }
  -
  \frac{ 1 }{ 2 }
\qquad
\text{and}
\qquad
  \sup_{ 
    x \in \mathbb{R}^d
  }  
  \left(
  \frac{
    \| x \|^{ q } 
  }{
    V(x)
  }
  \right)
  < \infty ,
\end{equation}
then \eqref{eq:momentbound_incEuler}
follows immediately.
We thus assume 
in the following 
that 
$ 
  q 
  \in 
  (0,\infty)
$
is a real number
which satisfies
\eqref{eq:q_to_fulfill}.
Next observe that
Corollary~\ref{cor:SVincrement2}
shows
that the function
$
  \Phi \colon
  \mathbb{R}^d
  \times [0,T]
  \times \mathbb{R}^m
  \rightarrow 
  \mathbb{R}^d
$
given by
\begin{equation}
  \Phi(x,t,y)
=
  x +
  \frac{
    \mu(x) t + \sigma(x) y
  }{
    \max\left( 1, 
      t \left\| \mu(x) t + \sigma(x) y \right\| 
    \right)
  }
\end{equation}
for all
$ x \in \mathbb{R}^d $,
$ t \in [0,T] $,
$ y \in \mathbb{R}^m $ 
is
$ 
  \alpha
$-semi $V$-stable
with respect to Brownian motion
with
\begin{equation}
  \alpha :=
  \frac{ p 
  }{
    \gamma_1 + 
    2 
    \max( 
      \gamma_0, \gamma_1, 1 / 2
    )
  }
  \in (0,\infty)
  .
\end{equation}
Moreover, note that
\begin{equation}
\begin{split}
&
  \big\| 
    \bar{Y}^N_t
  \big\|
\\ & \leq
  \| \xi \| 
  +
  \sum_{ n = 0 }^{ N - 1 }
  \frac{
    \big\|
    \mu( 
      \bar{Y}^N_{ n T / N } 
    ) 
    \frac{ T }{ N }
    + 
    \sigma( 
      \bar{Y}^N_{ n T / N } 
    ) 
    (
      W_{ (k + 1) T / N } 
      -
      W_{ k T / N } 
    )
    \big\|
  }{
    \max\!\big( 1, 
      \frac{ T }{ N }
      \| 
        \mu( 
          \bar{Y}^N_{ n T / N } 
        ) 
        \frac{ T }{ N }
        + 
        \sigma( 
          \bar{Y}^N_{ n T / N } 
        ) 
        (
          W_{ (k + 1) T / N } 
          -
          W_{ n T / N } 
        )
      \| 
    \big)
  }
\\ & \leq
  \| \xi \|
  +
  \frac{ N^2 }{ T }
\end{split}
\end{equation}
for all 
$ 
  t \in 
  [ 
    0, T
  ]
$
and all
$ N \in \mathbb{N} $
and therefore
\begin{equation}
\label{eq:bootstrap1}
  \sup_{ N \in \mathbb{N} }
  \sup_{ t \in [0,T] }
  \left(
    N^{ - 2 q } \,
    \big\|
      \| 
        \bar{Y}^N_t 
      \|^{ 
        q
      }
    \big\|_{ 
      L^{ r }( \Omega; \mathbb{R} )
    }
  \right)
  < \infty 
\end{equation}
for all $ r \in (0,\infty) $.
Next note that the inequality
\begin{equation}
  1 - \alpha
=
  - 2 
  \left(
    \frac{ \alpha }{ 2 } -
    \frac{ 1 }{ 2 }
  \right)
=
  - 2 
  \left(
  \frac{
      p 
    }{
      2 \gamma_1 + 
      4 
      \max( 
        \gamma_0, \gamma_1, \frac{1}{2}
      )
    }
  -
  \frac{ 1 }{ 2 }
  \right)
<
  - 2 q
\end{equation}
proves that there exists a real number
$ r \in (1, \infty ) $ such that
\begin{equation}
  \left( 
    1 - \alpha 
  \right)
  \left(
    1 - \tfrac{ 1 }{ r }
  \right)
<
  - 2 q .
\end{equation}
Combining this with
\eqref{eq:bootstrap1}
proves that
\begin{equation}
  \sup_{ N \in \mathbb{N} }
  \sup_{ t \in [0,T] }
  \left(
    N^{ ( 1 - \alpha ) ( 1 - 1/r) } \,
    \big\|
      \| 
        \bar{Y}^N_t 
      \|^{ 
        q
      }
    \big\|_{ 
      L^{ r }( \Omega; \mathbb{R} )
    }
  \right)
  < \infty .
\end{equation}
Corollary~\ref{cor:Semi.Stability}
can thus be applied to give
\begin{equation}
  \limsup_{ N \rightarrow \infty }
  \sup_{ n \in \{ 0, 1, \dots, N \} }
  \mathbb{E}\Big[
    \big\| 
      \bar{Y}^N_{ \frac{ n T }{ N } } 
    \big\|^q
  \Big] 
  < \infty .
\end{equation}
Combining this 
and \eqref{eq:bootstrap1} 
finally implies
\begin{equation}
  \sup_{ N \in \mathbb{N} }
  \sup_{ t \in [0,T] }
  \mathbb{E}\Big[
    \big\| 
      \bar{Y}^N_{ t } 
    \big\|^q
  \Big] 
  < \infty .
\end{equation}
The proof of 
Corollary~\ref{cor:apriori_increment}
is thus completed.
\end{proof}

\begin{cor}[Powers of
the Lyapunov-type function]
\label{cor:qMoments}
Let 
$ T \in (0,\infty) $,
$ d, m \in \mathbb{N} $,
$
  p \in [3,\infty)
$,
$
  q \in [1,\infty)
$,
$
  c,
  \gamma_0,
  \gamma_1
  \in [0,\infty)
$,
let
$
  \mu \colon \mathbb{R}^d
  \rightarrow \mathbb{R}^d
$,
$
  \sigma 
  \colon \mathbb{R}^d
  \rightarrow \mathbb{R}^{ d \times m }
$
be Borel measurable 
functions
and let
$ 
  V \in C^3_p( \mathbb{R}^d, [1,\infty) )
$
with
$
  \left\|
    \mu(x)
  \right\|
  \leq
  c \,
  |
    V(x)
  |^{
      ( \gamma_0 + 1 )/ p
  }
$,
$
  \|
    \sigma(x)
  \|_{
    L( \mathbb{R}^m, \mathbb{R}^d )
  }
  \leq
  c \,
  |
    V(x)
  |^{ 
      (\gamma_1 +2 )/( 2p )
  } 
$
and
\begin{equation}
\label{eq:generatorplus2_ass}
  ( \mathcal{G}_{ \mu, \sigma } V)(x)
  + 
  \frac{
    \left( q - 1 \right)
    \| V'(x) \sigma(x) \|^2_{ 
      HS( \mathbb{R}^m, \mathbb{R} )
    }
  }{
    2 \cdot V(x)
  }
  \leq c \cdot V(x) 
\end{equation}
for  
all
$ x \in \mathbb{R}^d $.
Moreover,
let 
$ 
  \left( \Omega, \mathcal{F},
  ( \mathcal{F}_t )_{ t \in [0,T] },
  \mathbb{P} \right) 
$
be a filtered probability space,
let
$
  W \colon [0,T] \times \Omega
  \rightarrow \mathbb{R}^m
$
be a standard
$
  ( \mathcal{F}_t )_{ t \in [0,T] }
$-Brownian motion,
let
$ 
  \xi \colon \Omega
  \rightarrow \mathbb{R}^d
$
be an
$ 
  \mathcal{F}_0 
$/$ \mathcal{B}(\mathbb{R}^d) 
$-measurable mapping with 
$
  \mathbb{E}\big[
    \| \xi \|^r
  \big]
  < \infty
$
for all $ r \in [0,\infty) $
and let
$
  \bar{Y}^N \colon 
  [0,T] \times
  \Omega
  \rightarrow \mathbb{R}^d
$,
$ N \in \mathbb{N} 
$,
be a sequence of stochastic
processes 
given by
$
  \bar{Y}^N_0 = \xi 
$
and \eqref{eq:defTamed}.
Then 
\begin{equation}
  \sup_{ N \in \mathbb{N} }
    \sup_{ t \in [0,T] }
    \mathbb{E}\big[
      \| 
        \bar{Y}^N_t
      \|^{ 
        r 
      }
    \big] 
  < \infty 
\end{equation}
for all 
$ 
  r
  \in 
  [0,\infty)
$
which satisfy
$
  r <
  \frac{
      p q
    }{
      2 \gamma_1 + 
      4 
      \max( 
        \gamma_0, \gamma_1, 1/2
      )
    }
  -
  \frac{ 1 }{ 2 }
$
and
$
  \sup_{ 
    x \in \mathbb{R}^d
  }
  \| x \|^{ r } /
  V(x)
  < \infty
$.
\end{cor}

\begin{proof}[Proof of 
Corollary~\ref{cor:qMoments}]
Define the function
$ 
  \hat{V} \colon \mathbb{R}^d
  \rightarrow [1,\infty)
$
by
$
  \hat{V}(x) 
  = 
  \left( V(x) \right)^{ q }
$
for all $ x \in \mathbb{R}^d $
and observe that
\begin{equation}
  \left\|
    \mu(x)
  \right\|
  \leq
  c \,
  |
    \hat{V}(x)
  |^{
    \left[
      \frac{ \gamma_0 + 1 }{ p q }
    \right]
  }
\qquad
  \text{and}
\qquad
  \|
    \sigma(x)
  \|_{
    L( \mathbb{R}^m, \mathbb{R}^d )
  }
  \leq
  c \,
  |
    \hat{V}(x)
  |^{ 
    \left[
      \frac{ \gamma_1 + 2 }{ 2p q }
    \right]
  } 
\end{equation}
for all $ x \in \mathbb{R}^d $.
Lemma~\ref{lem:potencies}
implies that $\hat{V}\in C^3_{pq}( \mathbb{R}^d, [1,\infty) )$
and
inequality~\eqref{eq:generatorplus2_ass}
yields that
$ 
  ( \mathcal{G}_{ \mu, \sigma } \hat{V} )( x ) 
  \leq q \cdot c \cdot \hat{V}(x)
$
for all
$
  x \in \R^d
$.
Therefore,
Corollary~\ref{cor:apriori_increment}
shows that
\begin{equation}
  \sup_{ N \in \mathbb{N} }
    \sup_{ t \in [0,T] }
    \mathbb{E}\big[
      \| 
        \bar{Y}^N_t
      \|^{ 
        r
      }
    \big] 
  < \infty 
\end{equation}
for all 
$ 
  r 
  \in 
  [0,\infty)
$
with
$
  r <
  \frac{
      pq 
    }{
      2 \gamma_1 + 
      4 
      \max( 
        \gamma_0, \gamma_1,1/2
      )
    }
  -
  \frac{ 1 }{ 2 }
$
and
$
  \sup_{ 
    x \in \mathbb{R}^d
  }
  \| x \|^{ r } /
  \hat{V}(x)
  < \infty
$.
The proof of 
Corollary~\ref{cor:qMoments}
is thus completed.
\end{proof}

We now illustrate the
moment bounds 
of
Corollary~\ref{cor:apriori_increment}
and 
of
Corollary~\ref{cor:qMoments}
by two simple corollaries.

\begin{cor}
\label{cor:SVstabilityT1}
Let 
$ c, T \in (0,\infty) $,
$ d, m \in \mathbb{N} $,
let
$
  \mu \colon \mathbb{R}^d
  \rightarrow \mathbb{R}^d
$,
$
  \sigma  
  \colon \mathbb{R}^d
  \rightarrow \mathbb{R}^{ d \times m }
$
be Borel measurable functions
and 
let
$ 
  V \in 
  \cup_{ p \in (0,\infty) } 
  C^3_p( \mathbb{R}^d, [1,\infty) )
$
be a function
with
$
  \limsup_{ 
    q \searrow 0
  }
  \sup_{ 
    x \in \mathbb{R}^d
  }
  \frac{ 
    \| x \|^{ q } 
  }{
    V(x)
  }
  < \infty
$
and with
\begin{equation}
  \sup_{ x \in \mathbb{R}^d }
  \left[
  \frac{
    ( \mathcal{G}_{ \mu, \sigma } V)(x)
  }{
    V(x)
  }
  +
  \frac{
    r \,
    \| 
      V'(x) \sigma(x) 
    \|_{ 
      L( \mathbb{R}^m, \mathbb{R} ) 
    }^2
  }{
    | V(x) |^2
  }
  \right]
  < \infty ,
\end{equation}
\begin{equation}
  \sup_{ x \in \mathbb{R}^d }
  \left[
  \frac{
    \|
      \mu(x)
    \|
    +
    \|
      \sigma(x)
    \|_{
      L( \mathbb{R}^m, \mathbb{R}^d )
    }
  }{
    \left(
      1 + \| x \|^{ c }
    \right)
  }
  \right]
  < \infty 
\end{equation}
for all $ r \in [0,\infty) $.
Moreover,
let 
$ 
  \left( \Omega, \mathcal{F},
  ( \mathcal{F}_t )_{ t \in [0,T] },
  \mathbb{P} \right) 
$
be a filtered probability space,
let
$
  W \colon [0,T] \times \Omega
  \rightarrow \mathbb{R}^m
$
be a standard
$
  ( \mathcal{F}_t )_{ t \in [0,T] }
$-Brownian motion,
let
$ 
  \xi \colon \Omega
  \rightarrow \mathbb{R}^d
$
be an
$ 
  \mathcal{F}_0 
$/$ \mathcal{B}(\mathbb{R}^d) 
$-measurable
mapping
with 
$ 
  \mathbb{E}\big[ 
    \| \xi \|^q 
  \big] < \infty 
$
for all $ q \in [0,\infty) $
and let
$
  \bar{Y}^N \colon 
  [0,T] \times
  \Omega
  \rightarrow \mathbb{R}^d
$,
$ N \in \mathbb{N} 
$,
be a sequence of stochastic
processes 
given by
$
  \bar{Y}^N_0 = \xi 
$
and \eqref{eq:defTamed}.
Then 
\begin{equation}
  \sup_{ N \in \mathbb{N} }
    \sup_{ t \in [0,T] }
    \mathbb{E}\big[
      \| 
        \bar{Y}^N_t
      \|^{ 
        q
      }
    \big] 
  < \infty 
\end{equation}
for all 
$ 
  q 
  \in 
  [0,\infty)
$.
\end{cor}

Corollary~\ref{cor:SVstabilityT1}
is an immediate consequence
of Corollary~\ref{cor:qMoments}.
Note also for every 
$ d \in \mathbb{N} $
that a function
$ 
  V \colon \mathbb{R}^d
  \rightarrow [1,\infty)
$
is in
$
  \cup_{ p \in (0,\infty) }
  C^3_p( \mathbb{R}^d, [1,\infty) )
$
if and only if it is twice differentiable  
with a locally Lipschitz continuous
second derivative and 
if there
exists a real number $ c \in (0,\infty) $
such that
$
  \sum_{ i = 1 }^{ 3 }
  \|
    V^{(i)}( x )
  \|_{
    L^{ (i) }( 
      \mathbb{R}^d, \mathbb{R} 
    )
  }
  \leq
  c \,
  |
    V(x)
  |^{ 
    [ 1 - 1 / c ]
  }
$
for 
$ \lambda_{ \mathbb{R}^d } $-almost
all 
$ x \in \mathbb{R}^d $.
The next corollary of 
Corollary~\ref{cor:apriori_increment}
is the counterpart to
Corollary~\ref{cor:SVstability2}

\begin{cor}[A special polynomial like
Lyapunov-type function]
\label{cor:SVstabilityT2}
Let $ T \in (0,\infty) $,
$ 
  c,
  \gamma_0,
  \gamma_1
  \in [0,\infty) 
$,
$ p \in [3,\infty) $,
$ d, m \in \mathbb{N} $,
let
$
  \mu \colon \mathbb{R}^d
  \rightarrow \mathbb{R}^d
$,
$
  \sigma 
  \colon \mathbb{R}^d
  \rightarrow \mathbb{R}^{ d \times m }
$
be Borel measurable functions with
\begin{equation}
\label{eq:p-generatorB}
  \left< x, \mu(x) \right>
  + 
  \tfrac{ \left( p - 1 \right) }{ 2 }
  \| \sigma(x) 
  \|^2_{
    HS( \mathbb{R}^m, \mathbb{R}^d )
  }
  \leq 
  c \left(
    1 + \left\| x \right\|^2
  \right) ,
\end{equation}
\begin{equation}
  \left\|
    \mu(x)
  \right\|
  \leq
  c \,
  \big( 
    1 + 
    \left\| x \right\|^{ 
      [ \gamma_0 + 1 ]
    }
  \big)
\qquad
  \text{and}
\qquad
  \| 
    \sigma(x) 
  \|_{ 
    L( \mathbb{R}^m , \mathbb{R}^d ) 
  }
  \leq
  c \,
  \big( 
    1 + 
    \left\| x \right\|^{ 
      \left[ \frac{ \gamma_1 + 2 }{ 2 } \right]
    }
  \big)
\end{equation}
for all
$ 
  x \in \mathbb{R}^d  
$.
Moreover,
let 
$ 
  \left( \Omega, \mathcal{F},
  ( \mathcal{F}_t )_{ t \in [0,T] },
  \mathbb{P} \right) 
$
be a filtered probability space,
let
$
  W \colon [0,T] \times \Omega
  \rightarrow \mathbb{R}^m
$
be a standard
$
  ( \mathcal{F}_t )_{ t \in [0,T] }
$-Brownian motion,
let
$ 
  \xi \colon \Omega
  \rightarrow \mathbb{R}^d
$
be an
$ 
  \mathcal{F}_0 
$/$ \mathcal{B}(\mathbb{R}^d) 
$-measurable mapping
with 
$ 
  \mathbb{E}[ \| \xi \|^q ] < \infty 
$
for all $ q \in [0,\infty) $
and let
$
  \bar{Y}^N \colon 
  [0,T] \times
  \Omega
  \rightarrow \mathbb{R}^d
$,
$ N \in \mathbb{N} 
$,
be a sequence of stochastic
processes 
given by
$
  \bar{Y}^N_0 = \xi 
$
and \eqref{eq:defTamed}.
Then
\begin{equation}
\label{eq:inc_tamed_Euler_momentbound_p}
  \sup_{ N \in \mathbb{N} }
    \sup_{ t \in [0,T] }
    \mathbb{E}\big[
      \| 
        \bar{Y}^N_t
      \|^{ 
        q
      }
    \big] 
  < \infty 
\end{equation}
for all 
$ 
  q 
  \in 
  [0,\infty)
$
which satisfy
$
  q <
    \frac{
      p
    }{
      2 \gamma_1 + 
      4 
      \max( 
        \gamma_0, 
        \gamma_1,
        1/2
      )
    }
    -
    \frac{ 1 }{ 2 }
$.
\end{cor}

\begin{proof}[Proof
of 
Corollary~\ref{cor:SVstabilityT2}]
Let
$
  V \colon \mathbb{R}^d
  \rightarrow [1,\infty)
$
be given by
$
  V(x) = 1 + \left\| x \right\|^p
$
for all $ x \in \mathbb{R}^d $.
Then note, as in the
proof of Corollary~\ref{cor:SVstability2}, 
that
\begin{equation}
  ( \mathcal{G}_{ \mu, \sigma } V)(x)
\leq
  2 \cdot
  p \cdot c \cdot V(x) ,
\end{equation}
\begin{equation}
  \left\|
    \mu(x)
  \right\|
  \leq
  2 c 
  \left|
    V(x)
  \right|^{
    \left[
      \frac{ \gamma_0 + 1 }{ p }
    \right]
  } 
\qquad
  \text{and}
\qquad
  \left\|
    \sigma(x)
  \right\|_{
    L( \mathbb{R}^m, \mathbb{R}^d )
  }
  \leq
  2c
  \left|
    V(x)
  \right|^{ 
    \left[
      \frac{
        \gamma_1 + 2
      }{ 2 p }
    \right]
  } 
\end{equation}
for all
$ x \in \mathbb{R}^d $.
Combining this,
the fact
$
  V \in C^3_p( \mathbb{R}^d, [1,\infty) )
$
and
Corollary~\ref{cor:apriori_increment} 
then shows that 
\begin{equation}
  \sup_{ N \in \mathbb{N} }
    \sup_{ t \in [0,T] }
    \mathbb{E}\big[
      \| 
        \bar{Y}^N_t
      \|^{ 
        q
      }
    \big] 
  < \infty 
\end{equation}
for all 
$ 
  q 
  \in 
  [0,\infty)
$
which satisfy
$
  q <
  \frac{
      p 
    }{
      2 \gamma_1 + 
      4 
      \max( 
        \gamma_0, \gamma_1,
        \frac{ 1 }{ 2 }
      )
    }
  -
  \frac{ 1 }{ 2 }
$
and
$
  \sup_{ 
    x \in \mathbb{R}^d
  }
  \| x \|^{ q } /
  V(x)
  < \infty
$.
The estimate
$
  \frac{
      p 
    }{
      2 \gamma_1 + 
      4 
      \max( 
        \gamma_0, \gamma_1,
        \frac{ 1 }{ 2 }
      )
    }
  -
  \frac{ 1 }{ 2 }
  < p
$
hence implies
\eqref{eq:inc_tamed_Euler_momentbound_p}
and this completes 
the proof of 
Corollary~\ref{cor:SVstabilityT2}.
\end{proof}

\section{Implicit 
approximation schemes}
\label{sec:implicit}

In this section, stability 
properties and moment 
bounds for implicit 
approximation
schemes are analyzed.
The main results of this 
section are
Corollary~\ref{c:Lyapunov.implicit.Euler}
for the fully 
drift-implicit Euler scheme
and Lemma~\ref{l:more.Lyapunov.implicit.Euler}
for partially 
drift-implicit approximation
schemes.

\subsection{Fully drift-implicit 
approximation schemes}
\label{sec:Full.drift.implicit.approximation.schemes}

Corollary~\ref{c:Lyapunov.implicit.Euler}
below proves uniform bounds on the
$ q $-th moments of fully 
drift-implicit Euler approximations
for a class of SDEs with globally one-sided Lipschitz continuous drift coefficients
where $ q \in [0,\infty) $.
This result generalizes,
in the case of the 
fully drift-implicit Euler 
scheme,
Theorem~3.6 of 
Mao \citationand~Szpurch~\cite{MaoSzpruch2012pre}
which establishes 
uniform bounds on
the second moments of the
numerical approximation 
processes.
First we prove
two auxiliary lemmas
(Lemma~\ref{l:estimate.F}
and Lemma~\ref{l:Lyapunov.implicit.Euler}).
The first lemma
is a slight generalization 
of Lemma~3.2
of Mao \citationand~Szpruch~\cite{MaoSzpruch2012pre}.

\begin{lemma} 
  \label{l:estimate.F}
  Let $ d \in \mathbb{N} $, 
  $ c \in (0,\infty) $
  and let
  $ \mu \colon \mathbb{R}^d
  \rightarrow \mathbb{R}^d $
  be a function
  with
  $
    \langle x, \mu(x) \rangle
    \leq 
    c \left( 1 + \|x\|^2 \right)
  $ 
  for all $ x \in \R^d $.
  Then
  \begin{equation}
    1+\left\| x - \mu(x) s \right\|^2
    \leq
    e^{ 4 c (t - s) }
    \left(
      1 + 
      \left\|
        x - \mu(x) t
      \right\|^2
    \right)
  \end{equation}
  for all
  $ x \in \R^d $
  and all 
  $ s, t \in [0, \frac{ 1 }{ 4 c } ] $
  with
  $ 
    s \leq t  
  $.
\end{lemma}

\begin{proof}[Proof of 
Lemma~\ref{l:estimate.F}]
Throughout this proof, let
  $
    F \colon [0,\infty) \times \R^d \to \R^d
  $ 
be a function defined through
\begin{equation}
  F_t(x) := x - \mu(x) t 
\end{equation}
for all 
$ (t,x) \in [0,\infty) \times \R^d $.
  The assumption
   $
    \langle x,\mu(x)\rangle\leq 
    c \left(1+\|x\|^2\right)
  $ 
  for all 
  $ x \in \R^d $ then
  implies
  \begin{equation}  
  \label{eq:F.lemma.Ft}
  \begin{split}  
    1+\left\|F_s(x)\right\|^2
    &=
    1+\left\|F_t(x)+(t-s)\mu(x)\right\|^2
    \\
    &=
    1+\left\|F_t(x)\right\|^2+2\langle x-t\mu(x),(t-s)\mu(x)\rangle
    +(t-s)^2\left\| \mu(x)\right\|^2
    \\
    &=
    1+\left\|F_t(x)\right\|^2+2(t-s)\langle x,\mu(x)\rangle
    -(t+s)(t-s)\left\| \mu(x)\right\|^2
    \\
    &\leq
    1+\left\|F_t(x)\right\|^2
  + 2 c \left( t - s \right)
    \left(1+\|x\|^2\right)
  \end{split}     \end{equation}
  for all $x\in\R^d$ and all
  $ s, t \in [0,\infty) $ with $ s \leq t $.
  The special case $s=0$ in 
  \eqref{eq:F.lemma.Ft} shows
  \begin{equation}  
  \begin{split}  
  \label{eq:F.lemma.x2}
    1+\|x\|^2
    &
    \leq 
    \frac{ 
      \left(
        1 + \left\| F_t(x) \right\|^2
      \right)
    }{ 
      ( 1 - 2 t c )
    }
  \end{split}     \end{equation}
  for all 
  $ x \in \R^d $
  and all
  $ t \in [0, \frac{ 1 }{ 2 c } ) $.
  Next we 
  apply~\eqref{eq:F.lemma.x2} 
  to~\eqref{eq:F.lemma.Ft} and 
  arrive at
  \begin{equation}  
  \begin{split}  
  \label{eq:F.lemma.Fs}
  &
    1 + \left\| F_s(x) \right\|^2
  \leq
    1 + 
    \left\| F_t(x) \right\|^2
    +
    2 c \left( t - s \right)
    \tfrac{ 
      ( 
        1 + \| F_t(x) \|^2
      )
    }{ 
      ( 1 - 2 c t ) 
    }
    \\
    &=
    \left(
      1 + \left\| F_t(x) \right\|^2
    \right)
    \left(
      1
      +
      \tfrac{ 
        2 c \left( t - s \right)
      }{ 
        \left( 1 - 2 c t \right) 
      }
    \right)
  \leq
    \left( 
      1 +
      \left\| F_t(x) \right\|^2
    \right)
    e^{ 
      \frac{ 2 c (t-s) }{ ( 1 - 2 c t ) } 
    }
  \end{split}     
  \end{equation}
  for all 
  $ x \in \R^d $
  and all
  $ s, t \in [ 0, \frac{ 1 }{ 2 c } ) $
  with $ s \leq t $.
  Combining this with the estimate
  \begin{equation}
    \frac{ 2 c \left( t - s \right) }{ 
      \left( 1 - 2 c t \right)
    }
    \leq 4 c \left( t - s \right)
  \end{equation}
  for all $ s, t \in [0, \frac{ 1 }{ 4 c } ] $
  with $ s \leq t $
  completes the proof
  of Lemma~\ref{l:estimate.F}.
\end{proof}

\begin{lemma}[Stability of
the fully drift-implicit 
Euler scheme]
\label{l:Lyapunov.implicit.Euler}
  Let
  $ d, m \in \mathbb{N} $,
  $ c \in (0,\infty) $,
  $ p \in [2,\infty) $
  and
  let
  $
    \mu \colon \mathbb{R}^d
    \rightarrow \mathbb{R}^d
  $,
  $
    \sigma 
    \colon \mathbb{R}^d
    \rightarrow \mathbb{R}^{ d \times m } 
  $
  be functions with
  \begin{equation}  
  \label{eq:one-sided.Lipschitz}
    \langle x,\mu(x)\rangle
  +
    \tfrac{ ( p - 1 ) }{ 2 }
    \|
      \sigma(x)
    \|_{ HS(\R^m,\R^d)
    }^2
    \leq
    c \left( 1 + \|x\|^2
    \right)
  \end{equation}
  for all $x\in\R^d$.
  Then there exists a real
  number $ \rho \in \mathbb{R} $
  such that
  \begin{equation}  \label{eq:Lyapunov.implicit.Euler}
    \E\!\left[
      \left(
        1 + 
        \left\| 
          x + \sigma(x) W_t
        \right\|^2
      \right)^q
    \right]
    \leq
    e^{ \rho t }
    \left(
      1 + 
      \left\|
        x - \mu(x) t
      \right\|^2
    \right)^q
  \end{equation}
  for all 
  $ x \in \R^d $,
  $ t \in [0,\frac{ 1 }{ 4 c }] $
  and all
  $ q \in [0,\tfrac{p}{2}] $
  where
  $
    W \colon [0,\infty) \times \Omega
    \rightarrow \mathbb{R}^m
  $
  is an arbitrary standard Brownian
  motion on a probability space
  $ 
    \left( \Omega, \mathcal{F},
    \mathbb{P} \right) 
  $.
\end{lemma}

\begin{proof}[Proof
of 
Lemma~\ref{l:Lyapunov.implicit.Euler}]
Throughout this proof,
let
  $
    F \colon [0,\infty) \times \R^d \to \R^d
  $ 
be a function defined through
  $ F_t(x) := x - \mu(x) t $ 
for all 
$ (t,x) \in [0,\infty) \times \R^d $
  and let 
  $ e^m_1 := (1, 0, \dots, 0) $, 
  $ \dots $, 
  $ e^m_m := (0, \dots, 0, 1) \in \R^m $
  be the canonical basis of
  $ \R^m $.
  It\^o's lemma yields
  \begin{equation}  
  \label{eq:F.diff_zero}
  \begin{split}  
  &
    \E\!\left[
      \left(
        1 + 
        \left\| x + \sigma(x) W_t
        \right\|^2
      \right)^q
    \right]
  \\ & =
    \left( 
      1 + \left\|x\right\|^2
    \right)^q
    +
    q 
    \left\| \sigma(x) \right\|^2_{
      HS( \R^m, \R^d )
    }
    \int_0^t 
    \E\!\left[
      \big(
        1 + 
        \left\| 
          x + \sigma(x) W_s
        \right\|^2
      \big)^{ (q -1) }
    \right]
    ds
  \\ & 
    +
    q \left( 2 q - 2 \right)
    \sum_{ k = 1 }^{ m } 
    \int_0^t 
    \E\!\left[
      \big(
        1 + 
        \left\| 
          x + \sigma(x) W_s
        \right\|^2
      \big)^{ (q -2) }
      \left|
      \left<
        x + \sigma(x) W_s, \sigma(x) e_k^{m}
      \right>
      \right|^2
    \right]
    ds
  \\ & \leq
    \left( 
      1 + \left\| x \right\|^2
    \right)^q
  \\ &
    +
    q \left( 2 q - 1 \right)
    \left\|
      \sigma(x)
    \right\|_{ HS(\R^m,\R^d) }^2
    \int_0^t 
      \E\!\left[
        \left(
          1 + 
          \left\| 
            x + \sigma(x) W_s
          \right\|^2
        \right)^{ (q - 1) }
      \right]
    ds
\end{split}
\end{equation}
for all 
$ t \in [0,\infty) $, 
$ x \in \R^d $ 
and all 
$ q \in [1,\infty) $.
  In addition, the fundamental theorem 
  of calculus implies
  \begin{equation}  
  \begin{split}  
  &
    \left(
      1 + 
      \left\| F_t(x) \right\|^2
    \right)^q
  \\ & =
    \left(
      1 + 
      \left\| F_0(x) \right\|^2
    \right)^q
    +
    2 q
    \int_0^t 
      \left(
        1 + 
        \left\| F_s(x) \right\|^2
      \right)^{ (q - 1) }
      \langle F_s(x), \tfrac{ \partial }{ \partial s } F_s( x ) \rangle
    \, ds
  \end{split}     
  \end{equation}
  and therefore
  \begin{equation}  
  \label{eq:F.fundamental}
  \begin{split}  
  &
    \left(
      1+ \left\| x \right\|^2
    \right)^q
  \\ & =
    \left(
      1 + 
      \left\| F_t(x) \right\|^2
    \right)^q
    +
    2 q
    \int_0^t 
      \left(
        1 + 
        \left\| F_s(x) \right\|^2
      \right)^{ (q - 1) }
      \langle F_s(x),\mu(x)\rangle
    \, ds
  \end{split}     
  \end{equation}
  for all 
  $ t \in [0,\infty) $, 
  $ x \in \R^d $ 
  and all 
  $ q \in [1,\infty) $.
Putting \eqref{eq:F.fundamental}
into \eqref{eq:F.diff_zero} 
gives
\begin{equation}
\label{eq:F.diff}
\begin{split}
  &
    \E\!\left[
      \left(
        1 + 
        \left\| x + \sigma(x) W_t
        \right\|^2
      \right)^q
    \right]
  \\ & \leq
    \left(
      1 + 
      \left\| F_t(x) \right\|^2
    \right)^q
    +
    2 q
    \int_0^t
    \left(
      1 + \left\|F_s(x)\right\|^2
    \right)^{ (q - 1) }
    \langle F_s(x),\mu(x)\rangle\,ds
  \\ & \quad
    +
    q \left( 2 q - 1 \right)
    \left\|
      \sigma(x)
    \right\|_{ HS(\R^m,\R^d) }^2
    \int_0^t 
    \E\!\left[
      \left(
        1 + 
        \left\|
          x + \sigma(x) W_s
        \right\|^2
      \right)^{ (q - 1) }
     \right]
     ds
  \end{split}     
  \end{equation}
  for all 
  $ t \in [0,\infty) $, 
  $ x \in \R^d $ 
  and all $ q \in [1,\infty) $.
  Roughly speaking,
  we now use \eqref{eq:F.diff}
  to prove 
  \eqref{eq:Lyapunov.implicit.Euler}
  by induction on $ q \in [0,\frac{p}{2}] $.
  More precisely, let
  $ 
    \kappa \colon [0, \frac{p}{2}] 
    \rightarrow [0,\infty) 
  $
  be a function defined recursively through
  $ \kappa( q ) := 6c q $
  for all  
  $ q \in [0,1] $
  and through
  \begin{equation}
    \kappa(q) :=
      2 p^3 c
      \exp\!\left( 
        \frac{ \kappa(q - 1) }{ 2 c } 
        + p
      \right)
  \end{equation}
  for all 
  $ 
    q \in 
    ( n, n+1 ]
    \cap [0, \frac{p}{2} ]
  $
  and all
  $ n \in \mathbb{N} $. 
  We then prove 
  \begin{equation}  
  \label{eq:Lyapunov.implicit.Euler2}
    \E\!\left[
      \left(
        1 + 
        \left\| 
          x + \sigma(x) W_t
        \right\|^2
      \right)^{ q }
    \right]
    \leq
    e^{ \kappa(q) t }
    \left(
      1 + 
      \left\|
        F_t(x)
      \right\|^2
    \right)^q
  \end{equation}
  for all 
  $ t \in [0, \tfrac{ 1 }{ 4 c }] $, 
  $ x \in \R^d $ 
  and all 
  $ q \in ( n, n+1 ] \cap [0,\frac{p}{2}] $
  by induction on 
  $ n \in \N_0 $. 
  For the case $ n = 0 $
  and $ q = 1 $, 
  we apply 
  assumption~\eqref{eq:one-sided.Lipschitz}
  and Lemma~\ref{l:estimate.F} 
  to the right-hand side 
  of~\eqref{eq:F.diff} 
  and get
  \begin{equation}  
  \label{eq:case.p1}
  \begin{split}  
  &
    \E\!\left[
      1+\left\|x+\sigma(x)W_t\right\|^2
    \right]
  \\ & \leq
    1 + 
    \left\| F_t(x) \right\|^2
    + 
    2
    \int_0^t 
      \langle F_s( x ), \mu(x) \rangle
    \, ds
    +
    \int_0^t 
      \left\|
        \sigma(x)
      \right\|_{ HS( \R^m, \R^d ) }^2
    ds
  \\ & = 
    1 + 
    \left\| F_t(x) \right\|^2
    +
    2
    \int_0^t 
    \left(
      \langle x,\mu(x) \rangle
      +
      \tfrac{ 1 }{ 2 }
      \left\|
        \sigma(x)
      \right\|_{ HS( \R^m, \R^d ) }^2
      - 
      s 
      \left\| \mu(x) \right\|^2
    \right)
    ds
  \\ & \leq
    1 + 
    \left\| F_t(x) \right\|^2
    +
    2
    \int_0^t 
    \left(
      \langle x,\mu(x) \rangle
      +
      \tfrac{ ( p - 1 ) }{ 2 }
      \left\|
        \sigma(x)
      \right\|_{ HS( \R^m, \R^d ) }^2
    \right)
    ds
  \\ & \leq 
    1 +
    \left\|
      F_t(x)
    \right\|^2
    +
    2 c 
    \int_0^t 
    \left( 1 + \|x\|^2 \right)
    ds
  \leq
    1 + 
    \left\| F_t(x) \right\|^2
    +
    2 t c \, e^{ 4 t c }
    \left(
      1 + \left\|F_t(x)\right\|^2
    \right)
  \\ & \leq
    \left(
      1 +
      2 e c t
    \right)
    \left(
      1 + \left\|F_t(x)\right\|^2
    \right)
  \leq
    e^{ 2 e c t }
    \left(
      1 + \left\| F_t(x) \right\|^2
    \right)
  \leq
    e^{ \kappa(1) t }
    \left(
      1 + \left\| F_t(x) \right\|^2
    \right)
  \end{split}     
  \end{equation}
  for all 
  $
    t \in [0, \tfrac{1}{4 c} ]
  $ and all 
  $ x \in \R^d $.
  Next observe 
  for every $ q \in (0,1) $ that
  the function
  $ 
    [0,\infty) \ni z \mapsto z^q \in [0,\infty)
  $ 
  is concave.
  Hence, Jensen's 
  inequality and~\eqref{eq:case.p1}
  imply
  \begin{equation}  
  \begin{split}
    \E\!\left[
      \left( 
        1 + \left\| x + \sigma(x) W_t \right\|^2
      \right)^q
    \right]
  & \leq
    \left(
      \E\!\left[
        1 + 
        \left\| x + \sigma(x) W_t \right\|^2
      \right]
    \right)^q
  \\ \leq
    \left(
      e^{ \kappa(1) t }
      \left( 
        1 + \left\| F_t(x) \right\|^2
      \right)
    \right)^q
    &=
      e^{ \kappa(1) q t }
      \left( 
        1 + \left\| F_t(x) \right\|^2
      \right)^q
    =
      e^{ \kappa(q) t }
      \left( 
        1 + \left\| F_t(x) \right\|^2
      \right)^q
  \end{split}     \end{equation}
  for all 
  $ t \in [0,\tfrac{1}{4 c}] $, 
  $ x \in \R^d $ 
  and all 
  $ q \in [0,1] $.
  This proves 
  \eqref{eq:Lyapunov.implicit.Euler2}
  in the base case 
  $ n = 0 $.
  For the induction step 
  $ n \to n + 1 $,
  apply the induction hypothesis 
  on the right-hand side of~\eqref{eq:F.diff}
  to obtain
  \begin{equation}  
  \begin{split}
  &
    \E\!\left[
      \left( 
        1 +
        \left\| x + \sigma(x) W_t
        \right\|^2
      \right)^q
    \right]
    -
    \left(
      1 + \left\| F_t(x) \right\|^2
    \right)^q
  \\ & \leq
    \int_0^t 
    \left( 
      1 + \left\| F_s(x) \right\|^2
    \right)^{ ( q - 1 ) }
  \\ & \cdot
    \left[
      2 q 
      \left< F_s(x), \mu(x) \right>
      +
      q \left( 2 q - 1 \right) 
      e^{ \kappa(q-1) s }
      \left\| \sigma(x) \right\|_{
        HS(\R^m,\R^d)
      }^2
    \right]
    ds
  \\ & =
    2 q
    \int_0^t
    \left(
      1 + \left\| F_s(x) \right\|^2
    \right)^{ (q - 1) }
    \Big[
      \langle x, \mu(x) \rangle
      +
      \tfrac{ (2 q - 1 ) }{ 2 }
      \left\| \sigma(x) 
      \right\|_{ HS(\R^m, \R^d ) }^2
    \Big]
    \, ds
  \\ & 
    +
    2 q 
    \int_0^t 
    \left(
      1 + 
      \left\| F_s(x) \right\|^2
    \right)^{ ( q - 1 ) }
  \\ & \cdot
    \left[
      \tfrac{ ( 2 q - 1 ) }{ 2 } 
      \left( 
        e^{ \kappa(q-1) s } - 1
      \right)
      \left\| \sigma(x)
      \right\|_{ HS(\R^m,\R^d) }^2
      - s \left\| \mu(x) \right\|^2
    \right]
    ds
  \end{split}     
  \end{equation}
  and 
  assumption~\eqref{eq:one-sided.Lipschitz}
  and
  the fact that the function
  $ 
    (0,\infty) \ni x \mapsto
    \frac{ ( e^x - 1 ) }{ x }
    \in (0,\infty)
  $
  is increasing hence show
  \begin{equation} 
  \label{eq:for.now} 
  \begin{split}
  &
    \E\!\left[
      \left( 
        1 +
        \left\| x + \sigma(x) W_t
        \right\|^2
      \right)^q
    \right]
    -
    \left(
      1 + \left\| F_t(x) \right\|^2
    \right)^q
  \\ & \leq
    2 q c
    \int_0^t
    \left(
      1 + \left\| F_s(x) \right\|^2
    \right)^{ (q - 1) }
    \left( 1 + \| x \|^2 \right)
    ds
  \\ & 
    +
    2 q 
    \int_0^t 
    \left(
      1 + 
      \left\| F_s(x) \right\|^2
    \right)^{ ( q - 1 ) }
  \\ & 
    \cdot
     s 
    \left[
      \left( q - \tfrac{ 1 }{ 2 } \right) 
      \tfrac{ 
        \left( 
          e^{ \kappa(q - 1) t } - 1
        \right)
      }{ t } 
      \left\| \sigma(x)
      \right\|_{ HS(\R^m,\R^d) }^2
      - \left\| \mu(x) \right\|^2
    \right]
    ds
  \\ & \leq
    2 q c
    \int_0^t
    \left(
      1 + \left\| F_s(x) \right\|^2
    \right)^{ (q - 1) }
    \left( 1 + \| x \|^2 \right)
    ds
  \\ & 
    +
    2 q 
    \int_0^t 
    \left(
      1 + 
      \left\| F_s(x) \right\|^2
    \right)^{ ( q - 1 ) }
  \\ & 
    \cdot
     s 
    \left[
      \left( q - \tfrac{ 1 }{ 2 } \right) 
      4 c
      \left( 
        e^{ \frac{ \kappa(q - 1) }{ 4 c } } - 1
      \right)
      \left\| \sigma(x)
      \right\|_{ HS(\R^m,\R^d) }^2
      - \left\| \mu(x) \right\|^2
    \right]
    ds
  \\ & \leq
    2 q c
    \int_0^t
    \left(
      1 + \left\| F_s(x) \right\|^2
    \right)^{ (q - 1) }
    \left( 1 + \| x \|^2 \right)
    ds
  \\ & 
    +
    2 q 
    \int_0^t 
    \left(
      1 + 
      \left\| F_s(x) \right\|^2
    \right)^{ ( q - 1 ) }
  \\ & \cdot
    s 
    \left[
      2 c p 
      \exp\!\left( 
        \tfrac{ \kappa(q - 1) }{ 4 c } 
      \right)
      \left\| \sigma(x)
      \right\|_{ HS(\R^m,\R^d) }^2
      - \left\| \mu(x) \right\|^2
    \right]
    ds
  \end{split}     
  \end{equation}
  for all 
  $ t \in [0, \tfrac{ 1 }{ 4 c }] $, 
  $ x \in \R^d $ 
  and all 
  $ q \in (n+1,n+2] \cap [0,\frac{p}{2}] $.
  Next observe that 
  Young's inequality 
  and again
  assumption~\eqref{eq:one-sided.Lipschitz}
  give
  \begin{equation}  
  \begin{split}  
  \label{eq:for.later}
  &
    r
    \left\| \sigma(x) 
    \right\|_{ HS(\R^m,\R^d) }^2
    - \left\| \mu(x) \right\|^2
  \\ & =
    \frac{ 2 r }{ (p - 1 ) }
    \left[
      \left< x, \mu(x) \right>
      +
      \tfrac{ (p - 1) }{ 2 } 
      \left\| \sigma(x) \right\|^2_{ 
        HS( \R^m, \R^d )
      }
    \right]
    -
    \frac{
      2 r \left< x, \mu(x) \right>
    }{ 
      (p - 1) 
    }
    - \left\| \mu(x) \right\|^2
  \\ & \leq
    \frac{ 
      2 r c 
      \left( 1 + \| x \|^2 \right)
    }{ (p - 1 ) }
    +
    \frac{
      2 r \left\| x \right\| 
      \left\| \mu(x) \right\|
    }{ 
      (p - 1) 
    }
    - \left\| \mu(x) \right\|^2
  \leq
    \frac{ 
      2 r c 
      \left( 1 + \| x \|^2 \right)
    }{ (p - 1 ) }
    +
    \frac{
      r^2 \left\| x \right\|^2 
    }{ 
      \left( p - 1 \right)^2 
    }
  \\ & \leq
    \left(
      \frac{ 
        2 r c 
      }{ (p - 1 ) }
      +
      \frac{
        r^2 
      }{ 
        \left( p - 1 \right)^2 
      }
    \right)
    \left( 1 + \| x \|^2 \right)
  \leq
    \left(
      2 r c 
      +
      r^2 
    \right)
    \left( 1 + \| x \|^2 \right)
  \end{split}     
  \end{equation}
  for all $ x \in \R^d $ and all
  $ r \in [0,\infty) $.
  Combining 
  \eqref{eq:for.now} 
  and 
  \eqref{eq:for.later} 
  implies
  \begin{equation}  
  \begin{split}
  &
    \E\!\left[
      \left( 
        1 +
        \left\| x + \sigma(x) W_t
        \right\|^2
      \right)^q
    \right]
    -
    \left(
      1 + \left\| F_t(x) \right\|^2
    \right)^q
  \\ & \leq
    2 q c
    \int_0^t
    \left(
      1 + \left\| F_s(x) \right\|^2
    \right)^{ (q - 1) }
    \left( 1 + \| x \|^2 \right)
    ds
  \\ & 
    +
    2 q 
    \int_0^t 
    \left(
      1 + 
      \left\| F_s(x) \right\|^2
    \right)^{ ( q - 1 ) }
    \cdot s \cdot
    \left(
      4 c^2 p 
      +
      4 c^2 p^2 
    \right)
      \exp\!\left( 
        \tfrac{ \kappa(q - 1) }{ 2 c } 
      \right)
    \left(
      1 + \| x \|^2
    \right)
     \,
    ds
  \\ & =
    2 q c
    \int_0^t
    \left(
      1 + \left\| F_s(x) \right\|^2
    \right)^{ (q - 1) }
    \left[
    1 +
    4 c s 
    \left(
      p 
      +
      p^2 
    \right)
      \exp\!\left( 
        \tfrac{ \kappa(q - 1) }{ 2 c } 
      \right)
    \right]
    \left( 1 + \| x \|^2 \right)
    ds
  \end{split}     
  \end{equation}
  and 
  Lemma~\ref{l:estimate.F}
  hence gives
  \begin{equation}  
  \begin{split}
  &
    \E\!\left[
      \left( 
        1 +
        \left\| x + \sigma(x) W_t
        \right\|^2
      \right)^q
    \right]
    -
    \left(
      1 + \left\| F_t(x) \right\|^2
    \right)^q
  \\ & \leq 
    2 q c
    \left(
      1
      +
      \tfrac{ 3 p^2 }{ 2 } 
      \exp\!\left( 
        \tfrac{ \kappa(q - 1) }{ 2 c } 
      \right)
    \right)
    \int_0^t 
    \left(
      1 + 
      \left\| F_s(x) \right\|^2
    \right)^{ ( q - 1 ) }
    \left(
      1 + \| x \|^2
    \right)
    ds
  \\ & \leq
    p c
    \left(
      1
      +
      \tfrac{ 3 p^2 }{ 2 } 
      \exp\!\left( 
        \tfrac{ \kappa(q - 1) }{ 2 c } 
      \right)
    \right)
    \int_0^t 
    e^{ 
      \left(
        4 c \left( t - s \right) \left( q - 1 \right) 
        + 4 c t
      \right)
    }
    \left(
      1 + 
      \left\| F_t(x) \right\|^2
    \right)^{ q }
    ds
  \\ & \leq
    p c
    \left(
      1
      +
      \tfrac{ 3 p^2 }{ 2 } 
    \right)
      \exp\!\left( 
        \tfrac{ \kappa(q - 1) }{ 2 c } 
      \right)
    \int_0^t 
    e^{ 
      4 c t q 
    }
    \left(
      1 + 
      \left\| F_t(x) \right\|^2
    \right)^{ q }
    ds
  \\ & \leq
    2 p^3 c \, t
      \exp\!\left( 
        \tfrac{ \kappa(q - 1) }{ 2 c } 
        + \tfrac{ p }{ 2 }
      \right)
    \left(
      1 + 
      \left\| F_t(x) \right\|^2
    \right)^{ q }
  \end{split}     
  \end{equation}
  for all 
  $ t \in [0, \tfrac{ 1 }{ 4 c }] $, 
  $ x \in \R^d $ 
  and all 
  $ q \in (n+1,n+2] \cap [0,\frac{p}{2}] $.  
  The estimate
  $ 1 + r \leq e^{ r } $ for all
  $ r \in \mathbb{R} $ therefore shows
  \begin{equation}  
  \begin{split}
  &
    \E\!\left[
      \left( 
        1 +
        \left\| x + \sigma(x) W_t
        \right\|^2
      \right)^q
    \right]
  \\ & \leq
    \left[
      1 +
      2 p^3 c \, t
      \exp\!\left( 
        \tfrac{ \kappa(q - 1) }{ 2 c } 
        + p
      \right)
    \right]
    \left(
      1 + 
      \left\| F_t(x) \right\|^2
    \right)^{ q }
  \\ & \leq
    e^{  t \cdot \kappa(q) }
    \left(
      1 + 
      \left\| F_t(x) \right\|^2
    \right)^{ q }
  \end{split}     
  \end{equation}
  for all 
  $ t \in [0, \tfrac{ 1 }{ 4 c }] $, 
  $ x \in \R^d $ 
  and all 
  $ q \in (n+1,n+2] \cap [0,\frac{p}{2}] $.
  This finishes the induction 
  step and  
  inequality~\eqref{eq:Lyapunov.implicit.Euler2}
  thus holds for all $ n \in \mathbb{N}_0 $.
  In particular,
  we get from 
  inequality~\eqref{eq:Lyapunov.implicit.Euler2}
  that
  \begin{equation}  
    \E\!\left[
      \left(
        1 + 
        \left\| 
          x + \sigma(x) W_t
        \right\|^2
      \right)^{ q }
    \right]
    \leq
    \exp\!\left( 
      t\cdot
        \sup\nolimits_{ r \in [0,\frac{p}{2}] } 
        \kappa(r) 
    \right)
    \left(
      1 + 
      \left\|
        x - \mu(x) t
      \right\|^2
    \right)^q
  \end{equation}
  for all 
  $ t \in [0, \tfrac{ 1 }{ 4 c }] $, 
  $ x \in \R^d $ 
  and all 
  $ q \in [0,\frac{p}{2}] $.
  This and the estimate
  $ 
    \sup_{ r \in [0,\frac{p}{2}] } 
    \kappa(r)
    < \infty
  $ 
  then
  complete the proof of
  Lemma~\ref{l:Lyapunov.implicit.Euler}.
\end{proof}

Now we apply
Lemma~\ref{l:Lyapunov.implicit.Euler}
and Corollary~\ref{cor:stability2}
to obtain moment bounds for 
fully drift-implicit Euler approximations.

\begin{cor}  
\label{c:Lyapunov.implicit.Euler}
  Let 
  $ d, m \in \mathbb{N} $,
  $ c, T \in (0,\infty) $,
  $ p \in [2,\infty) $
  be real numbers,
  let
  $
    \left( 
      \Omega, \mathcal{F}, 
      ( \mathcal{F}_t )_{ t \in [0,T] },
      \mathbb{P}
    \right)
  $
  be a filtered probability space,
  let 
  $ 
    W \colon [0,T] \times
    \Omega \rightarrow \mathbb{R}^m 
  $
  be a standard
  $ ( \mathcal{F}_t )_{ t \in [0,T] } $-Brownian
  motion, 
  let 
  $ 
    \xi\colon\Omega\to\R^d
  $ be an 
  $\mathcal{F}_0/\mathcal{B}(\R^d)
  $-measurable
  function 
  and let
  $
    \mu \colon \mathbb{R}^d 
    \rightarrow \mathbb{R}^d  
  $,
  $
    \sigma \colon \mathbb{R}^d
    \rightarrow \mathbb{R}^{ d \times m }
  $
  be Borel measurable functions with
  $
    \E\big[ \|\mu(\xi)\|^p \big]<\infty
  $
  and
  \begin{align}
    \langle x-y, \mu(x)-\mu(y) \rangle
    & \leq
    c \left\| x - y \right\|^2 ,
    \\
    \langle x,\mu(x) \rangle
    +
    \tfrac{ (p - 1) }{ 2 }
    \| 
      \sigma(x)
    \|_{ HS(\R^m, \R^d) }^2
  & \leq
    c \left( 1 + \|x\|^2 \right)
  \end{align}
  for all $ x, y \in \R^d $.
  Then there exists a unique 
  family 
  $
    Y^N
    \colon \{ 0, 1, \ldots, N \}
    \times \Omega \to \R^d
  $, 
  $ 
    N \in \mathbb{N} \cap ( c T, \infty )
  $,   
  of stochastic processes
  satisfying
  $ Y_0^N = \xi $ 
  and 
    \begin{equation}
      Y_{n+1}^N=
      Y_n^N
      +
      \mu( Y^N_{ n + 1 } )
      \tfrac{ T }{ N }
      +
      \sigma(Y_n^N )
      \big(
        W_{ (n + 1) T / N }
        - W_{ n T / N }
      \big)
    \end{equation}
  for all $ n \in \{0, 1, \ldots, N-1 \}$
  and all
  $ N \in \N \cap ( c T, \infty ) $
  and there exists a real number $\rho\in(0,\infty)$ such that
  \begin{equation} 
  \label{eq:moment.bound.implicit.Euler}
    \limsup_{ N \to \infty }
    \sup_{ n \in \{0,1,\ldots,N\} }
    \E\Big[
      \big\{
        1+ \|Y_{n}^N \|^2
      \big\}^{ q }
    \Big]
    \leq
    e^{\rho T}
    \cdot
    \E\!\left[
      \left\{ 
        1 + \left\| \xi \right\|^2
      \right\}^{ q }
    \right]
  \end{equation}
  for all $ q \in [0,\tfrac{p}{2}] $.
\end{cor}

\begin{proof}[Proof 
  of 
  Corollary~\ref{c:Lyapunov.implicit.Euler}]
  Throughout this proof, let
  $
    F \colon [0,\infty) \times \R^d \to \R^d
  $ 
  be a function defined through
  $ F_t(x) := x - \mu(x) t $ 
for all 
$ (t,x) \in [0,\infty) \times \R^d $.
  Next observe that
  \begin{equation}
    \left< 
      x - y, 
      \left( \mu(x) t - x \right) - 
      \left( \mu(y) t - y \right)
    \right> 
  \leq
    \left( c t - 1 \right)
    \left\| x - y \right\|^2
  \end{equation}
  for all $ x, y \in \mathbb{R}^d $
  and all $ t \in [0,\infty) $.
  This inequality ensures the unique 
  existence of the stochastic
  processes
  $
    Y^N
    \colon \{ 0, 1, \ldots, N \}
    \times \Omega \to \R^d
  $, 
  $ 
    N \in \mathbb{N} \cap ( c T, \infty )
  $.
  In the next step, note that
  \begin{equation}
    \big\|
      F_{ T / N }( Y^N_{ n + 1 } )
    \big\|^2 =
    \big\|
      Y_n^N
      +
      \sigma(Y_n^N )
      \big(
        W_{ ( n + 1 ) T / N }
        - W_{ n T / N }
      \big)
    \big\|^2
  \end{equation}
  for all 
  $ 
    n \in \{ 0, 1,\ldots,N - 1 \} 
  $ and all
  $ N \in \N \cap ( c T , \infty ) $.
  Lemma~\ref{l:Lyapunov.implicit.Euler} 
  hence implies the existence of 
  a real number $ \rho \in \mathbb{R} $
  such that
  \begin{equation}  \label{eq:cor.Lyapunov.implicit.Euler}
    \E\Big[
      \big\{
        1 + 
        \| 
          F_{ T / N }(
            Y_{n+1}^N
          )
        \|^2
      \big\}^{ q } \,
      \big| \, Y_n^N
    \Big]
  \leq
    \exp\!\left(
      \tfrac{ \rho T }{ N } 
    \right)
    \cdot
    \left\{
      1 + 
      \|
        F_{ T / N }(
          Y_n^N
        )
      \|^2
    \right\}^{ q }
  \end{equation}
  $ \mathbb{P} $-a.s.\ for 
  all 
  $ 
    n \in \{ 0, 1,\ldots,N - 1 \} 
  $,
  $ N \in \N \cap [ 4 c T , \infty ) $
  and all $ q \in [0,\frac{p}{2}] $.
  Next fix a real number
  $ q \in [0,\frac{p}{2}] $
  and we now prove
  \eqref{eq:moment.bound.implicit.Euler}
  for this
  $ q \in [0,\frac{p}{2}] $.
  If 
  $ 
    \E\big[ 
      \| \xi \|^{ 2 q } 
    \big] = \infty
  $,
  then \eqref{eq:moment.bound.implicit.Euler} 
  is trivial.
  We thus assume
  $
    \E\big[
      \|\xi\|^{ 2 q }
    \big]
    < \infty
  $ 
  for the rest of this proof.
  Hence, we obtain that
  \begin{equation}
    \E\big[
      \| \xi \|^{ 2 q }
      +
      \| \mu(\xi) \|^{ 2 q }
    \big]
    < \infty .
  \end{equation}
  Now we 
  apply
  Corollary~\ref{cor:stability2}
  with the Lyapunov-type function
  $ 
    V \colon \R^d \rightarrow [0,\infty) 
  $
  given by
  \begin{equation}
    V( x ) = 
    \big\{
      1 + \| F_{ T / N }( x ) \|^2
    \big\}^{ q }
  \end{equation}
  for all $ x \in \R^d $,
  with the truncation function
  $
    \zeta \colon [0,\infty)
    \rightarrow (0,\infty]
  $
  given by
  $ 
    \zeta(t) = \infty
  $
  for all $ t \in [0,\infty) $
  and with the sequence
  $ t_n \in \R $, $ n \in \N_0 $,
  given by
  $ 
    t_n = \min( n T / N, T )
  $ 
  for all 
  $ n \in \N_0 $
  to obtain
  \begin{equation}
    \sup_{ n \in \{ 0, 1, \dots, N \} }
    \E\Big[
      \big\{
        1 + 
        \| 
          F_{ T / N }( Y_{n}^N )
        \|^2
      \big\}^{ q }
    \Big]
  \leq
    e^{ \rho T }
    \cdot
    \E\Big[
      \big\{
        1 +
        \| 
          F_{ T / N }( \xi  )
        \|^2
      \big\}^{ q }
    \Big]
  \end{equation}
  for all $ n \in \{ 0, 1, \ldots, N \} $
  and all $ N \in \N \cap [ 4 c T, \infty ) $.
  Lemma~\ref{l:estimate.F}
  and the dominated convergence
  theorem  hence give
  \begin{equation}  
  \begin{split}
  &
    \limsup_{ N \to \infty }
    \sup_{ 
      n \in\{ 0, 1, \ldots, N \} 
    }
    \E\Big[
      \big\{
        1 + \| Y_n^N \|^2
      \big\}^{ q }
    \Big]
  \\ & \leq
    \lim_{ N \to \infty }
    \left(
    e^{ \frac{ 4 c T }{ N } }
    \cdot
    e^{ \rho T }
    \cdot
    \E\Big[
      \big\{
        1 + 
        \| 
          F_{ T / N }( \xi )
        \|^2
      \big\}^{ q }
    \Big]
    \right)
  \\ & =
    e^{ \rho T }
    \cdot
    \E\Big[
      \lim_{N\to\infty}
      \big\{
        1 + 
        \|
          F_{ T / N }( \xi )
        \|^2
      \big\}^{ q }
    \Big]
   =
    e^{ \rho T }
    \cdot
    \E\Big[
      \big\{ 
        1 +
        \| \xi \|^2
      \big\}^{ q }
    \Big] 
  \end{split}     
  \end{equation}
  and this completes the proof
  of 
  Corollary~\ref{c:Lyapunov.implicit.Euler}.
\end{proof}

\subsection{Partially
drift-implicit approximation 
schemes}
\label{sec:Partial.drift.implicit.approximation.schemes}

In Subsection \ref{sec:Full.drift.implicit.approximation.schemes}
above,
moment bounds
for the fully 
drift-implicit Euler
schemes have been 
established.
This subsection concentrates on partially
drift-implicit schemes.

\begin{lemma}[Partially 
drift-implicit
schemes for SDEs 
with at most
linearly growing diffusion coefficients]
  \label{l:more.Lyapunov.implicit.Euler}
  Let 
  $ h \in (0,\infty) $,
  $ c \in [1,\infty) $
  be real numbers with
  $ h c \leq \frac{ 1 }{ 4 } $, 
  let
  $ 
    \left( 
      \Omega, \mathcal{F}, 
      \mathbb{P}
    \right)
  $
  be a probability space,
  let
  $
    W \colon [0,\infty) \times
    \Omega \rightarrow \mathbb{R}^m
  $
  be a standard Brownian motion,
  let 
  $
    \varphi \colon
    \R^d \times \R^d \to \R^d
  $,
  $
    \sigma \colon
    \R^d \to \R^{ d \times m }
  $
  be Borel measurable   
  functions with
\begin{equation}  
\label{eq:phi}
    \langle y, \varphi(x,y) \rangle
  \leq
    c
    \left( 2 + \| x \|^2 + \| y \|^2 \right) ,
  \qquad
    \left\| \sigma(x) \right\|_{ 
      HS( \R^m, \R^d ) 
    }^2
    \leq
    c 
    \left( 1 + \|x\|^2 \right)
\end{equation}
for all $ x, y \in \R^d $ and let
  $
    Y \colon \mathbb{N}_0
    \times \Omega \to \R^d
  $
  be a
  stochastic process with
  \begin{equation}
    Y_{ n + 1 }
    =
    Y_n
    +
    \varphi\!\left( Y_n, Y_{n+1} \right)
    h
    +
    \sigma( Y_n ) \,
    \big(
      W_{ (n+1) h } - 
      W_{ n h }
    \big)
  \end{equation}
  for all 
  $ n \in \N_0 $. 
  Then
  \begin{equation}  
  \label{eq:more.Lyapunov.implicit}
    \bigg\|
      \sup_{ 
        k \in\{ 0, 1, \ldots, n \}
      }
      \left\| Y_k \right\|
    \bigg\|_{
      L^{ p }( \Omega; \R)
    }
  \leq
    2
    \left(
      1 + 
      \left\| Y_0 \right\|_{
        L^{ p }(\Omega; \R^d)
      }
    \right)
    \exp\!\big(
      (
        4 + 
        p^4
        \chi_{ p / 2  } 
      )
      c^2
      n h
    \big)
  \end{equation}
  for all $ p \in [4,\infty) $
  and all $ n \in \N_0 $.
\end{lemma}

\begin{proof}[Proof
of
Lemma~\ref{l:more.Lyapunov.implicit.Euler}]
  First of all, 
  let 
  $ p \in [4,\infty) $ be arbitrary.
  If 
  $ \E\big[ \| Y_0 \|^{ p } \big] = \infty 
  $, 
  then 
  inequality~\eqref{eq:more.Lyapunov.implicit}
  is trivial.
  Thus we assume 
  $ \E\big[ \| Y_0 \|^{ p } \big] < \infty $ 
  for the rest of this proof.
  In the sequel, we will show
  \eqref{eq:more.Lyapunov.implicit} 
  by an application
  of Corollary~\ref{cor:stability}.
  For this we define a stochastic process
  $ 
    Z \colon \N \times \Omega \to \R^d
  $ 
  through
  \begin{equation}  
  \label{eq:defZ}
  \begin{split}
    Z_{ n }
  := &
    \frac{ 1 }{
      ( 1 - 2 h c  ) 
    }
    \Big( 
      2 
      \left\langle 
        Y_{ n - 1 },
        \sigma( Y_{ n - 1 } )
        \big(
          W_{ n h } - 
          W_{ (n - 1) h }
        \big)
      \right\rangle
  \\ & +
      \big\|
        \sigma( Y_{ n - 1 } )
        \big(
          W_{ n h } -
          W_{ (n - 1) h }
        \big)
      \big\|^2
      -
      h
      \left\|
        \sigma( Y_{ n - 1 } )
      \right\|_{ HS( \R^m, \R^d ) }^2
  \Big)
  \end{split}     
  \end{equation}
  for all 
  $ n \in \N $ and
  we then verify 
  inequality~\eqref{eq:semi.Lyapunov.martingal}
  and
  inequality~\eqref{eq:variationZbounded.by.Y}.
  For
  inequality~\eqref{eq:semi.Lyapunov.martingal}
  note that
  Young's inequality 
  and assumption~\eqref{eq:phi}
  give
  \begin{equation}  
  \label{eq:estfirst}
  \begin{split}
  &
    2 \left\| Y_{ n } \right\|^2
    =
    2 \left\langle Y_{ n }, Y_{ n - 1 }
      + 
      \varphi\!\left( Y_{ n - 1 }, Y_{ n } \right) h
      + \sigma( Y_{ n - 1 } )
      \big(
        W_{ n h } -
        W_{ (n - 1) h }
      \big)
      \right\rangle
  \\
    & = 
    2 \left\langle Y_{ n }, Y_{ n - 1 }
      + \sigma( Y_{ n - 1 } )
      \big(
        W_{ n h } -
        W_{ (n - 1) h }
      \big)
      \right\rangle
      +
      2 h  
      \left\langle 
        Y_{ n },
        \varphi\!\left( Y_{ n - 1 }, Y_{ n } \right)
      \right\rangle
  \\ & \leq
     \left\| 
       Y_n
     \right\|^2
    +
    \left\|
      Y_{ n - 1 } + 
      \sigma( Y_{ n - 1 } )
      \big(
        W_{ n h } -
        W_{ (n - 1) h }
      \big)
    \right\|^2
  \\ & \quad
    +
    2 h c 
    \left( 
      2 + \| Y_{ n - 1 } \|^2 + 
      \| Y_n \|^2
    \right)
  \end{split}     
  \end{equation}
  for all 
  $ n \in \N $. 
  Rearranging \eqref{eq:estfirst}
  yields
  \begin{equation}  
\label{eq:lastest}
  \begin{split}
  &
    1 + \left\| Y_n \right\|^2
  \\
    &\leq
     \frac{
        1 +
       \left\|
         Y_{ n - 1 } +
         \sigma( Y_{ n - 1 } )
         \big(
           W_{ n h } -
           W_{ (n - 1) h }
         \big)
       \right\|^2
      + 2 h c 
      \left( 1 + \| Y_{ n - 1 } \|^2\right)
      }{ ( 1 - 2 h c ) }
    \\
    &=
     \frac{
        1 + \left\| Y_{ n - 1 } \right\|^2
        +
        h
        \left\| 
          \sigma( Y_{ n - 1 } ) 
        \right\|_{
          HS( \R^m, \R^d )
        }^2
        + 2 h c 
        \left(
          1 + \| Y_{ n - 1 } \|^2
        \right)
      }{ ( 1 - 2 h c ) }
      + Z_{ n }
  \\ & \leq
     \frac{
        \left(
          1 + \| Y_{ n - 1 } \|^2
        \right)
        \left( 1 + 3 h c \right)
      }{ ( 1 - 2 h c ) }
      + Z_{ n }
    \leq
      e^{ 7 h c }
      \left( 1 + \| Y_{ n - 1 } \|^2 \right)
      + Z_{ n }
  \end{split}     
  \end{equation}
  for all $ n \in \N $ where the last
inequality follows from the estimate
$ 
  \frac{ 1 + 3 x }{ 1 - 2 x } \leq e^{ 7 x} 
$
for all $ x \in [0,\frac{1}{4}] $.
Estimate~\eqref{eq:lastest} is
inequality 
\eqref{eq:semi.Lyapunov.martingal}
with $ \rho = 7 c \in [0,\infty) $.
Combinig \eqref{eq:lastest}, 
\eqref{eq:defZ},
\eqref{eq:phi} and the 
assumption that
$
  \E\big[
    \| Y_0 \|^p
  \big]
  < \infty
$
then shows that 
$
  \E\big[ \| Y_n \|^p \big] < \infty
$
for all $ n \in \N_0 $
and that
$
  \E\big[ | Z_n | \big] < \infty
$
for all $ n \in \N $.
In addition, note that
  \begin{equation}
    \E\big[
      Z_{ n } 
      \, | \, 
      ( Z_k )_{
        k \in \{ 0, 1, \ldots, n - 1 \}
      }
    \big]=0
  \end{equation}
  $ \mathbb{P} $-a.s.\ for 
  all $ n \in \N $.
  It thus remains to verify 
  inequality~\eqref{eq:variationZbounded.by.Y}
  to complete the proof of Lemma~\ref{l:more.Lyapunov.implicit.Euler}.
  For this observe that the estimate
  \begin{equation}
  \left\| 
    X - \mathbb{E}[ X ]
  \right\|_{ L^q( \Omega; \R ) }
  \leq
  2
  \left\| 
    X
  \right\|_{ L^q( \Omega; \R ) }
  \end{equation}
  for all $ q \in [1,\infty) $
  and all 
  $ \mathcal{F} $/$ \mathcal{B}( \R ) 
  $-measurable mappings
  $ X \colon \Omega \rightarrow \R $
  with $ \E\big[ | X | \big] < \infty $
  and
  Lemma~7.7 in Da Prato
  \citationand\ Zabzcyk~\cite{dz92} 
  give
  \begin{equation}  
  \begin{split}
    \left\|
      Z_{ n } 
    \right\|_{ 
      L^{ p / 2 }( \Omega; \R ) 
    }
  & \leq 
    \frac{ 
      \left\|
        \left<
          Y_{ n - 1 } ,
          \sigma( Y_{ n - 1 } )
          ( W_{ n h } - 
            W_{ (n - 1) h } 
          )
        \right>
      \right\|_{ 
        L^{ p / 2 }( \Omega; \R ) 
      }
    }{
      \left( 1/2 - h c \right)
    }
  \\ & \quad + 
    \frac{ 
      \big\|
          \sigma( Y_{ n - 1 } )
          \big(
            W_{ n h } -
            W_{ (n - 1) h  }
          \big)
      \big\|_{ 
        L^{ p } ( \Omega; \R^d ) 
      }^2  
    }{
      \left( 1/2 - h c \right)
    }
  \\ & \leq 
    \frac{ 
      \frac{ p \sqrt{ h } 
      }{ \sqrt{ 8 } } 
      \left\|   
          \left< Y_{ n - 1 },
          \sigma( Y_{ n - 1 } )( 
            \cdot 
          )
        \right>
      \right\|_{ 
        L^{ p / 2 
        }( \Omega; HS( \R^m, \R ) ) 
      }
    }{
      \left( 1/2 - h c \right)
    }
\\ & \quad
    +
    \frac{ 
      \frac{ p \left( p - 1 \right) h 
      }{ 2 }
      \left\|
        \sigma( Y_{ n - 1 } )
      \right\|_{ 
        L^{ p }( 
          \Omega; HS( \R^m, \R^d ) 
        ) 
      }^2
    }{
      \left( 1/2 - h c \right)
    }
  \end{split}
  \end{equation}
  and hence
  \begin{equation}  
  \begin{split}
    \left\|
      Z_{ n } 
    \right\|_{ 
      L^{ p / 2 }( \Omega; \R ) 
    }
  & \leq 
    \frac{
      \frac{ p \sqrt{ h } 
      }{ \sqrt{ 2 } } 
      \,
      \big\|   
        \| Y_{ n - 1 } \|
        \| \sigma( Y_{ n - 1 } ) \|_{ 
          HS( \R^m, \R^d ) 
        }
      \big\|_{ 
        L^{ p/2 }( \Omega; \R ) 
      }
    }{
      \left( 1 - 2 h c \right)
    }
  \\ & \quad
    +
    \frac{
      c h p ( p - 1 )
      \left\|
        1 + \| Y_{ n - 1 } \|^2
      \right\|_{ 
        L^{ p/2 }( \Omega; \R ) 
      }
    }{
      \left( 1 - 2 h c \right)
    }
  \\ & \leq 
    \frac{
      \frac{ p \sqrt{ h } 
      }{ \sqrt{ 2 } } 
      \,
      \big\|
        \frac{ \| Y_{ n - 1 } \|^2 
        }{ 2 }
        +
        \frac{ c }{ 2 }
        \left[
          1 + \| Y_{ n - 1 } \|^2
        \right]
      \big\|_{ 
        L^{ p/2 }( \Omega; \R ) 
      }
    }{
      \left( 1 - 2 h c \right)
    }
  \\ & \quad
    +
    \frac{
      c h p ( p - 1 )
      \left\|
        1 + \| Y_{ n - 1 } \|^2
      \right\|_{ 
        L^{ p/2 }( \Omega; \R ) 
      }
    }{
      \left( 1 - 2 h c \right)
    }
\end{split}
\end{equation}
and therefore
\begin{equation}
\begin{split}
  &
    \left\|
      Z_{ n } 
    \right\|_{ 
      L^{ p / 2 }( \Omega; \R ) 
    }
  \\ & \leq 
    \frac{
      \frac{ p \sqrt{ h } c
      }{ \sqrt{ 2 } } 
      \left\|      
        1 + \| Y_{ n - 1 } \|^2
      \right\|_{ 
        L^{ p / 2 }( \Omega; \R ) 
      }
      +
      c h p ( p - 1 )
      \left\|
        1 + \| Y_{ n - 1 } \|^2
      \right\|_{ 
        L^{ p / 2 }( \Omega; \R ) 
      }
    }{
      \left( 1 - 2 h c \right)
    }
\\ & =
  p c
  \sqrt{h}
  \left[
    \frac{
      \frac{ 1 }{ 
        \sqrt{2}
      }
      +
      \sqrt{h} \left( p - 1 \right)
    }{
      \left( 1 - 2 h c \right)
    }
    \right]
      \left\|
        1 + \| Y_{ n - 1 } \|^2
      \right\|_{ 
        L^{ p / 2 }( \Omega; \R ) 
      }
\\ & \leq
  p 
  (
    p +
    \sqrt{2} - 1 
  )
  c
  \sqrt{h}
      \left\|
        1 + \| Y_{ n - 1 } \|^2
      \right\|_{ 
        L^{ p / 2 }( \Omega; \R ) 
      }
  \\ &
  \leq
    c p^2 
    \sqrt{2 h}
      \left\|
        1 + \| Y_{ n - 1 } \|^2
      \right\|_{ 
        L^{ p / 2 }( \Omega; \R ) 
      }
  \end{split}     
  \end{equation}
  for all $ n \in \N $.
  This implies 
  inequality~\eqref{eq:variationZbounded.by.Y}
  with 
  $
    \nu \colon \mathbb{N}
    \rightarrow [0,\infty)
  $
  given by
  $ 
    \nu_n = 
    c p^2 
    \sqrt{2 h}
  $ 
  for all $ n \in \N $.
  Now we apply 
  Corollary~\ref{cor:stability}
  with
  $ \rho = 7 c $,
  with the Lyapunov-type function
  $ V \colon \R^d \rightarrow [0,\infty) $
  given by
  \begin{equation}
    V(x) = 1 + \| x \|^2 
  \end{equation}
  for all $ x \in \R^d $,
  with the truncation function
  $ \zeta \colon [0,\infty)
  \rightarrow [0,\infty] $ given by
  $ \zeta(t) = \infty $ for all 
  $ t \in [0,\infty) $
  and with the sequence
  $ t_n \in \R $, $ n \in \N_0 $,
  given by
  $
    t_n = n h 
  $
  for all
  $ n \in \N_0 $
  to obtain
  \begin{equation}  
  \begin{split}
  &
    \bigg\|
      \sup_{ k \in \{ 0, 1, \ldots, n \} }
      \left\| 
        Y_k 
      \right\|^2
    \bigg\|_{ L^{ p / 2 }( \Omega; \R) }
  \\ & \leq
    e^{ 7 c n h }
    \bigg\|
      \sup_{ k \in \{0, 1, \ldots, n \} }
      e^{ - 7 c k h }
      \left(
        1 + 
        \left\| 
          Y_k
        \right\|^2
      \right)
    \bigg\|_{ 
      L^{ p / 2 }(\Omega; \R)
    }
  \\ & \leq
    2
    e^{ 7 c n h }
    \left\|
      1 + \left\| Y_0 \right\|^2
    \right\|_{ L^{ p / 2 }(\Omega; \R) }
    \exp\!\left(
      \chi_{ \frac{ p }{ 2 } } 
      \left[
      \sum_{k=1}^n 
       2 
       p^4 c^2
       h
      \right]
    \right)
  \\ & \leq
    2
    \left(
      1 + 
      \left\| Y_0 \right\|_{
        L^{ p }(\Omega; \R^d)
      }^2
    \right)
    \exp\!\left(
      \big( 
        7 +   
        2 p^4 \chi_{ \frac{ p }{ 2 } }
      \big)
      c^2
      n h
    \right)
  \end{split}     
  \end{equation}
  for all $ n \in \N_0 $.
  This finishes the proof
  of
  Lemma~\ref{l:more.Lyapunov.implicit.Euler}.
\end{proof}

\begin{lemma}[A class
of linear implicit schemes
for one-dimensional SDEs] \label{l:Lyapunov.linear.implicit.Euler}
  Let 
  $ 
    c \in [0,\infty)  
  $,
  $ T \in (0,\infty) $,
  let
  $
    \left( 
      \Omega, \mathcal{F}, 
      ( \mathcal{F}_t )_{ t \in [0,T] },
      \mathbb{P} 
    \right)
  $
  be a filtered probability space,
  let
  $
    W \colon [0,T] \times
    \Omega \rightarrow \mathbb{R}
  $
  be a standard 
  $ ( \mathcal{F}_t )_{ t \in [0,T] } 
  $-Brownian motion,
  let
  $ 
    \xi \colon \Omega \rightarrow \mathbb{R} 
  $
  be an $ \mathcal{F}_0 
  $/$ \mathcal{B}( \mathbb{R} ) 
  $-measurable function
  and
  let 
  $ a, b, \sigma \colon \R \to \R $ 
  be Borel measurable functions
  with
  \begin{equation}  
  \label{eq:one.sided.linear.growth.d1}
    x \left( a(x) \, x + b(x) \right)
    +
    \tfrac{ 1 }{ 2 }
    | \sigma(x) |^2
    \leq
    c \left( 1 + x^2 \right),
  \quad
    a(x) \leq c ,
  \quad
    | b(x) |^2 \leq
    c \left( 1 + x^2 \right)
  \end{equation}
  for all $x\in\R$.
  Then there exists a
  unique family
  $
    Y^N \colon
    \{ 0, 1, \dots, N \}
    \times \Omega
    \to \R
  $,  
  $ N \in \mathbb{N} \cap (c T, \infty) $,
  of stochastic processes
  satisfying $ Y_0^N = \xi $ 
  and
  \begin{equation}  
  \label{eq:linear.implicit.implicit}
    Y_{n+1}^N
    =
    Y_n^N
  +
    \big(
      a(Y_n^N) \, Y_{n+1}^N
      +
      b(Y_n^N)
    \big) \,
    \tfrac{ T }{ N }
    +
    \sigma( Y_n^N ) \,
    \big(
      W_{ (n+1) T / N } - 
      W_{ n T / N }
    \big)
  \end{equation}
  for all $ n \in \{ 0, 1, \dots, N \} $ and all
  $ N \in \N \cap (c T, \infty) $
  and it holds
  \begin{equation}  
  \label{eq:linear.implicit}
  \sup_{ N \in \N \cap [2 c T, \infty) }
  \sup_{ n \in\{ 0, 1, \ldots, N \} }
  \E\!\left[
    1 + \left| Y_n^N \right|^2
  \right]
  \leq
    e^{ 
      ( 8 c + 2 ) T
    }
    \cdot
    \E\!\left[ 1 + |\xi|^2 \right] .
  \end{equation}
\end{lemma}

\begin{proof}[Proof
of 
Lemma~\ref{l:Lyapunov.linear.implicit.Euler}]
First, note that the assumption
$ a(x) \leq c \in [0,\infty) $ 
for all $ x \in \R $ 
(see 
\eqref{eq:one.sided.linear.growth.d1})
ensures the
unique existence of 
a family
$ 
  Y^N \colon \{ 0, 1, \dots, N \} \times
  \Omega \rightarrow \R
$, $ N \in \N \cap (cT,\infty) $,
of stochastic processes
satisfying 
$ Y_0 = \xi $
and \eqref{eq:linear.implicit.implicit}.
The stochastic processes
$
  ( Y^N )_{ N \in \N }
$
thus fulfill
\begin{equation}
  Y_{n+1}^N
  =
  \frac{
    Y_n^N
  +
    b( Y_n^N ) \,
    \tfrac{ T }{ N }
    +
    \sigma( Y_n^N ) \,
    \big(
      W_{ (n+1) T / N } - 
      W_{ n T / N }
    \big)
  }{
    \left(
      1 -
      a( Y_n^N ) \tfrac{ T }{ N }
    \right)
  }
\end{equation}
for all $ n \in \{ 0, 1, \dots, N - 1 \} $
and all $ N \in \N \cap ( c T , \infty ) $.
Next observe that~\eqref{eq:one.sided.linear.growth.d1}
implies
  \begin{equation}
  \begin{split}
  &
    \mathbb{E}\!\left[
      \frac{ 
        \left(
          x 
          +
          b(x) t
          +
          \sigma(x) W_t
        \right)^2
      }{
        \left( 1 - a(x) t \right)^2
      }
    \right]
  =
      \frac{ 
        \left(
          x 
          +
          b(x) t
        \right)^2
          +
        \left| \sigma(x) \right|^2 t
      }{
        \left( 1 - a(x) t \right)^2
      }
  \\ & =
      \frac{ 
        x^2
        \left( 1 - 2 a(x) t 
        \right)
      }{
        \left( 1 - a(x) t \right)^2
      }
      +
      \frac{
        2 t \left\{ 
        x \left( a(x) \, x + b(x) \right) 
        +
        \frac{ 1 }{ 2 } 
        \left| \sigma(x) \right|^2 
        \right\}
        + \left| b(x) \right|^2 t^2
      }{
        \left( 1 - a(x) t \right)^2
      }
  \\ & \leq
    x^2 
    +
      \frac{
        2 t c \left( 1 + x^2 \right)
        + c \left( 1 + x^2 \right)
        t^2
      }{
        \left( 1 - a(x) t \right)^2
      }
  =
    x^2 
    +
    \left( 1 + x^2 \right) 
    \left(
      \frac{
        \left( 
          2 c + c t
        \right) t
      }{
        \left( 1 - a(x) t \right)^2
      }
    \right)
  \end{split}
  \end{equation}
  and therefore
  \begin{equation}
  \begin{split}
  &
    \mathbb{E}\!\left[
      1 +
      \frac{ 
        \left(
          x 
          +
          b(x) t
          +
          \sigma(x) W_t
        \right)^2
      }{
        \left( 1 - a(x) t \right)^2
      }
    \right]
  \\ & \leq
    \left( 1 + x^2 \right) 
    \left(
      1 +
      \frac{
        \left( 
          2 c + c t
        \right) t
      }{
        \left( 1 - a(x) t \right)^2
      }
    \right)
    \leq
    \left(
      1 +
      x^2
    \right)
    \exp\!\left(   
      \frac{
        \left( 2 c + c t \right) t
      }{
        \left( 1 - c t \right)^2
      }
    \right)
  \end{split}
  \end{equation}
  for all $ t \in [0,\frac{1}{c}) $
  and all $ x \in \R $.
  Hence, we obtain
  \begin{equation}
  \label{eq:vstab.lin.implicit}
  \begin{split}
  &
    \mathbb{E}\!\left[
      1 
      +
      \frac{ 
        \left(
          x 
          +
          b(x) t
          +
          \sigma(x) W_t
        \right)^2
      }{
        \left( 1 - a(x) t \right)^2
      }
    \right]
  \leq
    \left(
      1 +
      x^2
    \right)
    e^{ ( 8 c + 2 ) t }
  \end{split}
  \end{equation}
  for all $ t \in [0,\frac{ 1 }{ 2 c }] $
  and all $ x \in \R $.
  Note that \eqref{eq:vstab.lin.implicit}
  shows that the linear implicit
  scheme
  \begin{equation}
    \R \times [0,\tfrac{ 1 }{ 2 c }] \times
    \R \ni
    (x,t,y) \mapsto
    \frac{ 
      x + b(x) t + \sigma(x) y
    }{
      \left( 1 - a(x) t \right)^2
    }
    \in \R
  \end{equation}
  is
  $
    ( 1 + x^2 )_{ x \in \R }
  $-stable with respect to
  Brownian motion.
  Moreover,
  combining \eqref{eq:vstab.lin.implicit}
  and 
  Corollary~\ref{cor:stability2}
  results in \eqref{eq:linear.implicit}
  and this completes the proof
  of
  Lemma~\ref{l:Lyapunov.linear.implicit.Euler}.
\end{proof}

\chapter{Convergence 
properties 
of approximation processes 
for SDEs}
\label{chap:convergence}

With the integrability 
properties of Chapter~\ref{sec:integrability properties}
at hand, we now prove
convergence in probability, 
strong convergence
and weak convergence 
results for  
numerical approximation 
processes
for SDEs.
For this, our central assumption 
on the numerical approximation
method is a certain consistency 
property in
Definition~\ref{def:consistent} below.

\section{Setting
and assumptions}
\label{sec:convergencesetting}

%
%
Let $ T \in (0,\infty) $, 
$ d, m \in \mathbb{N} $,
let 
$ ( \Omega, \mathcal{F}, \mathbb{P} ) $
be a probability space
with a normal filtration
$ ( \mathcal{F}_t )_{ t \in [0,T] } $
and 
let 
$ W \colon [0,T] \times \Omega
\rightarrow \mathbb{R}^m $
be a
standard 
$ ( \mathcal{F}_t )_{ t \in [0,T] } 
$-Brownian motion.
Moreover, let 
$ D \subset \mathbb{R}^d $ 
be a non-empty open set,
let $ \mu \colon D 
\rightarrow \mathbb{R}^d $
and
$ 
  \sigma \colon 
  D
  \rightarrow \mathbb{R}^{ d \times m }
$
be locally
Lipschitz continuous functions
and let
$ X \colon [0,T] \times \Omega
\rightarrow D $ be an
$ ( \mathcal{F}_t )_{ t \in [0,T] } 
$-adapted stochastic
process with continuous
sample paths satisfying
\begin{equation}
\label{eq:SDE}
  X_t
=
  X_0
+
  \int_0^t 
    \mu( X_s )
  \, ds
+
  \int_0^t 
    \sigma( X_s )
  \, dW_s
\end{equation}
$ \mathbb{P} $-a.s.\ for
all $ t \in [0,T] $.
Note that we assume 
existence of a solution 
process staying in the 
open set $ D $.
Next let 
$ 
  \phi \colon
    \mathbb{R}^d \times
    [0,T] \times \mathbb{R}^m
  \rightarrow 
    \mathbb{R}^d
$
be a
Borel measurable function
and let
$ 
  \bar{Y}^N \colon 
  [0,T] \times \Omega
  \rightarrow \mathbb{R}^d 
$,
$ N \in \mathbb{N} $,
be a sequence of 
stochastic processes
defined through 
$ 
  \bar{Y}^N_0 := X_0 
$ 
and
\begin{equation}
\label{eq:recX}
  \bar{Y}^N_t
:=
  \bar{Y}^N_{ 
    \frac{ n T }{ N } 
  }
  +
  \big( 
    \tfrac{ t N }{ T } - n
  \big)
  \cdot
  \phi\big( 
    \bar{Y}^N_{ 
      \frac{ n T }{ N } 
    },
    \tfrac{T}{N}, 
    W_{ \frac{ (n+1) T }{ N } }
    -
    W_{ \frac{ n T }{ N } }
  \big)
\end{equation}
for all 
$ 
  t \in 
  \big(
    \frac{ n T }{ N },
    \frac{ (n+1) T }{ N }
  \big] 
$,
$ n \in \{ 0, 1, \dots, N-1 \} $
and all $ N \in \mathbb{N} $.
A central goal of this chapter is to give
sufficient conditions to ensure
that the stochastic
processes 
$ 
  \bar{Y}^N \colon [0,T]
  \times \Omega
  \rightarrow \mathbb{R}^d
$,
$ N \in \mathbb{N} $,
converge in a suitable sense 
to the solution process 
$ 
  X \colon [0,T] \times \Omega
  \rightarrow D
$
of the SDE~\eqref{eq:SDE}.

\section{Consistency}
\label{sec:consistent}

This section
introduces a 
consistency property 
of the increment function
$
  \phi \colon \mathbb{R}^d
  \times
  [0,T] \times \mathbb{R}^m
  \rightarrow \mathbb{R}^d
$
from Section~\ref{sec:convergencesetting}
which ensures that
the stochastic
processes 
$ 
  \bar{Y}^N \colon [0,T]
  \times \Omega
  \rightarrow \mathbb{R}^d
$,
$ N \in \mathbb{N} $,
defined in \eqref{eq:recX}
converge in probability
to the solution
process 
$ 
  X \colon [0,T] \times \Omega
  \rightarrow D
$
of the 
SDE~\eqref{eq:SDE}
(see
Theorem~\ref{thm:convergence}
below).

\begin{definition}[Consistency
of numerical methods
for SDEs driven by standard
Brownian motions]
\label{def:consistent}
Let $ T \in (0,\infty) $,
$ d, m \in \mathbb{N} $,
let $ D \subset \mathbb{R}^d $
be an open set and let
$ 
  \mu \colon 
  D \rightarrow \mathbb{R}^d
$
and
$
  \sigma \colon
  D \rightarrow 
  \mathbb{R}^{ d \times m }
$
be functions.
A Borel measurable 
function
$ 
  \phi \colon
    \mathbb{R}^d \times
    [0,T] \times \mathbb{R}^m
  \rightarrow 
    \mathbb{R}^d
$
is then said to
be 
$ (\mu, \sigma) 
$-consistent with 
respect to Brownian motion
if 
\begin{equation}
\label{ass:consistency_1}
  \limsup_{ t \searrow 0 }
  \left(
    \tfrac{1}{\sqrt{t}}
    \cdot
    \sup_{ 
      x \in K
    }
    \mathbb{E}\Big[
    \big\|
      \sigma(x) 
      W_t
      -
        \phi( 
          x, 
          t, 
          W_t
        )
    \big\|
    \Big]
  \right)
  = 0
\end{equation}
and
\begin{equation}
\label{ass:consistency_2}
  \limsup_{ t \searrow 0 }
  \left(
  \sup_{ 
    x \in K
  }
  \left\|
    \mu(x)
    -
    \tfrac{1}{t}
    \cdot 
     \mathbb{E}\big[
      \phi( 
        x, 
        t, 
        W_t
      )
    \big]
  \right\|
  \right)
  = 0
\end{equation}
for all non-empty compact sets
$ K \subset D $
where
$
  W \colon [0,T] \times
  \Omega \rightarrow 
  \mathbb{R}^m 
$
is an arbitrary standard Brownian
motion on a probability
space 
$
  \left( 
    \Omega, \mathcal{F}, \mathbb{P}
  \right)
$.
\end{definition}
%
%
%

Note that
\eqref{ass:consistency_1}
in 
Definition~\ref{def:consistent}
assures
that the expectation in
\eqref{ass:consistency_2} 
is well-defined.
Further consistency notions
for numerical approximation
schemes for SDEs and results
for such schemes in the
case of SDEs
with globally Lipschitz
continuous coefficients
can, e.g., be found in 
Section~9.6 of
Kloeden \citationand\ Platen~\cite{kp92},
in Chapter~1 of Milstein~\cite{m95},
in Beyn \& Kruse~\cite{bk10},
in Kruse~\cite{k11}
and in the references therein.
The next lemma gives 
a simple
characterization of 
$ (\mu,\sigma) $-consistency 
with respect to Brownian motion.
Its proof 
is straightforward and therefore
omitted.

\begin{lemma}
\label{lem:consistent}
Let $ T \in (0,\infty) $,
$ d, m \in \mathbb{N} $,
let $ D \subset \mathbb{R}^d $
be an open set and let
$ 
  \mu \colon 
  D \rightarrow \mathbb{R}^d
$
and
$
  \sigma \colon
  D \rightarrow 
  \mathbb{R}^{ d \times m }
$
be functions.
A Borel measurable 
function
$ 
  \phi \colon
    \mathbb{R}^d \times
    [0,T] \times \mathbb{R}^m
  \rightarrow 
    \mathbb{R}^d
$
is then 
$ (\mu, \sigma) $-consistent with respect to Brownian motion
if and only if 
\begin{equation}
\label{ass:consistency_1B}
  \limsup_{ t \searrow 0 }
  \left(
  \tfrac{1
  }{
    \sqrt{t}
  }
  \cdot
    \sup\nolimits_{ 
      x \in D_v
    }
    \mathbb{E}\big[
    \|
      \sigma(x) 
      W_t
      -
        \phi( 
          x, 
          t, 
          W_t
        )
    \|
    \big]
  \right)
  = 0
\end{equation}
and
\begin{equation}
\label{ass:consistency_2B}
  \limsup_{ t \searrow 0 }
  \left(
  \tfrac{1
  }{ t }
  \cdot
  \sup_{ 
    x \in D_v
  }
  \big\|
    \mu(x) t
    -
     \mathbb{E}\big[
      \phi( 
        x, 
        t, 
        W_t
      )
    \big]
  \big\|
  \right)
  = 0
\end{equation}
for all $ v \in \mathbb{N} $
where
the sets
$ D_v \subset D $,
$ v \in \mathbb{N} $, 
are given by
$ 
  D_v :=
  \{ 
    v \in D 
    \colon 
    \| x \| < v
    \text{ and }
    \text{dist}(x, D^c) > \frac{ 1 }{ v } 
  \}
$
for all $ v \in \mathbb{N} $
and where
$
  W \colon [0,T] \times
  \Omega \rightarrow 
  \mathbb{R}^m 
$
is an arbitrary standard Brownian
motion on a probability
space 
$
  \left( 
    \Omega, \mathcal{F}, \mathbb{P}
  \right)
$.
\end{lemma}

A list of numerical schemes 
that
are $ ( \mu, \sigma ) $-consistent
with respect to Brownian motion
can be found in 
Section~\ref{sec:schemes}.

\section{Convergence in probability}
\label{sec:convergence in probability}

The next theorem
shows that the stochastic
processes 
$ 
  \bar{Y}^N \colon [0,T]
  \times \Omega
  \rightarrow \mathbb{R}^d
$,
$ N \in \mathbb{N} $,
in \eqref{eq:recX}
converge in probability
to the solution
process 
$ 
  X \colon [0,T] \times \Omega
  \rightarrow D
$
of the SDE~\eqref{eq:SDE}
if the increment function
$
  \phi \colon \mathbb{R}^d
  \times
  [0,T] \times \mathbb{R}^m
  \rightarrow \mathbb{R}^d
$
is $ ( \mu, \sigma ) $-consistent
with respect to Brownian
motion (see
Definition~\ref{def:consistent}).

\begin{theorem}[Convergence in probability]
\label{thm:convergence}
Assume that the setting
in Section~\ref{sec:convergencesetting}
is fulfilled
and that
$ 
  \phi \colon \mathbb{R}^d
  \times [0,T] \times \mathbb{R}^m
  \rightarrow \mathbb{R}^d
$ is
$ ( \mu, \sigma ) $-consistent with respect to Brownian motion.
Then
\begin{equation}
  \lim_{ N \rightarrow \infty }
  \mathbb{P}\Bigg[
    \sup_{ t \in [0,T] }
    \big\| 
      X_t 
      - 
      \bar{Y}^N_t
    \big\|
    \geq
    \varepsilon
  \Bigg]
= 
  0 
\end{equation}
for all 
$ \varepsilon \in (0,\infty) $.
\end{theorem}

The proof of 
Theorem~\ref{thm:convergence}
is presented in the following
subsection.
More results on convergence in
probability and pathwise 
convergence of 
temporal numerical
approximation processes
for SDEs with non-globally
Lipschitz continuous coefficients
can, e.g., be found in
\cite{Krylov1990,GyoengyKrylov1996,g98b,g99,p01,ps05,jkn09a,j08b,cv12}
and in the references 
therein.

Theorem~\ref{thm:convergence}
proves convergence in probability 
for a class of one-step numerical 
approximation processes~\eqref{eq:recX}
in the case of finite dimensional 
SDEs with locally Lipschitz
continuous coefficients
$ \mu $ and $ \sigma $.
The locally Lipschitz assumptions
on $ \mu $ and $ \sigma $
ensure that solutions of
the SDE~\eqref{eq:SDE}
are unique up to indistinguishability.
We expect that it is possible to
generalize 
Theorem~\ref{thm:convergence}
to a more general class of 
possibly infinite dimensional SDEs
and also to 
replace the locally Lipschitz
assumptions on $ \mu $ and
$ \sigma $ by 
weaker conditions such
as local monotonicity
(see, e.g.,
Krylov~\cite{Krylov1990},
Gy\"{o}ngy
\citationand\ Krylov~\cite{GyoengyKrylov1996}
in the finite dimensional case
and
Liu \citationand\
R\"{o}ckner~\cite{LiuRoeckner2010,LiuRoeckner2012}
in the infinite dimensional case)
which ensure that solutions of
the considered SDE
are unique.

\subsection{Proof of 
Theorem~\ref{thm:convergence}}
\label{sec:proof.convergence}

For the proof of
Theorem~\ref{thm:convergence},
we introduce more notation.
First, we define 
mappings 
$ Y_n^N 
\colon \Omega \rightarrow 
\mathbb{R}^d $, 
$ n \in \{ 0, 1, \ldots, N \} $, 
$ N \in \mathbb{N} $, through 
$ Y_n^N := 
\bar{Y}_{ n T / N }^N $ 
for all $ n \in \{ 0, 1, \ldots, N \} $ 
and all $ N \in \mathbb{N} $. 
In addition, we use the 
mappings 
$ 
  \bar{\mu} \colon \mathbb{R}^d 
  \rightarrow \mathbb{R}^d 
$, 
$ 
  \bar{\sigma} 
  = 
  \left( \bar{\sigma}_{ i, j } ( x ) 
  \right)_{ 
    i \in \{ 1, \ldots, d \}, 
    j \in \{ 1, \ldots, m \} 
  } 
  \colon 
  \mathbb{R}^d \rightarrow 
  \mathbb{R}^{ d \times m } 
$ 
and 
$ 
  \bar{\sigma}_i \colon 
  \mathbb{R}^d \rightarrow 
  \mathbb{R}^d 
$, 
$ 
  i \in \{ 1, 2, \ldots, m \} 
$, 
defined by 
\begin{equation}
  \bar{\mu}(x) := \mu(x) 
  \qquad
  \text{and}
  \qquad 
  \bar{\sigma}(x) := \sigma(x) 
\end{equation}
for all 
$ x \in D $, 
by 
\begin{equation}
  \bar{\mu}(x) := 0 
\qquad
  \text{and}
\qquad   
  \bar{\sigma}(x) := 0 
\end{equation}
for all 
$ x \in D^c $ 
and by 
$ \bar{\sigma}_i( x ) := 
\left( \bar{\sigma}_{1, i}( x ), \ldots, \bar{\sigma}_{d, i}( x ) \right) $ 
for all $ x \in \mathbb{R}^d $ 
and all $ i \in \{ 1, 2, \ldots, m \} $. 
Next let $ D_v \subset D $, $ v \in \N $,
be a sequence of open sets defined by
\begin{equation}
  D_v :=
  \left\{
    x \in D
    \colon
    \left\| x \right\| < v
    \text{ and }
    \text{dist}(x,D^c) > \tfrac{ 1 }{ v }
  \right\}
\end{equation}
for all $ v \in \N $.
The assumption that
$ 
  \phi \colon
  \mathbb{R}^d \times
  [0,T] \times \mathbb{R}^m
  \rightarrow 
  \mathbb{R}^d
$
is
$ (\mu, \sigma) 
$-consistent with 
respect to Brownian motion
and
Lemma~\ref{lem:consistent}
then ensure that there exists
a sequence
$ t_v \in (0,T] $,
$ v \in \mathbb{N} $,
of real numbers
such that
\begin{equation}
\label{eq:def_tv}
  \sup_{ t \in [0,t_v] }
  \sup_{ x \in D_v }
  \mathbb{E}\Big[
    \| \phi( x, t, W_t) \|
  \Big]
  < \infty
\end{equation}
for all 
$ v \in \mathbb{N} $
with $ D_v \neq \emptyset $.
Then let $ c_v \in [0, \infty) $, 
$ v \in \mathbb{N} $, be a family 
of real numbers defined by
\begin{equation}
\label{eq:def_cr}
  c_v
  :=
  \begin{cases}
  \sup_{ 
    \substack{ 
      x, y \in \bar{D}_v \\ x \neq y 
    } 
  }
  \frac{ \left\| \mu( x ) - \mu( y ) \right\| }{ \left\| x - y \right\| }
  +
  \sum_{ i = 1 }^m
  \left(
    \sup_{ 
      \substack{ 
        x, y \in \bar{D}_v \\ x \neq y 
      } 
    }
    \frac{ 
      \left\| \sigma_i( x ) - \sigma_i( y ) 
      \right\| 
    }{ 
      \left\| x - y \right\| 
    }
  \right)
  &
  \colon
  D_v \neq \emptyset
  \\
  0
  &
  \colon \text{else}
  \end{cases}
\end{equation}
for all $ v \in \mathbb{N} $. 
Using Lebesgue's number 
lemma one can indeed show that
$ c_v < \infty $ 
for all $ v \in \mathbb{N} $ 
since 
$ 
  \mu \colon D \rightarrow 
  \mathbb{R}^d 
$ 
and 
$ 
  \sigma \colon 
  D \rightarrow \mathbb{R}^{ d \times m } 
$ 
are assumed to 
be locally Lipschitz continuous 
and since 
$ \bar{D}_v \subset D $, 
$ v \in \mathbb{N} $, is a 
sequence of compact sets.

Roughly speaking, the consistency
condition of Definition~\ref{def:consistent}
requires the increment function
$
  \phi \colon \mathbb{R}^d
  \times
  [0,T] \times \mathbb{R}^m
  \rightarrow \mathbb{R}^d
$
from 
Section~\ref{sec:convergencesetting}
to be close to the increment function
of the respective 
Euler-Maruyama approximation
method.
For this reason,
we estimate the distance 
of the exact solution 
$ X \colon [0,T] \times 
\Omega \rightarrow D $ 
of the SDE~\eqref{eq:SDE} and 
of the Euler-Maruyama approximations
(see Lemma~\ref{l:CE} below)
and we estimate the 
distance of 
the Euler-Maruyama approximations
and 
of the numerical approximations 
$ 
Y_n^N 
\colon \Omega \rightarrow 
\mathbb{R}^d $, 
$ n \in \{ 0, 1, \ldots, N \} $, 
$ N \in \mathbb{N} $, (see Lemma~\ref{l:CA} below).
The triangle inequality 
will then yield an estimate for 
$ \| X_{ \frac{ nT }{ N } } - 
Y_n^N \| $, 
$ n \in \{ 0, 1, \ldots, N \} $, $ N \in \mathbb{N} $ 
(see Corollary~\ref{lem_C} below). 
For this strategy,
we now introduce
suitable 
Euler-Maruyama 
approximations
for the SDE~\eqref{eq:SDE}. 
More formally, let 
$ Z_n^N \colon 
\Omega \rightarrow \mathbb{R}^d $, 
$ n \in \{ 0, 1, \ldots, N \} $, 
$ N \in \mathbb{N} $, 
be defined recursively through 
$ Z_0^N := X_0 $ and 
\begin{equation}
  Z_{ n + 1 }^N := Z_n^N 
  + \bar{\mu}( Z_n^N ) 
  \cdot 
  \tfrac{T}{N} 
  + 
  \bar{\sigma}( Z_n^N 
  )\left( 
    W_{ \frac{ (n+1)T}{N} } - 
    W_{ \frac{nT}{N} } 
  \right) 
\end{equation} 
for all $ n \in \{ 0, 1, \ldots, N - 1 \} $ 
and all $ N \in \mathbb{N} $. 
Furthermore, let 
$ \tilde{Z}^N \colon 
[0,T] \times \Omega \rightarrow \mathbb{R}^d $, $ N \in \mathbb{N} $, 
be given by 
\begin{equation} 
\label{eq:def:Ztilde}
  \tilde{Z}^N_t = Z_n^N 
  + 
  \bar{\mu}( Z_n^N )
  \cdot \big( t - \tfrac{ nT}{N} \big)   
  + 
  \bar{\sigma}( Z_n^N 
  )\left( W_t - W_{ \frac{nT}{N} } \right) 
\end{equation} 
for all $ t \in [ \frac{ nT }{ N }, 
\frac{ (n+1)T}{N} ] $, 
$ n \in \{ 0, 1, \ldots, N - 1 \} $ 
and all $ N \in \mathbb{N} $. 
Finally, let 
$ \tau_v^N \colon \Omega 
\rightarrow [0,T] $, 
$ v, N \in \mathbb{N} $, 
and 
$ 
  \delta_v^N \colon \Omega 
  \rightarrow \{ 0, 1, \ldots, N \} 
$, 
$ v, N \in \mathbb{N} $, 
be defined by
\begin{equation}
\begin{split}
&
  \tau_v^N( \omega )
  :=
  \inf\!\left(
    \{ T \} 
    \cup
    \big\{
      t \in [0,T] \colon 
      X_t( \omega ) \notin D_v
    \big\}
    \cup 
    \big\{
      t \in [0,T] \colon
      \tilde{Z}_t^N(\omega) 
      \notin D_v
    \big\}
  \right)
\end{split}
\end{equation}
and by
\begin{multline}
  \delta_v^N( \omega )
  :=
  \min\!\Big(
    \{ N \} 
    \cup
    \big\{
      n \in \{ 0, 1, \ldots, N \} 
      \colon
      Z_n^N( \omega ) \notin D_v
    \big\}
\\
    \cup 
    \big\{
      n \in \{ 0, 1, \ldots, N \} 
      \colon
      Y_n^N( \omega ) \notin D_v
    \big\}
  \Big)
\end{multline}
for all $ \omega \in \Omega $
and all
$ v, N \in \mathbb{N} $.
Using the notation introduced
above, the proof of 
Theorem~\ref{thm:convergence} 
is divided into the 
following lemmas.

\begin{lemma}
\label{l:CE}
Assume that the setting
in 
Section~\ref{sec:convergencesetting}
is fulfilled,
that the setting
in the beginning of Subsection~\ref{sec:proof.convergence}
is fulfilled
and that
$ 
  \phi \colon \mathbb{R}^d
  \times [0,T] \times \mathbb{R}^m
  \rightarrow \mathbb{R}^d
$ is
$ ( \mu, \sigma ) $-consistent with respect to Brownian motion.
Then
\begin{equation}
  \lim_{ N \rightarrow \infty }
  \mathbb{P}\Bigg[
    \sup_{ t \in [0,T] }
    \big\| 
      X_t 
      - 
      \tilde{Z}_t^N
    \big\|
    \geq
    \varepsilon
  \Bigg]
= 
  0 
\end{equation}
for all 
$ \varepsilon \in (0,\infty) $.
\end{lemma}

The proof
of Lemma~\ref{l:CE}
is literally the same as 
the proof of
Corollary~2.6 of Gy\"ongy 
\citationand\ Krylov~\cite{GyoengyKrylov1996}
  (replace assumption \textit{(ii)} in~\cite{GyoengyKrylov1996}
  by the weaker assumption 
  of the existence 
  of an exact solution;
  see also Section~2 in
  \cite{jkn09a}).
  The proof of Lemma~\ref{l:CE} is therefore omitted.

\begin{lemma}
\label{lem_tau}
Assume that the setting
in Section~\ref{sec:convergencesetting}
is fulfilled,
that the setting
in the beginning of Subsection~\ref{sec:proof.convergence}
is fulfilled
and that
$ 
  \phi \colon \mathbb{R}^d
  \times [0,T] \times \mathbb{R}^m
  \rightarrow \mathbb{R}^d
$ is
$ ( \mu, \sigma ) $-consistent with respect to Brownian motion.
Then
\begin{equation}
\label{eq:lem_tau}
  \lim_{ v \rightarrow \infty }
  \limsup_{ N \rightarrow \infty }
  \mathbb{P}\!\left[
    \tau_v^N < T
  \right]
  = 0.
\end{equation}
\end{lemma}

\begin{lemma}
\label{lem_A}
Assume that the setting
in Section~\ref{sec:convergencesetting}
is fulfilled,
that the setting
in the beginning of Subsection~\ref{sec:proof.convergence}
is fulfilled
and that
$ 
  \phi \colon \mathbb{R}^d
  \times [0,T] \times \mathbb{R}^m
  \rightarrow \mathbb{R}^d
$ is
$ ( \mu, \sigma ) $-consistent with respect to Brownian motion.
Then
\begin{equation}
  \lim_{ N \rightarrow \infty }
  \mathbb{E}\!\left[
    \sup_{ 
      n \in 
      \{ 0, 1, \ldots, \delta_v^N \} 
    }
    \left\|
      Z_n^N - Y_n^N
    \right\|
  \right]
  = 0
\end{equation}
for all $ v \in \mathbb{N} $.
\end{lemma}

\begin{lemma}
\label{lem_sigma}
Assume that the setting
in Section~\ref{sec:convergencesetting}
is fulfilled,
that the setting
in the beginning of Subsection~\ref{sec:proof.convergence}
is fulfilled
and that
$ 
  \phi \colon \mathbb{R}^d
  \times [0,T] \times \mathbb{R}^m
  \rightarrow \mathbb{R}^d
$ is
$ ( \mu, \sigma ) $-consistent with respect to Brownian motion.
Then
\begin{equation}
\label{eq:lem_delta}
  \lim_{ v \rightarrow \infty }
  \limsup_{ N \rightarrow \infty }
  \mathbb{P}\!\left[
    \delta_v^N < N
  \right]
  = 0.
\end{equation}
\end{lemma}

\begin{lemma}
\label{l:CA}
Assume that the setting
in Section~\ref{sec:convergencesetting}
is fulfilled,
that the setting
in the beginning of Subsection~\ref{sec:proof.convergence}
is fulfilled
and that
$ 
  \phi \colon \mathbb{R}^d
  \times [0,T] \times \mathbb{R}^m
  \rightarrow \mathbb{R}^d
$ is
$ ( \mu, \sigma ) $-consistent with respect to Brownian motion.
Then
\begin{equation}
  \lim_{ N \rightarrow \infty }
  \mathbb{P}\Bigg[
      \sup_{ n \in \{ 0, 1, \ldots, N \} }
      \left\|
        Z_n^N - Y_n^N
      \right\|
    \geq
    \varepsilon
  \Bigg]
= 
  0 
\end{equation}
for all 
$ \varepsilon \in (0,\infty) $.
\end{lemma}

The proofs of Lemmas~\ref{lem_tau}--\ref{l:CA} 
are given below.
The next corollary is an immediate
consequence of
Lemma~\ref{l:CE} and 
Lemma~\ref{l:CA}.

\begin{cor}
\label{lem_C}
Assume that the setting
in Section~\ref{sec:convergencesetting}
is fulfilled,
that the setting
in the beginning of Subsection~\ref{sec:proof.convergence}
is fulfilled
and that
$ 
  \phi \colon \mathbb{R}^d
  \times [0,T] \times \mathbb{R}^m
  \rightarrow \mathbb{R}^d
$ is
$ ( \mu, \sigma ) $-consistent with respect to Brownian motion.
Then
\begin{equation}
  \lim_{ N \rightarrow \infty }
  \mathbb{P}\Bigg[
      \sup_{ n \in \{ 0, 1, \ldots, N \} }
      \big\|
        X_{ \frac{nT}{N} } - Y_n^N
      \big\|
    \geq
    \varepsilon
  \Bigg]
= 
  0 
\end{equation}
for all 
$ \varepsilon \in (0,\infty) $.
\end{cor}

Using Corollary~\ref{lem_C}, 
the proof of 
Theorem~\ref{thm:convergence} is  completed at the end
of this subsection.
We now present the proofs of Lemmas~\ref{lem_tau}--\ref{l:CA}.
Let us begin with the proof
of Lemma~\ref{lem_tau}.


\begin{proof}[Proof of Lemma~\ref{lem_tau}]
Observe that subadditivity 
and monotonicity 
of the probability 
measure 
$ \mathbb{P} $ show that
\begin{equation}
\begin{split}
  &
  \mathbb{P}\!\left[
    \tau_v^N < T
  \right]
  \\ & \leq
  \mathbb{P}\Big[
    \exists \, t \in [0,T]
    \colon
    X_t \notin D_{ w }
  \Big]
  +
  \mathbb{P}\Big[
    \tau_v^N < T, \,
    \forall \,
    t \in [0, T ]
    \colon
    X_t \in D_{ w }
  \Big]
\\&\leq
  \mathbb{P}\Big[
    \exists \, t \in [0,T]
    \colon
    X_t \notin D_{ w }
  \Big]
  +
  \mathbb{P}\Big[
    \tau_v^N < T, \,
    X_{ \tau_v^N } \in 
    D_{ w }, \,
    \tilde{Z}_{ \tau_v^N }^N
    \in 
    \left( D_v \right)^c
  \Big]
\\&\leq
  \mathbb{P}\Big[
    \exists \, t \in [0,T]
    \colon
    X_t \notin D_{ w }
  \Big]
  +
  \mathbb{P}\!\left[
    \sup_{t\in[0,T]}
    \|
      X_{ t }  
      -
      \tilde{Z}_{ t }^N
    \|
    \geq
    \text{dist}\big( 
      D_{ w }, 
      \left( D_v \right)^c
    \big)
  \right]
\end{split}
\end{equation}
for all $ v, w, N \in \mathbb{N} $ 
with $ w \leq v $. 
Lemma~\ref{l:CE} 
and the estimate
$
  \text{dist}\big( 
      D_{ w }, 
      \left( D_v \right)^c
    \big)
  > 0
$
for all $ v, w \in \N $ 
with $ w < v $ hence give
\begin{equation}
\begin{split}
  \limsup_{ v \rightarrow \infty }
  \limsup_{ N \rightarrow \infty }
  \mathbb{P}\!\left[
    \tau_v^N < T
  \right]
 \leq
  \mathbb{P}\Big[
    \exists \, t \in [0,T]
    \colon
    X_t \notin D_{ w }
  \Big]
\end{split}
\end{equation}
for all $ w \in \mathbb{N} $.
The continuity of the sample 
paths 
of $ X \colon [0, T ] \times 
\Omega \rightarrow D $ 
therefore yields
\begin{equation}
  \limsup_{ v \rightarrow \infty }
  \limsup_{ N \rightarrow \infty }
  \mathbb{P}\!\left[
    \tau_v^N < T
  \right]
\leq
  \lim_{ w \rightarrow \infty }
  \mathbb{P}\big[
    \exists \, t \in [0,T]
    \colon
    X_t \notin D_{ w }
  \big]
  = 0
\end{equation}
and this completes
the proof of 
Lemma~\ref{lem_tau}.
\end{proof}


\begin{proof}[Proof of 
Lemma~\ref{lem_A}]
Throughout this proof, the 
mappings 
$ 
  \Delta W_n^N 
  \colon 
  \Omega \rightarrow 
  \mathbb{R}^m 
$, 
$ n \in $  
$ \{ 0, 1, $ $ \ldots, $ $ N - 1 \} 
$, 
$ N \in \mathbb{N} $, 
defined by 
$ 
  \Delta W_n^N 
  := 
  W_{ \frac{ (n+1) T }{ N } } - 
  W_{ \frac{ n T }{ N } } 
$ 
for all 
$ n \in \{ 0, 1, \ldots N -1 \} $ 
and all 
$ N \in \mathbb{N} $
are used.
This notation, in particular, ensures
\begin{equation}
\label{eq:Zequation}
\begin{split}
  Z_{ k \wedge \delta_v^N }^N
 &=
  X_0
  +
  \sum_{ l = 0 }^{ 
    ( k \wedge \delta_v^N ) - 1 
  }
  \Big(
    \bar{\mu}( Z_l^N ) \cdot
    \tfrac{T}{N}
    +
    \bar{\sigma}( Z_l^N ) \,
    \Delta W_l^N
  \Big)
\\&=
  X_0
  +
  \sum_{ l = 0 }^{ k - 1 }
  \mathbbm{1}_{ \{ \delta_v^N > l \} }
  \Big(
    \bar{\mu}( Z_l^N ) \cdot
    \tfrac{T}{N}
    +
    \bar{\sigma}( Z_l^N ) \,
    \Delta W_l^N
  \Big)
\end{split}
\end{equation}
and
\begin{equation}
\label{eq:Yequation}
  Y_{ k \wedge \delta_v^N }^N
  =
  X_0
  +
  \sum_{ l = 0 }^{ k - 1 }
  \mathbbm{1}_{ \{ \delta_v^N > l \} } 
  \,
  \phi\!\left(
    Y_l^N,
    \tfrac{T}{N},
    \Delta W_l^N
  \right)
\end{equation}
for all $ k \in \{ 0, 1, \ldots, N \} $, 
$ N \in \mathbb{N} $ and all 
$ v \in \mathbb{N} $. 
In addition, the mappings
$ \phi_v \colon D_v \times [0,t_v] \to \R^d $,
$ v \in \N $,
defined through
\begin{equation}
\label{eq:defphiv}
  \phi_v( x, t )
  :=
  \E\big[
    \phi( x, t, W_t )
  \big]
\end{equation}
for all $ (x,t) \in D_v \times [0,t_v] $
and all $ v \in \N $
are used throughout this proof.
Observe that the definition of 
$ t_v $, $ v \in \N $,
(see \eqref{eq:def_tv})
ensures that the expectation in \eqref{eq:defphiv}
is well defined and thus that the
mappings $ \phi_v $, $ v \in \N $,
are well defined.
In addition, note that
\begin{equation}
\label{eq:integrability}
\begin{split}
&
  \E\Big[
    \left\|
    \mathbbm{1}_{ \{ \delta_v^N > l \} } 
    \,
    \phi\!\left(
      Y_l^N,
      \tfrac{T}{N},
      \Delta W_l^N
    \right)
    \right\|
  \Big]
=
  \E\!\Big[
    \mathbbm{1}_{ \{ \delta_v^N > l \} } 
    \left\|
    \phi\!\left(
      Y_l^N,
      \tfrac{T}{N},
      \Delta W_l^N
    \right)
    \right\|
  \Big]
\\ & \leq
  \sup_{ x \in D_v }
  \E\!\Big[
    \left\|
    \phi\!\left(
      x,
      \tfrac{T}{N},
      \Delta W_l^N
    \right)
    \right\|
  \Big]
  < \infty
\end{split}
\end{equation}
for all 
$ l \in \{ 0, 1, \dots, N - 1 \} $,
$ N \in \N $
with $ \frac{ T }{ N } \leq t_v $
and all
$ v \in \N $
and that
\begin{equation}
\label{eq:Bedingte_Erwartung}
  \E\Big[
    \mathbbm{1}_{ \{ \delta_v^N > l \} } 
    \,
    \phi\!\left(
      Y_l^N,
      \tfrac{T}{N},
      \Delta W_l^N
    \right)
    \big| \,
    \mathcal{F}_{ \frac{ l T }{ N } }
  \Big]
  =
  \mathbbm{1}_{ \{ \delta_v^N > l \} } 
  \,
  \phi_v\!\left(
    Y^N_l ,
    \tfrac{T}{N}
  \right)
\end{equation}
$ \P $-a.s.\ for all 
$ l \in \{ 0, 1, \dots, N - 1 \} $,
$ N \in \N $
with $ \frac{ T }{ N } \leq t_v $
and all
$ v \in \N $.
Combining \eqref{eq:Zequation} and \eqref{eq:Yequation}
with the triangle inequality 
then implies
\begin{align*}
  &
  \big\|
    Z_{ k \wedge \delta_v^N }^N
    -
    Y_{ k \wedge \delta_v^N }^N
  \big\|
\\ & \leq
  \frac{T}{N}
  \sum_{ l = 0 }^{ k - 1 }
  \mathbbm{1}_{ \{ \delta_v^N > l \} }
  \left\|
    \bar{\mu}( Z_l^N ) - \bar{\mu}( Y_l^N ) 
  \right\|
  +
  \left\|
    \sum_{ l = 0 }^{ k - 1 }
    \mathbbm{1}_{ \{ \delta_v^N > l \} }
    \left(
      \bar{\sigma}( Z_l^N ) - 
      \bar{\sigma}( Y_l^N ) 
    \right)
    \Delta W_l^N
  \right\|
\\ & +
  \frac{T}{N}
  \sum_{ l = 0 }^{ k - 1 }
  \mathbbm{1}_{ \{ \delta_v^N > l \} } 
  \left\|
    \bar{\mu}( Y_l^N ) 
    -
    \tfrac{N}{T}
    \cdot
    \phi_v\!\left(
      Y_l^N,
      \tfrac{T}{N}
    \right)
  \right\|
\\ & +
  \Bigg\|
    \sum_{ l = 0 }^{ k - 1 }
    \mathbbm{1}_{ \{ \delta_v^N > l \} } \,
    \bigg(
      \bar{\sigma}( Y_l^N ) \,
      \Delta W_l^N
      +
      \phi_v\!\left(
        Y_l^N,
        \tfrac{T}{N}
      \right)
    -
      \phi\!\left(
        Y_l^N,
        \tfrac{T}{N},
        \Delta W_l^N
      \right)
    \bigg)
  \Bigg\|
\end{align*}
for all 
$ k \in \{ 0, 1, \ldots, N \} $, 
$ N \in \mathbb{N} $ with 
$ \frac{T}{N} \leq t_v $ 
and all 
$ v \in \mathbb{N} $. 
The definition of 
$ c_v \in [0, \infty) $, 
$ v \in \mathbb{N} $,
(see~\eqref{eq:def_cr}) 
hence yields
\begin{align*}
  &
  \sup_{ k \in \{ 0, 1, \ldots, n \} }
  \big\|
    Z_{ k \wedge \delta_v^N }^N
    -
    Y_{ k \wedge \delta_v^N }^N
  \big\|
\\&\leq
  \frac{T c_v}{N}
  \sum_{ l = 0 }^{ n - 1 }
  \big\|
    Z_{ l \wedge \delta_v^N }^N
    -
    Y_{ l \wedge \delta_v^N }^N
  \big\|
  +
  \sup_{ k \in \{ 0, 1, \ldots, n \} }
  \left\|
    \sum_{ l = 0 }^{ k - 1 }
    \mathbbm{1}_{ \{ \delta_v^N > l \} }
    \left(
      \bar{\sigma}( Z_l^N ) - 
      \bar{\sigma}( Y_l^N ) 
    \right)
    \Delta W_l^N
  \right\|
\\ & +
  T 
  \left(
    \sup_{ x \in D_v }
    \left\|
      \mu( x ) 
      -
      \tfrac{N}{T}
      \cdot
        \phi_v(
          x,
          \tfrac{T}{N}
        )
    \right\|
  \right)
\\&+
  \sup_{ 0 \leq k \leq N }
  \Bigg\|
    \sum_{ l = 0 }^{ k - 1 }
    \mathbbm{1}_{ \{ \delta_v^N > l \} }
    \Big(
      \bar{\sigma}( Y_l^N ) \,
      \Delta W_l^N
      +
      \phi_v\!\left(
        Y_l^N,
        \tfrac{T}{N}
      \right)
      -
      \phi\!\left(
        Y_l^N,
        \tfrac{T}{N},
        \Delta W_l^N
      \right)
    \Big)
  \Bigg\|
\end{align*}
for all $ n \in \{ 0, 1, \ldots, N \} $, 
$ N \in \mathbb{N} $ 
with $ \frac{T}{N} \leq t_v $ 
and all 
$ v \in \mathbb{N} $. 
Combining this and \eqref{eq:Bedingte_Erwartung}
with the Burkholder-Davis-Gundy
inequality
(see, e.g., 
Theorem~48
in Protter~\cite{Protter2004})
then implies 
the existence of a real 
number $ \kappa \in [0, \infty) $ 
such that
\begin{align*}
  &
  \left\|
    \sup_{ k \in \{ 0, 1, \ldots, n \} }
    \left\|
      Z_{ k \wedge \delta_v^N }^N
      -
      Y_{ k \wedge \delta_v^N }^N
    \right\|
  \right\|_{ L^1( \Omega; \mathbb{R} ) }^2
\\&\leq
  \frac{4 T^2 | c_v |^2 }{ N }
  \sum_{ l = 0 }^{ n - 1 }
  \left\|
    Z_{ l \wedge \delta_v^N }^N
    -
    Y_{ l \wedge \delta_v^N }^N
  \right\|_{ 
    L^1( \Omega; \mathbb{R}^d ) 
  }^2
\\&+
  4 \kappa
  \sum_{ l = 0 }^{ n - 1 }
  \left\|
    \mathbbm{1}_{ \{ \delta_v^N > l \} }
    \left(
      \bar{\sigma}( Z_l^N ) - 
      \bar{\sigma}( Y_l^N ) 
    \right)
    \Delta W_l^N
  \right\|_{ 
    L^1( \Omega; \mathbb{R}^d ) 
  }^2
\\&+
  4 T^2 
  \left(
    \sup_{ x \in D_v }
    \left\|
      \mu( x ) 
      -
      \tfrac{N}{T}
      \cdot
      \phi_v(
        x,
        \tfrac{T}{N}
      )
    \right\|^2
  \right)
\\&+
  4 \kappa
  \sum_{ l = 0 }^{ N - 1 }
  \Big\|
    \mathbbm{1}_{ \{ \delta_v^N > l \} }
    \Big(
      \bar{\sigma}( Y_l^N ) \,
      \Delta W_l^N
      +
      \phi_v\!\left(
        Y_l^N,
        \tfrac{T}{N}
      \right)
      -
      \phi\!\left(
        Y_l^N,
        \tfrac{T}{N},
        \Delta W_l^N
      \right)
    \Big)
  \Big\|_{ 
    L^1( \Omega; \mathbb{R}^d ) 
  }^2
\end{align*}
for all $ n \in \{ 0, 1, \ldots, N \} $, 
$ N \in \mathbb{N} $ with 
$ \frac{T}{N} \leq t_v $ and all 
$ v \in \mathbb{N} $.
The estimate
\begin{equation}
\begin{split}
&
  \mathbb{E}\bigg[
  \left\|
    \mathbbm{1}_{ \{ \delta_v^N > l \} }
    \Big(
      \bar{\sigma}( Y_l^N ) \,
      \Delta W_l^N
      +
      \phi_v\!\left(
        Y_l^N,
        \tfrac{T}{N}
      \right)
      -
      \phi\!\left(
        Y_l^N,
        \tfrac{T}{N},
        \Delta W_l^N
      \right)
    \Big)
  \right\|
  \bigg]
\\ & =
  \mathbb{E}\bigg[
    \mathbbm{1}_{ \{ \delta_v^N > l \} }
  \left\|
      \bar{\sigma}( Y_l^N ) \,
      \Delta W_l^N
      +
      \phi_v\!\left(
        Y_l^N,
        \tfrac{T}{N}
      \right)
      -
      \phi\!\left(
        Y_l^N,
        \tfrac{T}{N},
        \Delta W_l^N
      \right)
  \right\|
  \bigg]
\\ & \leq
    \sup_{ x \in D_v }
    \mathbb{E}\bigg[
      \Big\|
        \sigma( x )
        W_{ T / N }
        +
        \phi_v(
          x,
          \tfrac{T}{N}
        )
    -
        \phi(
          x,
          \tfrac{T}{N},
          W_{ T/N }
        )
      \Big\|
    \bigg]
\end{split}
\end{equation}
for all 
$ l \in \{ 0, 1, \dots, N-1 \} $,
$ N \in \mathbb{N} $
with 
$ \frac{ T }{ N } \leq t_v $
and all $ v \in \mathbb{N} $
hence shows
\begin{equation}
\begin{split}
  &
  \left\|
    \sup_{ k \in \{ 0, 1, \ldots, n \} }
    \left\|
      Z_{ k \wedge \delta_v^N }^N
      -
      Y_{ k \wedge \delta_v^N }^N
    \right\|
  \right\|_{ 
    L^1( \Omega; \mathbb{R} ) 
  }^2
\\&\leq
  \frac{
    4 T^2 | c_v |^2 
  }{ 
    N 
  }
  \sum_{ l = 0 }^{ n - 1 }
  \left\|
    Z_{ l \wedge \delta_v^N }^N
    -
    Y_{ l \wedge \delta_v^N }^N
  \right\|_{ 
    L^1( \Omega; \mathbb{R}^d ) 
  }^2
\\ & +
  \frac{
    4 T m \kappa 
  }{ N }
  \left(
  \sum_{ l = 0 }^{ n - 1 }
  \sum_{ i = 1 }^{ m }
  \left\|
    \mathbbm{1}_{ 
      \{ \delta_v^N > l \} 
    }
    \left(
      \bar{\sigma}_i( Z_l^N ) - 
      \bar{\sigma}_i( Y_l^N ) 
    \right)
  \right\|_{ 
    L^1( \Omega; \mathbb{R}^d )
  }^2
  \right)
\\ & +
  4 T^2 \!
  \left(
    \sup_{ x \in D_v }
    \left\|
      \mu( x ) 
      -
      \tfrac{N}{T}
      \cdot
      \phi_v(
        x,
        \tfrac{T}{N}
      )
    \right\|^2
  \right)
\\ & +
  4 \kappa N
  \left(
    \sup_{ x \in D_v }
      \Big\|
        \sigma( x )
        W_{ T / N }
        +
        \phi_v(
          x,
          \tfrac{T}{N}
        )
        -
        \phi(
          x,
          \tfrac{T}{N},
          W_{ T/N }
        )
  \Big\|_{
    L^1(\Omega;\R^d)
  }^2
  \right)
\end{split}
\end{equation}
for all $ n \in \{ 0, 1, \ldots, N \} $, 
$ N \in \mathbb{N} $ with 
$ \frac{T}{N} \leq t_v $ and all 
$ v \in \mathbb{N} $. 
The definition 
of $ c_v \in [0, \infty) $, 
$ v \in \mathbb{N} $, 
(see~\eqref{eq:def_cr}) 
hence gives
\begin{equation}
\begin{split}
\label{eq:proof_lem_A_P}
  &
  \left\|
    \sup_{ k \in \{ 0, 1, \ldots, n \} }
    \left\|
      Z_{ k \wedge \delta_v^N }^N
      -
      Y_{ k \wedge \delta_v^N }^N
    \right\|
  \right\|_{ 
    L^1( \Omega; \mathbb{R} ) 
  }^2
\\&\leq
  \left(
    \frac{
      4 T m^2 | c_v |^2 
      ( T + \kappa )
    }{ N }
  \right)
  \left(
  \sum_{ l = 0 }^{ n - 1 }
  \left\|
    Z_{ l \wedge \delta_v^N }^N
    -
    Y_{ l \wedge \delta_v^N }^N
  \right\|_{ 
    L^1( \Omega; \mathbb{R}^d ) 
  }^2
  \right)       
\\&+
  4 T^2 \!
  \left(
    \sup_{ x \in D_v }
    \left\|
      \mu( x ) 
      -
      \tfrac{N}{T}
      \cdot
        \phi_v(
          x,
          \tfrac{T}{N}
        )
    \right\|^2
  \right)
\\&+
  4 T \kappa
  \left(
    \frac{ N }{ T }
    \cdot
    \sup_{ x \in D_v }
      \Big\|
        \sigma( x )
        W_{ T / N }
        +
        \phi_v(
          x,
          \tfrac{T}{N}
        )
      -
        \phi(
          x,
          \tfrac{T}{N},
          W_{ T/N }
        )
      \Big\|_{
    L^1(\Omega;\R^d)
  }^2
  \right)
\end{split}
\end{equation}
for all $ n \in \{ 0, 1, \ldots, N \} $, 
$ N \in \mathbb{N} $ with 
$ \frac{T}{N} \leq t_v $ and 
all $ v \in \mathbb{N} $. 
Moreover, note that~\eqref{ass:consistency_2B} 
in Lemma~\ref{lem:consistent}
implies
\begin{equation}
  \lim_{ t \searrow 0 }
  \left(
    \frac{ 1 }{ \sqrt{t } }
    \cdot
    \sup_{ x \in D_v }
    \big\| 
        \phi_v(
          x,
          t
        )
    \big\|
  \right)
  = 0
\end{equation}
for all $ v \in \mathbb{N} $. 
This and~\eqref{ass:consistency_1B} 
in Lemma~\ref{lem:consistent}
then yield
\begin{equation}
\label{eq:proof_lem_A_Q}
  \lim_{ t \searrow 0 }
  \left(
    \frac{1}{t}
    \cdot
    \sup_{ x \in D_v }
      \Big\|
        \sigma(x) W_t
        +
          \phi_v(
            x,
            t
          )
        -
        \phi(
          x,
          t,
          W_t
        )
      \Big\|_{
    L^1(\Omega;\R^d)
  }^2
  \right)
  = 0
\end{equation}
for all $ v \in \mathbb{N} $. 
Combining~\eqref{ass:consistency_2B},
\eqref{eq:proof_lem_A_P},
\eqref{eq:proof_lem_A_Q} and 
Gronwall's lemma 
then shows
\begin{equation}
  \lim_{ N \rightarrow \infty }
  \left\|
    \sup_{ n \in \{ 0, 1, \ldots, N \} }
    \left\|
      Z_{ n \wedge \delta_v^N }^N
      -
      Y_{ n \wedge \delta_v^N }^N
    \right\|
  \right\|_{ 
    L^1( \Omega; \mathbb{R} ) 
  }^2
  = 0
\end{equation}
for all $ v \in \mathbb{N} $. 
This completes the proof 
of Lemma~\ref{lem_A}.
\end{proof}


\begin{proof}[Proof of Lemma~\ref{lem_sigma}]
Note that subadditivity and monotonicity
of the probability 
measure 
$ \mathbb{P} $ show that
\begin{equation}
\begin{split}
  \mathbb{P}\!\left[
    \delta_v^N < N
  \right]
  &\leq
  \mathbb{P}\!\left[
    \tau_{ w }^N < T
  \right]
  +
  \mathbb{P}\!\left[
    \tau_{ w }^N = T, \,
    \delta_v^N < N
  \right]
\\ & \leq
  \mathbb{P}\!\left[
    \tau_{ w }^N < T
  \right]
  +
  \mathbb{P}\!\left[
    Z_{ \delta_v^N }^N 
    \in D_{ w }, \,
    Y_{ \delta_v^N }^N \notin D_{ v }
  \right]
\\ & \leq
  \mathbb{P}\!\left[
    \tau_{ w }^N < T
  \right]
  +
  \mathbb{P}\!\left[
    \|
      Z_{ \delta_v^N }^N 
      -
      Y_{ \delta_v^N }^N
    \|
    \geq
    \text{dist}\big(
      D_{ w }, \left( D_{ v } \right)^c
    \big)
  \right]
\end{split}
\end{equation}
for all $ v,w, N \in \mathbb{N} $
with $ w \leq v $.
Markov's inequality 
and Lemma~\ref{lem_A}
therefore yield
\begin{equation}
\label{eq:uselimit}
\begin{split}
&
  \limsup_{ v \rightarrow \infty }
  \limsup_{ N \rightarrow \infty }
  \mathbb{P}\!\left[
    \delta_v^N < N
  \right]
\\ & \leq
  \limsup_{ N \rightarrow \infty }
  \mathbb{P}\!\left[
    \tau_{ w }^N < T
  \right]
  +
  \limsup_{ v \rightarrow \infty }
  \limsup_{ N \rightarrow \infty }
  \left(
  \frac{
    \mathbb{E}\big[
      \|
        Z_{ \delta_v^N }^N 
        -
        Y_{ \delta_v^N }^N
      \|
    \big]
  }{
    \text{dist}\big(
      D_{ w }, 
      \left( D_{ v } \right)^c
    \big)
  }
  \right)
\\ & =
  \limsup_{ N \rightarrow \infty }
  \mathbb{P}\!\left[
    \tau_{ w }^N < T
  \right]
\end{split}
\end{equation}
for all $ w \in \mathbb{N} $.
Combining \eqref{eq:uselimit}
and 
Lemma~\ref{lem_tau}
completes the proof
of Lemma~\ref{lem_sigma}.
\end{proof}

\begin{proof}[Proof of 
Lemma~\ref{l:CA}]
The identity 
$ 
  \{ \delta_v^N < N \}
  \uplus
  \{ \delta_v^N = N \}
  =\Omega
$ 
and Markov's inequality imply
\begin{equation}  
\begin{split}
  &
  \mathbb{P}\!\left[
    \sup_{ 
      n \in \{ 0, 1, \ldots, N \} 
    }     
    \left\| Z_n^N - Y_n^N \right\|
    \geq
    \varepsilon
  \right]
 \\
 &\leq
  \mathbb{P}\!\left[
    \delta_v^N < N
  \right]
  +
  \mathbb{P}\!\left[
    \delta_v^N = N,
    \sup_{ 
      n \in 
      \{ 0, 1, \ldots, \delta_v^N  \} 
    }
    \left\| Z_n^N - Y_n^N \right\|
    \geq
    \varepsilon
  \right]
 \\
 &\leq
  \mathbb{P}\!\left[
    \delta_v^N < N
  \right]
  +
  \frac{ 1 }{ \varepsilon }
  \cdot
  \mathbb{E}\!\left[
    \sup_{ 
      n \in 
      \{ 0, 1, \ldots, \delta_v^N  \} 
    }
    \left\| Z_n^N - Y_n^N \right\|
  \right]
\end{split}     
\end{equation}
for all $v,N\in\N$ and 
all $ \varepsilon \in (0,\infty) $. 
Lemma~\ref{lem_sigma}
and
Lemma~\ref{lem_A}
therefore yield
\begin{equation}
  \lim_{ N \rightarrow \infty }
  \mathbb{P}\!\left[
    \sup_{ n \in \{ 0, 1, \ldots, N \} }
    \left\| Z_n^N - Y_n^N \right\|
    \geq
    \varepsilon
  \right]
  =
  0
\end{equation}
for all $ \varepsilon \in (0,\infty) $. 
This completes the 
proof of Lemma~\ref{l:CA}.
\end{proof}


\begin{proof}[Proof of Theorem~\ref{thm:convergence}]
In order to show Theorem~\ref{thm:convergence}, 
let 
$ \bar{X}^N \colon [0,T] 
\times \Omega \rightarrow \mathbb{R}^d 
$, 
$ N \in \mathbb{N} $, 
be a sequence of 
stochastic processes 
defined by 
\begin{equation}
  \bar{X}_t^N 
  := 
  \left( n + 1 - \tfrac{ t N }{ T } \right) 
  X_{ \frac{ nT }{ N } } 
  + 
  \left( \tfrac{ t N }{ T } - n \right) 
  X_{ \frac{ (n+1)T }{ N } } 
\end{equation}
for all 
$ 
  t \in 
  [ \frac{ nT }{ N }, \frac{ (n+1)T }{ N } ] 
$, 
$ n \in \{ 0,1,\ldots, N-1 \} $ 
and all $ N \in \mathbb{N} $. 
Then 
\begin{equation}
  \sup_{ t \in [0, T ] }
  \left\|
    \bar{X}_t^N - \bar{Y}_t^N
  \right\|
  =
  \sup_{ n \in \{ 0,1,\ldots, N \} }
  \left\|
    X_{ \frac{ nT }{ N } } - Y_n^N
  \right\|
\end{equation}
for all $ N \in \mathbb{N} $.
Moreover, the continuity of the sample 
paths of 
$ 
  X \colon [0,T] \times \Omega 
  \rightarrow D 
$ 
yields
\begin{equation}
\label{eq:thm2_p2}
  \lim_{ N \rightarrow \infty }
  \!
  \left(
    \sup_{ t \in [0, T ] }
    \left\|
      X_t - \bar{X}_t^N
    \right\|
  \right)
  = 0
\end{equation}
$ \mathbb{P} $-a.s.. 
This implies 
$   
  \lim_{ N \to \infty }
  \mathbb{P}\big[
    \sup_{ t \in [0, T] }
    \|
      X_t - \bar{X}_t^N
    \|
    \geq \varepsilon
  \big]
  = 0
$
for all 
$ 
  \varepsilon \in (0,\infty)
$
and
Corollary~\ref{lem_C}
hence shows
\begin{equation}  
\begin{split}
  &
  \limsup_{ N \to \infty }
  \mathbb{P}\Bigg[
    \sup_{ t \in [0,T] }
    \big\| 
      X_t 
      - 
      \bar{Y}^N_t
    \big\|
    \geq
    \varepsilon
  \Bigg]
\\ & \leq 
  \limsup_{ N \to \infty }
  \mathbb{P}\Bigg[
    \sup_{ t \in [0,T] }
    \big\| 
      X_t 
      - 
      \bar{X}_t^N
    \big\|
    +
    \sup_{ t \in [0,T] }
    \big\| 
      \bar{X}_t^N
      - 
      \bar{Y}^N_t
    \big\|
    \geq
    \varepsilon
  \Bigg]
\\ & \leq 
  \limsup_{ N \to \infty }
  \mathbb{P}\Bigg[
    \sup_{ t \in [0,T] }
    \big\| 
      X_t 
      - 
      \bar{X}_t^N
    \big\|
    \geq
    \tfrac{\varepsilon}{2}
  \Bigg]
\\ & \quad
  +
  \limsup_{ N \to \infty }
  \mathbb{P}\Bigg[
  \sup_{ n \in \{ 0,1,\ldots, N \} }
  \left\|
    X_{ \frac{ nT }{ N } } - Y_n^N
  \right\|
    \geq
    \tfrac{\varepsilon}{2}
  \Bigg]
  = 0
\end{split}     
\end{equation}
for all 
$ \varepsilon \in (0,\infty) $.
This completes the proof of Theorem~\ref{thm:convergence}.
\end{proof}

\section{Strong convergence}
\label{sec:strongconvergence}

In this section,
we combine the convergence in probability result
of Theorem~\ref{thm:convergence}
with moment bounds 
for the numerical approximation
processes
in \eqref{eq:recX}
to obtain strong convergence
of the numerical approximation
processes
in \eqref{eq:recX}.

\subsection{Strong convergence
based on moment bounds}
\label{sec:strongCbased}

This subsection presents
strong convergence results
under the assumption
that the
numerical approximation
processes in \eqref{eq:recX}
satisfy suitable
moment bounds.
Below in 
Subsections~\ref{sec:strongCtamed}
and
\ref{sec:strongcsemi1},
we will give more concrete
conditions on the drift
coefficient $ \mu $,
the diffusion coefficient
$ \sigma $ and the
numerical method which
allow us to avoid to impose 
these
moment bound assumptions.
The strong convergence
results in this subsection
use the following well-known
modification of
Fatou's lemma.
For completeness its proof
is given below.

\begin{lemma}[A modified
version of Fatou's lemma]
\label{lem:fatou}
Let 
$ 
  \left( \Omega, \mathcal{F},
  \mathbb{P} \right)
$
be a probability space,
let
$
  \left( E, d_E \right)
$
be a separable metric space
and let
$ 
  Z_N \colon \Omega
  \rightarrow E 
$,
$ N \in \N $,
and
$
  Z \colon \Omega \rightarrow
 E
$
be 
$ \mathcal{F} $/$ \mathcal{B}(E) 
$-measurable mappings
with
$ 
  \lim_{ N \to \infty }
  \mathbb{P}\big[
    d_E( Z_N, Z ) \geq \varepsilon 
  \big]
  = 0
$
for all
$ \varepsilon \in (0,\infty) $.
Then
\begin{equation}
  \mathbb{E}\big[
    \varphi( Z )
  \big]
  \leq
  \liminf_{ N \to \infty }
  \mathbb{E}\big[
    \varphi( Z_N )
  \big]
\end{equation}
for all continuous
functions 
$ 
  \varphi \colon E
  \rightarrow [0,\infty]
$.
\end{lemma}

\begin{proof}[Proof
of Lemma~\ref{lem:fatou}]
First, let 
$ 
  \varphi \colon E
  \to [0,\infty]
$
be an arbitrary continuous
function.
Then let
$ N(k) \in \N $,
$ k \in \N $,
be an increasing sequence
of natural numbers
such that
\begin{equation}
\label{eq:fatou1}
  \liminf_{ N \to \infty }
  \mathbb{E}\big[
    \varphi( Z_N )
  \big]
  =
  \lim_{ k \to \infty }
  \mathbb{E}\big[
    \varphi( Z_{ N(k) } )
  \big] .
\end{equation}
Next note 
that
$ 
  \lim_{ k \to \infty }
  \mathbb{P}\big[
    d_E( Z_{ N(k) }, Z ) 
    \geq \varepsilon 
  \big]
  = 0
$
for all 
$   
  \varepsilon \in (0,\infty)
$
by assumption.
Consequently,
there exists
an increasing
sequence
$ k_l \in \N $,
$ l \in \N $,
of natural numbers
such that
$
  \lim_{ l \to \infty }
  Z_{ N( k_l ) }
  =
  Z
$
$ \mathbb{P} $-a.s..
The continuity of
$ \varphi $
hence implies
$
  \lim_{ l \to \infty }
  \varphi( Z_{ N( k_l ) } )
  =
  \varphi( Z )
$
$ \mathbb{P} $-a.s..
Combining this,
Fatou's lemma 
and
\eqref{eq:fatou1}
then
gives
\begin{equation}
\begin{split}
  \mathbb{E}\big[
    \varphi(Z)
  \big]
&  =
  \mathbb{E}\!\left[
    \lim_{ l \to \infty }
    \varphi( Z_{ N(k_l) } )
  \right]
  \leq
  \liminf_{ l \to \infty }
  \mathbb{E}\big[
    \varphi( Z_{ N(k_l) } )
  \big] 
\\ & =
  \lim_{ k \to \infty }
  \mathbb{E}\big[
    \varphi( Z_{ N(k) } )
  \big] 
  =
  \liminf_{ N \to \infty }
  \mathbb{E}\big[
    \varphi( Z_N )
  \big] .
\end{split}
\end{equation}
The proof of
Lemma~\ref{lem:fatou}
is thus completed.
\end{proof}

Let us now present the
promised strong convergence
results which use the
assumption of moment bounds
for the numerical approximation
processes in \eqref{eq:recX}.

\begin{cor}[Strong final value
convergence
based on moment bounds]
\label{cor:convergence1}
Assume that the setting
in Section~\ref{sec:convergencesetting}
is fulfilled, 
suppose that
$ 
  \phi \colon \mathbb{R}^d
  \times [0,T] \times \mathbb{R}^m
  \rightarrow \mathbb{R}^d
$ is
$ ( \mu, \sigma ) $-consistent with respect to Brownian motion
and
let 
$ p \in (0,\infty ) $
be a real number such that
\begin{equation}
\label{eq:cor1a}
  \limsup_{ N \rightarrow \infty }
  \mathbb{E}\Big[
    \big\| \bar{Y}^N_T \big\|^p
  \Big]
  < \infty .
\end{equation}
Then 
$ \E\big[ \| X_T \|^p \big] < \infty $
and
\begin{equation}
\label{eq:cor1b}
  \lim_{ N \rightarrow \infty }
    \mathbb{E}\Big[
      \big\| 
        X_T
        - 
        \bar{Y}^N_T
      \big\|^q
    \Big]
= 
  0 
\end{equation}
for all $ q \in (0,p) $.
\end{cor}

\begin{proof}[Proof
of Corollary~\ref{cor:convergence1}]
First, let $ q \in (0,p) $ be 
arbitrary.
Next observe that 
inequality~\eqref{eq:cor1a} 
implies that there exists a natural
number $ N_0 \in \mathbb{N} $
such that
\begin{equation}
\label{eq:momentcor1}
  \sup_{ 
    N \in \{ N_0, N_0+1, \dots \}
  }
  \mathbb{E}\Big[
    \big\| \bar{Y}^N_T \big\|^p
  \Big]
  < \infty .
\end{equation}
Theorem~\ref{thm:convergence}
and Lemma~\ref{lem:fatou}
hence imply that
\begin{equation}
\label{eq:fatou}
\begin{split}
  \E\big[
    \| X_T \|^p
  \big]
  \leq
  \liminf_{ N \to \infty }
  \E\Big[
    \big\|
      \bar{Y}_T^{ N }
    \big\|^p
  \Big]
  \leq
  \sup_{ 
    N \in \{ N_0, N_0+1, \dots \}
  }
  \mathbb{E}\Big[
    \big\| \bar{Y}^N_T \big\|^p
  \Big]
  < \infty .
\end{split}
\end{equation}
This together with
inequality~\eqref{eq:momentcor1}
shows that the family of 
random variables
\begin{equation}
  \| X_T - \bar{Y}^N_T \|^q ,
  \quad
  N \in \{ N_0, N_0+1, \dots \} ,
\end{equation}
is bounded in 
$ L^{ p / q }(\Omega; \R)
$
and, therefore, 
uniformly integrable 
(see, e.g., Corollary~6.21 in 
Klenke~\cite{k08b}).
Theorem~6.25
in Klenke~\cite{k08b} 
and Theorem~\ref{thm:convergence}
hence 
imply \eqref{eq:cor1b}.
This completes the proof
of Corollary~\ref{cor:convergence1}.
\end{proof}

\begin{cor}[Strong convergence
based on moment 
bounds]
\label{cor:convergence2}
Assume that the setting
in Section~\ref{sec:convergencesetting}
is fulfilled,
suppose that
$ 
  \phi \colon \mathbb{R}^d
  \times [0,T] \times \mathbb{R}^m
  \rightarrow \mathbb{R}^d
$ is
$ ( \mu, \sigma ) $-consistent with respect to Brownian motion
and
let 
$ p \in (0,\infty ) $
be a real number such that
\begin{equation}
\label{eq:cor2a}
  \limsup_{ N \rightarrow \infty }
  \sup_{ n \in \{ 0, 1, \dots, N \} }
  \mathbb{E}\Big[
    \big\| 
      \bar{Y}^N_{ 
        \frac{ n T }{ N }
      } 
    \big\|^p
  \Big]
  < \infty .
\end{equation}
Then
$
  \sup_{ t \in [0,T] }
  \E\big[
    \| X_t \|^p
  \big]
  < \infty
$
and
\begin{equation}
\label{eq:cor2b}
  \lim_{ N \rightarrow \infty }
  \Bigg(
    \sup_{ t \in [0,T] }
    \mathbb{E}\Big[
      \big\| 
        X_t 
        - 
        \bar{Y}^N_t
      \big\|^q
    \Big]
  \Bigg)
= 
  0 
\end{equation}
for all $ q \in (0,p) $.
\end{cor}

\begin{proof}[Proof
of Corollary~\ref{cor:convergence2}]
Inequality~\eqref{eq:cor2a} 
implies
\begin{equation}
\label{eq:momentcor2_0}
  \limsup_{ N \rightarrow \infty }
  \sup_{ t \in [0,T] }
  \mathbb{E}\Big[
    \big\| \bar{Y}^N_{ t } \big\|^p
  \Big]
  < \infty 
\end{equation}
and this yields that there 
exists a natural
number $ N_0 \in \mathbb{N} $
such that
\begin{equation}
\label{eq:momentcor2}
  \sup_{ 
    N \in \{ N_0, N_0+1, \dots \}
  }
  \sup_{ t \in [0,T] }
  \mathbb{E}\Big[
    \big\| \bar{Y}^N_t \big\|^p
  \Big]
  < \infty .
\end{equation}
Theorem~\ref{thm:convergence}
and Lemma~\ref{lem:fatou}
therefore show that
\begin{equation}  
\begin{split}  
\label{eq:XLp.finite}
  \sup_{ t \in [0,T] }
  \mathbb{E}\Big[
    \big\| X_t \big\|^p
  \Big]
  & \leq
  \sup_{ t \in [0,T] }
  \liminf_{ N \to \infty }
  \mathbb{E}\Big[
    \big\| \bar{Y}^N_t \big\|^p
  \Big]
\leq
  \limsup_{ N \to \infty }
  \sup_{ t \in [0,T] }
  \mathbb{E}\Big[
    \big\| \bar{Y}^N_t \big\|^p
  \Big] < \infty .
\end{split}     
\end{equation}
Moreover, the identity
\begin{equation}
  \mathbbm{1}_{
    \left\{
      \sup_{ s \in [0,T] } 
      \| X_s - \bar{Y}^N_s \|
      < 1 
    \right\}
  }
+
  \mathbbm{1}_{
    \left\{
      \sup_{ s \in [0,T] } 
      \| X_s - \bar{Y}^N_s \|
      \geq 1 
    \right\}
  }
\equiv 
  1 
\end{equation}
and H\"{o}lder's inequality
imply
\begin{align*}
&
    \sup_{ t \in [0,T] }
    \mathbb{E}\Big[
      \left\| 
        X_t 
        - 
        \bar{Y}^N_t
      \right\|^q
    \Big]
\\
& \leq
    \sup_{ t \in [0,T] }
    \mathbb{E}\bigg[
  \mathbbm{1}_{
    \left\{
      \sup_{ s \in [0,T] } 
      \| X_s - \bar{Y}^N_s \|
      < 1 
    \right\}
  }
      \left\| 
        X_t 
        - 
        \bar{Y}^N_t
      \right\|^q
    \bigg]
\\ & +
    \sup_{ t \in [0,T] }
    \mathbb{E}\bigg[
  \mathbbm{1}_{
    \left\{
      \sup_{ s \in [0,T] } 
      \| X_s - \bar{Y}^N_s \|
      \geq 1 
    \right\}
  }
      \left\| 
        X_t 
        - 
        \bar{Y}^N_t
      \right\|^q
    \bigg]
\\ & \leq
    \mathbb{E}\bigg[
  \mathbbm{1}_{
    \left\{
      \sup_{ t \in [0,T] } 
      \| X_t - \bar{Y}^N_t \|
      < 1 
    \right\}
  }
    \sup_{ t \in [0,T] }
      \left\| 
        X_t 
        - 
        \bar{Y}^N_t
      \right\|^q
    \bigg]
\\ & +
  \left(
    \mathbb{P}\bigg[
      \sup_{ t \in [0,T] } 
      \| X_t - \bar{Y}^N_t \|
      \geq 1 
    \bigg]
  \right)^{ 
    \!\! \frac{ (p - q) }{ p }
  }
\!\!
  \left(
    \sup_{
      M \in \{ N_0, N_0+1,\dots\}
    }
    \sup_{ t \in [0,T] }
    \mathbb{E}\Big[
      \left\| 
        X_t 
        - 
        \bar{Y}^M_t
      \right\|^p
    \Big]
  \right)^{ 
    \!\! \frac{ q }{ p }
  }
\end{align*}
for all 
$ N \in \{ N_0, N_0+1, \dots 
\} $
and all $ q \in (0,p) $.
The estimate
$ 
  \left| a + b \right|^p
  \leq 2^p \left| a \right|^p + 2^p 
  \left| b \right|^p 
$
for all 
$ a, b \in \mathbb{R} $
and 
\eqref{eq:momentcor2}
and~\eqref{eq:XLp.finite}
therefore give
\begin{align}
\label{eq:help0}
&
    \sup_{ t \in [0,T] }
    \mathbb{E}\Big[
      \left\| 
        X_t 
        - 
        \bar{Y}^N_t
      \right\|^q
    \Big]
\\
& \leq
\nonumber
  \mathbb{E}\bigg[
    \min\!\bigg( 1,
      \sup_{ t \in [0,T] }
      \left\| 
        X_t 
        - 
        \bar{Y}^N_t
      \right\|^q
    \bigg)
  \bigg]
\\ & +
\nonumber
  2^p
  \left(
    \mathbb{P}\bigg[
      \sup_{ t \in [0,T] } 
      \| X_t - \bar{Y}^N_t \|
      \geq 1 
    \bigg]
  \right)^{ 
    \!\! \frac{ (p - q) }{ p }
  }
  \left(
    \sup_{
      M \in \{ N_0, N_0+1,\dots\}
    }
    \sup_{ t \in [0,T] }
    \mathbb{E}\Big[
      \big\| 
        \bar{Y}^M_t
      \big\|^p
    \Big]
  \right)^{ 
    \!\! \frac{ q }{ p }
  }
\\ & +
\nonumber
  2^p
  \left(
    \mathbb{P}\bigg[
      \sup_{ t \in [0,T] } 
      \| X_t - \bar{Y}^N_t \|
      \geq 1 
    \bigg]
  \right)^{ 
    \!\! \frac{ (p - q) }{ p }
  }
  \left(
    \sup_{ t \in [0,T] }
    \mathbb{E}\Big[
      \big\| 
        X_t 
      \big\|^p
    \Big]
  \right)^{ 
    \!\! \frac{ q }{ p }
  }
  < \infty
\end{align}
for all 
$ N \in \{ N_0, N_0+1, \dots 
\} $
and all $ q \in (0,p) $.
Moreover, 
Theorem~\ref{thm:convergence}
implies
\begin{equation}
\label{eq:help1}
  \lim_{ N \rightarrow \infty }
    \mathbb{P}\bigg[
      \sup_{ t \in [0,T] } 
      \| X_t - \bar{Y}^N_t \|
      \geq 1 
    \bigg]
= 0 
\end{equation}
and
\begin{equation}
\label{eq:help2}
  \lim_{ N \rightarrow \infty }
  \mathbb{E}\bigg[
    \min\!\bigg( 1,
      \sup_{ t \in [0,T] }
      \left\| 
        X_t 
        - 
        \bar{Y}^N_t
      \right\|^q
    \bigg)
  \bigg]
  = 0
\end{equation}
for all $ q \in (0,\infty) $.
Combining 
\eqref{eq:help0}--\eqref{eq:help2}
then shows
\eqref{eq:cor2b}.
This completes the proof
of Corollary~\ref{cor:convergence2}.
\end{proof}

\begin{cor}[Uniform strong
convergence based on a 
priori moment bounds]
\label{cor:convergence4}
Assume that the setting
in 
Section~\ref{sec:convergencesetting}
is fulfilled,
suppose that
$ 
  \phi \colon \mathbb{R}^d
  \times [0,T] \times \mathbb{R}^m
  \rightarrow \mathbb{R}^d
$ is
$ ( \mu, \sigma ) $-consistent with respect to Brownian motion
and let $ p \in (0,\infty) $ be a real
number such that
\begin{equation}
\label{eq:cor4a}
  \limsup_{ N \rightarrow \infty }
  \mathbb{E}\Bigg[
    \sup_{ n \in \{ 0, 1, \dots, N \} }
    \big\| 
      \bar{Y}^N_{ 
        \frac{ n T }{ N }
      } 
    \big\|^p
  \Bigg]
  < \infty .
\end{equation}
Then
$
  \E\big[
  \sup_{ t \in [0,T] } \| X_t \|^p
  \big]
  < \infty
$
and
\begin{equation}
\label{eq:cor4b}
  \lim_{ N \rightarrow \infty }
    \mathbb{E}\Bigg[
      \sup_{ t \in [0,T] }
      \big\| 
        X_t 
        - 
        \bar{Y}^N_t
      \big\|^q
    \Bigg]
= 
  0 
\end{equation}
for all $ q \in (0,p) $.
\end{cor}

The proof of 
Corollary~\ref{cor:convergence4}
is analogous to the proof of
Corollary~\ref{cor:convergence1}
and therefore omitted.
Corollary~\ref{cor:convergence4}
extends Theorem~2.2
of Higham, Mao \citationand\ 
Stuart~\cite{hms02}
in several ways. First, the moment bound
condition on the exact solution in Assumption~2.1
in \cite{hms02} is omitted
in Corollary~\ref{cor:convergence4}.
This assumption can be omitted
in Corollary~\ref{cor:convergence4}
since \eqref{eq:cor4a}
and
Lemma~\ref{lem:fatou}
imply 
$ 
  \E\big[ 
    \sup_{ t \in [0,T] } \| X_t \|^p
  \big]
  < \infty
$.
Moreover, Theorem~2.2 in \cite{hms02}
proves uniform
strong mean square convergence
for the Euler-Maruyama scheme while
Corollary~\ref{cor:convergence4}
proves uniform
strong $ L^q $-convergence
with $ q \in (0,p) $
for 
$ (\mu,\sigma) 
$-consistent 
one-step schemes of the
form \eqref{eq:recX}.

\subsection{Strong convergence 
based on semi stability}
\label{sec:strongcsemi1}

Roughly speaking, the next 
corollary asserts
that semi stability 
with respect to 
Brownian motion
(see Definition~\ref{def:SVstability})
together 
with
consistency
with respect to Brownian motion
(see
Definition~\ref{def:consistent})
and with the a priori growth bound
\eqref{eq:a.priori.growth.bound}
implies strong convergence
of the numerical approximation 
processes~\eqref{eq:recX}
to the solution of the
SDE~\eqref{eq:SDE}.
Its proof is a direct consequence of
Corollary~\ref{cor:Semi.Stability}
and of Corollary~\ref{cor:convergence2}
and is therefore omitted.

\begin{cor}
\label{cor:strongConvergence}
Assume that the setting
in Section~\ref{sec:convergencesetting}
is 
fulfilled,
let 
$ \alpha, r \in (1,\infty] $,
$ p \in (0,\infty) $, 
$ \theta \in (0,T] $,
let 
$ 
  V \colon \mathbb{R}^d
  \rightarrow [0,\infty)
$
be a Borel measurable
function
with
$
   \sup_{ 
      x \in \mathbb{R}^d
    }
    \frac{
    \| x \|^{ p }  
    }{
      1 + V(x)
    }
    < \infty
$
and
$
  \E[ V( X_0 ) ]
  < \infty
$,
assume that
$ 
  \phi \colon \mathbb{R}^d
  \times [0,T] \times
  \mathbb{R}^m \rightarrow 
  \mathbb{R}^d
$
is $ \left( \mu, \sigma \right) 
$-consistent with respect to 
Brownian motion,
assume that
\begin{equation}
    \mathbb{R}^d
    \times [0,\theta] \times
    \mathbb{R}^m
    \ni (x,t,y)  
    \mapsto
    x + \phi(t,x,y) \in \mathbb{R}^d
\end{equation}
is $ \alpha $-semi
$ V $-stable
with respect to Brownian
motion
and assume that
\begin{equation}  \label{eq:a.priori.growth.bound}
  \limsup_{ N \rightarrow \infty }
  \sup_{ n \in \{ 0, 1, \dots, N \} }
  \left(
  N^{ 
    ( 1 - \alpha ) ( 1 / p - 1 / ( p r ) )
  }
  \|
    \bar{Y}^N_{ n T / N }
  \|_{ 
    L^{ p r }( \Omega; \mathbb{R}^d ) 
  }
  \right)
  < \infty.
\end{equation}
Then
$
  \limsup_{ N \to \infty }
  \sup_{ t \in [0,T] }
    \mathbb{E}\big[
      \| 
        X_t
      \|^p
      +
      \| \bar{Y}^N_t 
      \|^{ 
        p
      }
    \big] 
  < \infty
$
and
\begin{equation}
  \lim_{ N \rightarrow \infty }
    \sup_{ t \in [0,T] }
    \mathbb{E}\big[
      \| 
        X_t - \bar{Y}^N_t 
      \|^{ 
        q
      }
    \big] 
  = 0
\end{equation}
for all 
$
  q \in (0,p)
$.
\end{cor}

\subsection{Strong convergence
of an increment-tamed
Euler-Maruyama\\ scheme}
\label{sec:strongCtamed}

In Subsection~\ref{sec:strongCbased},
strong convergence of numerical
methods of the form \eqref{eq:recX}
has been proved under the
assumption of suitable moment
bounds
for the numerical
approximation
processes~\eqref{eq:recX}.
The next result proves
strong convergence of an
increment-tamed Euler
method and imposes appropriate
assumptions on the coefficients
$ \mu $ and $ \sigma $ of
the SDE~\eqref{eq:SDE}.

\begin{theorem}[Strong convergence
of an
increment-tamed 
Euler-Maruyama
scheme]
\label{thm:strongC}
Assume that the setting
in Section~\ref{sec:convergencesetting}
is fulfilled,
assume
$
  \E\big[
    \| X_0 \|^r
  \big]
  < \infty
$
for all $ r \in [0,\infty) $,
let
$
  p \in [3,\infty)
$,
$
  c,
  \gamma_0,
  \gamma_1
  \in [0,\infty)
$,
let 
$ 
\bar{\mu} \colon \mathbb{R}^d
\rightarrow \mathbb{R}^d 
$,
$
  \bar{\sigma} 
  \colon \mathbb{R}^d
  \rightarrow \mathbb{R}^{ d \times m }
$
be Borel measurable functions 
with
$
  \bar{\mu}|_D = \mu
$
and
$
  \bar{\sigma}|_D = \sigma
$,
let
$ 
  V \in C^3_p( \mathbb{R}^d, [1,\infty) )
$
with 
\begin{equation}
  ( \mathcal{G}_{ 
    \bar{\mu}, \bar{\sigma} 
  } V)(x)
\leq
  c \cdot V(x) ,
\;
  \left\|
    \bar{\mu}(x)
  \right\|
  \leq
  c \,
  |
    V(x)
  |^{
    \left[
       \frac{ \gamma_0 + 1 }{ p }
    \right]
  } ,
\;
  \|
    \bar{\sigma}(x)
  \|_{
    L( \mathbb{R}^m, \mathbb{R}^d )
  }
  \leq
  c \,
  |
    V(x)
  |^{ 
    \left[
      \frac{ \gamma_1 + 2 }{ 2p } 
    \right]
  } 
\end{equation}
for all $ x \in \mathbb{R}^d $
and let
$
  \bar{Y}^N \colon 
  [0,T] \times
  \Omega
  \rightarrow \mathbb{R}^d
$,
$ N \in \mathbb{N} 
$,
satisfy
\begin{equation}
\label{eq:deftamedBAR}
  \bar{Y}^N_{ t }
=
  \bar{Y}^N_{ 
    \frac{ n T }{ N }
  }
+
  \frac{
  \left(
    \tfrac{ t N }{ T } - n
  \right) 
    \left(
    \bar{\mu}( 
      \bar{Y}^N_{ n T / N } 
    ) 
    \frac{ T }{ N }
    + 
    \bar{\sigma}( 
      \bar{Y}^N_{ n T / N } 
    ) 
    (
      W_{ (n + 1) T / N } 
      -
      W_{ n T / N } 
    )
    \right)
  }{
    \max\!\big( 1, 
      \frac{ T }{ N }
      \| 
        \bar{\mu}( 
          \bar{Y}^N_{ n T / N } 
        ) 
        \frac{ T }{ N }
        + 
        \bar{\sigma}( 
          \bar{Y}^N_{ n T / N } 
        ) 
        (
          W_{ (n + 1) T / N } 
          -
          W_{ n T / N } 
        )
      \| 
    \big)
  }
\end{equation}
for all 
$
  t \in 
  (
    n T / N,
    (n + 1) T / N 
  ]
$,
$ 
  n \in \{ 0, 1, \dots, N - 1 \} 
$
and all
$ N \in \mathbb{N} $.
Then
$
  \sup_{ t \in [0,T] }
  \E\big[
    \| X_t \|^q
  \big] < \infty
$
and
\begin{equation}
  \lim_{ N \rightarrow \infty }
    \sup_{ t \in [0,T] }
    \mathbb{E}\big[
      \| 
        X_t - \bar{Y}^N_t 
      \|^{ 
        q
      }
    \big] 
  = 0
\end{equation}
for all 
$
  q \in (0,\infty)
$
with
$
  q 
  <
  \frac{
      p 
    }{
      2 \gamma_1 + 
      4 
      \max( 
        \gamma_0, \gamma_1, 1/2
      )
    }
  -
  \frac{ 1 }{ 2 }
$
and
$
  \limsup_{ r \searrow q }
   \sup_{ 
      x \in \mathbb{R}^d
    }
    \| x \|^{ r }  
    / V(x)
$
$
    < \infty
$.
\end{theorem}

Theorem~\ref{thm:strongC}
is a direct consequence
of
Corollary~\ref{cor:apriori_increment},
Corollary~\ref{cor:convergence2}
and Lemma \ref{lem:eta_phi}
(see \eqref{eq:incrementtamed3}).
The next result is a special case 
of Theorem~\ref{thm:strongC}.

\begin{cor}[Powers of
the Lyapunov-type function]
\label{cor:PotencyLyap}
Assume that the setting
in Section~\ref{sec:convergencesetting}
is fulfilled, 
assume
$
  \E\big[
    \| X_0 \|^r
  \big]
  < \infty
$
for all $ r \in [0,\infty) $,
let
$
  p \in [3,\infty)
$,
$
  q \in [1,\infty)
$,
$
  c,
  \gamma_0,
  \gamma_1
  \in [0,\infty)
$,
let 
$ 
\bar{\mu} \colon \mathbb{R}^d
\rightarrow \mathbb{R}^d 
$,
$
  \bar{\sigma} 
  \colon \mathbb{R}^d
  \rightarrow \mathbb{R}^{ d \times m }
$
be Borel measurable functions 
with
$
  \bar{\mu}|_D = \mu
$
and
$
  \bar{\sigma}|_D = \sigma
$,
let
$ 
  V \in 
  C^3_p( \mathbb{R}^d, [1,\infty) )
$
with 
$
  \left\|
    \bar{\mu}(x)
  \right\|
  \leq
  c \,
  |
    V(x)
  |^{
    \left[
      \frac{ \gamma_0 + 1 }{ p }
    \right]
  }
$,
$
  \|
    \bar{\sigma}(x)
  \|_{
    L( \mathbb{R}^m, \mathbb{R}^d )
  }
  \leq
  c \,
  |
    V(x)
  |^{ 
    \left[
      \frac{ \gamma_1 + 2 }{ 2p }
    \right]
  } 
$
and
\begin{equation}
  ( \mathcal{G}_{ 
      \bar{\mu}, 
      \bar{\sigma} 
    } 
  V)(x)
  + 
  \frac{
    \left( q - 1 \right)
  }{
    2 V(x)
  }
    \| 
      V'(x) \bar{\sigma}(x) 
    \|^2_{ 
      HS( \mathbb{R}^m, \mathbb{R} ) 
    }
  \leq c \cdot V(x) 
\end{equation}
for  
all
$ x \in \mathbb{R}^d $
and let
$
  \bar{Y}^N \colon 
  [0,T] \times
  \Omega
  \rightarrow \mathbb{R}^d
$,
$ N \in \mathbb{N} 
$,
satisfy
\eqref{eq:deftamedBAR}.
Then 
$
  \sup_{ t \in [0,T] }
  \E\big[
    \| X_t \|^r
  \big] < \infty
$
and
\begin{equation}
  \lim_{ N \rightarrow \infty }
    \sup_{ t \in [0,T] }
    \mathbb{E}\big[
      \| 
        X_t - \bar{Y}^N_t 
      \|^{ 
        r
      }
    \big] 
  = 0
\end{equation}
for all 
$
  r \in (0,\infty)
$
with
$
  r
  <
  \frac{
      p q 
    }{
      2 \gamma_1 + 
      4 
      \max( 
        \gamma_0, \gamma_1, 1/2
      )
    }
$
$
  -
  \frac{ 1 }{ 2 }
$
and
$
  \limsup_{ v \searrow r }
   \sup_{ 
      x \in \mathbb{R}^d
    }
    \frac{
    \| x \|^{ v }  
    }{
    V(x)
    }
    < \infty
$.
\end{cor}

Corollary~\ref{cor:PotencyLyap}
follows immediately
Corollary~\ref{cor:qMoments},
Corollary~\ref{cor:convergence2}
and Lemma \ref{lem:eta_phi} (see 
\eqref{eq:incrementtamed3}).
The next corollary gives sufficient conditions
for strong 
$ 
  L^q
$-convergence 
of the increment-tamed
Euler-Maruyama 
method~\eqref{eq:deftamedBAR}
for all $ q \in (0,\infty) $.

\begin{cor}
\label{cor:CT1}
Assume that the setting
in Section~\ref{sec:convergencesetting}
is fulfilled, 
assume
$
  \E\!\big[
    \| X_0 \|^r
  \big]
  < \infty
$
for all $ r \in [0,\infty) $,
let
$ c \in (0,\infty) $,
let 
$ 
\bar{\mu} \colon \mathbb{R}^d
\rightarrow \mathbb{R}^d 
$,
$
  \bar{\sigma} 
  \colon \mathbb{R}^d
  \rightarrow \mathbb{R}^{ d \times m }
$
be Borel measurable functions 
with
$
  \bar{\mu}|_D = \mu
$
and
$
  \bar{\sigma}|_D = \sigma
$, 
let
$ 
  V \colon \mathbb{R}^d
  \rightarrow [1,\infty)
$
be a twice differentiable function 
with a locally Lipschitz continuous
second derivative,
with
$
  \limsup_{ 
    q \searrow 0
  }
  \sup_{ 
    x \in \mathbb{R}^d
  }
  \frac{ 
    \| x \|^{ q } 
  }{
    V(x)
  }
  < \infty
$,
with
$
  \sum_{ i = 1 }^{ 3 }
  \|
    V^{(i)}( x )
  \|_{
    L^{ (i) }( 
      \mathbb{R}^d, \mathbb{R} 
    )
  }
  \leq
  c \,
  |
    V(x)
  |^{ 
    [ 1 - 1 / c ]
  }
$
for 
$ \lambda_{ \mathbb{R}^d } $-almost
all 
$ x \in \mathbb{R}^d $
and with
\begin{equation}
  \sup_{ x \in \mathbb{R}^d }
  \left[
  \frac{
    ( \mathcal{G}_{ \bar{\mu}, \bar{\sigma} 
    } V)(x)
  }{ V(x) }
  +
  \frac{
   r \,
    \| 
      V'(x) \bar{\sigma}(x) 
    \|_{ 
      L( \mathbb{R}^m, \mathbb{R} ) 
    }^2
  }{
    \left| V(x) \right|^2
  }
  \right]
  < \infty ,
\end{equation}
\begin{equation}
  \sup_{ x \in \mathbb{R}^d }
  \left[
  \frac{
    \|
      \bar{\mu}(x)
    \|
    +
    \|
      \bar{\sigma}(x)
    \|_{
      L( \mathbb{R}^m, \mathbb{R}^d )
    }
  }{
    \left(
      1 + \| x \|^{ c }
    \right)
  }
  \right]
  < \infty 
\end{equation}
for all $ r \in [0,\infty) $
and let
$
  \bar{Y}^N \colon 
  [0,T] \times
  \Omega
  \rightarrow \mathbb{R}^d
$,
$ N \in \mathbb{N} 
$,
satisfy
\eqref{eq:deftamedBAR}.
Then 
$
  \sup_{ t \in [0,T] }
  \E\big[
    \| X_t \|^q
  \big] < \infty
$
and
$
  \lim_{ N \rightarrow \infty }
    \sup_{ t \in [0,T] }
    \mathbb{E}\big[
      \| 
        X_t - \bar{Y}^N_t 
      \|^{ 
        q
      }
    \big] 
  = 0
$
for all 
$
  q \in (0,\infty)
$.
\end{cor}

Corollary~\ref{cor:CT1}
follows directly from
Corollary~\ref{cor:PotencyLyap}.
The next
corollary
specializes
Theorem~\ref{thm:strongC}
to the function
$
  1 + \| x \|^p
$,
$ x \in \R^d $,
with $ p \in [3,\infty) $
appropriate
as Lyapunov-type function
and
is the counterpart
to Corollary~\ref{cor:SVstabilityT2}.
It is an immediate consequence
of Corollary~\ref{cor:SVstabilityT2},
Corollary~\ref{cor:convergence2}
and Lemma~\ref{lem:eta_phi}
(see \eqref{eq:incrementtamed3}).

\begin{cor}[A special polynomial
like Lyapunov-type function]
\label{cor:CT2}
Assume that the setting
in Section~\ref{sec:convergencesetting}
is fulfilled, 
assume
$
  \E\big[
    \| X_0 \|^r
  \big]
  < \infty
$
for all $ r \in [0,\infty) $,
let
$ 
  c,
  \gamma_0,
  \gamma_1
  \in [0,\infty) 
$,
$ p \in [3,\infty) $,
let 
$ 
\bar{\mu} \colon \mathbb{R}^d
\rightarrow \mathbb{R}^d 
$,
$
  \bar{\sigma} 
  \colon \mathbb{R}^d
  \rightarrow \mathbb{R}^{ d \times m }
$
be Borel measurable functions 
with
$
  \bar{\mu}|_D = \mu
$,
$
  \bar{\sigma}|_D = \sigma
$ and
\begin{equation}
  \left< x, \bar{\mu}(x) \right>
  + 
  \tfrac{ \left( p - 1 \right) }{ 2 }
  \| \bar{\sigma}(x) 
  \|^2_{
    HS( \mathbb{R}^m, \mathbb{R}^d )
  }
  \leq 
  c \left(
    1 + \left\| x \right\|^2
  \right) ,
\end{equation}
\begin{equation}
  \left\|
    \bar{\mu}(x)
  \right\|
  \leq
  c \,
  \big( 
    1 + 
    \left\| x \right\|^{ 
      \left[
         \gamma_0 + 1
      \right]
    }
  \big)
\quad
  \text{and}
\quad
  \| 
    \bar{\sigma}(x) 
  \|_{ 
    L( \mathbb{R}^m , \mathbb{R}^d ) 
  }
  \leq
  c \,
  \big( 
    1 + 
    \left\| x \right\|^{ 
      \left[
         \frac{ \gamma_1 + 2 }{ 2 }
      \right]
    }
  \big)
\end{equation}
for all
$ 
  x \in \mathbb{R}^d  
$
and let
$
  \bar{Y}^N \colon 
  [0,T] \times
  \Omega
  \rightarrow \mathbb{R}^d
$,
$ N \in \mathbb{N} 
$,
satisfy
\eqref{eq:deftamedBAR}.
Then 
$
  \sup_{ t \in [0,T] }
  \E\big[
    \| X_t \|^q
  \big] < \infty
$
and
\begin{equation}
  \lim_{ N \rightarrow \infty }
    \sup_{ t \in [0,T] }
    \mathbb{E}\big[
      \| 
        X_t - \bar{Y}^N_t 
      \|^{ 
        q
      }
    \big] 
  = 0
\end{equation}
for all 
$ 
  q 
  \in 
  (0,\infty)
$
with
$
  q <
    \frac{
      p
    }{
      2 \gamma_1 + 
      4 
      \max( 
        \gamma_0, 
        \gamma_1,
        1/2
      )
    }
    -
    \frac{ 1 }{ 2 }
$.
\end{cor}

\section{Weak convergence}
\label{sec:weakconvergence}

Convergence in probability
implies stochastic weak convergence.
The next corollary is thus an immediate
consequence of 
Theorem~\ref{thm:convergence}.

\begin{cor}[Weak convergence
with bounded test functions]
\label{cor:weakconvergence}
Assume that the setting
in Section~\ref{sec:convergencesetting}
is fulfilled,
suppose that
$ 
  \phi \colon \mathbb{R}^d
  \times [0,T] \times \mathbb{R}^m
  \rightarrow \mathbb{R}^d
$ is
$ ( \mu, \sigma ) $-consistent with respect to Brownian motion
and let
$
  \left( 
    E, \left\| \cdot \right\|_E
  \right) 
$
be a separable 
$ \mathbb{R} $-Banach space.
Then
\begin{equation}
  \lim_{ N \rightarrow \infty }
  \left\|
  \mathbb{E}\big[
    f( X )
  \big]
  -
  \mathbb{E}\!\left[
    f( \bar{Y}^N )
  \right]
  \right\|_E
= 
  0 
\end{equation}
for all
bounded and continuous
functions 
$
  f \colon C( [0,T], \mathbb{R}^d )
  \rightarrow E
$.
\end{cor}

Corollary~\ref{cor:weakconvergence}
follows directly from 
Theorem~\ref{thm:convergence}
since the test functions 
in Corollary~\ref{cor:weakconvergence}
are assumed to be bounded.
The case of unbounded test functions
is more subtle and is analyzed in
the sequel. The next lemma
proves weak convergence 
restricted to events whose probabilities converge
to one sufficiently fast with
possibly unbounded test function
under the assumption that
the discrete-time
stochastic processes
$ \bar{Y}^N_{ n T / N } $,
$ n \in \{ 0, 1, \dots, N \} $,
$ N \in \mathbb{N} $,
are $ 0 $-semi 
$ V $-bounded
with 
$ V \colon \mathbb{R}^d
\rightarrow [0,\infty) $
appropriate.

\begin{lemma}[Semi weak convergence]
\label{lem:weakconvergence}
Assume that the setting
in Section~\ref{sec:convergencesetting}
is fulfilled,
suppose that
$ 
  \phi \colon \mathbb{R}^d
  \times [0,T] \times \mathbb{R}^m
  \rightarrow \mathbb{R}^d
$ is
$ ( \mu, \sigma ) $-consistent with respect to Brownian motion,
let
$
  \left( 
    E, \left\| \cdot \right\|_E
  \right) 
$
be a separable 
$ \mathbb{R} $-Banach space,
let
$
  V \colon \mathbb{R}^d
  \rightarrow [0,\infty)
$
be a continuous function
and assume that 
the sequence
$ 
  \bar{Y}^N_{ n T / N }   
  \colon
  \Omega
  \rightarrow \mathbb{R}^d
$,
$ n \in \{ 0, 1, \dots, N \} $,
$ N \in \mathbb{N} $,
of discrete-time 
stochastic processes is
$ 0 $-semi
$ V $-bounded, i.e., 
let
$ 
  \Omega_N \in 
  \sigma_{ \Omega}( \bar{Y}^N ) 
$, 
$ N \in \mathbb{N} $,
be a sequence of events such that
\begin{equation}
\label{eq:semiVbounded.semiweak}
  \limsup_{ N \to \infty }
  \sup_{ n \in \{ 0, 1, \dots, N \} }
    \mathbb{E}\big[
      \mathbbm{1}_{ \Omega_N }
      V\big( 
        \bar{Y}^N_{ n T / N } 
      \big)
    \big]
  < \infty
\quad
  \text{and}
\quad
  \lim_{ N \to \infty }
    \mathbb{P}\big[
      \Omega_N 
    \big]
  = 1 .
\end{equation}
Then 
$
  \E\big[
    V( X_T ) 
    +
    \| f( X_T ) \|_E
  \big]
  < \infty
$
and
\begin{equation}
\label{eq:weakconvergence}
\begin{split}
&
  \lim_{ N \rightarrow \infty }
  \mathbb{E}\!\left[
    \mathbbm{1}_{ \Omega_N }
    \|
    f( X_T )
    -
    f( \bar{Y}^N_T )
    \|_E
  \right]
\\ & =
  \lim_{ N \rightarrow \infty }
  \left\|
  \mathbb{E}\big[
    \mathbbm{1}_{ \Omega_N }
    f( X_T )
  \big]
  -
  \mathbb{E}\!\left[
    \mathbbm{1}_{ \Omega_N }
    f( \bar{Y}^N_T )
  \right]
  \right\|_E
= 
  0 
\end{split}
\end{equation}
for all continuous functions
$
  f \colon \mathbb{R}^d 
  \rightarrow E
$
with
$
  \limsup_{ r \nearrow 1 }
  \sup_{ x \in \mathbb{R}^d }
  \frac{
    \| f(x) \|_E
  }{
    ( 1 + V(x) )^r
  }
  < \infty
$.
\end{lemma}

\begin{proof}[Proof
of 
Lemma~\ref{lem:weakconvergence}]
First of all, let
$
  f \colon \mathbb{R}^d 
  \rightarrow E
$
be a continuous function
which satisfies
\begin{equation}
  \limsup_{ r \nearrow 1 }
  \sup_{ x \in \mathbb{R}^d }
  \frac{
    \| f(x) \|_E
  }{
    ( 1 + V(x) )^r
  }
  < \infty
  .
\end{equation}
This ensures that
there exists
a real number $ r \in (0,1) $
such that
$
  \eta :=
$
$
  \sup_{ x \in \mathbb{R}^d }
$
$
  \frac{
    \| f(x) \|_E
  }{
    ( 1 + V(x) )^r
  }
  < \infty
$.
We thus obtain that
\begin{equation}
\label{eq:fest}
  \left\|
    f(x)
  \right\|_E
\leq
  \eta 
  \left(
    1 + V(x)
  \right)^r
\end{equation}
for all $ x \in \mathbb{R}^d $.
Next observe that
Theorem~\ref{thm:convergence}
and 
assumption~\eqref{eq:semiVbounded.semiweak}
show
\begin{equation}
\begin{split}
&
  \limsup_{ N \to \infty }
  \P\!\left[
    \| 
      X_T 
      - \mathbbm{1}_{ \Omega_N }
      \bar{Y}^N_T
    \|
    \geq \varepsilon
  \right]
\\ & \leq
  \limsup_{ N \to \infty }
  \P\big[
    \| 
      X_T 
      -
      \mathbbm{1}_{ \Omega_N }
      X_T 
    \|
    \geq \tfrac{ \varepsilon }{ 2 }
  \big]  
  +
  \limsup_{ N \to \infty }
  \P\!\left[
    \mathbbm{1}_{ \Omega_N }
    \| 
      X_T 
      - 
      \bar{Y}^N_T 
    \|
    \geq \tfrac{ \varepsilon }{ 2 }
  \right]  
\\ & \leq
  \limsup_{ N \to \infty }
  \P\big[
    \mathbbm{1}_{ ( \Omega_N )^c }
    \| 
      X_T 
    \|
    \geq \tfrac{ \varepsilon }{ 2 }
  \big]  
  +
  \limsup_{ N \to \infty }
  \P\!\left[
    \| 
      X_T 
      - 
      \bar{Y}^N_T 
    \|
    \geq \tfrac{ \varepsilon }{ 2 }
  \right]  
\\ & \leq
  \limsup_{ N \to \infty }
  \P\big[
    ( \Omega_N )^c
  \big]
  = 0
\end{split}
\end{equation}
for all $ \varepsilon \in (0,\infty) $.
Lemma~\ref{lem:fatou}
hence implies that
\begin{equation}
\begin{split}
  \E\big[
    V( X_T )
  \big]
& \leq
  \liminf_{ N \to \infty }
  \E\big[
    V\big( 
      \mathbbm{1}_{ \Omega_N }
      \bar{Y}^N_T
    \big)
  \big]
\\ & \leq
  V(0) 
  +
  \limsup_{ N \to \infty }
  \E\big[
    \mathbbm{1}_{ \Omega_N }
    V( 
      \bar{Y}^N_T
    )
  \big]
< \infty 
\end{split}
\end{equation}
and
Jensen's inequality 
and \eqref{eq:fest}
therefore yield that
\begin{equation}
\label{eq:fxfinite}
\begin{split}
  \mathbb{E}\big[
    \| f( X_T ) \|_E
  \big]
& \leq
  \eta
  \cdot
  \mathbb{E}\Big[
    \big( 
      1 +
      V( X_T )
    \big)^r
  \Big]
  \leq
  \eta
  \cdot
    \big( 
      1 +
      \E[
        V( X_T )
      ]
    \big)^r
< \infty .
\end{split}
\end{equation}
In addition,
observe that 
Theorem~\ref{thm:convergence}
implies 
\begin{equation}
  \lim_{ N \rightarrow \infty }
  \mathbb{P}\!\left[
    \|
    f( X_T )
    -
    f( \bar{Y}^N_T )
    \|_E
    > \varepsilon
  \right] = 0
\end{equation}
for all 
$ \varepsilon \in (0,\infty) $.
Therefore, we obtain that
\begin{equation}
\label{eq:omegaProbC}
  \lim_{ N \rightarrow \infty }
  \mathbb{P}\!\left[
    \mathbbm{1}_{ \Omega_N }
    \|
    f( X_T )
    -
    f( \bar{Y}^N_T )
    \|_E
    > \varepsilon
  \right] = 0
\end{equation}
for all 
$ \varepsilon \in (0,\infty) $.
Next note that
estimate~\eqref{eq:fest} 
ensures that
\begin{equation}
\begin{split}
  \limsup_{ N \to \infty }
  \left\| 
    \mathbbm{1}_{ \Omega_N }
    f( \bar{Y}^N_T )
  \right\|_{ L^{ 1 / r }( \Omega; E ) }
& \leq
  \eta \cdot
  \left(
    1 +
    \limsup_{ N \to \infty }
    \mathbb{E}\Big[ 
      \mathbbm{1}_{ \Omega_N }
      V\big( \bar{Y}^N_T \big)
    \Big]
  \right)^{ \! r }
< \infty .
\end{split}
\end{equation}
This implies that there exists
a natural number $ N_0 \in \N $
such that
\begin{equation}
\label{eq:uniformintegrable}
  \sup_{
    N \in \{ N_0, N_0 + 1, \dots \}
  }
  \left\| 
    \mathbbm{1}_{ \Omega_N }
    f( \bar{Y}^N_T )
  \right\|_{ L^{ 1 / r }( \Omega; E ) }
  < \infty .
\end{equation}
Inequality~\eqref{eq:uniformintegrable},
the fact that
$ \frac{ 1 }{ r } > 1 $
and Corollary 6.21 of Klenke~\cite{k08b}
show that the family 
$
  \mathbbm{1}_{ \Omega_N }
  \| 
    f( \bar{Y}^N_T )
  \|_E
$,
$ 
  N \in \{ N_0, N_0 + 1, \dots \} 
$,
of 
$ 
  \mathcal{F} 
$/$
  \mathcal{B}( \mathbb{R} )
$-measurable mappings
is uniformly integrable.
This together with 
\eqref{eq:fxfinite}
yields that the sequence
$
  \mathbbm{1}_{ \Omega_N }
  \| 
    f( X_T ) 
    -
    f( \bar{Y}^N_T )
  \|_E
$,
$ 
  N \in \{ N_0, N_0 + 1, \dots \} 
$,
is uniformly integrable.
Equation~\eqref{eq:omegaProbC}
hence shows
\eqref{eq:weakconvergence}.
The proof of 
Lemma~\ref{lem:weakconvergence}
is thus completed.
\end{proof}

\subsection{Convergence
of Monte Carlo methods}

In \eqref{eq:weakconvergence}
in Lemma~\ref{lem:weakconvergence},
the convergence holds restricted
to a sequence of events whose
probabilities converge to one.
In Proposition~\ref{p:weakconvergence} below,
we will get rid of the
restriction to these events 
and prove convergence of the
corresponding Monte Carlo method
on an event of probability one.
For proving this result,
we first need
a minor generalization of 
Lemma~2.1 in
Kloeden \citationand\ Neuenkirch~\cite{kn07}.
More precisely,
Lemma~2.1
in Kloeden \citationand\ Neuenkirch~\cite{kn07}
proves that convergence in the
$ p $-th mean with order 
$ \beta \in (0,\infty) $ 
for all $ p \in (0,\infty) $
implies almost sure convergence
with order 
$ \beta - \varepsilon $
where
$ \varepsilon \in (0,\infty) $
is arbitrarily small.
The next result is a minor generalization
of this result and, in particular, 
proves 
for every arbitrary fixed $ p \in (0,\infty) $
that convergence in the 
$ p $-th mean with order 
$ \beta \in (\frac{1}{p},\infty) $
implies almost sure convergence
with order 
$ \beta - 1/p - \varepsilon $
for every arbitrarily small
$ \varepsilon \in (0,\infty) $.

\begin{lemma}[$ L^p $-convergence
with order 
$ 
  \beta 
  \in (1/p,\infty) 
$ 
implies almost sure convergence]
\label{lem:almostsure}
Let 
$ M \in \N $,
let
$
  \left( E, \mathcal{E}, \mu \right)
$
be a measurable space and let
$
  Y_N \colon E
  \rightarrow \mathbb{R}
$,
$ N \in \{ M, M+1, \dots \} $,
be a family
of 
$ \mathcal{E} 
$/$ \mathcal{B}( \mathbb{R} ) 
$-measurable mappings.
Then
\begin{equation}
\label{eq:almostsure}
\begin{split}
&
  \left\|
    \sup_{ 
      N \in \{ M, M + 1, \dots \} 
    }
      N^{ \alpha } 
      \cdot
      \left| Y_N \right|
  \right\|_{ 
    L^p( E; \mathbb{R} )
  }
\\ & \leq
  \left[
    \sum_{ N = M }^{ \infty }
    N^{ \left( \alpha - \beta\right) p } 
  \right]^{ \! 1 / p }
  \left[
    \sup_{ 
      N \in \{ M, M + 1, \dots \} 
    }
      N^{ \beta }
      \cdot
      \left\| Y_N \right\|_{ 
        L^p( E; \mathbb{R} ) 
      }
  \right]
\end{split}
\end{equation}
for all 
$ 
  \alpha, \beta \in \mathbb{R}
$,
$ 
  p \in (0,\infty) 
$.
In particular, if
$
  \sup_{ 
    N \in \{ M, M + 1, \dots \} 
  }
  (
    N^{ \beta } \,
    \| Y_N \|_{ 
      L^p( E; \mathbb{R} ) 
    }
  )
  < \infty
$
for one
$ p \in (0,\infty) $
and one
$
  \beta \in \mathbb{R}
$,
then
\begin{equation}
\label{eq:almostsure2}
  \int_E
    \left\{
    \sup_{ N \in \{ M, M+1, \dots \} }
    (
      N^{ \alpha } \cdot
      | Y_N |
    )
    \right\}^{ \! p }
  d \mu < \infty ,
\;
    \sup_{ N \in \{ M, M+1, \dots \} }
    \big(
      N^{ \alpha } \cdot
      | Y_N |
    \big)
    < \infty
  \;\;
  \mu\text{-a.s.}
\end{equation}
for all
$ \alpha \in (-\infty,\beta-\frac{1}{p}) $.
Moreover, if 
$
  \sup_{ 
    N \in \{ M, M + 1, \dots \} 
  }
  (
    N^{ \beta } \,
    \| Y_N \|_{ 
      L^p( E; \mathbb{R} ) 
    }
  )
  < \infty
$
for all
$ p \in (0,\infty) $
and one 
$
  \beta \in \mathbb{R}
$,
then
\begin{equation}
\label{eq:almostsure3}
  \int_E
    \left\{
    \sup_{ 
      N \in \{ M, M + 1, \dots \} 
    }
    (
      N^{ \alpha } \cdot
      | Y_N |
    )
    \right\}^{ \! p }
  d \mu < \infty ,
\;
    \sup_{ 
      N \in \{ M, M + 1, \dots \} 
    }
    \big(
      N^{ \alpha } \cdot
      | Y_N |
    \big)
    < \infty
  \;\;
  \mu\text{-a.s.}
\end{equation}
for all
$ \alpha \in (-\infty,\beta) $,
$ p \in (0,\infty) $.
\end{lemma}

\begin{proof}[Proof
of Lemma~\ref{lem:almostsure}]
Note that
\begin{equation}
\begin{split}
&
  \int_{ E }
    \left\{
      \sup_{ 
        N \in \{ M, M + 1, \dots \} 
      }
      \left(
        N^{ \alpha } \cdot
        \left| Y_N \right|
      \right)
    \right\}^{ \! p }
  d \mu
\\ &  =
  \int_{ E }
  \left\{
    \sup_{ 
      N \in \{ M, M + 1, \dots \} 
    }
    \left(
      N^{ \alpha p } \cdot
      \left| Y_N \right|^p
    \right)
  \right\}
  d \mu
\\ & \leq
  \int_{ E }
  \left\{
    \sum_{ N = M }^{ \infty }
    \left(
      N^{ \alpha p } \cdot
      \left| Y_N \right|^p
    \right)
  \right\}
  d \mu
  =
  \sum_{ N = M }^{ \infty }
    \left(
      N^{ \alpha p } 
      \cdot
      \int_{ E }
        \left| Y_N \right|^p
      d \mu
    \right)
\\ & \leq
  \left(
    \sum_{ N = M }^{ \infty }
    N^{ \left( \alpha - \beta\right) p } 
  \right)
  \left(
  \sup_{ 
    N \in \{ M, M + 1, \dots \} 
  }
    N^{ \beta p }
    \cdot
    \int_{ E }
      \left| Y_N \right|^p
    d \mu
  \right)
\end{split}
\end{equation}
for all $ \alpha, \beta \in \mathbb{R} $
and all $ p \in (0,\infty) $.
This proves 
inequality~\eqref{eq:almostsure}.
The assertions
in \eqref{eq:almostsure2}
and \eqref{eq:almostsure3}
then follow immediately from
inequality~\eqref{eq:almostsure}.
The proof of Lemma~\ref{lem:almostsure}
is thus completed.
\end{proof}

Note also that the last assertion in Lemma~\ref{lem:almostsure},
i.e., the assertion in \eqref{eq:almostsure3},
is essentially 
Lemma~2.1 in \cite{kn07}
generalized to arbitrary measure
spaces.

\begin{prop}[Convergence
of Monte Carlo methods]
\label{p:weakconvergence}
Assume that the setting
in Section~\ref{sec:convergencesetting}
is fulfilled,
suppose that
$ 
  \phi \colon \mathbb{R}^d
  \times [0,T] \times \mathbb{R}^m
  \rightarrow \mathbb{R}^d
$ is
$ ( \mu, \sigma ) $-consistent with respect to Brownian motion,
let 
$ n \in \mathbb{N} $,
$ \alpha \in (n+1,\infty) $,
let
$
  V \colon \mathbb{R}^d 
  \rightarrow [0,\infty)
$
be a continuous function,
assume that 
the sequence
$ 
  (
    \bar{Y}^N_{ k T / N }   
  )_{
    k \in \{ 0, 1, \dots, N \} 
  }
$,
$ N \in \mathbb{N} $,
of discrete-time 
stochastic processes is
$ \alpha $-semi
$ V $-bounded
and for every $ N \in \mathbb{N} $
let
$
  \bar{Y}^{ N, m }
  \colon [0,T] \times
  \Omega \rightarrow \mathbb{R}^d
$,
$ m \in \mathbb{N} $,
be independent
stochastic processes with
$
  \mathbb{P}_{ \bar{Y}^N }
=
  \mathbb{P}_{ \bar{Y}^{N,m} }
$
for all $ m \in \mathbb{N} $.
Then 
\begin{equation}
\label{eq:weakfinal}
  \lim_{ N \rightarrow \infty }
  \left|
  \mathbb{E}\big[
    f( X_T )
  \big]
  -
  \frac{
    \sum_{ m = 1 }^{ N^n }
    f( \bar{Y}^{ N, m }_T )
  }{
    N^n
  }
  \right|
= 
  0 
\end{equation}
$ \mathbb{P} $-a.s.\ for 
all
continuous
functions 
$
  f \colon \mathbb{R}^d 
  \rightarrow \mathbb{R}
$
with
$
  \limsup_{ 
    r \nearrow \frac{ n \wedge 2 }{ 2 } 
  }
  \sup_{ x \in \mathbb{R}^d }
  \frac{
    | f(x) |
  }{
    ( 1 + V(x) )^r
  }
$
$
  < \infty
$.
\end{prop}

\begin{proof}[Proof
of 
Proposition~\ref{p:weakconvergence}]
Let 
$
  f \colon \mathbb{R}^d 
  \rightarrow \mathbb{R}
$
be a continuous
function
which satisfies
\begin{equation}
\label{eq:f_assumption_prop}
  \limsup_{ 
    r \nearrow 
    \frac{ n \wedge 2 }{ 2 } 
  }
  \sup_{ x \in \mathbb{R}^d }
  \frac{
    | f(x) |
  }{
    ( 1 + V(x) )^r
  }
  < \infty
  .
\end{equation}
Next let 
$ M \in \N $
be a natural number and
let
$
  A_N \in 
  \mathcal{B}( \mathbb{R}^d )^{
    \otimes \{ 0, 1, \dots, N \}
  }
$,
$ N \in \mathbb{N} $,
be a sequence of sets
such that
\begin{equation}
\begin{split}
  c := 
&
  \sup_{ N \in \{ M, M + 1, \dots \} }
  \bigg(
    \mathbb{E}\Big[
      \mathbbm{1}_{   
        \left\{
          ( 
            \bar{Y}^N_{ k T / N }
          )_{ k \in \{ 0, 1, \dots, N \} }
          \in 
          A_N
        \right\}
      }
      V( \bar{Y}^N_T )
    \Big]
\\ & 
    +
    N^{ \alpha } \cdot
    \mathbb{P}\Big[
      ( 
        \bar{Y}^N_{ k T / N }
      )_{ k \in \{ 0, 1, \dots, N \} }
      \in 
      \left( A_N \right)^c
    \Big]
  \bigg)
  < \infty .
\end{split}
\end{equation}
Such a sequence of sets indeed
exists since
the sequence
$ 
  (
    \bar{Y}^N_{ k T / N }   
  )_{ k \in \{ 0, 1, \dots, N \} }
$,
$ N \in \mathbb{N} $,
of discrete-time 
stochastic processes 
is assumed to be
$ \alpha $-semi
$ V $-bounded.
We define now two families
$ \Omega_N \in \mathcal{F} $,
$ N \in \mathbb{N} $,
and
$ 
  \Omega_{N,m} \in \mathcal{F} 
$,
$ N, m \in \mathbb{N} $,
of events by
\begin{equation}
  \Omega_{ N } :=
  \Big\{
    ( 
      \bar{Y}^{ N }_{ k T / N } 
    )_{ 
      k \in \{ 0, 1, \dots, N \} 
    }
    \in A_N
  \Big\}
\quad
  \text{and}
\quad
  \Omega_{ N, m } :=
  \left\{
    ( 
      \bar{Y}^{ N, m }_{ k T / N } 
    )_{ 
      k \in \{ 0, 1, \dots, N \} 
    }
    \in A_N
  \right\}
\end{equation}
for all $ N, m \in \mathbb{N} $.
Note that
$
  \mathbb{P}\!\left[
    \Omega_N
  \right]  
  =
  \mathbb{P}\!\left[
    \Omega_{ N, m }
  \right]  
$
for all $ N, m \in \mathbb{N} $.
Moreover, observe that
Lemma~\ref{lem:weakconvergence}
implies
\begin{equation}
\label{eq:weakfinalZERO}
  \lim_{ N \rightarrow \infty }
  \mathbb{E}\big[
    \mathbbm{1}_{ 
      \left( \Omega_N \right)^c 
    }
    |
      f( X_T )
    |
  \big]
  = 0 
\quad
  \text{and}
\quad
  \lim_{ N \rightarrow \infty }
  \left|
  \mathbb{E}\big[
    \mathbbm{1}_{ \Omega_N }
    f( X_T )
  \big]
  -
  \mathbb{E}\!\left[
    \mathbbm{1}_{ \Omega_N }
    f( \bar{Y}^N_T )
  \right]
  \right|
= 
  0 .
\end{equation}
In the next step note that
condition \eqref{eq:f_assumption_prop}
shows the existence of a real
number 
$ 
  r \in 
  (0,
    \frac{ n \wedge 2 }{ 2 } 
  ) 
$
such that
\begin{equation}
  \eta :=
  \sup_{ x \in \mathbb{R}^d }
  \frac{
    | f(x) |
  }{
    ( 1 + V(x) )^r
  }
  < \infty
  .
\end{equation}
The Burkholder-Davis-Gundy
inequality 
(see, e.g., 
Theorem~48
in Protter~\cite{Protter2004})
then shows the
existence of a real
number $ \kappa \in [0,\infty) $
such that
\begin{equation}
\begin{split}
&
  \left\|
  \frac{ 
    \sum_{ m = 1 }^{ N^n }
    \big\{
      \mathbbm{1}_{ \Omega_{ N, m } }
      f( \bar{Y}^{ N, m}_T )
      -
      \mathbb{E}\big[
        \mathbbm{1}_{ \Omega_N }
        f( \bar{Y}^N_T )
      \big]
    \big\}
  }{
    N^n
  }
  \right\|_{ 
    L^{ 1 / r }( \Omega; \mathbb{R} ) 
  }
\\ & \leq
  \frac{ 
    \kappa
  }{
    N^n
  }
    \left(
      \sum_{ m = 1 }^{ N^n }
      \big\|
        \mathbbm{1}_{ \Omega_{ N, m } }
        f( \bar{Y}^{ N, m}_T )
        -
        \mathbb{E}\big[
        \mathbbm{1}_{ \Omega_N }
        f( \bar{Y}^N_T )
        \big]
      \big\|_{ 
        L^{ 1 / r }( \Omega; \mathbb{R} ) 
      }^2
    \right)^{ \! 1 / 2 }
\\ & \leq
  \frac{ 
    2 \kappa
      \big\|
        \mathbbm{1}_{ \Omega_N }
        f( \bar{Y}^N_T )
      \big\|_{ 
        L^{ 1 / r }( \Omega; \mathbb{R} ) 
      }
  }{
    \sqrt{ N^n }
  }
\leq
  \frac{ 
    2 \kappa \eta \,
      \big\|
        \mathbbm{1}_{ \Omega_N }
        \big| 
          1 + V( \bar{Y}^N_T ) 
        \big|^{ r }
      \big\|_{ 
        L^{ 1 / r }( \Omega; \mathbb{R} ) 
      }
  }{
    \sqrt{ N^n }
  }
\\ & =
  \frac{ 
    2 \kappa \eta 
    \left\{
      \mathbb{E}\big[
        \mathbbm{1}_{ \Omega_N }
        \big( 
          1 + V( \bar{Y}^N_T ) 
        \big)
      \big]
    \right\}^r
  }{
    N^{ n / 2 }
  }
\leq
  \frac{ 
    2 \kappa \eta \,
        \big( 
          1 + 
          \mathbb{E}\big[
            \mathbbm{1}_{ \Omega_N }
            V( \bar{Y}^N_T ) 
          \big]
        \big)^r
  }{
    N^{ n / 2 }
  }
\\ & \leq
  \frac{ 
    2 \kappa \eta 
        \left( 
          1 + 
          c
        \right)^r
  }{
    N^{ n / 2 }
  } 
\end{split}
\end{equation}
for all 
$ 
  N \in \{ M, M + 1, \dots \} 
$.
Therefore, we obtain
\begin{equation}
  \sup_{ N \in \{ M, M + 1, \dots \} }
  \left(
  N^{ \frac{ n }{ 2 } }
  \left\|
  \frac{ 
    \sum_{ m = 1 }^{ N^n }
    \big\{
      \mathbbm{1}_{ \Omega_{ N, m } }
      f( \bar{Y}^{ N, m}_T )
      -
      \mathbb{E}\big[
        \mathbbm{1}_{ \Omega_N }
        f( \bar{Y}^N_T )
      \big]
    \big\}
  }{
    N^n
  }
  \right\|_{ 
    L^{ 1 / r }( \Omega; \mathbb{R} ) 
  }
  \right)
  < \infty .
\end{equation}
The condition
$
  \frac{ n }{ 2 } > 
  r
$
and the second inequality
in \eqref{eq:almostsure2}
in Lemma~\ref{lem:almostsure}
hence yield
\begin{equation}
\label{eq:weakfinalONE}
  \lim_{ N \rightarrow \infty }
  \left(
  \frac{ 
    \sum_{ m = 1 }^{ N^n }
    \big\{
      \mathbbm{1}_{ \Omega_{ N, m } }
      f( \bar{Y}^{ N, m}_T )
      -
      \mathbb{E}\big[
        \mathbbm{1}_{ \Omega_N }
        f( \bar{Y}^N_T )
      \big]
    \big\}
  }{
    N^n
  }
  \right)
  = 0
\end{equation}
$ \mathbb{P} $-a.s..
Furthermore, observe that
\begin{equation}
\label{eq:weakfinalTWOa}
\begin{split}
&
  \lim_{ N \rightarrow \infty }
  \mathbbm{1}_{ 
    \left[
    \cup_{ K = N }^{ \infty }
    \cup_{ m = 1 }^{ K^n }
    ( \Omega_{ K, m } )^c 
    \right]
  }
  (\omega)
\\ & =
  \mathbbm{1}_{ 
    \left[
    \cap_{ N = 1 }^{ \infty }
    \cup_{ K = N }^{ \infty }
    \cup_{ m = 1 }^{ K^n }
    ( \Omega_{ K, m } )^c 
    \right]
  }
  (\omega)
 =
  \mathbbm{1}_{ 
    \left[
    \limsup_{ N \rightarrow \infty }
    \left(
    \cup_{ m = 1 }^{ N^n }
    ( \Omega_{ N, m } )^c 
    \right)
    \right]
  }
  (\omega)
\end{split}
\end{equation}
for all $ \omega \in \Omega $.
In addition, the condition
$ \alpha > n + 1 $
gives
\begin{equation}
  \sum_{ 
    N = M
  }^{ 
    \infty 
  }
  \mathbb{P}\!\left[
    \cup_{ m = 1 }^{ N^n }
    ( \Omega_{ N, m } )^c 
  \right]
\leq
  \sum_{ 
    N = M
  }^{ 
    \infty 
  }
  N^n \cdot
  \mathbb{P}\!\left[
    ( \Omega_{ N } )^c 
  \right]
\leq
  c
  \sum_{ N = M }^{ \infty }
  N^{ ( n - \alpha ) }
< \infty
\end{equation}
and the Borel-Cantelli Lemma
thus shows
\begin{equation}
\label{eq:weakfinalTWOb}
  \mathbb{P}\!\left[
    \limsup_{ N \rightarrow \infty }
    \left(
    \cup_{ m = 1 }^{ N^n }
    ( \Omega_{ N, m } )^c 
    \right)
  \right]
  = 0 .
\end{equation}
Combining 
\eqref{eq:weakfinalTWOa}
with the estimate
\begin{equation}
\begin{split}
&	
  \frac{
    \sum_{ m = 1 }^{ N^n }
    \mathbbm{1}_{ 
      ( \Omega_{ N, m } )^c 
    }
    |
      f( \bar{Y}^{ N, m }_T )
    |
  }{
    N^n
  }
\\ & \leq
  \frac{
    \left(
      \sum_{ m = 1 }^{ N^n }
      |
        f( \bar{Y}^{ N, m }_T )
      |
    \right)
    \left(
      \max_{ m \in \{ 1, 2, \dots, N^n \} }
      \mathbbm{1}_{ 
        ( \Omega_{ N, m } )^c 
      }
    \right)
  }{
    N^n
  }
\\ & \leq
  \frac{
    \left[
      \sum_{ m = 1 }^{ N^n }
      |
        f( \bar{Y}^{ N, m }_T )
      |
    \right]
    \mathbbm{1}_{ 
      \cup_{ K = N }^{ \infty }
      \cup_{ m = 1 }^{ K^n }
      ( \Omega_{ K, m } )^c 
    }
  }{
    N^n
  }
\end{split}
\end{equation}
for all $ N \in \mathbb{N} $ 
results in
\begin{equation}
  \lim_{ N \rightarrow \infty }
  \left(
  \frac{
    \sum_{ m = 1 }^{ N^n }
    \mathbbm{1}_{ 
      ( \Omega_{ N, m } )^c 
    }( \omega )
    \,
    |
      f( \bar{Y}^{ N, m }_T( \omega ) )
    |
  }{
    N^n
  }
  \right)
= 0
\end{equation}
for all
$
  \omega \in 
  \left[
    \cap_{ N = 1 }^{ \infty }
    \cup_{ K = N }^{ \infty }
    \cup_{ m = 1 }^{ K^n }
    ( \Omega_{ K, m } )^c
  \right]^c
  =
  \left[
    \limsup_{ N \to \infty }
    \left(
    \cup_{ m = 1 }^{ N^n }
      ( \Omega_{ N, m } )^c
    \right)
  \right]^c
$.
This and 
\eqref{eq:weakfinalTWOb}
then show
\begin{equation}
\label{eq:weakfinalTWO}
  \lim_{ N \rightarrow \infty }
  \left(
  \frac{
    \sum_{ m = 1 }^{ N^n }
    \mathbbm{1}_{ 
      ( \Omega_{ N, m } )^c 
    }
    |
      f( \bar{Y}^{ N, m }_T )
    |
  }{
    N^n
  }
  \right)
= 0
\end{equation}
$ \mathbb{P} $-a.s..
In the next step observe 
that the triangle inequality
implies
\begin{equation}
\label{eq:weakfinalTRIANGLE}
\begin{split}
&
  \left|
  \mathbb{E}\big[
    f( X_T )
  \big]
  -
  \frac{
    \sum_{ m = 1 }^{ N^n }
    f( \bar{Y}^{ N, m }_T )
  }{
    N^n
  }
  \right|
\\ & \leq
  \left|
  \mathbb{E}\big[
    f( X_T )
  \big]
  -
  \mathbb{E}\big[
    \mathbbm{1}_{ \Omega_N }
    f( X_T )
  \big]
  \right|
  +
  \left|
  \mathbb{E}\big[
    \mathbbm{1}_{ \Omega_N }
    f( X_T )
  \big]
  -
  \mathbb{E}\big[
    \mathbbm{1}_{ \Omega_N }
    f( \bar{Y}^N_T )
  \big]
  \right|
\\ & \quad +
  \left|
  \mathbb{E}\big[
    \mathbbm{1}_{ \Omega_N }
    f( \bar{Y}^N_T )
  \big]
  -
  \frac{
    \sum_{ m = 1 }^{ N^n }
    \mathbbm{1}_{ \Omega_{ N, m } }
    f( \bar{Y}^{ N, m }_T )
  }{
    N^n
  }
  \right|
  +
  \left|
  \frac{
    \sum_{ m = 1 }^{ N^n }
    \left(
      \mathbbm{1}_{ \Omega_{ N, m } }
      - 1 
    \right)
    f( \bar{Y}^{ N, m }_T )
  }{
    N^n
  }
  \right|
\\ & \leq
  \mathbb{E}\big[
    \mathbbm{1}_{ 
      \left( \Omega_N \right)^c 
    }
    |
      f( X_T )
    |
  \big]
  +
  \left|
  \mathbb{E}\big[
    \mathbbm{1}_{ \Omega_N }
    f( X_T )
  \big]
  -
  \mathbb{E}\big[
    \mathbbm{1}_{ \Omega_N }
    f( \bar{Y}^N_T )
  \big]
  \right|
\\ & \quad +
  \frac{ 
  \big|
    \!
    \sum_{ m = 1 }^{ N^n }
    \big\{
    \mathbbm{1}_{ \Omega_{ N, m } }
    f( \bar{Y}^{ N, m}_T )
    -
    \mathbb{E}\big[
      \mathbbm{1}_{ \Omega_N }
      f( \bar{Y}^N_T )
    \big]
    \big\}
  \big|
  }{
    N^n
  }
  +
  \frac{
    \sum_{ m = 1 }^{ N^n }
    \mathbbm{1}_{ 
      ( \Omega_{ N, m } )^c 
    }
    |
      f( \bar{Y}^{ N, m }_T )
    |
  }{
    N^n
  }
\end{split}
\end{equation}
for all $ N \in \mathbb{N} $.
Combining \eqref{eq:weakfinalZERO},
\eqref{eq:weakfinalONE}, 
\eqref{eq:weakfinalTWO}
and 
\eqref{eq:weakfinalTRIANGLE}
finally implies \eqref{eq:weakfinal}.
The proof of 
Proposition~\ref{p:weakconvergence}
is thus completed.
\end{proof}

\subsection{Convergence
of the Monte Carlo Euler method}
\label{sec:montecarloeuler}

Next we prove almost sure convergence of the Monte Carlo Euler method.
This result generalizes 
Theorem 2.1 in
\cite{hj11}
which assumes $\mu$ to be globally one-sided Lipschitz continuous,
$\sigma$ to be globally Lipschitz continuous
and $\mu$ and $\sigma$ to grow at most polynomially fast.
Corollary~\ref{cor:MCEmethod}
is a direct
consequence of
Theorem~\ref{thm:SVstability},
of Corollary~\ref{cor:FV}
and of
Proposition~\ref{p:weakconvergence}
and its proof is therefore
omitted.

\begin{cor}[Convergence
of the Monte Carlo Euler method]
\label{cor:MCEmethod}
Assume that the setting
in Section~\ref{sec:convergencesetting}
is fulfilled,
let
$
  p \in [3,\infty)
$,
$
  c,
  \gamma_0,
  \gamma_1
  \in [0,\infty)
$
be real numbers
with
$ 
    \gamma_1 + 
    2 ( \gamma_0 \vee \gamma_1 ) 
  < p/4
$,
let
$
  \bar{\mu} 
  \colon \mathbb{R}^d
  \rightarrow \mathbb{R}^d
$,
$
  \bar{\sigma} 
  \colon \mathbb{R}^d
  \rightarrow \mathbb{R}^{ d \times m }
$
be Borel measurable functions 
with 
$
  \bar{\mu}|_D = \mu
$,
$
  \bar{\sigma}|_D = \sigma
$
and
$
  \phi(x,t,y) =
  \bar{\mu}(x) t + 
  \bar{\sigma}(x) y 
$
for all $ x \in \mathbb{R}^d $,
$ t \in [0,T] $,
$ y \in \mathbb{R}^m $,
let
$ 
  V \in C^3_p( \mathbb{R}^d , [1,\infty) )
$
with
$
  \mathbb{E}[
    V( X_0 )
  ]
  < \infty
$
and with
\begin{equation*}
  ( \mathcal{G}_{ \bar{\mu}, 
    \bar{\sigma} } V)(x)
\leq
  c \cdot V(x) ,
\;
  \left\|
    \bar{\mu}(x)
  \right\|
  \leq
  c \,
  |
    V(x)
  |^{
    \left[
      \frac{ \gamma_0 + 1 }{ p }
    \right]
  } ,
\;
  \|
    \bar{\sigma}(x)
  \|_{
    L( \mathbb{R}^m, \mathbb{R}^d )
  }
  \leq
  c \,
  |
    V(x)
  |^{ 
    \left[
      \frac{ \gamma_1 + 2 }{ 2p } 
    \right]
  } 
\end{equation*}
for all $ x \in \mathbb{R}^d $
and for every $ N \in \mathbb{N} $
let
$
  \bar{Y}^{ N, m }
  \colon [0,T] \times
  \Omega \rightarrow \mathbb{R}
$,
$ m \in \mathbb{N} $,
be 
independent
stochastic processes with
$
  \mathbb{P}_{ \bar{Y}^N }
=
  \mathbb{P}_{ \bar{Y}^{N,m} }
$
for all $ m \in \mathbb{N} $.
Then 
\begin{equation}
  \lim_{ N \rightarrow \infty }
  \left|
  \mathbb{E}\big[
    f( X_T )
  \big]
  -
  \frac{
    \sum_{ m = 1 }^{ N^2 }
    f( \bar{Y}^{ N, m }_T )
  }{
    N^2
  }
  \right|
= 
  0 
\end{equation}
$ \mathbb{P} $-a.s.\ for 
all
continuous
functions 
$
  f \colon \mathbb{R}^d 
  \rightarrow \mathbb{R}
$
with
$
  \limsup_{ 
    r \nearrow 1 
  }
  \sup_{ x \in \mathbb{R}^d }
  \frac{
    | f(x) |
  }{
    ( 1 + V(x) )^r
  }
  < \infty
$.
\end{cor}

More results for approximating
statistical quantities of solutions
of SDEs with non-globally Lipschitz
continuous coefficients can, e.g.,
be found in 
Yan~\cite{y02},
Milstein 
\citationand\ Tretyakov~\cite{mt05},
D\"{o}rsek~\cite{Doersek2011}
and in the references
therein.

\section{Numerical
schemes for SDEs}
\label{sec:schemes}

The purpose of this 
section is to present a few examples
of numerical schemes which are
$ ( \mu, \sigma ) $-consistent
with respect to Brownian motion.
Theorem~\ref{thm:convergence} 
then shows
that the approximation
processes~\eqref{eq:recX} 
corresponding to such schemes
converge in probability to the exact
solution of the SDE~\eqref{eq:SDE}.
We now go into detail and
describe the setting that we use
in this section.
Throughout this section,
assume that the setting
in Section~\ref{sec:convergencesetting}
is fulfilled and
let 
$ \bar{\mu} \colon \mathbb{R}^d \rightarrow \mathbb{R}^d $ and 
$ \bar{\sigma} 
\colon \mathbb{R}^d \rightarrow \mathbb{R}^{ m \times d } 
$ 
be two arbitrary Borel 
measurable functions 
satisfying 
\begin{equation}
  \bar{\mu}( x ) = \mu( x )
\qquad
  \text{and}
\qquad 
  \bar{\sigma}( x ) = \sigma( x ) 
\end{equation}
for all $ x \in D $.
Moreover, define 
mappings 
$ Y_n^N 
\colon \Omega \rightarrow 
\mathbb{R}^d $, 	
$ n \in \{ 0, 1, \ldots, N \} $, 
$ N \in \mathbb{N} $, by 
$ Y_n^N := 
\bar{Y}_{ n T / N }^N $ 
for all $ n \in \{ 0, 1, \ldots, N \} $ 
and all $ N \in \mathbb{N} $
and define mappings
$ 
  \Delta W_n^N \colon 
  \Omega \rightarrow \mathbb{R}^m 
$, 
$ n \in  \{ 0, 1, \ldots, N-1 \} $, 
$ N \in \mathbb{N} $, 
by 
$ 
  \Delta W_n^N 
  := 
  W_{ (n+1) T / N } - 
  W_{ n T / N } 
$ 
for all 
$ n \in \{ 0, 1, \ldots N -1 \} $ 
and all 
$ N \in \mathbb{N} $. 
Using this notation,
we get from \eqref{eq:recX} that
\begin{equation}
\label{eq:schemeXX}
  Y^N_{ n + 1 }
  =
  Y^N_{ n }
  +
  \phi\big(
    Y^N_{ n },
    \tfrac{ T }{ N },
    \Delta W^N_n
  \big)
\end{equation}
for all $ n \in \{ 0, 1, \ldots, N - 1 \} $ 
and all $ N \in \mathbb{N} $. 
The following subsections provide
examples 
of Borel measurable 
functions
$
  \phi \colon \mathbb{R}^d
  \times [0,T] \times \mathbb{R}^m
  \rightarrow \mathbb{R}^d
$
and numerical approximation
schemes
of the form \eqref{eq:schemeXX},
respectively,
which are 
$ (\mu,\sigma) 
$-consistent with respect to Brownian motion.

\subsection{A few Euler-type 
schemes for SDEs}

Let 
$ 
  \eta_0 \colon 
  \mathbb{R}^d \times [0,T] \times 
  \mathbb{R}^m \rightarrow \mathbb{R}^d 
$, 
$ 
  \eta_1 \colon \mathbb{R}^d 
  \times [0,T] \times \mathbb{R}^m 
  \rightarrow \mathbb{R}^{d \times d} 
$ 
and 
$ 
  \eta_2 \colon \mathbb{R}^d \times 
  [0,T] \times \mathbb{R}^m \rightarrow 
  \mathbb{R}^{d \times d} 
$ 
be Borel measurable functions. 
The next lemma then gives
sufficient conditions to
ensure that
schemes of the form
\begin{equation}
\label{eq:simple_scheme_1}
\begin{split}
  Y_{ n + 1 }^N
& =
  Y_n^N
  +
  \eta_0\!\left(
    Y_n^N, \tfrac{T}{N}, \Delta W_n^N
  \right)
  +
  \eta_1\!\left(
    Y_n^N, \tfrac{T}{N}, \Delta W_n^N
  \right)
  \bar{\mu}( Y_n^N )
  \, \tfrac{T}{N}
\\ & \quad
  +
  \eta_2\!\left(
    Y_n^N, \tfrac{T}{N}, \Delta W_n^N
  \right)
  \bar{\sigma}( Y_n^N ) \,
  \Delta W_n^N
\end{split}
\end{equation}
for all $ n \in \{ 0, 1, \ldots, N-1 \} $ 
and all $ N \in \mathbb{N} $
are 
$ \left( \mu, \sigma \right) 
$-consistent with respect to Brownian motion.

\begin{lemma}
\label{lem:eta_012}
Assume that the setting
in Section~\ref{sec:convergencesetting}
is fulfilled, let 
$ 
  \bar{\mu} \colon \mathbb{R}^d 
  \rightarrow \mathbb{R}^d 
$,
$ \bar{\sigma} 
\colon \mathbb{R}^d \rightarrow \mathbb{R}^{ m \times d } 
$, 
$ 
  \eta_1, \eta_2 \colon \mathbb{R}^d 
  \times [0,T] \times \mathbb{R}^m 
  \rightarrow \mathbb{R}^{d \times d} 
$ 
and 
$ 
  \eta_0 \colon 
  \mathbb{R}^d \times [0,T] \times 
  \mathbb{R}^m \rightarrow \mathbb{R}^d 
$
be Borel measurable functions 
with $ \bar{\mu}|_D = \mu $
and $ \bar{\sigma}|_D = \sigma $
and assume that
\begin{equation}
\label{eq_prop_0A}
  \limsup_{ t \searrow 0 }
  \left(
    \sup_{ x \in K }
    \mathbb{E}\big[
      \left\| 
        \eta_2( x, t, W_t ) 
        \sigma( x ) W_t 
      \right\|
    \big]
  \right)
  < \infty ,
\end{equation}
\begin{equation}
\label{eq_prop_1A}
  \limsup_{ t \searrow 0 }
  \left(
    \sup_{ x \in K }
    \mathbb{E}\Big[
      \left\| 
        \eta_1( x, t, W_t ) - I 
      \right\|_{ L( \mathbb{R}^d ) }
      +
      \left\| 
        \eta_2( x, t, W_t ) - I 
      \right\|_{ L( \mathbb{R}^d ) 
      }^2
    \Big]
  \right)
  = 0,
\end{equation}
\begin{equation}
\label{eq_prop_2A}
  \limsup_{ t \searrow 0 }\left(
    \frac{1}{\sqrt{t}}
    \cdot
    \sup_{ x \in K }
    \mathbb{E}\big[
      \left\| 
        \eta_0( x, t, W_t ) 
      \right\|
    \big]
  \right)
  =
  \limsup_{ t \searrow 0 }\left(
    \frac{1}{t}
    \cdot
    \sup_{ x \in K }
    \big\|
    \mathbb{E}\big[
      \eta_0( x, t, W_t ) 
    \big]
    \big\|
  \right)
  = 0,
\end{equation}
\begin{equation}
\label{eq_prop_4A}
  \limsup_{ t \searrow 0 }\left(
    \frac{1}{t}
    \cdot
    \sup_{ x \in K }
    \big\|
    \mathbb{E}\big[
      \eta_2( x, t, W_t ) \sigma( x ) W_t 
    \big]
    \big\|
  \right)
  = 0
\end{equation}
for all non-empty compact sets $ K \subset D $. 
Then 
$ 
  \phi \colon 
  \mathbb{R}^d \times [0,T] 
  \times \mathbb{R}^m 
  \rightarrow \mathbb{R}^d 
$ 
given by
\begin{equation}
  \phi( x, t, y ) = 
  \eta_0( x, t, y ) + 
  \eta_1( x, t, y ) \bar{\mu}( x ) t 
  + \eta_2( x, t, y) \bar{\sigma}( x ) y 
\end{equation}
for all 
$ x \in \mathbb{R}^d $, 
$ t \in [0,T] $,
$ y \in \mathbb{R}^m $ 
is 
$ (\mu, \sigma) $-consistent with respect to Brownian motion.
\end{lemma}

\begin{proof}[Proof of Lemma~\ref{lem:eta_012}]
The triangle inequality 
and the H{\"o}lder inequality give
\begin{equation}
\label{eq:Z}
\begin{split}
  & \frac{ 1 }{ \sqrt{ t } }
  \cdot 
  \sup_{ x \in K }
  \mathbb{E}\big[
    \|
      \sigma( x ) W_t - 
      \phi( x, t, W_t )
    \|
  \big]
\\ & \leq
  \frac{ 1 }{ \sqrt{t} }
  \cdot 
  \sup_{ x \in K }
  \mathbb{E}\big[
    \|
      \eta_0( x, t, W_t )
    \|
  \big]
  +
  \frac{ 1 }{ \sqrt{t} }
  \cdot 
  \sup_{ x \in K }
  \mathbb{E}\big[
    \|
      \eta_1( x, t, W_t ) \mu( x ) t
    \|
  \big]
\\ & \quad +
  \frac{ 1 }{ \sqrt{t} }
  \cdot 
  \sup_{ x \in K }
  \E\big[
    \|
      ( I - \eta_2( x, t, W_t ) ) 
      \sigma( x ) W_t
     \|
  \big]
\\ & \leq
  \frac{ 1 }{ \sqrt{t} }
  \cdot 
  \sup_{ x \in K }
  \mathbb{E}\big[
    \| \eta_0( x, t, W_t ) \|
  \big]
  +
  \sqrt{ t }
  \left(
    \sup_{ x \in K }
    \mathbb{E}\big[
      \| 
        \eta_1( x, t, W_t ) 
      \|_{
        L( \mathbb{R}^d )
      }
    \big]
  \right)
  \left(
    \sup_{ x \in K }
    \left\| \mu( x ) \right\|
  \right)
\\ & \quad+
  \frac{
    \| W_t \|_{
      L^2( \Omega; \R^m )
    }
  }{ 
    \sqrt{ t }
  }
  \left(
    \sup_{ x \in K }
    \| 
      \eta_2( x, t, W_t ) - I 
    \|_{ 
      L^2( 
        \Omega; 
        L(\mathbb{R}^d) 
      ) 
    }
  \right)
  \left(
    \sup_{ x \in K }
    \| \sigma( x ) 
    \|_{
      L( 
        \mathbb{R}^m, 
        \mathbb{R}^d 
      )
    }
  \right)
\end{split}
\end{equation}
for all $ t \in (0, T] $ 
and all non-empty compact sets $ K \subset D $. 
Combining \eqref{eq:Z},
\eqref{eq_prop_1A} and 
\eqref{eq_prop_2A} then
shows~\eqref{ass:consistency_1}. 
It thus remains to
establish~\eqref{ass:consistency_2}
in order to complete the proof of 
Lemma~\ref{lem:eta_012}. 
To this end note that
\begin{equation}
\label{eq:W}
\begin{split}
&
  \limsup_{ t \searrow 0 }
  \left(
  \sup_{ x \in K }
  \left\|
    \mu( x )
    -
    \tfrac{ 1 }{ t }
    \cdot
    \mathbb{E}\big[
      \phi( x, t, W_t )
    \big]
  \right\|
  \right)
\\
& \leq
  \limsup_{ t \searrow 0 }
  \left(
  \tfrac{1}{t}
  \cdot
  \sup_{ x \in K }
  \big\|
    \mathbb{E}\big[
      \eta_0( x, t, W_t )
    \big]
  \big\|
  \right)
\\ & \quad
  +
  \left(
    \limsup_{ t \searrow 0 }
    \sup_{ x \in K }
    \mathbb{E}\big[
      \| \eta_1( x, t, W_t ) - I \|_{
        L( \mathbb{R}^d )
      }
    \big]
  \right)
  \left(
    \sup_{ x \in K }
    \left\| \mu( x ) \right\|
  \right)
\\ & \quad
  +
  \limsup_{ t \searrow 0 }
  \left(
  \tfrac{1}{t}
  \cdot 
  \sup_{ x \in K }
  \big\|
    \mathbb{E}\big[
      \eta_2( x, t, W_t ) \sigma( x ) W_t
    \big]
  \big\|
  \right)
\end{split}
\end{equation}
for all non-empty
compact sets $ K \subset D $. 
Inequality~\eqref{eq:W} 
and equations~\eqref{eq_prop_1A}, 
\eqref{eq_prop_2A} 
and~\eqref{eq_prop_4A} then
show~\eqref{ass:consistency_2}. 
This completes the proof of Lemma~\ref{lem:eta_012}.
\end{proof}

In the remainder of this
subsection,
a few numerical 
schemes of the form 
\eqref{eq:simple_scheme_1}
are presented.

\medskip

\paragraph{{\bf The Euler-Maruyama scheme}}

In the case 
$ \eta_0( x, t, y ) = 0 $,
$ \eta_1( x, t, y ) = I$
and
$\eta_2( x, t, y ) = I $ 
for all $ x \in \mathbb{R}^d $, 
$ t \in [0,T] $ and all 
$ y \in \mathbb{R}^m $,
the 
numerical 
scheme~\eqref{eq:simple_scheme_1} 
is the well-known 
Euler-Maruyama 
scheme
\begin{equation}
\label{eq:EulerMaruyama}
  Y_{n+1}^N
  =
  Y_n^N
  +
  \bar{\mu}( Y_n^N )
  \tfrac{T}{N}
  +
  \bar{\sigma}( Y_n^N )
  \Delta W_n^N
\end{equation}
for all $ n \in \{ 0, 1, \ldots, N-1 \} $ 
and all $ N \in \mathbb{N} $
(see Maruyama~\cite{m55}). 
Of course, this choice satisfies the assumptions 
of 
Lemma~\ref{lem:eta_012}. %
Combining 
Lemma~\ref{lem:eta_012} 
and Theorem~\ref{thm:convergence}
thus shows that the 
Euler-Maruyama
scheme~\eqref{eq:EulerMaruyama} 
converges in probability
to the exact solution of the
SDE~\eqref{eq:SDE}.
In the literature convergence
in probability and also
pathwise convergence
of the Euler-Maruyama scheme
has already been proved even in a
more general setting than
the setting considered here;
see, e.g., 
Krylov~\cite{Krylov1990}
and
Gy\"{o}ngy~\cite{g98b}.
Strong convergence of the Euler-Maruyama
scheme, however, often fails to
hold if the coefficients 
$ \mu $ and $ \sigma $
of 
the SDE~\eqref{eq:SDE} grow
more than linearly (see \cite{hjk11}) 
and therefore, we are interested
in appropriately modified
Euler-Maruyama schemes 
which are truncated or tamed
in a suitable way and which 
therefore
do converge strongly even
for SDEs with superlinearly growing
coefficients.
Although strong convergence
fails to hold, the corresponding
Monte Carlo Euler method
does converge with probability
one for a large class
of SDEs with possibly
superlinearly growing
coefficients; see 
Corollary~\ref{cor:MCEmethod}
above.

\medskip

\paragraph{{\bf A drift-truncated Euler scheme}}
In the case 
$ \eta_0( x, t, y ) = 0 $,
$   
  \eta_1( x, t, y ) = 
  \frac{ 1 }{
    \max\left(
      1, t \left\| \bar{\mu}(x) \right\|
    \right)
  } 
  I
$
and
$ \eta_2( x, t, y ) = I $ 
for all $ x \in \mathbb{R}^d $, 
$ t \in [0,T] $ and all 
$ y \in \mathbb{R}^m $, 
the numerical
scheme~\eqref{eq:simple_scheme_1} 
reads as
\begin{equation}
\label{eq:truncated}
  Y_{n+1}^N
  =
  Y_n^N
  +
  \frac{
    \bar{\mu}( Y_n^N )
    \tfrac{T}{N}
  }{
    \max\!\left(
      1, \frac{ T }{ N }
      \left\|
        \bar{\mu}( Y_n^N )
      \right\|
    \right)
  }
  +
  \bar{\sigma}( Y_n^N )
  \Delta W_n^N
\end{equation}
for all $ n \in \{ 0, 1, \ldots, N-1 \} $ 
and all $ N \in \mathbb{N} $. 
In the case where the drift
$ \bar{\mu} $ in \eqref{eq:truncated}
is of gradient type and the noise
is additive,
i.e., 
$ \bar{\sigma}(x) = \bar{\sigma}(0) $
and
$
  \bar{\mu}(x) = - (\nabla U)(x)
$
for all $ x \in \mathbb{R}^d $
and some 
appropriately smooth 
function
$ 
  U \colon \mathbb{R}^d
  \rightarrow \mathbb{R} 
$,
the scheme \eqref{eq:truncated}
has been used as proposal 
for the 
Metropolis-Adjusted-Truncated-Langevin-Algorithm 
(MATLA; see Roberts \citationand\ 
Tweedie~\cite{rt96}).
The choice
in \eqref{eq:truncated}
also 
satisfies the assumptions of 
Lemma~\ref{lem:eta_012}.
Combining 
Lemma~\ref{lem:eta_012} 
and Theorem~\ref{thm:convergence}
hence proves that the 
drift-truncated Euler 
scheme~\eqref{eq:truncated} 
converges in probability
to the exact solution of the
SDE~\eqref{eq:SDE}.
If the diffusion coefficient
$ 
  \bar{\sigma} 
$ in \eqref{eq:truncated}
grows at most linearly, then
moment bounds and
strong convergence of the
drift-truncated Euler 
scheme~\eqref{eq:truncated}
can be
studied by combining
Theorem~\ref{thm:SVstability},
Lemma~\ref{lem:compareSVstability},
Corollary~\ref{cor:Semi.Stability},
Lemma~\ref{lem:eta_012}
and 
Corollary~\ref{cor:strongConvergence}.
More precisely,
Theorem~\ref{thm:SVstability} 
and Lemma~\ref{lem:compareSVstability}
can be used to prove,
under suitable assumptions
on $ \bar{\mu} $
and $ \bar{\sigma} $
(see Theorem~\ref{thm:SVstability} 
for more details), 
that the drift-truncated 
Euler 
scheme~\eqref{eq:truncated}
is
$ \alpha $-semi $ V $-stable
with respect to Brownian motion
with $ \alpha \in (0,\infty) $
and
$ 
  V \colon \mathbb{R}^d
  \rightarrow [0,\infty)
$
appropriate.
Combining this and
Corollary~\ref{cor:strongConvergence}
with
the fact
that \eqref{eq:truncated} is
$ (\mu, \sigma) $-consistent with respect to Brownian motion
according
to Lemma~\ref{lem:eta_012} 
finally proves,
under the additional assumption
that $ \bar{\sigma} $ grows 
at most linearly, 
strong convergence of the
drift-truncated Euler 
scheme~\eqref{eq:truncated}.

\medskip

\paragraph{{\bf A drift-tamed Euler scheme}}
A slightly different variant of 
the drift-trun\-cated Euler 
scheme \eqref{eq:truncated}
is the drift-tamed Euler-type 
method 
considered in \cite{hjk10b}.
More precisely,
in the case 
$ \eta_0( x, t, y ) = 0 $,
$   
  \eta_1( x, t, y ) = 
  \frac{ 1 }{
    1 + t \left\| \bar{\mu}(x) \right\|
  } 
  I
$
and
$ \eta_2( x, t, y ) = I $ 
for all $ x \in \mathbb{R}^d $, 
$ t \in [0,T] $ and all 
$ y \in \mathbb{R}^m $,
the numerical
scheme~\eqref{eq:simple_scheme_1} 
reads as
\begin{equation}
\label{eq:tamed}
  Y_{n+1}^N
  =
  Y_n^N
  +
  \frac{
    \bar{\mu}( Y_n^N )
    \tfrac{T}{N}
  }{
      1 + \frac{ T }{ N }
      \left\|
        \bar{\mu}( Y_n^N )
      \right\|
  }
  +
  \bar{\sigma}( Y_n^N )
  \Delta W_n^N
\end{equation}
for all $ n \in \{ 0, 1, \ldots, N-1 \} $ 
and all $ N \in \mathbb{N} $. 
If $ D = \mathbb{R}^d $,
if $ \sigma $ is globally Lipschitz 
continuous
and if $ \mu $ is continuously
differentiable and
globally one-sided
Lipschitz continuous with an 
at most polynomially growing 
derivative, then strong convergence
of the drift-tamed Euler
scheme~\eqref{eq:tamed}
with the standard rate 
$ \frac{ 1 }{ 2 } $
has been proved in \cite{hjk10b}.
Moment bounds and
strong convergence of the
drift-tamed Euler 
scheme~\eqref{eq:tamed} 
in a more general setting
can be obtained in precisely the
same way as illustated
for the drift-truncated
Euler
scheme~\eqref{eq:truncated}.

\medskip

\paragraph{{\bf The Milstein 
scheme}}

In addition to the setting 
described in the beginning of
Section~\ref{sec:schemes},
assume in this paragraph 
that 
$ 
  \bar{\sigma} 
  =
  ( \bar{\sigma}_{ i, j } )_{
    i \in \{ 1, \dots, d \},
    j \in \{ 1, \dots, m \}
  }
  =
  ( \bar{\sigma}_{ j } )_{
    j \in \{ 1, \dots, m \}
  }
  \colon 
  \mathbb{R}^d
  \rightarrow 
  \mathbb{R}^{ d \times m }
$
is continuously differentiable.
In the case 
$ 
  \eta_1( x, t, y ) = 
  \eta_2( x, t, y ) = I 
$ 
and
\begin{equation}
\begin{split}
&
  \eta_0( x, t, y ) 
\\ & =  
  \frac{ 1 }{ 2 }
  \sum_{ k = 1 }^{ d }
  \sum_{ i, j = 1 }^{ m }
  \big(
    \tfrac{ \partial }{ \partial x_k }
    \bar{\sigma}_i
  \big)( x )  
  \cdot
  \bar{\sigma}_{ k , j }( x )
  \cdot
  y_i
  \cdot
  y_j
  -
  \frac{ t }{ 2 }
  \sum_{ k = 1 }^{ d }
  \sum_{ i = 1 }^{ m }
  \big(
    \tfrac{ \partial }{ \partial x_k }
    \bar{\sigma}_i
  \big)( x )  
  \cdot
  \bar{\sigma}_{ k , i }( x )
\end{split}
\end{equation}
for all 
$ 
  x = (x_1, \dots, x_d) 
  \in \mathbb{R}^d 
$, 
$ t \in [0,T] $ and all 
$ 
  y = (y_1, \dots, y_m) 
  \in \mathbb{R}^m 
$,
the numerical 
scheme~\eqref{eq:simple_scheme_1} 
reads as
\begin{equation}
\label{eq:milstein}
\begin{split}
  Y_{n+1}^N
& =
  Y_n^N
  +
  \bar{\mu}( Y_n^N )
  \tfrac{T}{N}
  +
  \bar{\sigma}( Y_n^N )
  \Delta W_n^N
\\ & \quad
  +
  \frac{ 1 }{ 2 }
  \sum_{ k = 1 }^{ d }
  \sum_{ i, j = 1 }^{ m }
  \big(
    \tfrac{ \partial }{ \partial x_k }
    \bar{\sigma}_i
  \big)( Y^N_n )  
  \cdot
  \bar{\sigma}_{ k , j }( Y^N_n )
  \cdot
  \Delta W^{ N, i }_n
  \cdot
  \Delta W^{ N, j }_n
\\ & \quad
  -
  \frac{ T }{ 2 N }
  \sum_{ k = 1 }^{ d }
  \sum_{ i = 1 }^{ m }
  \big(
    \tfrac{ \partial }{ \partial x_k }
    \bar{\sigma}_i
  \big)( Y^N_n )  
  \cdot
  \sigma_{ k , i }( Y^N_n )
\end{split}
\end{equation}
for all $ n \in \{ 0, 1, \ldots, N-1 \} $ 
and all $ N \in \mathbb{N} $
where
$
  ( 
    \Delta W^{N,1}_n, \dots, 
    \Delta W^{N,m}_n 
  )
  = \Delta W^N_n
$
for all 
$ n \in \{ 0, 1, \dots, N - 1 \} $
and all $ N \in \mathbb{N} $.
This choice satisfies the assumptions 
of Lemma~\ref{lem:eta_012}.
Moreover, if the commutativity 
condition
(see, e.g., 
(3.13) in Section~10.3 in
Kloeden \citationand\ Platen~\cite{kp92})
\begin{equation}
\label{eq:commutativity}
  \sum_{ k = 1 }^{ d }
  \big(
    \tfrac{ \partial }{ \partial x_k }
    \bar{\sigma}_i
  \big)( x )  
  \cdot
  \bar{\sigma}_{ k , j }( x )
=
  \sum_{ k = 1 }^{ d }
  \big(
    \tfrac{ \partial }{ \partial x_k }
    \bar{\sigma}_j
  \big)( x )  
  \cdot
  \bar{\sigma}_{ k , i }( x )
\end{equation}
for all 
$ 
  x 
  \in \mathbb{R}^d 
$ 
and all
$ i, j \in \{ 1, 2, \dots, m \} $
is fulfilled,
then the scheme
\eqref{eq:milstein} is
nothing else but
the well-known
Milstein scheme
(see Milstein~\cite{m74}
or, e.g., 
(3.16) in Section~10.3
in Kloeden \citationand\ Platen~\cite{kp92}).
Combining 
Lemma~\ref{lem:eta_012} 
and Theorem~\ref{thm:convergence}
thus shows that 
the 
scheme~\eqref{eq:milstein} 
converges in probability
to the exact solution of the
SDE~\eqref{eq:SDE}.
In the literature,
almost sure convergence
with rate $ 1 - \varepsilon $
for $ \varepsilon \in (0,1) $
arbitrarily small 
and thus also convergence 
in probability of the Milstein 
scheme has already been 
proved 
in \cite{jkn09a}.
However, as in the case of 
the Euler-Maruyama
scheme (see \eqref{eq:EulerMaruyama}),
strong convergence of the 
Milstein scheme often fails 
to hold if at least one of the 
coefficients 
$ \mu $ and $ \sigma $
of the SDE~\eqref{eq:SDE} grows
more than linearly 
(see \cite{hjk11}).
Nonetheless, if the
drift term is tamed
appropriately as in
\eqref{eq:tamed},
then strong convergence
of the corresponding
drift-tamed Milstein 
scheme
has been established
in Gang \citationand\ Wang~\cite{gw11}
for a class of SDEs with
possibly superlinearly
growing drift coefficients.

\medskip

\paragraph{{\bf Balanced implicit 
methods}}

Milstein, Platen and Schurz~\cite{mps98}
introduced the following class
of balanced implicit methods. 
Let 
$
  c_0,c_1,\ldots,c_m
  \colon 
  \R^d \to 
  \R^{d\times d}
$
be Borel measurable functions such that the matrix
\begin{equation}
  I + c_0(x) t + 
  \sum_{j=1}^m 
  c_j(x) |y_j| 
  \in 
  \R^{ d \times d }
\end{equation}
is invertible for all
$ x \in \R^d $,
$ t \in [0,T] $
and all
$ y = (y_1, \dots, y_m) \in \R^m $.
This condition is, e.g., satisfied 
if the matrices 
$ c_0(x,t,y) ,\ldots, c_m(c,t,y) $,
$(x,t,y)\in\R^d\times[0,T]\times\R^m$,
are positive 
semi-definite.
The associated
balanced implicit method 
is then given through
\begin{equation}
\label{eq:bim}
\begin{split}
&
  Y_{n+1}^N
\\ & =
  Y_n^N
  +
  \bar{\mu}( Y_n^N )
  \tfrac{T}{N}
  +
  \bar{\sigma}( Y_n^N )
  \Delta W_n^N
\\ & 
  +
  \left(
    c_0( Y^N_n ) \tfrac{ T }{ N } 
    + 
    \sum\nolimits_{j=1}^m 
    c_j( Y^N_n ) 
    \left| \Delta W^{ N, j }_n \right| 
  \right)
  \left( 
    Y_{n}^N - Y_{n+1}^N 
  \right)
\\ & =
  Y^N_n 
\\ & 
  +
  \left(
    I + 
    c_0( Y^N_n ) \tfrac{ T }{ N } 
    + 
    \sum\nolimits_{j=1}^m 
    c_j( Y^N_n ) 
    \left| \Delta W^{ N, j }_n \right| 
  \right)^{ \! - 1 }
  \!
  \Big(
    \bar{\mu}( Y_n^N )
    \tfrac{T}{N}
    +
    \bar{\sigma}( Y_n^N )
    \Delta W_n^N
  \Big)
\end{split}
\end{equation}
for all $ n \in \{ 0, 1, \dots, N - 1 \} $
and all $ N \in \mathbb{N} $
where
$
  ( 
    \Delta W^{N,1}_n, \dots, 
    \Delta W^{N,m}_n 
  )
  = \Delta W^N_n
$
for all $ n \in \{ 0, 1, \dots, N - 1 \} $
and all $ N \in \mathbb{N} $.
For the case of at most linearly growing coefficients $\bar{\mu}$ and $\bar{\sigma}$
and uniformly bounded matrices $c_0,c_1,\ldots,c_m$,
Theorem 4.1 of Schurz~\cite{Schurz2005} implies uniformly bounded moments
of the balanced implicit 
method~\eqref{eq:bim}.
Moreover, under additional assumptions such as global Lipschitz continuity
of $\bar{\mu}$ and $\bar{\sigma}$, Theorem 5.1 of
Schurz~\cite{Schurz2005} implies strong mean square convergence
of the balanced implicit 
method with 
convergence order 
$ \frac{ 1 }{ 2 } $.
In order to apply the above theory, write~\eqref{eq:bim}
in the form~\eqref{eq:simple_scheme_1} 
with
$
  \eta_0(x,t,y) = 0
$
and 
\begin{equation}
  \eta_1(x,t,y)
=
  \eta_2(x,t,y)
=
  \left(
    I + 
    c_0( x ) t
    + 
    \sum_{j=1}^m 
    c_j( x ) 
    | y_j | 
  \right)^{ \! \! - 1 }
\end{equation}
for all 
$ x \in \R^d $,
$ t \in [0,T] $,
$ y = (y_1, \dots, y_m) \in \R^m $.
Lemma~\ref{lem:eta_012} can be applied 
to derive conditions
%
on 
the functions $c_0,c_1,\ldots,c_m$ 
which imply 
$ (\mu, \sigma) $-consistency
with respect to
Brownian motion.
Theorem~\ref{thm:convergence}
then shows convergence 
in probability of the balanced implicit
method~\eqref{eq:bim}.
Moreover,
similar as described above for the 
drift-truncated Euler scheme,
moment bounds and
strong convergence of the
balanced implicit
method~\eqref{eq:bim}
can be
studied by combining
Theorem~\ref{thm:SVstability},
Lemma~\ref{lem:compareSVstability},
Corollary~\ref{cor:Semi.Stability},
Lemma~\ref{lem:eta_012}
and 
Corollary~\ref{cor:strongConvergence}.

\subsection{Comparison results 
for numerical schemes for SDEs}

The next lemma
considers two numerical approximation
schemes of the form \eqref{eq:recX}
and shows that if one of the two
schemes is 
$ \left( \mu, \sigma \right) $-consistent with respect to Brownian motion
and if the two schemes are close
to each other in an appropriate
sense (see \eqref{eq_prop_A1}),
then the other scheme is 
also 
$ \left( \mu, \sigma \right) 
$-consistent with respect to Brownian motion.
Its proof is straightforward and
hence omitted.

\begin{lemma}[A comparison
result for consistency]
\label{lem:phiphi} 
Assume that the setting
in 
Section~\ref{sec:convergencesetting}
is fulfilled and
let
$ 
  \hat{\phi} 
  \colon 
  \mathbb{R}^d \times [0,T] 
  \times \mathbb{R}^m 
  \rightarrow \mathbb{R}^d 
$ 
be a function
which is
$ \left( \mu, \sigma \right) 
$-consistent with respect to Brownian motion
and which satisfies
\begin{equation}
\label{eq_prop_A1}
  \limsup_{ t \searrow 0 }
  \left(
    \frac{1}{t}
    \cdot
    \sup_{ x \in K }
    \mathbb{E}\left[
      \big\| 
        \phi( x, t, W_t ) - 
        \hat{\phi}( x, t, W_t ) 
      \big\|
    \right]
  \right)
  = 0
\end{equation}
for all non-empty
compact sets $ K \subset D $.
Then 
$ 
  \phi \colon \mathbb{R}^d 
  \times [0,T] \times 
  \mathbb{R}^m \rightarrow 
  \mathbb{R}^d 
$ 
is $ 
  \left( \mu, \sigma \right) 
$-consistent with respect to Brownian motion
too.
\end{lemma}

The next lemma, in particular, 
illustrates that convex combinations 
of schemes
that are
$ \left( \mu, \sigma \right) 
$-consistent with 
respect to Brownian motion
are 
$ \left( \mu, \sigma \right) 
$-consistent with respect to Brownian motion too.
More precisely,
the next lemma shows that
schemes of the form
\begin{equation}
  Y^N_{ n + 1 }
  =
  Y^N_n 
  +
  \eta_1 \cdot
  \phi_1\big(
    Y^N_n, \tfrac{ T }{ N },
    \Delta W^N_n
  \big) 
  +
  \eta_2 \cdot
  \phi_2\big(
    Y^N_n, \tfrac{ T }{ N },
    \Delta W^N_n
  \big) 
\end{equation}
for all
$ n \in \{ 0, 1, \dots, N \} $
and all
$ N \in \mathbb{N} $
are 
$ ( \mu, \sigma ) $-consistent
with respect to Brownian motion
provided that
$ \eta_1, \eta_2 \in \mathbb{R} $
are real numbers with 
$ \eta_1 + \eta_2 = 1 $
and provided that
$
  \phi_1, \phi_2 \colon
  \mathbb{R}^d \times [0,T]
  \times \mathbb{R}^m
$
are 
$ \left( \mu, \sigma \right) 
$-consistent with respect to Brownian motion.
Its proof is clear
and therefore omitted.

\begin{lemma}[Generalized
convex combinations of
numerical schemes
for SDEs]
\label{lem_cc}
Assume that the setting
in 
Section~\ref{sec:convergencesetting}
is fulfilled,
let 
$ 
  \phi_1, \phi_2 \colon \mathbb{R}^d 
  \times [0,T] \times \mathbb{R}^m 
  \rightarrow \mathbb{R}^d 
$ 
be 
$ \left( \mu, \sigma \right) 
$-consistent functions
with respect to Brownian motion
and let 
$ \eta_1, \eta_2 \in \mathbb{R} $ 
be two real numbers with 
$ \eta_1 + \eta_2 = 1 $. 
Then 
$ 
  \phi \colon \mathbb{R}^d \times 
  [0,T] \times \mathbb{R}^m 
  \rightarrow \mathbb{R}^d 
$ 
given by 
\begin{equation}
  \phi( x, t, y ) = 
  \eta_1 \cdot \phi_1 ( x, t, y ) 
  + \eta_2 \cdot \phi_2( x, t, y ) 
\end{equation}
for all 
$ x \in \mathbb{R}^d $, $ t \in [0,T] $, 
$ y \in \mathbb{R}^m $ 
is
$ \left( \mu, \sigma \right) 
$-consistent with respect to Brownian motion.
\end{lemma}

Finally, the next lemma gives
a simple characterization of
$ (\mu, \sigma) 
$-consistency with respect to Brownian motion.
Its proof is straightforward and
hence omitted.

\begin{lemma}[A characterization of consistency]
Let $ T \in (0,\infty) $,
$ d, m \in \mathbb{N} $,
let $ D \subset \mathbb{R}^d $
be an open set and let
$ 
  \mu \colon 
  D \rightarrow \mathbb{R}^d
$
and
$
  \sigma \colon
  D \rightarrow 
  \mathbb{R}^{ d \times m }
$
be locally Lipschitz continuous 
functions.
A Borel measurable 
function
$ 
  \phi \colon
    \mathbb{R}^d \times
    [0,T] \times \mathbb{R}^m
  \rightarrow 
    \mathbb{R}^d
$
is then 
$ \left( \mu, \sigma \right) 
$-consistent with respect to Brownian motion
if and only if 
there exists a
Borel measurable function
$ 
  \hat{\phi} \colon
    \mathbb{R}^d \times
    [0,T] \times \mathbb{R}^m
  \rightarrow 
    \mathbb{R}^d
$
which is
$ \left( \mu, \sigma \right) 
$-consistent
with respect to 
Brownian motion
and which satisfies
\begin{equation}
  \lim_{ t \searrow 0 }
  \left(
    \tfrac{1}{\sqrt{t}}
    \cdot
    \sup_{ 
      x \in K
    }
    \mathbb{E}\Big[
    \big\|
        \hat{\phi}( 
          x, 
          t, 
          W_t
        )
      -
        \phi( 
          x, 
          t, 
          W_t
        )
    \big\|
    \Big]
  \right)
  = 0
\end{equation}
and
\begin{equation}
  \lim_{ t \searrow 0 }
  \left(
  \tfrac{1}{t}
  \cdot 
  \sup_{ 
    x \in K
  }
  \big\|
     \mathbb{E}\big[
        \hat{\phi}( 
          x, 
          t, 
          W_t
        )
     \big]
    -
     \mathbb{E}\big[
      \phi( 
        x, 
        t, 
        W_t
      )
    \big]
  \big\|
  \right)
  = 0
\end{equation}
for all non-empty 
compact sets $ K \subset D $
where
$
  W \colon [0,T] \times
  \Omega \rightarrow 
  \mathbb{R}^m 
$
is an arbitrary standard Brownian
motion on a probability
space 
$
  \left( 
    \Omega, \mathcal{F}, \mathbb{P}
  \right)
$.
\end{lemma}

\subsection{Taming principles
for numerical schemes for SDEs}
\label{sec:taming}

Let 
$ 
  \hat{\phi} \colon \mathbb{R}^d 
  \times [0,T] \times 
  \mathbb{R}^m \rightarrow 
  \mathbb{R}^d 
$ 
be a function which is
$ 
  \left( \mu, \sigma \right) 
$-consistent with respect to Brownian motion
and consider a sequence
$ 
  Z^N \colon \{ 0, 1, \dots N \}
  \times \Omega \to \R^d
$,
$ N \in \N $,
of stochastic
processes given by
$ Z^N_0 = X_0 $
and
\begin{equation}
\label{eq:originalscheme}
  Z^N_{ n + 1 } =
  Z^N_n+
  \hat{\phi}\big( 
    Z^N_{ 
      n 
    },
    \tfrac{T}{N}, 
    \Delta W^N_n
  \big)
\end{equation}
for all $ n \in \{ 0, 1, \dots, N - 1 \} $
and all $ N \in \N $.
The next lemma 
then
gives sufficient conditions
to ensure that numerical
approximation schemes of the form
\begin{equation}
\label{eq:incrementtamed}
  Y_{ n + 1 }^N
  =
  Y_n^N
  +
  \eta\big(
    Y^N_n, \tfrac{ T }{ N },
    \Delta W^N_n
  \big) \,
  \hat{\phi}\big(
    Y^N_n, \tfrac{ T }{ N },
    \Delta W^N_n
  \big)
\end{equation}
for all $ n \in \{ 0, 1, \ldots, N-1 \} $ 
and all $ N \in \mathbb{N} $
are 
$ \left( \mu, \sigma \right) 
$-consistent with respect to Brownian motion
where
$ 
  \eta \colon \mathbb{R}^d 
  \times [0,T] \times 
  \mathbb{R}^m \rightarrow 
  \mathbb{R}^{ d \times d } 
$ 
is a suitable Borel measurable
function.

\begin{lemma}[Increment
taming principle]
\label{lem:eta_phi}
Assume that the setting in 
Section~\ref{sec:convergencesetting}
is fulfilled, 
let 
$ 
  \eta \colon \mathbb{R}^d 
  \times [0,T] \times 
  \mathbb{R}^m \rightarrow 
  \mathbb{R}^{ d \times d } 
$ 
be a Borel measurable 
function, let
$ 
  \hat{\phi} \colon \mathbb{R}^d 
  \times [0,T] \times 
  \mathbb{R}^m \rightarrow 
  \mathbb{R}^d 
$ 
be a function which is 
$ \left( \mu, \sigma \right) $-consistent
with respect to Brownian motion
and assume that
\begin{equation}
\label{eq:propE_1}
  \limsup_{ t \searrow 0 }
  \left(
    \frac{1}{t}   
    \cdot
    \sup_{ x \in K }
    \mathbb{E}\!\left[
      \big\| \hat{\phi}( x, t, W_t ) \big\|^2
    \right]
  \right)
  < \infty
\end{equation}
and
\begin{equation}
\label{eq:propE_2}
  \limsup_{ t \searrow 0 }
  \left(
    \frac{1}{t}   
    \cdot
    \sup_{ x \in K }
    \mathbb{E}\Big[
      \big\| \eta( x, t, W_t ) - I    
      \big\|^2_{ L( \mathbb{R}^d ) }
    \Big]
  \right)
  =
  0
\end{equation}
for all non-empty 
compact sets $ K \subset D $. 
Then 
$ 
  \phi \colon \mathbb{R}^d 
  \times [0,T] \times 
  \mathbb{R}^m \rightarrow 
  \mathbb{R}^d 
$ 
given by 
\begin{equation}
  \phi( x, t, y ) = 
  \eta( x, t, y ) \, \hat{\phi}( x, t, y ) 
\end{equation}
for all $ x \in \mathbb{R}^d $, 
$ t \in [0, T ] $, 
$ y \in \mathbb{R}^m $ 
is
$ \left( \mu, \sigma \right) $-consistent with respect to Brownian motion.
\end{lemma}
\begin{proof}[Proof of 
Lemma~\ref{lem:eta_phi}]
Combining
H{\"o}lder's inequality,
equation~\eqref{eq:propE_1} 
and 
equation~\eqref{eq:propE_2} 
implies
\begin{equation}
\label{eq:propE_3}
\begin{split}
  &
  \limsup_{ t \searrow 0 }
  \left(
    \frac{1}{t}   
    \cdot
    \sup_{ x \in K }
    \mathbb{E}\Big[
      \big\| 
        \hat{\phi}( x, t, W_t ) - 
        \phi( x, t, W_t ) 
      \big\|
    \Big]
  \right)
\\ & \leq
  \limsup_{ t \searrow 0 }
  \left(
    \frac{ 1 }{ t }   
    \cdot
    \sup_{ x \in K }
    \mathbb{E}\Big[
      \| \eta( x, t, W_t ) - I \|_{
        L( \mathbb{R}^d )
      }
      \cdot
      \| \hat{\phi}( x, t, W_t ) \|
    \Big]
  \right)
\\ & \leq
  \left(
    \limsup_{ t \searrow 0 }
    \frac{
    \sup_{ x \in K }
    \mathbb{E}\big[
      \| \eta( x, t, W_t ) - I \|^2_{
        L( \mathbb{R}^d )
      }
    \big]
}{ t }
  \right)^{ \! 1 / 2 }
\\ & \quad
  \cdot
  \left(
    \limsup_{ t \searrow 0 }
    \frac{1}{t}   
    \cdot
    \sup_{ x \in K }
    \mathbb{E}\left[
      \| \hat{\phi}( x, t, W_t ) \|^2
    \right]
  \right)^{\frac{1}{2}}
  = 0
\end{split}
\end{equation}
for all non-empty 
compact sets $ K \subset D $.
Inequality~\eqref{eq:propE_3} and Lemma~\ref{lem:phiphi} 
then complete the proof of 
Lemma~\ref{lem:eta_phi}.
\end{proof}

Let us illustrate
Lemma~\ref{lem:eta_phi}
by an example.
More precisely,
in the special case
$
  \eta(x,t,y)
=
  \frac{ 1 }{
    \max( 1, 
      t \| \hat{\phi}(x,t,y) \|
    )
  }
  I
$
for all $ x \in \mathbb{R}^d $,
$ t \in [0,T] $, $ y \in \mathbb{R}^m $,
the scheme~\eqref{eq:incrementtamed} 
reads as
\begin{equation}
\label{eq:incrementtamed2}
  Y_{ n + 1 }^N
  =
  Y_n^N
  +
  \frac{
    \hat{\phi}\big(
      Y^N_n, \tfrac{ T }{ N },
      \Delta W^N_n
    \big)
  }{
    \max\big(
      1 , \frac{ T }{ N }
      \big\|
      \hat{\phi}\big(
        Y^N_n, \tfrac{ T }{ N },
        \Delta W^N_n
      \big)
      \big\|
    \big)
  }
\end{equation}
for all $ n \in \{ 0, 1, \ldots, N-1 \} $ 
and all $ N \in \mathbb{N} $.

If we now additionally assume
that the 
$ Z^N $, $ N \in \N $,
in \eqref{eq:originalscheme}
are Euler-Maruyama
approximations, i.e., that
$ 
  \hat{\phi}(x,t,y) = 
  \bar{ \mu }(x) t +
  \bar{ \sigma }(x) y
$
for all $ x \in \mathbb{R}^d $,
$ t \in [0,T] $, $ y \in \mathbb{R}^m $,
then \eqref{eq:incrementtamed2}
reads as
\begin{equation}
\label{eq:incrementtamed3}
  Y_{ n + 1 }^N
  =
  Y_n^N
  +
  \frac{
    \bar{\mu}(
      Y^N_n
    )
    \tfrac{ T }{ N }
    +
    \bar{ \sigma }( Y^N_n )
    \Delta W^N_n
  }{
    \max\!\big(
      1 , \frac{ T }{ N }
      \big\|
      \bar{\mu}(
        Y^N_n
      )
      \tfrac{ T }{ N }
      +
      \bar{ \sigma }( Y^N_n )
      \Delta W^N_n
      \big\|
    \big)
  }
\end{equation}
for all $ n \in \{ 0, 1, \ldots, N-1 \} $ 
and all $ N \in \mathbb{N} $.
This increment-tamed 
Euler-Maruyama scheme 
clearly satisfies the assumptions
of Lemma~\ref{lem:eta_phi}.
Theorem~\ref{thm:convergence}
hence shows that
the 
scheme~\eqref{eq:incrementtamed3}
convergences 
in probability to the exact
solution of the SDE~\eqref{eq:SDE}.
Note that this scheme
is frequently studied in this article.
Strong convergence of the
scheme~\eqref{eq:incrementtamed3}
is studied in
Subsection~\ref{sec:strongCtamed}
above (see also 
Chapter~\ref{sec:examples}
for a list examples of SDEs
in case of which the
scheme~\eqref{eq:incrementtamed3}
has been shown to converge
strongly).

In the case of spatially 
discretized semilinear
stochastic partial
differential equations, another
choice for the 
increment function 
$ 
  \hat{\phi} \colon \R^d \times
  [0,T] \times \R^m \rightarrow \R^d
$
in \eqref{eq:originalscheme}
naturally arises.
More precisely,
suppose that
\begin{equation}
  \bar{\mu}(x)
  =
  A x + F(x)
\end{equation}
for all $ x \in \R^d $
where
$ A \in \R^{ d \times d } $
is a $ d \times d $-matrix
with $ \det( I - t A ) \neq 0 $
for all $ t \in [0,\infty) $
and where
$
  F \colon \R^d \to \R^d
$
is a Borel measurable function
and suppose that
\begin{equation}
  \hat{\phi}(x,t,y) =
  \left( I - t A \right)^{ - 1 }
  \left(  
    x + F(x) t + \sigma(x) y
  \right)
  - x
\end{equation}
for all $ x \in \R^d $,
$ t \in [0,T] $
and all $ y \in \R^m $.
The approximations 
processes $ Z^N $, $ N \in \N $,
in equation~\eqref{eq:originalscheme}
thus reduce to the
linear implicit Euler approximations
\begin{equation}
\begin{split}
  Z^N_{ n + 1 } 
& =
  \left( I - \tfrac{ T }{ N } A \right)^{ - 1 }
  \left(
    Z^N_n
    +
    F( Z^N_n ) \tfrac{ T }{ N }
    +
    \sigma( Z^N_n ) \Delta W_n^N
  \right)
\\ & =
  Z^N_n
  +
  \left[
  \left( I - \tfrac{ T }{ N } A \right)^{ - 1 }
  \left(
    Z^N_n
    +
    F( Z^N_n ) \tfrac{ T }{ N }
    +
    \sigma( Z^N_n ) \Delta W_n^N
  \right)
    - Z^N_n
  \right]
\end{split}
\end{equation}
for all $ n \in \{ 0, 1, \dots, N - 1 \} $
and all $ N \in \N $
and the scheme
in \eqref{eq:incrementtamed2}
then reads as
\begin{equation}
\label{eq:spdetamed}
\begin{split}
  Y^N_{ n + 1 } 
& =
  Y^N_n
  +
  \frac{
    \left( I - \tfrac{ T }{ N } A \right)^{ - 1 }
    \left(
      Y^N_n
      +
      F( Y^N_n ) \tfrac{ T }{ N }
      +
      B( Y^N_n ) \Delta W_n^N
    \right)
    - Y^N_n
  }{
    \max\!\left( 1,
    \frac{ T }{ N }
    \big\| \!
      \left( I - \tfrac{ T }{ N } A 
      \right)^{ - 1 } \!
      \left(
        Y^N_n
        +
        F( Y^N_n ) \tfrac{ T }{ N }
        +
        B( Y^N_n ) \Delta W_n^N
      \right)
      - Y^N_n
    \big\|
    \right)
  }
\end{split}
\end{equation}
for all $ n \in \{ 0, 1, \dots, N - 1 \} $
and all $ N \in \N $.

In \eqref{eq:incrementtamed},
\eqref{eq:incrementtamed2},
\eqref{eq:incrementtamed3}
and \eqref{eq:spdetamed},
respectively, the increment
function
$ 
  \hat{\phi} 
  \colon 
  \mathbb{R}^d \times
  [0,T] \times \mathbb{R}^m
  \rightarrow \mathbb{R}^d
$
is tamed in a suitable way
so that the scheme does
not diverge strongly 
(see \cite{hjk11})
and 
moment bounds
can be obtained
(see Subsections~\ref{sec:moment_bounds000},
\ref{sec:onestep}
and \ref{sec:momentbounds2}).
Instead of the increment
function, also the whole scheme
can be tamed in an appropriate
way. This is the subject of the
next lemma.
More precisely,
the next lemma gives sufficient
conditions to ensure that
schemes of the form
\begin{equation}
\label{eq:fulltamed}
  Y_{ n + 1 }^N
  =
  \eta\big(
    Y^N_n, \tfrac{ T }{ N },
    \Delta W^N_n
  \big) 
  \Big(
    Y_n^N
    +
    \hat{\phi}\big(
      Y^N_n, \tfrac{ T }{ N },
      \Delta W^N_n
    \big)
  \Big)
\end{equation}
for all $ n \in \{ 0, 1, \ldots, N-1 \} $ 
and all $ N \in \mathbb{N} $
are 
$ \left( \mu, \sigma \right) 
$-consistent with respect to Brownian motion
where
$ 
  \eta \colon \mathbb{R}^d 
  \times [0,T] \times 
  \mathbb{R}^m \rightarrow 
  \mathbb{R}^{ d \times d } 
$ 
is a Borel measurable
function and where
$ 
  \hat{\phi} \colon \mathbb{R}^d 
  \times [0,T] \times 
  \mathbb{R}^m \rightarrow 
  \mathbb{R}^d 
$ 
is a function which
is
$ 
  \left( \mu, \sigma \right) 
$-consistent with respect to Brownian motion.

\begin{lemma}[Full taming 
principle]
\label{l:eta_I}
Assume that the setting in 
Section~\ref{sec:convergencesetting}
is fulfilled, 
let 
$ 
  \eta \colon \mathbb{R}^d 
  \times [0,T] \times 
  \mathbb{R}^m \rightarrow 
  \mathbb{R}^{ d \times d } 
$ 
be a Borel measurable 
function, let
$ 
  \hat{\phi} \colon \mathbb{R}^d 
  \times [0,T] \times 
  \mathbb{R}^m \rightarrow 
  \mathbb{R}^d 
$ 
be a function
which is
$ \left( \mu, \sigma \right) $-consistent
with respect to Brownian motion
and assume that
\begin{equation}
\label{eq:propEI_1}
  \limsup_{ t \searrow 0 }
  \left(
    \sup_{ x \in K }
    \mathbb{E}\Big[
      \| \hat{\phi}( x, t, W_t ) 
      \|^2
    \Big]
  \right)
  < \infty 
\end{equation}
and
\begin{equation}
\label{eq:propEI_2}
  \limsup_{ t \searrow 0 }
  \left(
    \frac{1}{t^2}   
    \cdot
    \sup_{ x \in K }
    \mathbb{E}\Big[
      \| \eta( x, t, W_t ) - I \|^2_{
        L( \mathbb{R}^d )
      }
    \Big]
  \right)
  = 0 
\end{equation}
for all non-empty
compact sets $ K \subset D $. 
Then 
$ 
  \phi \colon \mathbb{R}^d \times 
  [0,T] \times \mathbb{R}^m 
  \rightarrow \mathbb{R}^d 
$ 
given by 
\begin{equation}
  \phi( x, t, y ) = 
  \eta( x, t, y )( x + \hat{\phi}( x, t, y ) ) - x 
\end{equation}
for all 
$ x \in \mathbb{R}^d $, 
$ t \in [0,T] $ 
and all 
$ y \in \mathbb{R}^m $ 
is 
$ \left( \mu, \sigma \right) 
$-consistent with respect to Brownian motion.
\end{lemma}
\begin{proof}[Proof of 
Lemma~\ref{l:eta_I}]
H{\"o}lder's inequality, the 
estimate 
$ (a+b)^2 \leq 2a^2 + 2b^2 $ 
for all 
$ a, b \in \mathbb{R} $ 
and equations~\eqref{eq:propEI_1}
and~\eqref{eq:propEI_2} 
imply
\begin{equation}
\label{eq:propEI_3}
\begin{split}
  &
  \limsup_{ t \searrow 0 }
  \left(
    \frac{ 1 }{ t }   
    \cdot
    \sup_{ x \in K }
    \mathbb{E}\!\left[
      \big\| 
        \hat{\phi}( x, t, W_t ) - 
        \phi( x, t, W_t ) 
      \big\|
    \right]
  \right)
\\ & \leq
  \limsup_{ t \searrow 0 }
  \left(
    \frac{1}{t}   
    \cdot
    \sup_{ x \in K }
    \mathbb{E}\!\left[
      \left\| 
        \eta( x, t, W_t ) - I 
      \right\|_{ L( \R^d ) }
      \cdot
      \big\| 
        x + \hat{\phi}( x, t, W_t ) 
      \big\|
    \right]
  \right)
\\ & \leq
  2
  \left(
    \limsup_{ t \searrow 0 }
    \frac{1}{t^2}   
    \cdot
    \sup_{ x \in K }
    \mathbb{E}\!\left[
      \big\| 
        \eta( x, t, W_t ) - I 
      \big\|^2_{ L( \R^d ) }
    \right]
  \right)^{ \! 1 / 2 }
\\ & \quad
  \cdot
  \left(
    \sup_{ x \in K }\|x\|^2
    +
    \limsup_{ t \searrow 0 }
    \sup_{ x \in K }
    \mathbb{E}\!\left[
      \big\| 
        \hat{\phi}( x, t, W_t ) 
      \big\|^2
    \right]
  \right)^{ \! 1 / 2 }
  = 0
\end{split}
\end{equation}
for all non-empty
compact sets $ K \subset D $.
Inequality~\eqref{eq:propEI_3} 
and Lemma~\ref{lem:phiphi} 
then complete the proof of 
Lemma~\ref{l:eta_I}.
\end{proof}

As an example 
of Lemma~\ref{l:eta_I},
let $ r \in (1,\infty) $ be a real number
and consider
the choice
$
  \eta \colon 
  \R^d \times [0,T] \times \R^m \to \R^{ d \times d }
$
given by
\begin{equation}
\label{eq:eta_indicator}
  \eta(x,t,y) =
  \mathbbm{1}_{ 
    [0,1]
  }
  \big(
    t^r
    \| 
      x + \hat{ \phi }(x,t,y) 
    \|
  \big)
  \cdot
  I
\end{equation}
for all
$ x \in \mathbb{R}^d $,
$ t \in [0,T] $
and all
$ y \in \mathbb{R}^m $.
Note that
the choice \eqref{eq:eta_indicator}
satisfies 
\eqref{eq:propEI_2}
in Lemma~\ref{l:eta_I}
provided that
\eqref{eq:propEI_1}
is fulfilled.
Indeed, observe that
Markov's inequality and
\eqref{eq:propEI_1}
imply that
\begin{equation}
\begin{split}
&
  \limsup_{ t \searrow 0 }
  \left(
    \frac{1}{t^2}   
    \cdot
    \sup_{ x \in K }
    \mathbb{E}\Big[
      \| \eta( x, t, W_t ) - I \|^2_{
        L( \mathbb{R}^d )
      }
    \Big]
  \right)
\\ & = 
  \limsup_{ t \searrow 0 }
  \left(
    \frac{1}{t^2}   
    \cdot
    \sup_{ x \in K }
    \mathbb{E}\Big[
    1 -
    \mathbbm{1}_{ 
      [0,1]
    }
    \big(
      t^r
      \| 
        x + \hat{ \phi }(x,t,W_t) 
      \|
    \big)
    \Big]
  \right)
\\ & = 
  \limsup_{ t \searrow 0 }
  \left(
    \frac{1}{t^2}   
    \cdot
    \sup_{ x \in K }
    \P\Big[
      t^r
      \| 
        x + \hat{ \phi }(x,t,W_t) 
      \|
      > 1
    \Big]
  \right)
\\ & = 
  \limsup_{ t \searrow 0 }
  \left(
    t^{ - 2}   
    \cdot
    \sup_{ x \in K }
    \P\Big[
      \| 
        x + \hat{ \phi }(x,t,W_t) 
      \|^2
      > t^{ - 2 r }
    \Big]
  \right)
\\ & \leq
  \limsup_{ t \searrow 0 }
  \left(
    t^{ (2r - 2) }
    \cdot
    \sup_{ x \in K }
    \E\Big[
      \| 
        x + \hat{ \phi }(x,t,W_t) 
      \|^2
    \Big]
  \right)
  = 0 
\end{split}
\end{equation}
for all non-empty
compact sets
$ K \subset D $.
Next note that 
in the case \eqref{eq:eta_indicator}
the scheme~\eqref{eq:fulltamed}
reads as
\begin{equation}
\label{eq:exfulltamed}
  Y_{ n + 1 }^N
  =
  \mathbbm{1}_{ 
    \left\{
      \|
      Y_n^N
      +
      \hat{\phi}(
        Y^N_n, \frac{ T }{ N },
        \Delta W^N_n
      )
      \|
      \leq 
      \frac{ N^r }{ T^r }
    \right\}
  }
  \Big(
    Y_n^N
    +
    \hat{\phi}\big(
      Y^N_n, \tfrac{ T }{ N },
      \Delta W^N_n
    \big)
  \Big)
\end{equation}
for all $ n \in \{ 0, 1, \ldots, N-1 \} $ 
and all $ N \in \mathbb{N} $.
A similiar class of approximations
has been proposed in
Milstein \citationand\ Tretyakov~\cite{mt05}
in which
the truncation barrier 
$ \frac{ N^r }{ T^r } $
in the indicator set in 
\eqref{eq:exfulltamed}
is replaced by a 
possibly large
real number 
$ R \in (0,\infty) $
which does not depend on $N\in\N$.

\subsection{Linear implicit numerical schemes for SDEs}

The next lemma gives 
sufficient conditions that ensure that
linear implicit numerical schemes 
of the form
\begin{equation}
\begin{split}
&
  Y_{ n + 1 }^N
\\ & =
  Y_n^N
  +
  \mathbbm{1}_{
    \left\{
    \det(
      I -
      A( 
        Y^N_n, \frac{ T }{ N },
        \Delta W^N_n
      )
    )
    \neq 0
    \right\}
  }
  A\!\left( 
    Y^N_n, \tfrac{ T }{ N },
    \Delta W^N_n
  \right)
  Y^N_{ n + 1 }
  +
  b\big(
    Y^N_n, \tfrac{ T }{ N },
    \Delta W^N_n
  \big)
\\[1ex] 
  & =
  \Big( 
    I - 
  \mathbbm{1}_{
    \left\{
    \det(
      I -
      A( 
        Y^N_n, \frac{ T }{ N },
        \Delta W^N_n
      )
    )
    \neq 0
    \right\}
  }
    A \!\left( 
     Y^N_n, \tfrac{ T }{ N },
      \Delta W^N_n
    \right)
  \Big)^{ \! - 1 } 
  \Big(
    Y_n^N
    +
    b\big(
      Y^N_n, \tfrac{ T }{ N },
      \Delta W^N_n
    \big)
  \Big)
\end{split}
\end{equation}
for all $ n \in \{ 0, 1, \ldots, N-1 \} $ 
and all $ N \in \mathbb{N} $
are 
$ \left( \mu, \sigma \right) $-consistent with respect to Brownian motion
where
$ 
A \colon \mathbb{R}^d \times [0,T] 
\times \mathbb{R}^m \rightarrow \mathbb{R}^{d \times d} 
$ 
and 
$ 
b \colon \mathbb{R}^d 
\times [0,T] \times 
\mathbb{R}^m \rightarrow 
\mathbb{R}^d 
$ 
are suitable Borel measurable 
functions (see 
Lemma~\ref{lem:Ab}
for the detailed assumptions).

\begin{lemma}[Linear implicit
numerical schemes for SDEs
driven by Brownian motions]
\label{lem:Ab}
Assume that the setting in 
Section~\ref{sec:convergencesetting}
is fulfilled, 
let 
$ 
  A \colon \mathbb{R}^d \times [0,T] 
  \times \mathbb{R}^m \rightarrow   
  \mathbb{R}^{ d \times d } 
$ 
and 
$ 
  b \colon \mathbb{R}^d 
  \times [0,T] \times 
  \mathbb{R}^m \rightarrow 
  \mathbb{R}^d 
$ 
be Borel measurable functions 
and let $ t_K \in (0, T] $, 
$ K \subset D $ non-empty compact set, 
be a family of real numbers 
such that 
$ 
  I - A( x, t, y ) 
  \in \mathbb{R}^{ d \times d } 
$ 
is invertible 
for all 
$ (x,t,y) \in K \times [0, t_K ] \times \mathbb{R}^m $ 
and all non-empty compact sets $ K \subset D $. 
Moreover, assume that 
\begin{equation}
  \mathbb{R}^d \times [0,T] 
  \times \mathbb{R}^m 
  \ni
  (x,t,y)
  \mapsto
  A( x, t, y ) \, x + b( x, t, y ) 
  \in \mathbb{R}^d
\end{equation}
is 
$ \left( \mu, \sigma \right) 
$-consistent with respect to Brownian motion,
that
\begin{equation}
  \phi( x, t, y ) = 
  \big( I - A( x, t, y ) 
  \big)^{ \! -1 }
  \big( x + b(x, t, y ) \big) 
  - x 
\end{equation}
for all $ (x,t,y) \in K \times [0,t_K] \times \mathbb{R}^m $
and all non-empty compact sets $ K \subset D $
and that
\begin{equation}
\label{eq:propA_1}
  \limsup_{ t \searrow 0 }\left(
    \frac{ 1 }{ t }   
    \cdot
    \sup_{ x \in K }
    \mathbb{E}\!\left[
      \big\|
        ( I - A( x, t, W_t ) 
        )^{ - 1 } 
        A( x, t, W_t )
      \big\|^2
    \right]
  \right)
  = 0 ,
\end{equation}
\begin{equation}
\label{eq:propA_2}
  \limsup_{ t \searrow 0 }\left(
    \frac{ 1 }{ t }   
    \cdot
    \sup_{ x \in K }
    \mathbb{E}\!\left[
      \big\|
        A( x, t, W_t ) \, x + b( x, t, W_t )
      \big\|^2
    \right]
  \right)
  < \infty
\end{equation}
for all non-empty compact sets $ K \subset D $.
Then
$ 
  \phi \colon \mathbb{R}^d 
  \times [0,T] \times 
  \mathbb{R}^m \rightarrow 
  \mathbb{R}^d 
$ 
is 
$ \left( \mu, \sigma \right) $-consistent with respect to Brownian motion.
\end{lemma}

\begin{proof}[Proof of 
Lemma~\ref{lem:Ab}]
Note that
\begin{equation}
\label{eq:propA_3}
\begin{split}
&
  \phi( x, t, y ) - 
  \left(
    A( x, t, y ) \, x + b( x, t, y )
  \right)
\\ & =
  \left(
    \left( I - A( x, t, y ) \right)^{-1} 
    -
    I
  \right)
  \left( x + b( x, t, y ) \right)
  -
  A( x, t, y ) \, x
\\ & =
  \left( 
    I - A( x, t, y ) 
  \right)^{-1}
  A( x, t, y )
  \left( x + b( x, t, y ) \right)
  -
  A( x, t, y ) \, x
\\ & =
  \left(
    \left( I - A( x, t, y ) \right)^{-1}
    -
    I
  \right)
  A( x, t, y ) \, x
  +
  \left( I - A( x, t, y ) \right)^{-1}
  A( x, t, y ) \,
  b( x, t, y )
\\ & =
  \left( I - A( x, t, y ) \right)^{-1}
  A( x, t, y )
  \left(
    A( x, t, y ) \, x
    +
    b( x, t, y )
  \right)
\end{split}  
\end{equation}
for all $ (x,t,y) \in K \times [0,t_K] \times \R^m $ 
and all non-empty compact sets
$ K \subset D $.
Combining 
equation~\eqref{eq:propA_3}, 
H{\"o}lder's inequality, equation~\eqref{eq:propA_1}, inequality~\eqref{eq:propA_2} 
and
Lemma~\ref{lem:phiphi} then 
completes the proof of 
Lemma~\ref{lem:Ab}.
\end{proof}

\subsection{Fully drift-implicit numerical schemes for SDEs}

The next lemma gives 
sufficient conditions to ensure that
the fully drift-implicit Euler scheme
described by
\begin{equation}
\begin{split}
  Y_{ n + 1 }^N
& =
  Y_n^N
  +
  \mu\!\left( 
    Y^N_{ n + 1 }
  \right)
  \tfrac{ T }{ N }
  +
  \sigma\!\left(
    Y^N_n
  \right)
  \Delta W^N_n
\end{split}
\end{equation}
for all $ n \in \{ 0, 1, \ldots, N-1 \} $ 
and all sufficiently large $ N \in \mathbb{N} $
is
$ \left( \mu, \sigma \right) $-consistent with respect to Brownian motion
where
$ D = \R^d $
and 
where
$ 
  \mu \colon \R^d \to \R^d 
$ 
satisfies a one-sided linear growth bound 
(see 
Lemma~\ref{lem:full_implicit}
for the detailed assumptions).

\begin{lemma}[Fully drift-implicit
Euler-Maruyama scheme for SDEs
driven by Brownian motions]
\label{lem:full_implicit}
Assume that the setting in 
Section~\ref{sec:convergencesetting}
is fulfilled, 
assume that $ D = \R^d $,
let $ c, \kappa \in [0,\infty) $, $ \theta \in (0,T] $ 
be real numbers
with 
$
  \left< x, \mu(x) \right>
  \leq 
  c 
  \left( 
    1 + \| x \|^2
  \right)
$
and
$
  \left\| 
    \mu(x) - \mu(y)
  \right\|
  \leq
  c
  \left( 
    1 
    + \left\| x \right\|^{ \kappa }
    + \left\| y \right\|^{ \kappa }
  \right)
  \left\| x - y \right\|
$
for all
$ x, y \in \R^d $
and assume that
\begin{equation}
  \phi( x, t, y ) = 
  \mu\big(
    x + \phi(x,t,y)
  \big)
  \, t
  +
  \sigma(x) \, y
\end{equation}
for all $ (x,t,y) \in \R^d \times [0,\theta] \times \R^m $.
Then
$ 
  \phi \colon \mathbb{R}^d 
  \times [0,T] \times 
  \mathbb{R}^m \rightarrow 
  \mathbb{R}^d 
$ 
is 
$ \left( \mu, \sigma \right) $-consistent with respect to Brownian motion.
\end{lemma}

\begin{proof}[Proof of 
Lemma~\ref{lem:full_implicit}]
First of all, observe that
the Cauchy-Schwarz inequality 
and the assumption that
$
  \left< v, \mu(v) \right>
  \leq 
  c \left( 1 + \| v \|^2 \right)
$
for all $ v \in \R^d $
imply that
\begin{equation}
\begin{split}
&
  \left\| x + \phi( x, t, y )
  \right\|^2
  =
  \left< 
    x + \phi( x, t, y ), 
    x + \phi( x, t, y )
  \right>
\\ & =
  \left< x + \phi( x, t, y ), x + \sigma(x) y \right>
  +
  \left< x + \phi( x, t, y ), \mu( x + \phi(x,t,y) ) \right> t
\\ & \leq
  \left\| x + \phi( x, t, y ) \right\| \left\| x + \sigma( x ) y \right\|
  +
  c t
  \left( 1 + \| x + \phi( x, t, y ) \|^2 \right)
\end{split}
\end{equation}
and Young's inequality hence gives that
\begin{equation}
\begin{split}
&
  \left\| x + \phi( x, t, y )
  \right\|^2
\\ & \leq
  \left\| x + \phi( x, t, y ) \right\| \left\| x + \sigma( x ) y \right\|
  +
  c t
  +
  c t
  \left\| x + \phi( x, t, y ) \right\|^2 
\\ & \leq
  \tfrac{ 1 }{ 2 }
  \left\| x + \phi( x, t, y ) \right\|^2 
  +
  \tfrac{ 1 }{ 2 }
  \left\| x + \sigma( x ) y \right\|^2
  +
  c t
  +
  c t
  \left\| x + \phi( x, t, y ) \right\|^2 
\end{split}
\end{equation}
for all $ (x, t, y) \in \R^d \times [0,\theta] \times \R^m $.
Rearranging therefore yields
\begin{equation}
\begin{split}
&
  \left(
    1 - c t - \tfrac{ 1 }{ 2 }
  \right)
  \left\| x + \phi( x, t, y )
  \right\|^2
\leq
  \tfrac{ 1 }{ 2 }
  \left\| x + \sigma( x ) y \right\|^2
  +
  c t
\end{split}
\end{equation}
for all $ (x, t, y) \in \R^d \times [0,\theta] \times \R^m $
and this shows that
\begin{equation}
\begin{split}
&
  \left\| x + \phi( x, t, y )
  \right\|^2
\leq
  \frac{
    \tfrac{ 1 }{ 2 }
    \left\| x + \sigma( x ) y \right\|^2
    +
    c t
  }{
    \left(
      1 - c t - \tfrac{ 1 }{ 2 }
    \right)
  }
\leq
  \frac{
    \left\| x \right\|^2 
    + 
    \left\| \sigma( x ) y \right\|^2
    +
    c T
  }{
    \left(
      1 - c t - \tfrac{ 1 }{ 2 }
    \right)
  }
\\ & \leq
  \frac{
    \left\| x \right\|^2 
    + 
    \left\| \sigma( x ) 
    \right\|^2_{
      L( \R^m , \R^d )
    } 
    \left\| y \right\|^2
    +
    c T
  }{
    \left(
      1 - c t - \tfrac{ 1 }{ 2 }
    \right)
  }
\end{split}
\end{equation}
for all $ (x, t, y) \in \R^d \times [0,\theta] \times \R^m $
with $ c t < \frac{ 1 }{ 2 } $
and hence
\begin{equation}
\label{eq:implicit_x_phi_bound}
\begin{split}
&
  \limsup_{ t \searrow 0 }
  \left(
  \E\!\left[
  \sup_{ x \in K }
  \left\| x + \phi( x, t, W_t )
  \right\|^r
  \right]
  \right)
  < \infty
\end{split}
\end{equation}
for all $ r \in [0,\infty) $
and all non-empty compact sets $ K \subset \R^d $.
Combining this with the estimate
\begin{equation}
\begin{split}
&
  \left\|
    \mu(x) 
  \right\|
\leq
  \left\|
    \mu(x)
    - \mu(0)
  \right\|
  +
  \left\|
    \mu(0) 
  \right\|
\leq
  c
  \left(
    2 + \left\| x \right\|^{ \kappa }
  \right)
  \left\| x \right\|
  +
  \left\|
    \mu(0) 
  \right\|
\\ & \leq
  c
  \left(
    2 
    \left\| x \right\|
    + \left\| x \right\|^{ ( \kappa + 1 ) }
  \right)
  +
  \left\|
    \mu(0) 
  \right\|
\leq
  3 c
  \left(
    1 
    + \left\| x \right\|^{ ( \kappa + 1 ) }
  \right)
  +
  \left\|
    \mu(0) 
  \right\|
\\ & \leq
  \left(
    3 c +
    \left\|
      \mu(0) 
    \right\|
  \right)
  \left(
    1 
    + \left\| x \right\|^{ ( \kappa + 1 ) }
  \right)
\end{split}
\end{equation}
for all $ x \in \R^d $
results in
\begin{equation}
\label{eq:implicit_mu_x_phi_bound}
\begin{split}
&
  \limsup_{ t \searrow 0 }
  \left(
  \E\!\left[
    \sup_{ x \in K }
    \left\| 
      \mu\big(
        x + \phi( x, t, W_t )
      \big)
    \right\|^r
  \right]
  \right)
\\ & \leq
  \left(
    6 c 
    +
    2
    \left\|
      \mu(0) 
    \right\|
  \right)^r
  \left(
    1 
    + 
    \limsup_{ 
      t \searrow 0
    }
    \E\!\left[
    \sup_{ x \in K }
    \left\| 
      x + \phi( x, t, W_t )
    \right\|^{ r ( \kappa + 1 ) }
    \right]
  \right)
  < \infty
\end{split}
\end{equation}
for all $ r \in [0,\infty) $
and all non-empty compact sets $ K \subset \R^d $.
This implies that
\begin{equation}
\label{eq:full_implicit_phi_0}
\begin{split}
&
  \limsup_{ t \searrow 0 }
  \left(
  \E\!\left[
    \sup_{ x \in K }
    \left\| 
      \phi( x, t, W_t )
    \right\|^r
  \right]
  \right)
\\ & =
  \limsup_{ t \searrow 0 }
  \left(
  \E\!\left[
    \sup_{ x \in K }
    \left\| 
      \mu\big(
        x +
        \phi( x, t, W_t )
      \big)
      \, t
      +
      \sigma( x ) W_t
    \right\|^r
  \right]
  \right)
\\ & \leq
  2^r
  \limsup_{ t \searrow 0 }
  \left(
  t^r \cdot
  \E\!\left[
    \sup_{ x \in K }
    \left\| 
      \mu\big(
        x +
        \phi( x, t, W_t )
      \big)
    \right\|^r
  \right]
  \right)
\\ & 
  +
  2^r
  \limsup_{ t \searrow 0 }
  \left(
  \sup_{ x \in K }
  \left\|
    \sigma( x ) 
  \right\|_{
    L( \R^m, \R^d )
  }^r
  \E\big[
    \left\|
      W_t
    \right\|^r
  \big]
  \right)
= 0
\end{split}
\end{equation}
for all $ r \in ( 0, \infty ) $ 
and all non-empty compact sets $ K \subset \R^d $.
In addition, \eqref{eq:implicit_mu_x_phi_bound}
shows that
\begin{equation}
\label{eq:full_implicit_consisA}
\begin{split}
&
  \lim_{ 
    t \searrow 0 
  }
  \left(
    \tfrac{ 1 }{ \sqrt{ t } }
    \cdot
    \sup_{ x \in K }
    \E\big[
      \left\|
        \phi( x, t, W_t )
        -
        \sigma(x) W_t
      \right\|
    \big]
  \right)
\\ & =
  \lim_{ 
    t \searrow 0 
  }
  \left(
    \sqrt{ t } 
    \cdot
    \sup_{ x \in K }
    \E\big[
      \left\|
        \mu( x + \phi( x, t, W_t ) )
      \right\|
    \big]
  \right)
\\ & \leq
  \left(
    \lim_{ t \searrow 0 }
    \sqrt{ t }
  \right)
  \left(
    \limsup_{ 
      t \searrow 0 
    }
    \sup_{ x \in K }
    \E\big[
      \left\|
        \mu( x + \phi( x, t, W_t ) )
      \right\|
    \big]
  \right)
  = 0
\end{split}
\end{equation}
for all non-empty compact sets $ K \subset \R^d $.
In the next step we note that
\begin{equation}
\begin{split}
&
  \limsup_{ t \searrow 0 }
  \sup_{ x \in K }
  \left\|
    \mu( x )
    -
    \tfrac{ 1 }{ t } 
    \cdot
    \E\big[ \phi( x, t, W_t ) \big] 
  \right\|
\\ & =
  \limsup_{ t \searrow 0 }
  \sup_{ x \in K }
  \left\|
    \E\big[   
      \mu( x )
      -
      \mu\big( 
        x + \phi( x, t, W_t ) 
      \big)
    \big] 
  \right\|
\\ & \leq
  \limsup_{ t \searrow 0 }
  \E\!\left[   
    \sup_{ x \in K }
    \left\|
      \mu( x )
      -
      \mu\big( 
        x + \phi( x, t, W_t ) 
      \big)
    \right\|
  \right] 
\\ & \leq
  \limsup_{ t \searrow 0 }
  \E\!\left[  
    c
    \left(
      1 
      +
      \sup_{ x \in K }
      \| x \|^{ \kappa }
      +
      \sup_{ x \in K }
      \left\|
        x + \phi( x, t, W_t ) 
      \right\|^{ \kappa }
    \right)
    \sup_{ x \in K }
    \left\|
      \phi( x, t, W_t ) 
    \right\|
  \right] 
\end{split}
\end{equation}
and H\"{o}lder's inequality 
together with 
\eqref{eq:implicit_x_phi_bound}
and 
\eqref{eq:full_implicit_phi_0} hence 
implies that
\begin{equation}
\label{eq:full_implicit_consisB}
\begin{split}
&
  \limsup_{ t \searrow 0 }
  \sup_{ x \in K }
  \left\|
    \mu( x )
    -
    \tfrac{ 1 }{ t } 
    \cdot
    \E\big[ \phi( x, t, W_t ) \big] 
  \right\|
\\ & \leq
  3 c
  \left(
  \limsup_{ t \searrow 0 }
  \E\!\left[  
    \left(
      1 
      +
      \sup_{ x \in K }
      \| x \|^{ 2 \kappa }
      +
      \sup_{ x \in K }
      \left\|
        x + \phi( x, t, W_t ) 
      \right\|^{ 2 \kappa }
    \right)
  \right]
  \right)^{ \! \frac{ 1 }{ 2 } }
\\ & \quad 
  \cdot
  \left(
  \limsup_{ t \searrow 0 }  
  \E\!\left[
    \sup_{ x \in K }
    \left\|
      \phi( x, t, W_t ) 
    \right\|^2
  \right] 
  \right)^{ \! \frac{ 1 }{ 2 } }
= 0
\end{split}  
\end{equation}
for all non-empty compact sets
$ K \subset \R^d $.
Combining 
\eqref{eq:full_implicit_consisA} and 
\eqref{eq:full_implicit_consisB}
completes the proof of 
Lemma~\ref{lem:full_implicit}.
\end{proof}

\chapter{Examples of SDEs}
\label{sec:examples}

In this chapter,
we apply the strong convergence results of
Chapter~\ref{chap:convergence}
to a selection of examples
of SDEs with non-globally Lipschitz
continuous coefficients.
In Section~\ref{sec:setting00},
we first describe 
the general setting in which
these examples appear 
and
then treat each example
separately in the subsequent 
sections.


\section{Setting and
assumptions}
\label{sec:setting00}

%
%
The following setting is used
throughout
Chapter~\ref{sec:examples}.
Let $ T \in (0,\infty) $, 
$ d, m \in \mathbb{N} $,
let 
$ ( \Omega, \mathcal{F}, \mathbb{P} ) $
be a probability space
with a normal filtration
$ ( \mathcal{F}_t )_{ t \in [0,T] } $
and 
let 
$ 
  W = ( W^{(1)}, \dots, W^{(m)} )
  \colon [0,T] \times \Omega
  \rightarrow \mathbb{R}^m 
$
be a standard 
$ ( \mathcal{F}_t )_{ t \in [0,T] } 
$-Brownian motion.
Moreover, let 
$ D \subset \mathbb{R}^d $ 
be an open set,
let 
$ 
  \mu \colon \mathbb{R}^d
  \rightarrow \mathbb{R}^d 
$
and
$ 
  \sigma \colon 
  \mathbb{R}^d
  \rightarrow \mathbb{R}^{ d \times m }
$
be Borel measurable functions 
such that
$ 
  \mu|_D \colon D 
  \rightarrow \mathbb{R}^d 
$
and
$
  \sigma|_D
  \colon 
  D
  \rightarrow \mathbb{R}^{ d \times m }
$ 
are locally
Lipschitz continuous
and let
$
   X = ( X^{(1)}, \dots, X^{(d)} )
  \colon [0,T] \times \Omega
  \rightarrow D 
$ be an
$ ( \mathcal{F}_t )_{ t \in [0,T] } 
$-adapted stochastic
process with continuous
sample paths satisfying
$
  \mathbb{E}\big[
    \| X_0 \|^p
  \big] < \infty
$
for all $ p \in [1,\infty) $
and
\begin{equation}
\label{eq:SDE00}
  X_t
=
  X_0
+
  \int_0^t 
    \mu( X_s )
  \, ds
+
  \int_0^t 
    \sigma( X_s )
  \, dW_s
\end{equation}
$ \mathbb{P} $-a.s.\ for
all $ t \in [0,T] $.
Thus we assume the existence of a solution process
of the SDE~\eqref{eq:SDE00}.
%
%
Our goal is then to approximate
the solution process
$
  X \colon [0,T] \times \Omega
  \rightarrow D 
$
of the SDE~\eqref{eq:SDE00}
in the strong sense.
For this we concentrate
for simplicity
on the increment-tamed 
Euler-Maruyama
approximations in 
Subsection~\ref{sec:strongCtamed}.
More precisely, let
$ \bar{Y}^N \colon 
[0,T] \times \Omega
\rightarrow \mathbb{R}^d $,
$ N \in \mathbb{N} $,
be a sequence of 
stochastic processes
defined through 
$ 
  \bar{Y}^N_0 := X_0 
$ 
and
\begin{equation}
\label{eq:scheme00}
  \bar{Y}^N_{ t }
=
  \bar{Y}^N_{ 
    \frac{ n T }{ N }
  }
+
  \frac{
  \left(
    \tfrac{ t N }{ T } - n
  \right) 
  \big(
    \mu( 
      \bar{Y}^N_{ n T / N } 
    ) 
    \frac{ T }{ N }
    + 
    \sigma( 
      \bar{Y}^N_{ n T / N } 
    ) 
    \,
    (
      W_{ (n + 1) T / N } 
      -
      W_{ n T / N } 
    )
  \big)
  }{
    \max\!\big( 1, 
      \frac{ T }{ N }
      \| 
        \mu( 
          \bar{Y}^N_{ n T / N } 
        ) 
        \frac{ T }{ N }
        + 
        \sigma( 
          \bar{Y}^N_{ n T / N } 
        ) \,
        (
          W_{ (n + 1) T / N } 
          -
          W_{ n T / N } 
        )
      \| 
    \big)
  }
\end{equation}
for all 
$ 
  t \in 
  \big(
    \frac{ n T }{ N },
    \frac{ (n+1) T }{ N }
  \big] 
$,
$ n \in \{ 0, 1, \dots, N-1 \} $
and all $ N \in \mathbb{N} $.
In the next sections,
we present several
examples from the
literature for the
SDE~\eqref{eq:SDE00} 
and we show under suitable
assumptions that the 
numerical approximation
processes~\eqref{eq:scheme00}
converge strongly to the solution
process of each of the following 
examples.
To the best of our knowledge,
this is the first result in the 
literature that proves strong 
convergence 
for the stochastic van der Pol oscillator~\eqref{eq:ex_Vanderpol}, 
for the stochastic Duffing-van der Pol oscillator~\eqref{eq:ex_Duffing},
for the stochastic Lorenz equation~\eqref{eq:Lorenz},
for the stochastic Brusselator~\eqref{eq:ex_Brusselator},
for the stochastic SIR model~\eqref{eq:ex_SIR},
for the SDE~\eqref{eq:experimental.psychology}
from experimental psychology
and for the Lotka-Volterra 
predator-prey model~\eqref{eq:ex_Volterra}.


\section{Stochastic van der Pol oscillator}
\label{sec:vanderpol}

The van der Pol oscillator 
was proposed 
to describe
stable oscillation;
see van der 
Pol~\cite{VanDerPol1926} 
and the references therein.
As, for instance, 
in Timmer 
et.\ al~\cite{TimmerEtAl2000},  
we consider
a stochastic version with additive noise
(see also 
Leung~\cite{Leung1995} 
for more general 
stochastic versions of the van 
der Pol oscillator and also 
equation~(4.1)
in Schurz~\cite{Schurz2003}
for a generalized stochastic
van der Pol oscillator with
multiplicative noise).
More formally, assume
that the setting in 
Section~\ref{sec:setting00}
is fulfilled,
let
$ 
  \alpha, \beta, \gamma, \delta
  \in (0,\infty)
$
be real numbers
and suppose that
$ d = 2 $,
$ m = 1 $,
$ D = \mathbb{R}^2 $
and 
\begin{equation}
  \mu\!\left( 
    \begin{array}{c}
      x_1
      \\ 
      x_2 
    \end{array}
  \right)
=
  \left( 
    \begin{array}{c}
        x_2
      \\
        \alpha 
        \left( \gamma - ( x_1 )^2 \right)
        x_2 
      - \delta x_1 
    \end{array}
  \right) ,
\qquad
  \sigma\!\left( 
    \begin{array}{c}
      x_1
      \\ 
      x_2 
    \end{array}
  \right)
=
  \left(
    \begin{array}{ccc}
      0 
    \\
      \beta  
    \end{array}
  \right)
\end{equation}
for all
$ 
  x = (x_1, x_2)
  \in \mathbb{R}^2
$.
Then the SDE~\eqref{eq:SDE00}
is the
stochastic 
van der Pol equation
\begin{equation}
\label{eq:ex_Vanderpol}
\begin{split}
  dX_t^{(1)}
  & =
  X_t^{(2)} \, dt,
\\
  dX_t^{(2)}
 &=
  \left[
    \alpha
    \left( 
      \gamma -
      ( 
        X_t^{(1)} 
      )^2 
    \right) \!
    X_t^{(2)}
    - 
    \delta X_t^{(1)}
  \right] dt
  +
  \beta \, dW_t
\end{split}
\end{equation}
for 
$ t \in [0,\infty) $.
In an abbreviated form, 
the SDE~\eqref{eq:ex_Vanderpol}
can also be written as
\begin{equation}
  \ddot{X}_t
  -
  \alpha 
  \left(
    \gamma - (X_t)^2
  \right) \!
  \dot{X}_t
  +
  \delta X_t
  = 
  \beta \dot{W}_t
\end{equation}
for 
$ t \in [0,\infty) $.
Clearly, the diffusion coefficient
$ 
  \sigma \colon \R^2 \to \R^2 
$
in 
\eqref{eq:ex_Vanderpol}
is globally Lipschitz 
continuous and 
the drift coefficient 
$ \mu \colon \R^2 \to \R^2 $ in 
\eqref{eq:ex_Vanderpol}
is not 
globally Lipschitz 
continuous.
It is also well known that
the drift coefficient 
$ \mu \colon \R^2 \to \R^2 $ in 
\eqref{eq:ex_Vanderpol}
is not
globally one-sided
Lipschitz continuous.
Indeed, note that
\begin{equation}
\begin{split}
&
  \left< 
    \left(
      \begin{array}{c}
        u \\ u
      \end{array}
    \right)
    -
    \left(
      \begin{array}{c}
        0 \\ 2 u
      \end{array}
    \right)
    ,
    \mu\!\left(
      \begin{array}{c}
        u \\ u
      \end{array}
    \right)
    -
    \mu\!\left(
      \begin{array}{c}
        0 \\ 2 u
      \end{array}
    \right)
  \right>
\\ & =
  \left< 
    \left(
      \begin{array}{c}
        u \\ - u
      \end{array}
    \right)
    ,
    \left(
      \begin{array}{c}
        u \\ \alpha \left( \gamma - u^2 \right) u - \delta u
      \end{array}
    \right)
    -
    \left(
      \begin{array}{c}
        2 u \\ 2 \alpha \gamma u
      \end{array}
    \right)
  \right>
\\ & =
  \left< 
    \left(
      \begin{array}{c}
        u \\ - u
      \end{array}
    \right)
    ,
    \left(
      \begin{array}{c}
        - u \\ - \alpha \gamma u - \alpha u^3 - \delta u 
      \end{array}
    \right)
  \right>
  =
  \left( \alpha \gamma + \delta - 1 \right) u^2 
  + \alpha u^4
\end{split}
\end{equation}
for all $ u \in \R $
and hence
there exists no real number 
$ c \in \R $ such that it holds for every $ x, y \in \R^2 $ that
\begin{equation}
\label{eq:one_sided_Lipschitz}
  \left< x - y, \mu( x ) - \mu( y ) \right>
  \leq
  c \left\| x - y \right\|^2 .
\end{equation}
However, the drift
coefficient in 
\eqref{eq:ex_Vanderpol}
satisfies the global
one-sided linear growth
condition
\begin{equation}
\label{eq:onesidedgrowth}
  \left< x, \mu(x) \right>
\leq
  c \left( 1 + \| x \|^2 \right)
\end{equation}
for all $ x \in \R^d $
and some $ c \in [0,\infty) $
(see, e.g., Example~16 in Subsection~3.2.2 in Boccara~\cite{boccara2010}).
Therefore,
Corollary~\ref{cor:CT1} applies
here with the Lyapunov-type 
function
$
  V \colon \mathbb{R}^2
  \rightarrow [1,\infty)
$
given by
\begin{equation}
  V(x) = 
  1 + \| x \|^2
\end{equation}
for all 
$ x = (x_1, x_2) \in \mathbb{R}^2 $
(cf., e.g., page~75 in Boccara~\cite{boccara2010})
and we obtain
\begin{equation}
  \lim_{ N \rightarrow \infty }
    \sup_{ t \in [0,T] }
    \mathbb{E}\big[
      \| 
        X_t - \bar{Y}^N_t 
      \|^{ 
        q
      }
    \big] 
  = 0
\end{equation}
for all 
$
  q \in (0,\infty)
$.

\section{Stochastic 
Duffing-van 
der Pol oscillator}

The Duffing equation is a further model for an oscillator.
The Duffing-van der 
Pol equation unifies 
both the Duffing
equation and 
the van der Pol equation 
and has been
used, e.g., in 
certain aeroelasticity 
problems;
see
Holmes 
\citationand\ Rand~\cite{hr80}. 
As, for instance, in 
Schenk-Hopp{\'e}~\cite{SchenkHoppe1996}, 
we consider a 
stochastic version with an 
affine-linear noise 
term
(see, e.g., also 
Arnold, Namachchivaya 
\citationand\ Schenk-Hopp{\'e}~\cite{ass96},
Section~9.4
in Arnold~\cite{a98} 
and Section~13.1
in Kloeden \citationand\ 
Platen~\cite{kp92};
we also refer to
\cite{ass96} for
details and references on the 
physical background
of the stochastic Duffing-van der Pol oscillator). 
More precisely, assume
that the setting in 
Section~\ref{sec:setting00}
is fulfilled,
let
$ 
  (\alpha_1, \alpha_2, \alpha_3),
  (\beta_1, \beta_2, \beta_3)
  \in \mathbb{R}^3
$
be two triples of real numbers
with $\alpha_3\geq0$
and suppose that
$ d = 2 $,
$ m = 3 $,
$ D = \mathbb{R}^2 $
and 
\begin{equation}
  \mu\!\left( 
    \begin{array}{c}
      x_1
      \\ 
      x_2 
    \end{array}
  \right)
=
  \left( 
    \begin{array}{c}
      x_2
      \\ 
      \alpha_1 x_1 - \alpha_2 x_2 
      - \alpha_3 x_2 \left( x_1 \right)^2
      - \left( x_1 \right)^3
    \end{array}
  \right) 
\end{equation}
and
\begin{equation}
  \sigma\!\left( 
    \begin{array}{c}
      x_1
      \\ 
      x_2 
    \end{array}
  \right)
=
  \left(
    \begin{array}{ccc}
      0 & 0 & 0 
    \\
      \beta_1 x_1 & \beta_2 x_2 & \beta_3 
    \end{array}
  \right)
\end{equation}
for all
$ 
  x = (x_1, x_2)
  \in \mathbb{R}^2
$.
Then the SDE~\eqref{eq:SDE00}
is the stochastic 
Duffing-van der Pol equation
\begin{equation}
\label{eq:ex_Duffing}
\begin{split}
  dX_t^{(1)}
 &=
  X_t^{(2)} \, dt,
\\
  dX_t^{(2)}
 &=
  \left[
    \alpha_1 
    X_t^{(1)}
    -
    \alpha_2 
    X_t^{(2)}
    -
    \alpha_3 X_t^{(2)}
    \big( X_t^{(1)} \big)^{ 2 }
    -
    \big( X_t^{(1)} \big)^{ 3 }
  \right] dt
\\&\quad+
  \beta_1 X_t^{(1)} dW_t^{(1)}
  +
  \beta_2 X_t^{(2)} dW_t^{(2)}
  +
  \beta_3 \, dW_t^{(3)}
\end{split}
\end{equation}
for 
$ t \in [0,\infty) $.
In an abbreviated form, 
the SDE~\eqref{eq:ex_Duffing}
can also be written as
\begin{equation}
  \ddot{X}_t
  -
  \alpha_1 X_t
  +
  \alpha_2 \dot{X}_t
  +   
  \alpha_3 
  \dot{X}_t 
  (X_t)^2
  + ( X_t )^3
  = 
  \beta_1 
  X_t 
  \dot{W}_t^{(1)}
  +
  \beta_2
  \dot{X}_t 
  \dot{W}_t^{(2)}
  +
  \beta_3
  \dot{W}_t^{(3)}
\end{equation}
for 
$ t \in [0,\infty) $.
Clearly, the diffusion coefficient
$ \sigma $
in \eqref{eq:ex_Duffing}
is globally Lipschitz
continuous.
In addition, it is well know that
the drift 
coefficient $ \mu $ 
in \eqref{eq:ex_Duffing}
is not
globally one-sided
Lipschitz continuous
and also fails to
satisfy the global one-sided
linear growth
condition~\eqref{eq:onesidedgrowth}.
Indeed, note that
\begin{equation}
\begin{split}
&
  \left<  
    \left(
      \begin{array}{c}
        u \\ - 1
      \end{array}
    \right)
    ,
    \mu\!\left(
      \begin{array}{c}
        u \\ - 1
      \end{array}
    \right)
  \right>
  =
  \left<  
    \left(
      \begin{array}{c}
        u \\ - 1
      \end{array}
    \right)
    ,
    \left(
      \begin{array}{c}
        - 1 \\ \alpha_1 u + \alpha_2 + \alpha_3 u^2 - u^3
      \end{array}
    \right)
  \right>
\\ & =
  - \left( 1 + \alpha_1 \right) u  
  - \alpha_2 - \alpha_3 u^2 + u^3
\end{split}
\end{equation}
for all $ u \in \R $
and hence that there exists no real number $ c \in \R $
such that it holds for every $ x \in \R^2 $ that
$
  \left< x, \mu(x) \right>
\leq
  c \left( 1 + \| x \|^2 \right) 
$.
The function 
\begin{equation}
  \R^2 \ni (x_1, x_2) \mapsto 
  1 + \left( x_1 \right)^2 + 
  \left( x_2 \right)^2 \in [1,\infty) 
\end{equation}
is thus no Lyapunov-type function
for the SDE~\eqref{eq:ex_Duffing}.
However,
Corollary~\ref{cor:CT1} applies
here with the 
Lyapunov-type function
$
  V \colon \mathbb{R}^2
  \rightarrow [1,\infty)
$
given by
\begin{equation}
  V(x_1,x_2) = 
  1 + \left( x_1 \right)^4
  + 
  2 \left( x_2 \right)^2
\end{equation}
for all 
$ x = (x_1, x_2) \in \mathbb{R}^2 $
(cf., e.g., (8) in 
Holmes \citationand\ Rand~\cite{hr80})
and we hence obtain
$
  \lim_{ N \rightarrow \infty }
    \sup_{ t \in [0,T] }
    \mathbb{E}\big[
      \| 
        X_t - \bar{Y}^N_t 
      \|^{ 
        q
      }
    \big] 
  = 0
$
for all 
$
  q \in (0,\infty)
$.

\section{Stochastic Lorenz equation}

Lorenz~\cite{Lorenz1963} 
derived a 
three-dimensional 
system
as a simplified model of convection rolls in the atmosphere
and this equation became famous for its chaotic behaviour.
As, for instance, in Schmalfu\ss~\cite{Schmalfuss1997}, 
we consider a stochastic
version hereof with multiplicative noise.
Assume
that the setting in 
Section~\ref{sec:setting00}
is fulfilled,
let
$ 
  (\alpha_1, \alpha_2, \alpha_3),
  (\beta_1, \beta_2, \beta_3)
  \in \mathbb{R}^3
$
be two triples of real numbers
and suppose that
$ d = m = 3 $,
$ D = \mathbb{R}^3 $
and
\begin{equation}
  \mu\!\left( 
    \begin{array}{c}
      x_1
      \\ 
      x_2 
      \\ 
      x_3 
    \end{array}
  \right)
=
  \left( 
    \begin{array}{c}
      \alpha_1 \left( x_2 - x_1 \right)
      \\ 
      \alpha_2 x_1 - x_2 - x_1 x_3
      \\ 
      x_1 x_2 - \alpha_3 x_3 
    \end{array}
  \right) 
\end{equation}
and
\begin{equation}
  \sigma\!\left( 
    \begin{array}{c}
      x_1
      \\ 
      x_2 
      \\ 
      x_3 
    \end{array}
  \right)
=
  \left(
    \begin{array}{ccc}
      \beta_1 x_1 & 0 & 0 
    \\
      0 & \beta_2 x_2 & 0 
    \\
      0 & 0 & \beta_3 x_3 
    \end{array}
  \right)
\end{equation}
for all
$ 
  x = (x_1, x_2, x_3)
  \in \mathbb{R}^3
$.
Under these assumptions,
the SDE~\eqref{eq:SDE00}
is thus the
stochastic Lorenz equation
\begin{equation}
\begin{split}
\label{eq:Lorenz}
   dX_t^{(1)}
  & =
  \left[
   \alpha_1 
   X_t^{(2)} - 
   \alpha_1
    X_t^{(1)} 
  \right] dt 
  +
  \beta_1
    X_t^{(1)} 
  \, dW^{(1)}_t,
 \\
   dX_t^{(2)}
  &=
   \left[
     \alpha_2 X_t^{(1)}
     -
     X_t^{(2)}
     -
     X_t^{(1)}
     X_t^{(3)}
   \right] dt
   + \beta_2
     X_t^{(2)}
   \, dW_t^{(2)} ,
 \\
   dX_t^{(3)}
  &=
   \left[
     X_t^{(1)}
     X_t^{(2)}
     -
     \alpha_3 X_t^{(3)}
   \right]
   dt
   +
   \beta_3
   X_t^{(3)}
   \, dW_t^{(3)}
\end{split}
\end{equation}
for $ t \in [0,\infty) $.
Clearly, the diffusion coefficient 
in \eqref{eq:Lorenz}
is globally Lipschitz
continuous.
In addition, it is well know that
the drift coefficient 
in \eqref{eq:Lorenz}
is not globally one-sided
Lipschitz continuous but
fulfills the global
one-sided linear growth
condition~\eqref{eq:onesidedgrowth}
(see, e.g., (29)
in Schmalfu\ss~\cite{Schmalfuss1997}).
Therefore,
Corollary~\ref{cor:CT1} applies
here with 
$
  V \colon \mathbb{R}^3
  \rightarrow [1,\infty)
$
given by
\begin{equation}
  V(x) = 1 + \| x \|^2
\end{equation}
for all $ x \in \mathbb{R}^3 $
to obtain
$
  \lim_{ N \rightarrow \infty }
    \sup_{ t \in [0,T] }
    \mathbb{E}\big[
      \| 
        X_t - \bar{Y}^N_t 
      \|^{ 
        q
      }
    \big] 
  = 0
$
for all 
$
  q \in (0,\infty)
$.

\section{Stochastic Brusselator in the well-stirred case}
\label{sec:Brusselator}

The Brusselator
is a model for a trimolecular chemical reaction
and has been studied in Prigogine \citationand~Lefever~\cite{PrigogineLefever1968} 
and by other scientists from Brussels
(Tyson~\cite{Tyson1973} proposed the name "Brusselator"
to appreciate the innovative contribution of these scientists from Brussel
on this model).
The following stochastic version hereof
in the ``well-stirred case''
(see, e.g., Section~1 in Scheutzow~\cite{Scheutzow1986} for further details)
has been proposed by Dawson~\cite{Dawson1981} 
(see also Scheutzow~\cite{Scheutzow1986}). 
Assume
that the setting in 
Section~\ref{sec:setting00}
is fulfilled,
let
$ \alpha, \delta \in (0,\infty) $
be two real numbers,
let
$ 
  g_1, g_2 \colon
  [0,\infty)
  \rightarrow \mathbb{R}
$
be two globally Lipschitz continuous
functions with 
$ g_1( 0 ) = g_2( 0 ) = 0 $
and 
$
  \sup_{ x \in [0,\infty) }
  | g_2(x) | < \infty
$
and 
suppose that
$ d = m = 2 $,
$ D = (0,\infty)^2 $,
$
  \mu(x) = 0 
$
and
$ \sigma(x) = 0 $
for all $ x \in D^c $
and 
\begin{equation}
  \mu\!\left( 
    \begin{array}{c}
      x_1
      \\ 
      x_2 
    \end{array}
  \right)
=
  \left( 
    \begin{array}{c}
      \delta - 
      \left( \alpha + 1 \right) x_1
      +
      x_2 \cdot ( x_1 )^2
      \\
      \alpha x_1
      - x_2 \cdot
      ( x_1 )^2
    \end{array}
  \right) ,
\quad
  \sigma\!\left( 
    \begin{array}{c}
      x_1
      \\ 
      x_2 
    \end{array}
  \right)
=
  \left(
    \begin{array}{ccc}
      g_1(x_1) & 0
    \\
      0 & g_2(x_2)
    \end{array}
  \right)
\end{equation}
for all
$ 
  x = (x_1, x_2)
  \in D
$.
Then the SDE~\eqref{eq:SDE00}
is the
stochastic Brusselator equation
\begin{equation}  
\label{eq:ex_Brusselator}
\begin{split}
  dX^{(1)}_t
&
  =
  \left[
    \delta- ( \alpha + 1 )
    X^{(1)}_t
    + 
    X^{(2)}_t
    ( X^{(1)}_t )^2 
  \right]
  dt 
  +
  g_1( X_t^{(1)} )
  \, dW^{(1)}_t
\\
  dX^{ (2) }_t
  &
  =
  \left[ 
    \alpha X^{(1)}_t 
    - 
    X^{(2)}_t   
    ( X_t^{(1)} )^2 
  \right] 
  dt 
  +
  g_2( X_t^{(2)} )
  \, dW^{(2)}_t
\end{split}     
\end{equation}
for $ t \in [0,\infty) $.
By assumption the diffusion coefficient
in \eqref{eq:ex_Brusselator}
is globally Lipschitz continuous.
Clearly,
the drift coefficient
in \eqref{eq:ex_Brusselator}
is not globally one-sided
Lipschitz continuous and
also fails to satisfy
the one-sided linear
growth
condition~\eqref{eq:onesidedgrowth}.
Instead here the squared sum 
of the coordinates is a Lyapunov-type function
on the state space $ (0,\infty)^2 $;
see, e.g., the proof of Theorem~2.1 a) in Scheutzow~\cite{Scheutzow1986}.
However,
in order to apply Corollary~\ref{cor:CT1}
in this example,
we need to construct
a Lyapunov-type
function
$ 
  V \colon \mathbb{R}^2
  \rightarrow [1,\infty) 
$
on the space $\mathbb{R}^2$.
For this let
$
  \phi \colon 
  \mathbb{R}
  \rightarrow [0,1]
$
and
$
  \psi \colon
  \mathbb{R}^2
  \rightarrow [0,1]
$
be two infinitely 
often differentiable
functions with
$
  \phi(x) = 0 
$
for all $ x \in (-\infty,0] $,
with
$
  \phi(x) = 1 
$
for all $ x \in [1,\infty) $
and with
\begin{equation}
  \psi(x_1,x_2) = 
  \phi(x_1) \cdot \phi(-x_2)
  +
  \phi(-x_1) \cdot \phi(x_2)
\end{equation}
for all 
$   
  x = (x_1, x_2)   
  \in \mathbb{R}^2 
$.
Then consider 
$ 
  V \colon \mathbb{R}^3
  \rightarrow [1,\infty) 
$
given by
\begin{equation}
\label{eq:BrusselatorV0}
\begin{split}
&
  V( x_1, x_2 ) 
\\ & = 
  \frac{ 5 }{ 2 } 
  + 
  \left( x_1 + x_2 \right)^2
  -
  2 \cdot x_1 \cdot x_2 \cdot
  \psi(x_1,x_2)
\\ & =
  \frac{ 5 }{ 2 } 
  + 
  \left\| x \right\|^2
  +
  2 x_1 x_2 
  \left( 1 - \psi(x_1,x_2) \right)
\\ & =
  \frac{ 5 }{ 2 } +
\begin{cases}
  \left( x_1 + x_2 \right)^2
& 
  \colon 
  x_1 x_2 \geq 0
\\ 
  \left\| x \right\|^2
& 
  \colon 
  \left( 
    x_1 x_2 < 0
  \right)
  \wedge
  \left(
    \min( |x_1|, |x_2| ) \geq 1
  \right)
\\ 
  \left\| x \right\|^2
  +
  2 x_1 x_2 
  \left( 1 - \psi(x_1,x_2) \right)
& 
  \colon 
  \text{else}
\end{cases}
\end{split}
\end{equation}
for all
$ 
  x = (x_1, x_2) 
  \in \mathbb{R}^2 
$
and observe that
$ V \colon \mathbb{R}^2
\rightarrow [1,\infty) $
is infinitely often differentiable 
and satisfies
\begin{equation}
\label{eq:BrusselatorV1}
  \sup_{ x \in \mathbb{R}^3 }
  \left(
  \tfrac{
    \sum_{ i = 1 }^{ 3 }
    \left\|
      V^{ (i) }( x )
    \right\|_{ 
      L^{(i)}( \mathbb{R}^2, \mathbb{R} ) 
    }
  }{
    \left( 1 + \| x \| \right)
  }
  \right)
  < \infty .
\end{equation}
Next note that the fact
$
  2 + a^2 - 2 a 
\geq
  \frac{ 
    a^2 
  }{ 
    2
  }
$
for all $ a \in [0,\infty) $
implies
\begin{equation}
\label{eq:BrusselatorV2}
\begin{split}
&
  V(x_1,x_2)
  =
  \tfrac{ 5 }{ 2 } 
  + 
  \left\| x \right\|^2
  +
  2 x_1 x_2
  \left(
    1
    -
    \psi(x_1,x_2)
  \right)
\\ & \geq
\begin{cases}
  \tfrac{ 5 }{ 2 } 
  + 
  \left\| x \right\|^2
  &
  \colon
  \left( x_1 x_2 \geq 0 \right)
  \vee
  \left( \min(|x_1|, |x_2|) \geq 1 \right)
\\
  \tfrac{ 5 }{ 2 } 
  + 
  \left\| x \right\|^2
  -
  2 \left| x_1 \right| 
  \left| x_2 \right|
  & 
  \colon
  \left( 
    x_1 x_2 < 0 
  \right)
  \wedge
  \left( 
    \min(|x_1|, |x_2|) < 1 
  \right)
\end{cases}
\\ & \geq
  \tfrac{ 5 }{ 2 } 
  + 
  \left\| x \right\|^2
  -
  2 
  \max\!\left(
    | x_1 |, 
    | x_2 |
  \right)
\\ & =  
  \Big[
  \tfrac{ 1 }{ 2 }
  +
  \left|
    \min\!\left(
      | x_1 |, 
      | x_2 |
    \right) 
  \right|^2
  \Big]
  +
  \Big[
  2
  + 
  \left|
    \max\!\left(
      | x_1 |, 
      | x_2 |
    \right) 
  \right|^2
  -
  2 
  \max\!\left(
    | x_1 |, 
    | x_2 |
  \right)
  \Big]
\\ & \geq
  \frac{
    1 
   +  
    \left|
      \min\!\left(
        | x_1 |, 
        | x_2 |
      \right) 
    \right|^2
    +
    \left|
      \max\!\left(
        | x_1 |, 
        | x_2 |
      \right) 
    \right|^2
  }{ 2 }
  =
  \frac{ 1 + \| x \|^2 }{ 2 }
\end{split}
\end{equation}
for all
$ 
  x = (x_1, x_2) 
  \in \mathbb{R}^2 
$.
In addition, 
observe that
$
  \left\|
    V'(x) \sigma(x)
  \right\|_{ 
    L( 
      \mathbb{R}^2, 
      \mathbb{R} 
    ) 
  } = 0
$
for all $ x \in D^c $
and hence
\begin{equation}
\label{eq:BrusselatorV3}
\begin{split} 
&
  \sup_{ x = (x_1, x_2) \in \R^2 }
  \left(
  \frac{
    \left\|
      V'(x) \sigma(x)
    \right\|_{ 
      L( 
        \mathbb{R}^2, 
        \mathbb{R} 
      ) 
    } 
  }{
    V(x)
  }
  \right)
\\ & \leq
  2
  \left[
  \sup_{ x = (x_1, x_2) \in D }
  \left(
  \frac{
  \left( x_1 + x_2 \right)
  \left(
    \left|
      g_1(x_1)
    \right|
    +
    \left|
      g_2(x_1)
    \right|
  \right)
  }{
    \left(
      1 + \| x \|^2
    \right)
  }
  \right)
  \right]
  < \infty .
\end{split}
\end{equation}
Combining~\eqref{eq:BrusselatorV1}--\eqref{eq:BrusselatorV3}
shows that
Corollary~\ref{cor:CT1} applies
here with the function
$
  V \colon \mathbb{R}^2
  \rightarrow [1,\infty)
$
given in \eqref{eq:BrusselatorV0}
and we therefore get
\begin{equation}
  \lim_{ N \rightarrow \infty }
    \sup_{ t \in [0,T] }
    \mathbb{E}\big[
      \| 
        X_t - \bar{Y}^N_t 
      \|^{ 
        q
      }
    \big] 
  = 0
\end{equation}
for all 
$
  q \in (0,\infty)
$.
Please also note that
other choices
than~\eqref{eq:BrusselatorV0} are possible
for the
Lyapunov-type
function 
$ 
  V \colon \R^2
  \rightarrow [1,\infty) 
$
and that the 
choice~\eqref{eq:BrusselatorV0} simply
ensures that the smoothness and
growth assumptions of
Corollary~\ref{cor:CT1} are met.

\section{Stochastic SIR model}

The SIR model from epidemiology for the total number
of susceptibles, infected and recovered individuals
has been introduced by Anderson \citationand\ May~\cite{AndersonMay1979}. 
The following stochastic version has been studied first
by Tornatore, Buccellato 
\citationand\ Vetro~\cite{TornatoreBuccellatoVetro2005}. 
Assume
that the setting in 
Section~\ref{sec:setting00}
is fulfilled,
let
$ 
  \alpha, \beta, \gamma, 
  \delta \in (0,\infty) 
$
be real numbers
and suppose that
$ d = 3 $,
$ m = 1 $,
$ 
  D = (0,\infty)^3 
$,
$
  \mu(x) = \sigma(x) = 0
$
for all $ x \in D^c $
and
\begin{equation}
  \mu\!\left( 
    \begin{array}{c}
      x_1
      \\ 
      x_2 
      \\ 
      x_3
    \end{array}
  \right)
=
  \left( 
    \begin{array}{c}
      - \alpha x_1 x_2
      - \delta x_1
      + \delta
      \\
      \alpha x_1 x_2
      - 
      ( \gamma + \delta ) x_2 
    \\
      \gamma x_2 - \delta x_3
    \end{array}
  \right) ,
\qquad
  \sigma\!\left( 
    \begin{array}{c}
      x_1
      \\ 
      x_2 
      \\ 
      x_3
    \end{array}
  \right)
=
  \left(
    \begin{array}{ccc}
      - \beta x_1 x_2
    \\
      \beta x_1 x_2
    \\
      0
    \end{array}
  \right)
\end{equation}
for all
$ 
  x = (x_1, x_2, x_3)
  \in D
$.
Under these assumptions,
the 
solution process
$ 
  ( 
    S_t, I_t, R_t 
  )
  :=
  ( 
    X^{(1)}_t, X^{(2)}_t, X^{(3)}_t 
  )
$, $t\in[0,\infty)$,
of the SDE~\eqref{eq:SDE00}
fulfills
\begin{equation}  
\label{eq:ex_SIR}
\begin{split}
  dS_t
&
  =
  \big[
    - \alpha S_t I_t - \delta S_t + \delta
  \big] 
  dt
  - \beta S_t I_t \, dW_t
\\
  d I_t
  &
  =
  \big[ 
    \alpha S_t I_t 
    - ( \gamma + \delta ) I_t
  \big]
  dt
  +
  \beta S_t I_t \, dW_t
\\
  d R_t
  &
  =
  \big[ 
    \gamma I_t - \delta R_t
  \big] dt
\end{split}     
\end{equation}
for $ t \in [0,\infty) $.
Clearly, both the drift coefficient
and the diffusion
coefficient grow superlinearly
in this example.
In addition, 
it is also 
obvious that
the drift coefficient
fails to satisfy the one-sided
linear growth condition~\eqref{eq:onesidedgrowth}.
We now construct
an appropriate
Lyapunov-type function
$ 
  V \colon \mathbb{R}^3
  \rightarrow [1,\infty) 
$
for this example
so that Corollary~\ref{cor:CT1}
can be applied.
The Lyapunov-type function
here is constructed similarly
as in in the case of 
the stochastic Brusselator in 
Section~\ref{sec:Brusselator} above.
More formally, let
$
  \phi \colon 
  \mathbb{R}
  \rightarrow [0,1]
$
and
$
  \psi \colon
  \mathbb{R}^2
  \rightarrow [0,1]
$
be two infinitely 
often differentiable
functions with
$
  \phi(x) = 0 
$
for all $ x \in (-\infty,0] $,
with
$
  \phi(x) = 1 
$
for all $ x \in [1,\infty) $
and with
\begin{equation}
  \psi(x_1,x_2) = 
  \phi(x_1) \cdot \phi(-x_2)
  +
  \phi(-x_1) \cdot \phi(x_2)
\end{equation}
for all 
$   
  x = (x_1, x_2)   
  \in \mathbb{R}^2 
$.
Then consider 
$ 
  V \colon \mathbb{R}^3
  \rightarrow [1,\infty) 
$
given by
\begin{equation}
\label{eq:V0}
  V( x_1, x_2, x_3 ) = 
  \frac{ 5 }{ 2 } 
  + 
  \left( x_1 + x_2 \right)^2
  + \left( x_3 \right)^2
  -
  2 \cdot x_1 \cdot x_2 \cdot
  \psi(x_1,x_2)
\end{equation}
for all
$ 
  x = (x_1, x_2,x_3) 
  \in \mathbb{R}^3 
$
(cf., e.g., Tornatore, Buccellato 
\citationand\ Vetro~\cite{TornatoreBuccellatoVetro2005})
and observe that
$ V \colon \mathbb{R}^3
\rightarrow [1,\infty) $
is infinitely often differentiable 
and satisfies
\begin{equation}
\label{eq:V1}
  \sup_{ x \in \mathbb{R}^3 }
  \left(
  \tfrac{
    \sum_{ i = 1 }^{ 3 }
    \left\|
      V^{ (i) }( x )
    \right\|_{ 
      L^{(i)}( \mathbb{R}^3, \mathbb{R} ) 
    }
  }{
    \left( 1 + \| x \| \right)
  }
  \right)
  < \infty .
\end{equation}
As in \eqref{eq:BrusselatorV2} it follows that
\begin{equation}
\label{eq:V2}
\begin{split}
&
  V(x)
\geq
  \tfrac{ 1 }{ 2 } 
  \left( 1 + \| x \|^2 \right)
\end{split}
\end{equation}
for all
$ 
  x 
  \in \mathbb{R}^3 
$.
In addition, 
observe that
$
  \left|
    V'(x) \sigma(x)
  \right| = 0
$
for all $ x \in D^c $
and
\begin{equation}
\begin{split}
&
  \left|
    V'(x) \sigma(x)
  \right|
\\ & =
  \left|
  -
  \big(
    \tfrac{ 
      \partial 
    }{
      \partial x_1
    }
    V
  \big)( x_1, x_2, x_3 )
  \cdot
  \beta \cdot x_1 \cdot x_2
  +
  \big(
    \tfrac{ 
      \partial 
    }{
      \partial x_2
    }
    V
  \big)( x_1, x_2, x_3 )
  \cdot
  \beta \cdot x_1 \cdot x_2
  \right|
\\ & =
  \beta
  \left| x_1 \right|
  \left| x_2 \right|
  \left|
  \big(
    \tfrac{ 
      \partial 
    }{
      \partial x_1
    }
    V
  \big)( x_1, x_2, x_3 )
  -
  \big(
    \tfrac{ 
      \partial 
    }{
      \partial x_2
    }
    V
  \big)( x_1, x_2, x_3 )
  \right|
\\ & =
  \beta
  \left| x_1 \right|
  \left| x_2 \right|
  \left|
    2 ( x_1 + x_2 )
    -
    2 ( x_1 + x_2 )
  \right|
  = 0
\end{split}
\end{equation}
for all
$ x = (x_1, x_2, x_3) \in D $
and we thus get
\begin{equation}
\label{eq:V3}
  \left| 
    V'(x) \sigma(x)
  \right| = 0
\end{equation}
for all 
$ 
  x \in \mathbb{R}^3 
$.
Combining~\eqref{eq:V1}--\eqref{eq:V3}
shows that
Corollary~\ref{cor:CT1} applies
here with the Lyapunov-type 
function
$
  V \colon \mathbb{R}^3
  \rightarrow [1,\infty)
$
given in \eqref{eq:V0}
and we therefore get that
\begin{equation}
  \lim_{ N \rightarrow \infty }
    \sup_{ t \in [0,T] }
    \mathbb{E}\big[
      \| 
        X_t - \bar{Y}^N_t 
      \|^{ 
        q
      }
    \big] 
  = 0
\end{equation}
for all 
$
  q \in (0,\infty)
$.
Note that
other choices
than~\eqref{eq:V0} are possible
for the
Lyapunov-type
function 
$ 
  V \colon \mathbb{R}^3
  \rightarrow [1,\infty) 
$
and that the 
choice~\eqref{eq:V0} simply
ensures that the smoothness and
growth assumptions of
Corollary~\ref{cor:CT1} are met.

\section{Experimental
psychology model}

The motivation for the following example from experimental psychology
is explained in 
Section~7.2 
in 
Kloeden
\citationand\ 
Platen~\cite{kp92}.
Assume
that the setting in 
Section~\ref{sec:setting00}
is fulfilled,
let
$ \alpha, \delta \in (0,\infty) $,
$ \beta \in \mathbb{R} $
be real numbers
and suppose that
$ d = 2 $,
$ m = 1 $,
$ D = \mathbb{R}^2 $
and 
\begin{equation}
  \mu\!\left( 
    \begin{array}{c}
      x_1
      \\ 
      x_2 
    \end{array}
  \right)
=
  \left( 
    \begin{array}{c}
      \left( x_2 \right)^2
      \left( \delta + 4 \alpha x_1 \right)
      - \frac{ \beta^2 x_1 }{ 2 }
      \\[1ex] 
      - x_1 x_2 
      \left( \delta + 4 \alpha x_1 \right)
      + \frac{ \beta^2 x_2 }{ 2 }
    \end{array}
  \right) ,
\qquad
  \sigma\!\left( 
    \begin{array}{c}
      x_1
      \\ 
      x_2 
    \end{array}
  \right)
=
  \left(
    \begin{array}{ccc}
      - \beta x_2
    \\
      \beta x_1
    \end{array}
  \right)
\end{equation}
for all
$ 
  x = (x_1, x_2)
  \in \mathbb{R}^2
$.
The SDE~\eqref{eq:SDE00}
is then the
Stratonovich stochastic
differential equation
\begin{equation}  
\label{eq:experimental.psychology}
\begin{split}
  dX_t^{(1)}
  &
  =
  \left[
    \big( X_t^{(2)} 
    \big)^2
    \big(
      \delta +
      4 \alpha
      X_t^{(1)}
    \big) 
  \right]
  dt
  - \beta X_t^{(2)} 
  \circ d W_t
\\
  dX_t^{(2)}
  &
  =
  \left[
  -
  X_t^{(1)}
  X_t^{(2)}
  \big(
    \delta +
    4 \alpha
    X_t^{(1)}
  \big)
  \right]
  dt
  + 
  \beta X_t^{ (1) }
  \circ d W_t
\end{split}     
\end{equation}
for $ t \in [0,\infty) $.
The SDE~\eqref{eq:experimental.psychology}
is a transformed version of a model
proposed in Haken, Kelso \citationand\ Bunz~\cite{Hakenetal1985} in the deterministic case and 
in Sch\"{o}ner, Haken \citationand\ Kelso~\cite{Schoeneretal1985} in the stochastic case
(see Section~7.2 in Kloeden
\& Platen~\cite{kp92} for details).
The diffusion coefficient
in \eqref{eq:experimental.psychology} 
is clearly globally Lipschitz continuous.
The drift coefficient
in \eqref{eq:experimental.psychology} 
is not globally one-sided
Lipschitz continuous but fulfills
the global one-sided linear growth
bound~\eqref{eq:onesidedgrowth}.
The drift coefficient
in \eqref{eq:experimental.psychology} 
even fulfills
$
  \left< x, \mu( x ) \right> = 0
$
for all $ x \in \R^2 $
and therefore the function
$
  V \colon \mathbb{R}^2
  \rightarrow [1,\infty)
$
given by
\begin{equation}
\label{eq:Lyapunov_experimentalPsy}
  V(x) = 
  1 
  + 
  \left\| x \right\|^2
\end{equation}
is a Lyapunov-type function 
for the SDE~\eqref{eq:experimental.psychology}
(see, e.g., Section~7.2 in Kloeden
\& Platen~\cite{kp92}).
Hence,
Corollary~\ref{cor:CT1} applies
here with 
$
  V \colon \mathbb{R}^2
  \rightarrow [1,\infty)
$
given by
\eqref{eq:Lyapunov_experimentalPsy}
and we therefore get
$
  \lim_{ N \rightarrow \infty }
    \sup_{ t \in [0,T] }
    \mathbb{E}\big[
      \| 
        X_t - \bar{Y}^N_t 
      \|^{ 
        q
      }
    \big] 
  = 0
$
for all 
$
  q \in (0,\infty)
$.

\section{Scalar stochastic Ginzburg-Landau equation}

The Ginzburg-Landau equation is from the theory of superconductivity 
and has been introduced by Ginzburg \citationand~Landau~\cite{gl50}
to describe a phase
transition.
As, for instance, in
(4.52) in Section~4.4 in 
Kloeden \citationand\ Platen~\cite{kp92}, 
we consider in this section a simplified scalar
version of the Ginzburg-Landau equation
pertubed by a multiplicative noise term.
More precisely, assume
that the setting in 
Section~\ref{sec:setting00}
is fulfilled,
let $ \alpha, \beta, \delta \in (0,\infty) $
be real numbers
and suppose that
$ d = m = 1 $,
$ D = \mathbb{R} $,
$
  \mu(x) = \alpha x - \delta x^3
$
and
$
  \sigma(x) = \beta x
$
for all
$ 
  x 
  \in \mathbb{R}
$.
Under these assumptions,
the SDE~\eqref{eq:SDE00}
reads as
\begin{equation}
\label{eq:ex_Ginzburg}
  d X_t = 
  \left[ 
    \alpha X_t - \delta X_t^3 
  \right] 
  dt 
  + 
  \beta X_t \, dW_t
\end{equation}
for $ t \in [0,\infty) $
(see, e.g., (4.52) in Section~4.4 in 
Kloeden \citationand\ Platen~\cite{kp92}).
Clearly, Corollary~\ref{cor:CT1} applies
here with 
$
  V \colon \mathbb{R}
  \rightarrow [1,\infty)
$
given by
\begin{equation}
  V(x) = 1 + x^2
\end{equation}
for all $ x \in \mathbb{R} $
to obtain
\begin{equation}
  \lim_{ N \rightarrow \infty }
    \sup_{ t \in [0,T] }
    \mathbb{E}\big[
      | 
        X_t - \bar{Y}^N_t 
      |^{ 
        q
      }
    \big] 
  = 0
\end{equation}
for all 
$
  q \in (0,\infty)
$.
Obviously,
this example has a globally one-sided Lipschitz continuous
drift coefficient and a globally Lipschitz continuous
diffusion coefficient. 
So, the convergence 
results, e.g., in~\cite{h96,hms02,Schurz2003,hjk10b,gw11} apply here
(see Chapter~\ref{sec:intro}
for more details).

\section{Stochastic
Lotka-Volterra equations}
\label{sec:lotka}

The Lotka-Volterra 
predator-prey model
(see Lotka~\cite{Lotka1920JAMCS} 
and Volterra~\cite{Volterra1926})
and the Lotka-Volterra model for competing species
have a quadratic drift term.
Here we study the following more 
general Lotka-Volterra model
as considered, for instance, in Section~7.1 
in Kloeden \citationand~Platen~\cite{kp92}
(see also Dobrinevski \citationand~Frey~\cite{fd10}).
Assume
that the setting in 
Section~\ref{sec:setting00}
is fulfilled,
let
$ 
  A, 
  c = (c_1, \dots, c_d) 
  \in \mathbb{R}^d
$,
$ 
  v = (v_1, \dots, v_d) 
  \in (0,\infty)^d 
$,
$
  B \in \mathbb{R}^{ d \times d }
$
and suppose that
$ d = m $,
$ 
  D = (0,\infty)^d
$,
$
  \mu(x) = \sigma(x) = 0
$
for all $ x \in D^c $,
that
\begin{equation}
\label{eq:VolterraCondition}
  \left< 
    x, 
    \left(
      \begin{array}{ccc}
        \! v_1 \! & & \! 0 \! \\[-1ex]
        \! & \ddots & \! \\[-1ex]
        \! 0 \! & & \! v_d \!
      \end{array}
    \right)
    B x 
  \right>
  \leq 0
\end{equation}
for all $ x \in \mathbb{R}^d $
and that
\begin{equation}
  \mu\!\left( 
    \begin{array}{c}
      x_1
      \\ 
      \vdots
      \\ 
      x_d
    \end{array}
  \right)
=
  \left(
    \begin{array}{ccc}
      x_1 & & 0 \\
      & \ddots & \\
      0 & & x_d 
    \end{array}
  \right)
  \left(
    A + B x
  \right) 
\end{equation}
and
\begin{equation}
  \sigma\!\left( 
    \begin{array}{c}
      x_1
      \\ 
      \vdots
      \\ 
      x_d
    \end{array}
  \right)
=
  \left(
    \begin{array}{ccc}
      c_1 x_1 & & 0 \\
      & \ddots & \\
      0 & & c_d x_d 
    \end{array}
  \right)
\end{equation}
for all
$ 
  x = (x_1, \dots, x_d)
  \in D
$.
The SDE~\eqref{eq:SDE00}
is then
the $ d $-dimensional 
stochastic Lotka-Volterra system 
\begin{equation}
\label{eq:ex_Volterra}
  dX_t
  =
  \left(
    \begin{array}{ccc}
      \! 
      X_t^{(1)} \! & & \! 0 \! \\[-1ex]
      & \ddots & \\[-1ex]
      \! 0 \! & & \! X_t^{(d)} \!
    \end{array}
  \right)
  \!
  \left(
    A + B X_t
  \right) 
  dt
  +
  \left(
    \begin{array}{ccc}
      \! c_1 
      X_t^{(1)} \! & & \! 0 \! \\[-1ex]
      & \ddots & \\[-1ex]
      \! 0 \! & & \! c_d X_t^{(d)} \!
    \end{array}
  \right)
  dW_t
\end{equation}
for $ t \in [0,\infty) $
(see, e.g., (1.6) in Section~7.1
in Kloeden \citationand\ Platen~\cite{kp92}).
The drift coefficient of
the SDE~\eqref{eq:ex_Volterra}
contains a quadratic term and
is therefore clearly not globally Lipschitz
continuous.
%
%
In order to apply 
Corollary~\ref{cor:CT1},
a Lyapunov-type function
$ 
  V \colon \mathbb{R}^d
  \rightarrow [1,\infty) 
$
needs to be constructed.
For this let
$
  \phi \colon 
  \mathbb{R}
  \rightarrow [0,1]
$
and
$
  \psi \colon
  \mathbb{R}^2
  \rightarrow [0,1]
$
be two infinitely 
often differentiable
functions with
$
  \phi(x) = 0 
$
for all $ x \in (-\infty,0] $,
with
$
  \phi(x) = 1 
$
for all $ x \in [1,\infty) $
and with
\begin{equation}
  \psi(x,y) = 
  \phi(x) \cdot \phi(-y)
  +
  \phi(-x) \cdot \phi(y)
\end{equation}
for all 
$   
  x, y   
  \in \mathbb{R} 
$.
Note that
$
  \psi(x,y) = \psi(y,x) 
$
for all
$ x, y \in \mathbb{R} $
and
$
  \psi(x,y) = 0
$
for all $ x, y \in \mathbb{R} $
with $ x \cdot y \geq 0 $.
Then consider 
$ 
  V \colon \mathbb{R}^d
  \rightarrow [1,\infty) 
$
given by
\begin{equation}
\label{eq:VolterraV0}
\begin{split}
&
  V( x ) 
=
  1
  +
  \tfrac{
    d^4 \left\| v \right\|^4
  }{
  \left|
    \min\left(
      v_1, \dots, v_d
    \right)
  \right|^2
  }
  + 
  \left| 
    \left< v, x \right> 
  \right|^2
  -
  \sum_{ 
    \substack{
      i, j \in \{ 1, 2, \dots, d \} 
    \\
      i \neq j
    }
  }
  v_i 
  \cdot
  v_j
  \cdot
  x_i 
  \cdot x_j
  \cdot
  \psi(x_i, x_j)
\\ & =
  1
  +
  \tfrac{
    d^4 \left\| v \right\|^4
  }{
  \left|
    \min\left(
      v_1, \dots, v_d
    \right)
  \right|^2
  }
  + 
  \sum_{ i = 1 }^{ d }
  \left| v_i x_i \right|^2
  +
  \sum_{ 
    \substack{
      i, j \in \{ 1, 2, \dots, d \} 
    \\
      i \neq j
    }
  }
  v_i 
  \cdot
  v_j
  \cdot
  x_i 
  \cdot x_j
  \cdot
  \left(
    1 -
    \psi(x_i, x_j)
  \right)
\end{split}
\end{equation}
for all
$ 
  x = (x_1, \dots, x_d) 
  \in \mathbb{R}^d
$
and observe that
$ V \colon \mathbb{R}^d
\rightarrow [1,\infty) $
is infinitely often differentiable 
and satisfies
\begin{equation}
\label{eq:VolterraV1}
  \sup_{ x \in \mathbb{R}^d }
  \left(
  \tfrac{
    \sum_{ i = 1 }^{ 3 }
    \left\|
      V^{ (i) }( x )
    \right\|_{ 
      L^{(i)}( 
        \mathbb{R}^d, 
        \mathbb{R} 
      ) 
    }
  }{
    \left( 1 + \| x \| \right)
  }
  \right)
  < \infty .
\end{equation}
In the next step note that
\begin{equation}
\begin{split}
&
  V( x )
\\ & \geq
  1
  +
  \tfrac{
    d^4 \left\| v \right\|^4
  }{
  \left|
    \min\left(
      v_1, \dots, v_d
    \right)
  \right|^2
  }
  + 
  \left|
  \min\!\left(
    v_1, \dots, v_d
  \right)
  \right|^2
  \left\| x \right\|^2
\\ & \quad
  -
  \left|
  \max\!\left( 
    v_1, \dots, v_d 
  \right)
  \right|^2
  \left(
    \sum_{ 
      i, j \in \{ 1, 2, \dots, d \} 
    }
    \max(|x_i|, |x_j|) 
  \right)
\\ & \geq
  1
  +
  \tfrac{
    d^4 \left\| v \right\|^4
  }{
  \left|
    \min\left(
      v_1, \dots, v_d
    \right)
  \right|^2
  }
  + 
  \left|
  \min\!\left(
    v_1, \dots, v_d
  \right)
  \right|^2
  \left\| x \right\|^2
  -
  \left|
  \max\!\left( 
    v_1, \dots, v_d 
  \right)
  \right|^2
  d^2
  \left\| x \right\|
\end{split}
\end{equation}
for all $ x \in \R^d $
and the estimate 
$
  a^2 + b^2 - a b 
  \geq \frac{ a^2 + b^2 }{ 2 }
$
for all $ a, b \in \R $
therefore proves that
\begin{equation}
\begin{split}
  V( x )
& \geq
  1
  +
  \tfrac{
    d^4 \left\| v \right\|^4
  }{
  \left|
    \min\left(
      v_1, \dots, v_d
    \right)
  \right|^2
  }
  + 
  \left|
    \min\!\left(
      v_1, \dots, v_d
    \right)
  \right|^2
  \left\| x \right\|^2
  -
  d^2
  \left\| v \right\|^2
  \left\| x \right\|
\\ & =
  1
  +
  \left|
    \min\left(
      v_1, \dots, v_d
    \right)
  \right|^2
  \left[
  \tfrac{
    d^4 \left\| v \right\|^4
  }{
  \left|
    \min\left(
      v_1, \dots, v_d
    \right)
  \right|^4
  }
  + 
  \left\| x \right\|^2
  -
  \tfrac{
    d^2
    \left\| v \right\|^2
  }{
    \left|
      \min\left(
        v_1, \dots, v_d
      \right)
    \right|^2
  }
  \left\| x \right\|
  \right]
\\ & \geq
  1
  +
  \tfrac{
  \left|
    \min\left(
      v_1, \dots, v_d
    \right)
  \right|^2
  }{ 2 }
  \left[
  \tfrac{
    d^4 \left\| v \right\|^4
  }{
  \left|
    \min\left(
      v_1, \dots, v_d
    \right)
  \right|^4
  }
  + 
  \left\| x \right\|^2
  \right]
\geq
  1 
  +
  \tfrac{
  \left|
    \min\left(
      v_1, \dots, v_d
    \right)
  \right|^2
  }{ 2 }
  \left\| x \right\|^2
\end{split}
\end{equation}
for all
$ 
  x \in \mathbb{R}^d
$.
This shows
\begin{equation}
\label{eq:Volterra_V_dominates_n2}
  \sup_{ x \in \mathbb{R}^d }
  \left(
  \frac{ \left\| x \right\|^2 }{
    V(x)
  }
  \right)
  < \infty 
  .
\end{equation}
In addition, 
observe that
\begin{equation}
\label{eq:VolterraV3}
  \sup_{ x \in \mathbb{R}^d }
  \left(
  \frac{
  \left\|
    V'(x) \sigma(x)
  \right\|_{
    L( 
      \mathbb{R}^d, \mathbb{R} 
    )
  } 
  }{
    \left(
      1 + \left\| x \right\|^2
    \right)
  }
  \right)
  < \infty .
\end{equation}
Next note that
assumption~\eqref{eq:VolterraCondition}
gives that
\begin{equation}
\label{eq:VolterraV4}
  \sup_{ x \in \mathbb{R}^d }
  \left(
    \frac{
      ( \mathcal{G}_{ 
        \mu, \sigma
      } V)(x)
    }{
      V(x)
    }
  \right)
  < \infty
\end{equation}
Combining~\eqref{eq:VolterraV1},
\eqref{eq:Volterra_V_dominates_n2},
\eqref{eq:VolterraV3}
and \eqref{eq:VolterraV4}
shows that
Corollary~\ref{cor:CT1} 
applies here with the function
$
  V \colon \mathbb{R}^d
  \rightarrow [1,\infty)
$
given in \eqref{eq:VolterraV0}
and we therefore get
\begin{equation}
  \lim_{ N \rightarrow \infty }
    \sup_{ t \in [0,T] }
    \mathbb{E}\big[
      \| 
        X_t - \bar{Y}^N_t 
      \|^{ 
        q
      }
    \big] 
  = 0
\end{equation}
for all 
$
  q \in (0,\infty)
$.

Finally, let us describe
two more specific examples
of the stochastic Lotka-Volterra
system~\eqref{eq:ex_Volterra}.

\subsection{Stochastic Verhulst equation}

In addition to the 
assumptions above,
let in this
subsection 
$ \eta, \lambda \in (0,\infty) $
be real numbers and
suppose 
that
$ d = m = 1 $,
$
  A = \eta + \frac{ c^2 }{ 2 }
$
and
$
  B = - \lambda
$.
The stochastic 
Lotka-Volterra
system~\eqref{eq:ex_Volterra}
thus
simplifies to
the one-dimensional 
Stratonovich SDE
\begin{equation}
\label{eq:Verhulst}
  dX_t
  =
  \Big[
    \eta \cdot
    X_t
    - \lambda \cdot 
    \left( X_t \right)^2
  \Big] \,
  dt +
  c \cdot X_t \circ dW_t
\end{equation}
for $ t \in [0,\infty) $.
Equation~\eqref{eq:Verhulst}
is referred to as
stochastic Verhulst 
equation in the literature
(see, e.g., 
Section 4.4 in Kloeden \citationand\ Platen~\cite{kp92}).
Note also that
\eqref{eq:VolterraCondition}
is fulfilled here with
$ v = 1 $, for instance.
Clearly, this example has a globally one-sided Lipschitz continuous
drift coefficient and a globally Lipschitz continuous
diffusion coefficient.
Hence, the convergence
results, e.g., in~\cite{h96,hms02,Schurz2003}
can be applied here (see Chapter~\ref{sec:intro}
for more details).

\subsection{Predator-prey 
model}

In addition to the 
assumptions above,
let in this
subsection 
$ 
  \alpha, \beta, \gamma,
  \delta \in (0,\infty)
$
be real numbers and
suppose that
$ d = m = 2 $,
$
  A = ( \alpha, - \delta )
$
and
\begin{equation}
  B = 
  \left(
    \begin{array}{cc}
      0 & - \beta
    \\
      \gamma & 0
    \end{array}
  \right) .
\end{equation}
The stochastic 
Lotka-Volterra
system~\eqref{eq:ex_Volterra}
is then
the two-dimensional 
SDE
\begin{equation}  
\label{eq:ex_PredatorPrey}
\begin{split}
  dX^{(1)}_t
& =
  X_t^{(1)} \,
  \big( 
    \alpha - \beta \cdot X^{(2)}_t 
  \big) \,
  dt
  +
  c_1 \cdot X_t^{(1)} \,
  dW_t^{(1)}
\\
  dX^{(2)}_t
  & =
  X^{(2)}_t \,
  \big( 
    \gamma \cdot X^{(1)}_t 
    -
    \delta 
  \big)
  \,
  dt
  +
  c_2 \cdot X_t^{(2)} \,
  dW_t^{ (2) }
\end{split}     
\end{equation}
for $ t \in [0,\infty) $.
Note that
\eqref{eq:VolterraCondition}
is fulfilled here with
$ v = ( \gamma, \beta ) $, 
for instance.
The deterministic case 
($ c_1 = c_2 = 0 $)
of this model
has been introduced 
by 
Lotka~\cite{Lotka1920JAMCS} 
and Volterra~\cite{Volterra1926}.

\section{Volatility 
processes}
\label{sec:vola}

There are a number of models in the literature on computational finance
which generalize the Black-Scholes model with a stochastic volatility process.
To unify some squared 
volatility processes
of these models, 
we consider the following SDE.
Assume that the setting
in Section~\ref{sec:setting00}
is fulfilled,
let 
$ 
  a \in [1,\infty) 
$,
$ 
  b \in [\frac{1}{2},\infty)
$,
$ 
  \alpha, \beta
  \in (0,\infty) 
$,
$ \gamma \in \R $,
$ \delta \in [0,\infty) $
be real numbers
with
\begin{equation}
\label{eq:assumptions_vola}
  a + 1 \geq 2b
\qquad
\text{and}
\qquad
  \delta \geq
  \mathbbm{1}_{ 
    \left\{ 
      \frac{ 1 }{ 2 } 
    \right\} 
  }(b)
  \cdot
  \frac{ \beta^2 }{ 2 }
\end{equation}
and suppose that
$ d = m = 1 $,
$ D = (0,\infty) $,
$ 
  \mu(x) = \delta 
$
and
$
  \sigma(x) = 0 
$
for all $ x \in (-\infty,0] $
and
\begin{equation}
  \mu(x) = \delta + \gamma \cdot x 
  - \alpha \cdot x^a
\qquad
\text{and}
\qquad
  \sigma(x) = \beta \cdot x^b
\end{equation}
for all $ x \in (0,\infty) $.
The SDE~\eqref{eq:SDE00}
then reads as
\begin{equation}
\label{eq:squaredvola}
  d X_t = 
  \big[
    \delta
    +
    \gamma X_t
    -  
    \alpha \,
    ( X_t )^{ a } 
  \big] \, dt
  +
  \beta \,
  ( X_t )^{ b } \,
  dW_t
\end{equation}
for $ t \in [0,\infty) $.
The assumption
$
  \delta \geq
  \mathbbm{1}_{ \{ 1 / 2 \} }(b)
  \cdot
  \frac{ \beta^2 }{ 2 }
$
in \eqref{eq:assumptions_vola}
ensures the existence
of an up to indistinguishability
unique strictly positive solution
of \eqref{eq:squaredvola}.
The drift
coefficient $ \mu \colon \R \to \R $ 
in this example
is globally one-sided Lipschitz
continuous. 
Indeed, note that
\begin{equation}
\begin{split}
&
  \left<
    x - y,
    \mu( x ) - \mu( y )
  \right>
\\ & =
  \begin{cases}
    \gamma \left( x - y \right)^2
    -
    \alpha
    \left(
      x^a - y^a
    \right)
    \left(
      x - y
    \right)
    \leq
    \gamma \left( x - y \right)^2
  &
    \colon
    x, y > 0
  \\
    \left< x - y, \mu( x ) - \delta \right>
    \leq
    \left( x - y \right)
    \gamma x 
    \leq
    \max( \gamma, 0 )
    \left( x - y \right)^2
  &
    \colon
    x > 0,
    y \leq 0
  \\
    \left< y - x, \mu( y ) - \delta \right>
    \leq
    \left( x - y \right)
    \gamma x 
    \leq
    \max( \gamma , 0 ) 
    \left( x - y \right)^2
  &
    \colon
    x \leq 0, y > 0
  \\
    \left< x - y, \delta - \delta \right>
    =
    0
    \leq
    \max( \gamma, 0 )
    \left( x - y \right)^2
  &
    \colon
    x, y \leq 0
  \end{cases}
\end{split}
\end{equation}
for all $ x, y \in \R $
and therefore
\begin{equation}
\label{eq:FINANCE_onesidedlipschitz}
  \left< x - y, \mu(x) - \mu(y)
  \right>
\leq
    \max( \gamma , 0 )
    \left( x - y \right)^2
\end{equation}
for all $ x, y \in \R $.
In the sequel, the application
of Corollary~\ref{cor:CT2} 
for the 
SDE~\eqref{eq:squaredvola}
is illustrated. For this define
$ p_0 \in (1,\infty] $
and $ q_0 \in (-\frac{1}{2},\infty] $
by
\begin{equation}
  p_0 :=
  \begin{cases}
    \infty & 
    \colon
    b \in [ \frac{ 1 }{ 2 } , 1 ] \cup [ \frac{ 1 }{ 2 } , \frac{ a + 1 }{ 2 } )
  \\
    \frac{ 2 \alpha + \beta^2 }{ \beta^2 }
  &
    \colon
    \text{otherwise}
  \end{cases}
\end{equation}
and
\begin{equation}
  q_0 :=
  \begin{cases}
    0
    & \colon
    p_0 \in (1,3)
  \\
    \frac{ 
      p_0
    }{ 
      4 \left\{ 
        b - 2 
        + \max( a, 3 / 2 )
      \right\}
    }
    - \frac{ 1 }{ 2 }
  &
    \colon
    p_0 \in [3,\infty)
  \\
    \infty & 
    \colon
    p_0 = \infty
  \end{cases} 
\end{equation}
and 
note that this definition
ensures 
$ p_0 \geq q_0 $. 
In addition, observe that
\begin{equation}
\begin{split}
&
  \left< x, \mu( x ) \right>
  +
  \tfrac{ ( p - 1 ) }{ 2 }
  \left\| \sigma(x) \right\|^2_{ HS( \R^m, \R^d ) }
=
  x \cdot \mu( x ) 
  +
  \tfrac{ ( p - 1 ) }{ 2 }
  \left| \sigma(x) \right|^2
\\ & =
    x 
    \left( 
      \delta + \gamma x - \alpha x^a
    \right)
    +
    \tfrac{ ( p - 1 ) \beta^2 x^{ 2 b } }{ 2 }
\\ & \leq
    \left(
      \delta + \max( \gamma, 0 )
    \right)
    \left(
      1 + x^2
    \right)
    -
    \alpha x^{ ( a + 1 ) }    
    +
    \left(
    \tfrac{ ( p - 1 ) \beta^2 }{ 2 }
    \right)
    x^{ 2 b }
\\ & =
    \left(
      \delta + \max( \gamma, 0 )
    \right)
    \left(
      1 + x^2
    \right)
    +
    x^{ 2 b }
    \left(
    \tfrac{ ( p - 1 ) \beta^2 }{ 2 }
    -
    \alpha \cdot x^{ ( a + 1 - 2 b ) }
    \right)
\end{split}
\label{eq:FINANCE_x_pos}
\end{equation}
for all $ x \in (0,\infty) $
and all $ p \in [0,\infty) $
and
\begin{equation}
\label{eq:FINANCE_x_neg}
  \left< x, \mu( x ) \right>
  +
  \tfrac{ ( p - 1 ) }{ 2 }
  \left\| \sigma(x) \right\|^2_{ HS( \R^m, \R^d ) }
  =
  x \cdot \delta
  \leq
  \delta \left( 1 + x^2 \right)
\end{equation}
for all
$
  x \in ( - \infty, 0 ]
$
and all $ p \in [0,\infty) $.
Combining \eqref{eq:FINANCE_x_pos} 
and \eqref{eq:FINANCE_x_neg}
then results in
\begin{equation}
  \sup_{ x \in \R }
  \left(
  \tfrac{
  \left< x, \mu( x ) \right>
  +
  \frac{ ( p - 1 ) }{ 2 }
  \left\| \sigma(x) \right\|^2_{ HS( \R^m, \R^d ) }
  }{
    \left( 
      1 + \left\| x \right\|^2 
    \right)
  }
  \right)
  < \infty
\label{eq:FINANCE_generator_estimate}
\end{equation}
for all $ p \in [0,p_0] \cap [0,\infty) $.
Next note 
for every $ p \in [0,\infty) $
that
\begin{equation}
\label{eq:vola_moments}
  \mathbb{E}\big[
    | X_t |^p
  \big]
  < \infty
\end{equation}
for all
$ t \in [0,\infty) $ 
if and only if $ p \leq p_0 $.
Furthermore,
observe that
if $ p_0 \geq 3 $, then
\begin{equation}
\begin{split}
&
  \sup_{
    p \in 
    [ 3 , \infty ) \cap 
    [0,p_0]
  }
  \left(
    \frac{ p }{
      2 \max( 0, 2 (b-1) ) +
      4 \max\!\left( a - 1, \max( 0, 2 ( b - 1 ) ), \frac{ 1 }{ 2 } 
      \right)
    }
    -
    \frac{ 1 }{ 2 }
  \right)
\\ & =
  \sup_{
    p \in 
    [ 3 , \infty ) \cap 
    [0,p_0]
  }
  \left(
  \frac{ p }{
    4 \max( 0, b - 1 ) +
    4 \max\!\left( a - 1, 2 b - 2, \frac{ 1 }{ 2 } 
    \right)
  }
  -
  \frac{ 1 }{ 2 }
  \right)
\\ & =
  \sup_{
    p \in 
    [ 3 , \infty ) \cap 
    [0,p_0]
  }
  \left(
  \frac{ p }{
    4
    \left\{ 
      \max( 0, b - 1 ) +
      \max\!\left( a - 1, \frac{ 1 }{ 2 } 
      \right)
    \right\}
  }
  -
  \frac{ 1 }{ 2 }
  \right)
\\ & =
  \sup_{
    p \in 
    [ 3 , \infty ) \cap 
    [0,p_0]
  }
  \left(
  \frac{ 
    p
  }{
    4
    \left\{ 
      \max( -1, b - 2 ) +
      \max\!\left( a, 3 / 2 
      \right)
    \right\}
  }
  -
  \frac{ 1 }{ 2 }
  \right)
  = q_0 .
\end{split}
\label{eq:FINANCE_p0q0_identity}
\end{equation}
Combining 
\eqref{eq:FINANCE_generator_estimate}
and
\eqref{eq:FINANCE_p0q0_identity}
with Corollary~\ref{cor:CT2}
then implies
\begin{equation}
  \lim_{ N \rightarrow \infty }
    \sup_{ t \in [0,T] }
    \mathbb{E}\big[
      \| 
        X_t - \bar{Y}^N_t 
      \|^{ 
        q
      }
    \big] 
  = 0
\end{equation}
for all 
$
  q \in (0,\infty)
$
with $ q < q_0 $.
Let us illustrate this by
three more specific examples.

\subsection{Cox-Ingersoll-Ross 
process}

In addition to the 
assumptions above,
suppose that
$ a = 1 $, 
$ b = \frac{ 1 }{ 2 } $
and $ \gamma = 0 $.
The SDE~\eqref{eq:squaredvola}
is then the 
Cox-Ingersoll-Ross process 
\begin{equation}
\label{eq:ex_CIR}
\begin{split}
  d 
  X_t
& =
  \big[
  \delta - 
  \alpha X_t
  \big] \, 
  dt
  +
  \beta
  \sqrt{ X_t } 
  \, dW_t
\end{split}
\end{equation}
for $ t \in [0,\infty) $
which has been introduced 
in Cox, Ingersoll \citationand\ 
Ross~\cite{cir85} as 
model for instantaneous 
interest rates.
Later
Heston~\cite{h93}
proposed this process 
as a model for the squared 
volatility in a Black-Scholes type 
market model.
Here we have $ p_0 = q_0 = \infty $
and hence get
$
  \lim_{ N \rightarrow \infty }
    \sup_{ t \in [0,T] }
    \mathbb{E}\big[
      \| 
        X_t - \bar{Y}^N_t 
      \|^{ 
        q
      }
    \big] 
  = 0
$
for all 
$
  q \in (0,\infty)
$.
Both the drift and the diffusion
coefficient clearly grow at most linearly.
Therefore, strong convergence 
of the 
Euler-Maruyama approximations
is well-known
(see, e.g., Krylov~\cite{Krylov1990}
and Gy\"ongy~\cite{g98b} for
convergence in probability and
pathwise convergence
respectively).
Strong convergence rates of
a drift-implicit
Euler method
for the SDE~\eqref{eq:ex_CIR}
are established in 
Theorem~1.1 
in Dereich, Neuenkirch
\citationand~Szpruch~\cite{DereichNeuenkirchSzpruch2012}.

\subsection{Simplified
Ait-Sahalia interest rate model}

In addition to the 
assumptions above,
suppose that
$ a = 2 $ and 
$ b < \frac{ 3 }{ 2 } $.
Under these additional
assumptions, the 
SDE~\eqref{eq:squaredvola}
reads as
\begin{equation}
\label{eq:ex_AitSahalia}
  d X_t
=
  \big[
    \delta + \gamma \, X_t
    - \alpha \left( X_t \right)^2 
  \big] \, dt
+
  \beta 
  \left( X_t \right)^b 
  dW_t
\end{equation}
for $ t \in [0,\infty) $.
A more general version hereof has been used
in Ait-Sahalia~\cite{AitSahalia1996}
for testing continuous-time
models of the spot interest rate.
Here we also 
have $ p_0 = q_0 = \infty $
and therefore
\begin{equation}
  \lim_{ N \rightarrow \infty }
    \sup_{ t \in [0,T] }
    \mathbb{E}\big[
      \| 
        X_t - \bar{Y}^N_t 
      \|^{ 
        q
      }
    \big] 
  = 0
\end{equation}
for all 
$
  q \in (0,\infty)
$.
A strong convergence result for this
SDE is Theorem 6.2 in Szpruch et al.~\cite{hmps11}.

\subsection{Volatility
process in the Lewis stochastic 
volatility model}
\label{sec:ex_Lewis32}

In addition to the 
assumptions above,
suppose that
$ a = 2 $,
$ b = \frac{ 3 }{ 2 } $,
$ \gamma \in [ 0, \infty ) $
and
$ \delta = 0 $.
The 
SDE~\eqref{eq:squaredvola}
is then the 
instantaneous 
variance process
in the Lewis stochastic
volatility model 
(see Lewis~\cite{l00}) 
\begin{equation}
\label{eq:lewis}
  d X_t
=
  \alpha \,
  X_t
  \left(
    \tfrac{ \gamma }{ \alpha }
    - X_t 
  \right) 
  dt
+
  \beta 
  \left( X_t
  \right)^{ \frac{ 3 }{ 2 } } 
  dW_t
\end{equation}
for $ t \in [0,\infty) $.
Here we get
\begin{equation}
  p_0
  =
  \tfrac{ 2 \alpha + \beta^2 }{ \beta^2 }
  \in (1,\infty)
\end{equation}
and
\begin{equation}
  q_0
=
  \max\!\big(
    \tfrac{ 
      p_0
    }{ 
      6
    }
    - \tfrac{ 1 }{ 2 } , 0
  \big)
  =
  \max\!\big(
  \tfrac{ 
    \alpha - \beta^2
  }{ 
    3 \beta^2 
  } , 0
  \big)
  \in [0,\infty)
  .
\end{equation}
In the case
$ 
  \alpha > \beta^2
$ 
we thus get
\begin{equation}
  \lim_{ N \rightarrow \infty }
    \sup_{ t \in [0,T] }
    \mathbb{E}\big[
      \| 
        X_t - \bar{Y}^N_t 
      \|^{ 
        q
      }
    \big] 
  = 0
\end{equation}
for all 
$
  q \in (0,q_0)
$.
The stochastic volatility 
model associated
to \eqref{eq:lewis}
is also known as $ 3/2 $-stochastic 
volatility model (see also
\cite{h07c,h11}).
Furthermore, we note that
Theorem~4.4 
in 
Mao \citationand\ 
Szpruch~\cite{MaoSzpruch2012pre} 
proves strong $ L^2 $-convergence of 
drift-implicit Euler
methods for the 
SDE~\eqref{eq:lewis} in the case 
$ 2 \alpha \geq \beta^2 $.
More formally,
\eqref{eq:FINANCE_onesidedlipschitz}
ensures that there exist
unique stochastic processes
$ 
  Z^N \colon \N_0
  \times \Omega
  \to \R
$,
$ 
  N \in 
  \N \cap ( \max\!\left( 0, \gamma \right) T, \infty ) 
$, satisfying
$ Z^N_0 = X_0 $
and
\begin{equation}
\label{eq:def_scheme_32Vola}
\begin{split}
  Z^N_{ n + 1 } 
& = 
  Z^N_n 
  + 
  \mu( Z^N_{ n + 1 } )
  \frac{ T }{ N }
  + 
  \sigma( Z^N_n )
  \left(
    W_{ \frac{ ( n + 1 ) T }{ N } }
    -
    W_{ \frac{ n T }{ N } }
  \right)
\\ & = 
  Z^N_n 
  + 
  \mathbbm{1}_{
    \left\{
      Z^N_{ n + 1 }
      > 0
    \right\}
  }
  \left(
    \gamma 
    Z^N_{ n + 1 }
    -
    \alpha
    \left( 
      Z^N_{ n + 1 } 
    \right)^2
  \right)
  \!
  \frac{ T }{ N }
  + 
  \sigma( Z^N_n ) 
  \left(
    W_{ \frac{ ( n + 1 ) T }{ N } }
    -
    W_{ \frac{ n T }{ N } }
  \right)
\end{split}
\end{equation}
for all
$ n \in \N_0 $
and all
$
  N \in
  \N \cap ( \max\!\left( 0, \gamma \right) T, \infty ) 
$
and 
Theorem~4.4
in 
Mao \citationand\ 
Szpruch~\cite{MaoSzpruch2012pre}
then,
in particular, proves 
that
in the case $ 2 \alpha \geq \beta^2 $
it holds that
\begin{equation}
\label{eq:vola_implicit1}
  \lim_{ N \to \infty }
  \E\big[
    \| X_T - \bar{Z}_T^N \|^p
  \big]
  = 0
\end{equation}
for all $ p \in (0,2) $
where
$ 
  \bar{Z}^N \colon
  [0,T] \times \Omega
  \to \R
$,
$ N \in \N $,
are defined through
\begin{equation}
  \bar{Z}^N_t
:=
  \big(
    n + 1 - \tfrac{tN}{T}
  \big)
  Z_n^N
  +
  \big(
    \tfrac{tN}{T}-n
  \big)
  Z_{ n + 1 }^N
\end{equation}
for all
$
  t \in 
  [
    n T / N,
    (n + 1) T / N 
  ]
$,
$ 
  n \in \{ 0, 1, \dots, N - 1 \} 
$
and all
$ N \in \mathbb{N} $.
Next 
we observe that
Corollary~\ref{c:Lyapunov.implicit.Euler}
and 
Lemma~\ref{l:Lyapunov.linear.implicit.Euler}
can be applied here to
prove moment bounds
and strong convergence
of implicit numerical
approximations methods 
for the 
SDE~\eqref{eq:lewis}.
In particular,
Corollary~\ref{c:Lyapunov.implicit.Euler} 
implies that
\begin{equation}
  \sup_{ N \in \N }
  \sup_{ n \in \{ 0, 1, \dots, N \} }
  \E\big[
    \| Z^N_n \|^p
  \big]
  < \infty
\end{equation}
for all $ p \in [0,p_0] $
in the case $ 2 \alpha \geq \beta^2 $.
This,
Corollary~\ref{cor:convergence2} 
and Lemma~\ref{lem:full_implicit}
then show 
in the case $ 2 \alpha \geq \beta^2 $
that
\begin{equation}
\label{eq:vola_implicit}
  \lim_{ N \to \infty }
  \sup_{ t \in [0,T] }
  \E\big[
    \| X_t - \bar{Z}_t^N \|^p
  \big]
  = 0
\end{equation}
for all $ p \in (0,p_0) $.
Equation~\eqref{eq:vola_implicit}
improves 
Theorem~4.4
in 
Mao \citationand\ 
Szpruch~\cite{MaoSzpruch2012pre}
in the case $ 2 \alpha > \beta^2 $.

%
%
%
%
%
%
%

\section{Langevin equation}
\label{sec:langevin}

A commonly used 
model 
for the motion
of molecules in a potential
is the Langevin equation (see, e.g.,
Subsection~2.1 
in Beskos \citationand~Stuart~\cite{bs09}).
Corollary~\ref{cor:CT1}
does not apply to completely 
arbitrary Langevin equations 
but requires
the following assumptions on the
potential.
Assume that the setting
in Section~\ref{sec:setting00}
is fulfilled, 
let 
$ \varepsilon, c \in (0,\infty) $
be real numbers,
let 
$
  U
  \colon\mathbbm{R}^d\to\mathbbm{R}
$
be a continuously differentiable function
with a locally Lipschitz continuous derivative
and suppose that
$ D = \mathbb{R}^d $,
$ d = m $
and
\begin{equation}
  \mu(x) = - (\triangledown U)(x)
\qquad
\text{and}
\qquad
  \sigma(x) = 
  \sqrt{ 2 \varepsilon } I 
\end{equation}
for all $ x \in \mathbb{R}^d $.
Moreover, let
$ 
  V \colon \mathbb{R}^d
  \rightarrow [1,\infty)
$
be a twice differentiable function 
with a locally Lipschitz continuous
second derivative,
with
\begin{equation}
  \limsup_{ 
    q \searrow 0
  }
  \sup_{ 
    x \in \mathbb{R}^d
  }
  \left(
  \frac{ 
    \| x \|^{ q } 
  }{
    V(x)
  }
  \right)
  < \infty
\end{equation}
and with 
\begin{equation}
  \langle \triangledown 
    U(x),\triangledown V(x)
  \rangle
  \geq - c \cdot V(x)
\quad
\text{and}
\quad
  \sum_{ i = 1 }^{ 3 }
  \|
    V^{(i)}( x )
  \|_{
    L^{ (i) }( 
      \mathbb{R}^d, \mathbb{R} 
    )
  }
  \leq
  c \,
  |
    V(x)
  |^{ 
    [ 1 - 1 / c ]
  }
\end{equation}
for 
$ \lambda_{ \mathbb{R}^d } $-almost
all 
$ x \in \mathbb{R}^d $.
Under these assumptions,
the SDE~\eqref{eq:SDE00}
reads as
\begin{equation}
\label{eq:ex_Langevin}
  dX_t
  =
  - ( \triangledown U )
  ( X_t ) \, dt
  +
  \sqrt{ 2 \varepsilon }
  \, dW_t
\end{equation}
for $ t \in [0,\infty) $.
Corollary~\ref{cor:CT1} then
implies
\begin{equation}
  \lim_{ N \rightarrow \infty }
    \sup_{ t \in [0,T] }
    \mathbb{E}\big[
      \| 
        X_t - \bar{Y}^N_t 
      \|^{ 
        q
      }
    \big] 
  = 0
\end{equation}
for all 
$
  q \in (0,\infty)
$.
If the force
$ ( - \triangledown U) $
in the Langevin equation
is
globally one-sided Lipschitz
continuous and satisfies suitable growth and regularity conditions, then the 
strong convergence
results, e.g., in~\cite{h96,hms02,Schurz2003,hjk10b,gw11} apply here
(see Chapter~\ref{sec:intro}
for more details).

\subsection*{Acknowledgement}

We gratefully acknowledge Lukas Szpruch
for several useful comments to an earlier
preprint version of this manuscript.

%
%
%

\backmatter

\def\cprime{$'$} \def\polhk\#1{\setbox0=\hbox{\#1}{{\o}oalign{\hidewidth
  \lower1.5ex\hbox{`}\hidewidth\crcr\unhbox0}}}
  \def\lfhook#1{\setbox0=\hbox{#1}{\ooalign{\hidewidth
  \lower1.5ex\hbox{'}\hidewidth\crcr\unhbox0}}} \def\cprime{$'$}

\end{document}